\documentclass[10pt,reqno]{amsart}
\usepackage[foot]{amsaddr}
\usepackage[margin=1in]{geometry}
\usepackage{amsmath, amsthm, amssymb, hyperref, enumerate, longtable}
\usepackage{graphicx}
\usepackage{xparse}
\usepackage{amssymb}
\usepackage{amsfonts}
\usepackage{bm}

\hypersetup{
    colorlinks,
    citecolor=blue,
    filecolor=black,
    linkcolor=black,
    urlcolor=black
}

\newenvironment{psmallmatrix}
  {\left(\begin{smallmatrix}}
  {\end{smallmatrix}\right)}

\allowdisplaybreaks

\def\R{\mathbb{R}}
\def\C{\mathbb{C}}
\def\eps{\varepsilon}
\def\G{\mathsf{G}}
\def\O{\mathsf{O}}
\def\SO{\mathsf{SO}}
\def\N{\mathcal{N}}
\def\orbit{\mathcal{O}}
\def\cA{\mathcal{A}}
\def\cI{\mathcal{I}}
\def\cJ{\mathcal{J}}
\def\cK{\mathcal{K}}
\def\cL{\mathcal{L}}
\def\cV{\mathcal{V}}
\def\cR{\mathcal{R}}
\def\cS{\mathcal{S}}
\def\sS{\mathcal{S}}
\def\der{\mathrm{d}}

\def\1{\mathbf{1}}
\def\E{\mathbb{E}}
\def\P{\mathbb{P}}
\def\HS{{\mathrm{HS}}}
\def\td{\tilde{d}}

\def\ts{\tilde{s}}
\def\tu{\tilde{u}}
\def\tz{\tilde{z}}
\def\ttheta{\tilde{\theta}}
\def\tvarphi{\tilde{\varphi}}
\def\tK{{\widetilde{K}}}
\def\tM{{\widetilde{M}}}
\def\tV{{\widetilde{V}}}
\def\tT{{\widetilde{T}}}
\def\tH{{\widetilde{H}}}
\def\tcI{\mathcal{\widetilde{I}}}
\def\tcR{\mathcal{\widetilde{R}}}
\def\tcV{\mathcal{\widetilde{V}}}
\def\frakg{\mathfrak{g}}
\def\frakh{\mathfrak{h}}
\DeclareMathOperator{\Id}{Id}
\DeclareMathOperator{\Unif}{Unif}
\DeclareMathOperator{\Var}{Var}
\DeclareMathOperator{\Cov}{Cov}
\DeclareMathOperator{\vectorize}{vec}
\DeclareMathOperator{\diag}{diag}
\DeclareMathOperator{\rank}{rank}
\DeclareMathOperator{\trdeg}{trdeg}

\DeclareMathOperator{\Tr}{Tr}
\DeclareMathOperator{\Type}{Type}
\DeclareMathOperator*{\argmin}{arg\,min}
\renewcommand{\i}{\mathbf{i}}
\renewcommand{\Re}{\operatorname{Re}}
\renewcommand{\Im}{\operatorname{Im}}
\newcommand{\hV}{\hat{V}}
\newcommand{\cQ}{\mathcal{Q}}
\newcommand{\cH}{\mathcal{H}}

\newcommand{\inner}[2]{\left\langle #1,#2 \right\rangle}
\newcommand{\rbr}[1]{\left(#1\right)}
\newcommand{\sbr}[1]{\left[#1\right]}

\newcommand{\abr}[1]{\left|#1\right|}

\makeatletter
\ifcase \@ptsize \relax% 10pt
  \newcommand{\miniscule}{\@setfontsize\miniscule{4}{5}}% \tiny: 5/6
\or% 11pt
  \newcommand{\miniscule}{\@setfontsize\miniscule{5}{6}}% \tiny: 6/7
\or% 12pt
  \newcommand{\miniscule}{\@setfontsize\miniscule{5}{6}}% \tiny: 6/7
\fi
\makeatother

\usepackage[dvipsnames]{xcolor}

\newtheorem{theorem}{Theorem}[section]
\newtheorem{lemma}[theorem]{Lemma}
\newtheorem{proposition}[theorem]{Proposition}
\newtheorem{fact}[theorem]{Fact}
\newtheorem{corollary}[theorem]{Corollary}
\theoremstyle{definition}
\newtheorem{remark}[theorem]{Remark}
\newtheorem*{remark*}{Remark}
\newtheorem*{example*}{Example}
\newtheorem{example}[theorem]{Example}
\newtheorem{definition}[theorem]{Definition}

\numberwithin{equation}{section}
\numberwithin{figure}{section}

\title[Group orbit estimation]{Maximum likelihood for high-noise
group orbit estimation and single-particle cryo-EM}
\author{Zhou Fan}\email{zhou.fan@yale.edu}
\author{Roy R. Lederman}\email{roy.lederman@yale.edu}
\author{Yi Sun}\email{yisun@statistics.uchicago.edu}
\author{Tianhao Wang}\email{tianhao.wang@yale.edu}
\author{Sheng Xu}\email{sheng.xu@yale.edu}
\address{ZF, RL, TW, SX: Department of Statistics and Data Science, Yale
University}
\address{YS: Department of Statistics, University of Chicago}
\date{\today}

\begin{document}

\begin{abstract}
Motivated by applications to single-particle cryo-electron microscopy (cryo-EM),
we study several problems of function estimation in a high noise
regime, where
samples are observed after random rotation and possible linear projection
  of the function domain. We describe a stratification of the Fisher information eigenvalues
according to transcendence degrees of graded pieces of the algebra of group invariants,
and we relate critical points of the log-likelihood landscape to a sequence
of moment optimization problems, extending previous results
for a discrete rotation group without projections.

We then compute the transcendence degrees and forms of these 
optimization problems for several examples of function
estimation under $\SO(2)$ and $\SO(3)$ rotations,
including a simplified model of cryo-EM as
introduced by Bandeira, Blum-Smith, Kileel, Perry, Weed, and Wein.
We affirmatively resolve conjectures that $3^\text{rd}$-order
moments are sufficient to locally identify a generic signal up to its
rotational orbit in these examples.

For low-dimensional approximations of the electric potential maps
of two small protein molecules, we empirically verify that the noise-scalings of the Fisher 
information eigenvalues conform with our theoretical predictions over a
range of SNR, in a model of $\SO(3)$ rotations without projections.
\end{abstract}

\maketitle

\tableofcontents

\section{Introduction}

We study several problems of function estimation in low dimensions, where the
function is observed under random and unknown rotations of its domain.
Let ${f:\mathcal{X} \to \R}$ be a function on the unit circle
$\mathcal{X}=\cS^1$, the unit sphere $\mathcal{X}=\cS^2$, or $\mathcal{X}=\R^3$.
Let $\G$ be a rotation group acting on $\mathcal{X}$.
We seek to estimate $f$ from samples of the form
\[f_\frakg(x)+\text{white noise}\]
where each sample consists of the function $f_\frakg(x)=f(\frakg^{-1} \cdot x)$
rotated by a uniformly random element $\frakg \in \G$ and observed with continuous
Gaussian white noise on $\mathcal{X}$. Equivalently, choosing an orthonormal
basis for $L_2(\mathcal{X})$, the basis coefficients of $f_{\frakg}$ are
observed with i.i.d.\ Gaussian noise, having some entrywise noise variance
$\sigma^2>0$. We focus on a regime of this problem
where each sample has high noise
$\sigma^2 \gtrsim \|f\|_{L_2}^2$, and the information from many rotated
samples must be combined to obtain an accurate estimate of $f$.
We study also a variant of this model where samples
are observed under an additional linear projection.

Our primary motivation is a formulation of this problem that models molecular
reconstruction in single-particle cryo-electron microscopy (cryo-EM)
\cite{dubochet1988cryo,henderson1990model,frank2006three}.
In this application, $f:\R^3 \to \R$ is the electric
potential of an unknown molecular structure. Tomographic projections of this
potential are measured for many samples of the molecule, each in a different
and unknown rotated orientation, typically with a high level of
measurement noise. The molecular structure is determined by estimating this
electric potential $f$ from the rotated and projected samples, and then fitting
an atomic model
\cite{bendory2020single,singer2020computational}. 
A brief introduction to cryo-EM and a discussion of its relation to the
problems studied in this work are presented in Appendix~\ref{sec:cryoem}.

Among computational
procedures for solving this reconstruction problem, regularized versions of
maximum likelihood estimation (MLE), as implemented via
expectation-maximization or stochastic gradient descent, are commonly used
\cite{sigworth1998maximum,scheres2007disentangling,scheres2012relion,punjani2017cryosparc}.
However, many theoretical properties of the optimization landscape and
reconstruction errors of these procedures are not fully
understood in cryo-EM applications \cite{bendory2020single}.

In this work, we study the Fisher information matrix and log-likelihood
function landscape associated with maximum likelihood estimation for
a basic model of the cryo-EM reconstruction problem, as well as several
simpler statistical models with qualitative similarities. These models may be of
independent interest while building up to the complexity of cryo-EM:
\begin{itemize}
\item (Continuous multi-reference alignment, Section \ref{sec:func-est-so2}.)
Estimating a function on the unit
circle $\mathcal{X}=\sS^1$, under $\SO(2)$ rotations of the circle
\cite{bandeira2020optimal,bendory2020super}.
\item (Spherical registration, Section \ref{subsec:sphericalregistration}.)
Estimating a function on the unit sphere
$\mathcal{X}=\sS^2$, under $\SO(3)$ rotations of the sphere \cite[Section
5.4]{bandeira2017estimation}.
\item (Unprojected cryo-EM, Section \ref{subsec:cryoEM}.)
Estimating a function on $\R^3$ under $\SO(3)$
rotations about the origin, without tomographic projection
\cite[Appendix B]{bandeira2017estimation}. Such a problem arises in a
related application of cryo-ET, discussed in Appendix \ref{sec:cryoET}.
\item (Cryo-EM, Section \ref{subsec:projectedcryoEM}.)
Estimating a function on $\R^3$ under $\SO(3)$
rotations about the origin, with tomographic projection
\cite[Section 5.5]{bandeira2017estimation}.
\end{itemize}

\subsection{Group orbit recovery and related literature}

Classical literature on function estimation has explored the rich interplay
between the complexity of infinite-dimensional function classes, the statistical
difficulty of estimation, and the role of regularization
\cite{ibragimov1981statistical,tsybakov2008introduction,johnstone2017gaussian}.
We restrict attention instead to a finite-dimensional function space for each
of the above models, with the goal of understanding connections between
estimation in these models having latent rotations and the algebraic
structure of the underlying rotational group.

Choosing a $d$-dimensional function basis, each of the
above function estimation problems may be restated as an \emph{orbit recovery
problem} \cite{abbe2018estimation,bandeira2017estimation}
of estimating the coefficients $\theta_* \in \R^d$ of $f$ in this basis,
from noisy observations of $\theta_*$ that are rotated by elements of a
subgroup $\G \subset \O(d)$. This subgroup $\G$ represents the transformation
of basis coefficients under rotations of the function domain $\mathcal{X}$.
A body of recent literature has
studied both specific and general instances of this orbit recovery problem
\cite{perry2019sample,abbe2019multireference,bandeira2020optimal,bandeira2017estimation,abbe2018estimation,pumir2019generalized,brunel2019learning,fan2020likelihood,sharon2020method,romanov2021multi}.
When $\G$ is the group of
cyclic rotations of coordinates (a.k.a.\ discrete multi-reference alignment),
\cite{perry2019sample} first proved that the optimal squared error for
estimating generic signals $\theta_* \in \R^d$ in high noise
is significantly larger than
that in the model without latent rotations, scaling as $\sigma^6$ rather than
as just the noise variance $\sigma^2$. This analysis was extended to non-generic signals for
continuous multi-reference alignment in \cite{bandeira2020optimal} and to
general group actions in \cite{abbe2018estimation,bandeira2017estimation}.

Our current work is inspired, in particular, by results of
Bandeira et al.\ \cite{bandeira2017estimation},
which placed cryo-EM and other examples of function estimation in this context,
and connected statistical properties of method-of-moments estimators in these
problems to properties of the invariant algebra of the group action. 
Further connections between the algebraic structure of invariants and the
geometry of the log-likelihood function landscape were developed in
\cite{fan2020likelihood,katsevich2020likelihood}. As a central technical
ingredient, these papers derived a
series expansion of the log-likelihood function in powers of $\sigma^{-1}$,
in \cite{fan2020likelihood} for orbit recovery models without
linear projection, and in \cite{katsevich2020likelihood} for more
general Gaussian mixture models that include the models we study in this work.
We discuss below several relevant results of
\cite{bandeira2017estimation,fan2020likelihood,katsevich2020likelihood} 
in further detail.

\subsection{Overview of results}

In Section \ref{sec:general}, we introduce the general orbit recovery
model both with and without a linear projection, and describe
results that connect geometric properties of the log-likelihood function
to properties of the invariant algebra of the group action.
In Sections \ref{sec:func-est-so2} and \ref{sec:func-est-so3}, we
apply this connection to study the preceding problems of function estimation,
including continuous multi-reference alignment (MRA) and cryo-EM. In
Section \ref{sec:simulations}, we report results of numerical
simulations for estimating the electric potential functions of two small
protein molecules in an unprojected cryo-EM model, which corroborate
predictions of our theory for the spectra of the Fisher information matrices.

Here, we provide a brief overview of these results in the context of related
literature.\\

{\noindent}{\bf Fisher information and log-likelihood function landscape.}

For general orbit recovery problems with $\theta_* \in \R^d$,
in a high-noise regime $\sigma^2 \gtrsim \|\theta_*\|^2$, results of
\cite{fan2020likelihood,katsevich2020likelihood} demonstrated that it is
informative to study properties of the (negative)
population log-likelihood function $R(\theta)$ via a series expansion in powers
of $\sigma^{-1}$ of the form
\begin{equation}\label{eq:seriesinformal}
R(\theta)=\sum_{k=0}^\infty \frac{1}{\sigma^{2k}} R_k(\theta).
\end{equation}
Each term $R_k(\theta)$ is a $\G$-invariant polynomial function of $\theta$
which may depend on a number of ``degrees-of-freedom'' of $\theta$
strictly smaller than the total dimension $d$. For a model with discrete group
$\G$ and no linear projection, \cite{fan2020likelihood} showed this number
exactly coincides with $\trdeg \cR_{\leq k}^\G$, the
transcendence degree of the invariant subalgebra generated by all
$\G$-invariant polynomials of degree $\leq k$. This implies a graded structure
of the Fisher information matrix $I(\theta_*)$ for generic signal vectors
$\theta_*$, where eigenvalues corresponding to different degrees-of-freedom have
different scalings with $\sigma^{-1}$. Furthermore, local
minimizers of $R(\theta)$ have a certain correspondence with successive local
minimizers of each function $R_k(\theta)$.

In this work, we first extend these results to a model where
$\G \subseteq \O(d)$ may be continuous, and samples may be observed
with an additional linear projection.  This extension encompasses a basic
formulation of the molecular reconstruction problem in cryo-EM.
The main result of \cite{katsevich2020likelihood} implies
that a series expansion analogous to (\ref{eq:seriesinformal}) continues to hold
for the population log-likelihood function $R(\theta)$ in such a model. However,
as anticipated from the structure of the expansion in
\cite{katsevich2020likelihood}, the algebraic properties of its terms
differ from the unprojected
setting studied in \cite{fan2020likelihood}.
We show here that the number of degrees-of-freedom associated to each
function $R_k(\theta)$ coincides with the transcendence degree of a possibly
reduced subalgebra generated by order-$k$ moments of the
projected signal (Theorem \ref{thm:FI}).  In addition, a version of the correspondence between local
minimizers of $R(\theta)$ and of successive local minimizers of $R_k(\theta)$
remains true over a bounded domain of $\R^d$ (Theorems
\ref{thm:benignlandscape} and \ref{thm:landscape}). When the group $\G$ is continuous,
we extend the arguments of \cite{fan2020likelihood} to address 
technical issues arising from the Fisher information matrix $I(\theta_*)$ being
singular, and the locus of minimizers of $R(\theta)$ being a manifold of
positive rather than zero dimension.\\

{\noindent}{\bf Multi-reference alignment and cryo-EM.}

These general results enable our study of maximum likelihood procedures
in specific function estimation problems, the main
focus of our work. In high-noise regimes of these problems,
statistical properties of the MLE are related to the
structures of the subalgebras $\cR_{\leq k}^\G$ and to
their transcendence degrees. In particular, the squared-error risk of the
MLE is dictated by the smallest non-zero eigenvalue
of the Fisher information matrix $I(\theta_*)$, and scales as
$\sigma^{2K}$ for generic signals $\theta_*$ where $K$ is the smallest
integer for which $\trdeg \cR_{\leq K}^\G$
equals $\trdeg \cR^\G$, the transcendence degree of the full
$\G$-invariant algebra. This connects
with a central result of \cite{bandeira2017estimation}, which showed that
$K$ is the lowest order moment needed to identify $\theta_*$ up to a finite
list of group orbits, and that
$\sigma^{2K}$ is also the scaling of the sample complexity required for
estimating $\theta_*$ up to such a finite list.  We apply our
general results to determine the explicit value of $K$ in several
function estimation examples.

For our model of continuous MRA on $\cS^1$,
we verify that $K=3$ (Theorem \ref{thm:MRA}).  This is expected from known
results about estimation using $3^\text{rd}$-order moments in similar
observation models for both discrete and continuous MRA in
\cite{perry2019sample,bandeira2020optimal}. We also show that the
optimization landscape of $R(\theta)$ may possess spurious
local minimizers even for generic Fourier coefficient vectors $\theta_* \in
\R^d$ (Theorem \ref{thm:MRAlandscape}) when the number of Fourier basis
functions $d$ exceeds a small constant.
This statement is analogous to results shown for discrete MRA in \cite{fan2020likelihood},
although our construction here in the continuous setting has a
different structure.

For spherical registration and projected and unprojected cryo-EM under
an $\SO(3)$-action, a primary contribution of our work is proving also that
$\trdeg \cR_{\leq K}^\G=\trdeg \cR^\G$ for $K=3$. An iterative algorithm for estimating $\theta_*$ from $3^\text{rd}$-order moments
in cryo-EM was first proposed by Kam in
\cite{kam1980reconstruction}, which implicitly assumed that
these moments are sufficient to identify $\theta_*$ (up to symmetries such
as chirality).
Formal conjectures that $K=3$ were stated in \cite{bandeira2017estimation} and
verified numerically for small values of the basis dimension $d$ in
exact-precision arithmetic. We prove
that $K=3$ for $d$ exceeding small absolute constants
(Theorems \ref{thm:S2registration}, \ref{thm:cryoEM}, and
\ref{thm:projectedcryoEM}), hence resolving several of these conjectures that
$3^\text{rd}$-order moments are sufficient to locally identify the orbit
of $\theta_*$.

Writing the terms $R_k(\theta)$ of (\ref{eq:seriesinformal}) as
\[R_k(\theta)=s_k(\theta)+q_k(\theta)\]
where $s_k(\theta)$ depends on the additional degrees-of-freedom of $\theta$
beyond those which define $R_{k-1}(\theta)$, our proofs of $K=3$
leverage a connection between $\trdeg \cR_{\leq k}^\G$ and the generic
ranks of the Hessians $\nabla^2 s_k(\theta)$ (Lemma \ref{lem:trdeg}).
We show that $\nabla^2 s_3(\theta)$ is generically of full rank
by using an inductive ``frequency marching'' argument on the dimension $d$ and
explicitly analyzing $\rank(\nabla^2 s_3(\theta))$ for special choices of
$\theta \in \R^d$. As a by-product of these analyses, we derive
explicit forms for $s_3(\theta)$, which define optimization problems analogous 
to bispectrum inversion problems studied in
MRA models \cite{bendory2017bispectrum}.

In an unprojected spherical registration model over $\cS^2$, recent
independent work of \cite{liu2021algorithms} provides a more quantitative
version of this inductive frequency marching argument. The result of
\cite[Lemma 5.6]{liu2021algorithms} implies
that above some absolute constant dimension $\underline{d}$, the increase in
$\rank(\nabla^2 s_3(\theta))$ from dimension $\underline{d}$ to $d$ must be exactly
$d-\underline{d}$ for generic $\theta \in \R^d$, and
\cite{liu2021algorithms} obtained a quantitative lower bound on the smallest
singular value in a smoothed analysis over $\theta$.
Our proofs show versions of this statement that
are less quantitative but more explicit about the value of $\underline{d}$,
holding down to small enough $\underline{d}$ where the full-rank condition for
$\nabla^2 s_3(\theta)$ may be explicitly checked.
We carry this out for both spherical registration and cryo-EM, and in
particular, our inductive argument in the projected cryo-EM model is
more complex than in the unprojected models and uses different ideas.\\

{\noindent}{\bf Simulations of the Fisher information for small proteins.}

To empirically investigate the predictions of this body of theory in a cryo-EM
example, we computed in simulation the observed Fisher information matrices for
the electric potential functions of two small protein molecules---a
rotavirus VP6 trimer and hemoglobin---in a model without tomographic
projection.

We developed and employed a procedure of adaptively constructing a radial
function basis in the Fourier domain (Appendix \ref{appendix:simulations})
so as to reduce the dimension of the function space needed to
approximate the true potential. Applying this construction, we obtained
function bases of dimension $d \approx 400$ that capture the coarse trimer 
structure of the rotavirus example, and of dimension $d \approx 4000$ that
capture the secondary structures of both proteins up to spatial
resolutions of 7--8 Angstroms.
At these dimensions and spatial resolutions, the theoretically predicted $\sigma^{-2}$, $\sigma^{-4}$, and $\sigma^{-6}$ scalings of the Fisher information
eigenvalues were apparent in simulation for sufficiently high noise.
We observe deviations from these predictions at lower levels of noise, and also
in higher-dimensional function spaces that may be necessary to approximate the
potentials to better spatial resolutions.

\subsection*{Notation}

We use the conventions $\langle u,v \rangle=\sum_i
\overline{u_i}v_i$ for the complex inner product, $\|u\|$
for the (real or complex) $\ell_2$-norm, $\|M\|$ for the
$\ell_2 \to \ell_2$ operator norm for matrices, and $\i=\sqrt{-1}$ for the
imaginary unit.

For a measure space $(X,\mu)$, $L_2(X,\C)$ is the $L_2$-space of
functions $f:X \to \C$ with inner-product $\int_X
\overline{f(x)}g(x)\mu(\der x)$. We write $L_2(X)=L_2(X,\R)$ for
the analogous $L_2$-space of real-valued functions.
$\sS^1$ and $\sS^2$ are the unit circle and unit sphere.

For differentiable $f:\R^d \to \R^k$, $\der f(x) \in \R^{k \times d}$ is
its derivative or Jacobian at $x$. For twice-differentiable
$f:\R^d \to \R$, $\nabla f(x)=\der f(x)^\top \in \R^d$ is its gradient, and
$\nabla^2 f(x) \in \R^{d \times d}$ is its Hessian. We will write
$\der_x,\nabla_x,\nabla_x^2$ to clarify that the variable of differentiation is
$x$. For a subset of coordinates $y$, we write $\nabla_y f(x)$ and
$\nabla_y^2 f(x)$ as the components of this gradient and Hessian in $y$.

For a smooth manifold $\mathcal{M}$ and twice-differentiable
$f:\mathcal{M} \to \R$, we write $\nabla f(x)|_{\mathcal{M}}$
and $\nabla^2 f(x)|_{\mathcal{M}}$ for its gradient and Hessian evaluated in
any choice of local chart at $x \in \mathcal{M}$. We will often not make the
choice of chart explicit when referring to properties of
$\nabla f(x)|_{\mathcal{M}}$ and $\nabla^2 f(x)|_{\mathcal{M}}$ that do not
depend on the specific choice of chart.

\subsection*{Acknowledgments}

We would like to thank Fred Sigworth for helpful discussions about cryo-EM, and
for suggesting to us the hemoglobin example. ZF was supported in part
by NSF DMS-1916198. RRL was supported in part by NIH/NIGMS 1R01GM136780-01.
YS was supported in part by NSF DMS-1701654, DMS-2039183, and DMS-2054838.

\section{The general orbit recovery model in high noise}\label{sec:general}

\subsection{Model and likelihood}

Let $\theta_* \in \R^d$ be an unknown signal of interest.
Let $\G \subseteq \O(d)$ be a known compact subgroup of the orthogonal group of
dimension $d$. We denote by $\Lambda$ the unique Haar probability measure on
$\G$, satisfying
\[\Lambda(\G)=1, \qquad \Lambda(g \cdot S)=\Lambda(S \cdot g)=\Lambda(S),\]
for any $g \in \G$ and Borel measurable subset $S \subseteq \G$. In the \emph{unprojected orbit recovery model}, we observe $n$ noisy and
rotated samples of $\theta_*$, given by
\begin{equation}\label{eq:unprojectedmodel}
Y_i=g_i \cdot \theta_*+\sigma \eps_i \in \R^d, \qquad i=1,\ldots,n
\end{equation}
where $g_1,\ldots,g_n \overset{iid}{\sim} \Lambda$ are Haar-uniform random
elements of $\G$, and $\eps_1,\ldots,\eps_n \overset{iid}{\sim} \N(0,\Id_{d
\times d})$ are Gaussian noise vectors independent of $g_1,\ldots,g_n$.
The signal $\theta_*$ is identifiable only up to an arbitrary
rotation in $\G$, i.e.\ it is identifiable up to its orbit
\[\orbit_{\theta_*}=\{g \cdot \theta_*:\,g \in \G\}.\]
Our goal is to estimate $\orbit_{\theta_*}$ from the observed rotated
samples $Y_1,\ldots,Y_n$.

In the \emph{projected orbit recovery model}, we consider an additional known
linear map $\Pi:\R^d \to \R^{\td}$.  (Note that $\Pi$ may not necessarily be an
orthogonal projection; our terminology is borrowed from the example of tomographic
projection in cryo-EM.) We observe $n$ samples
\begin{equation}\label{eq:projectedmodel}
Y_i=\Pi(g_i \cdot \theta_*)+\sigma \eps_i \in \R^{\td}, \qquad i=1,\ldots,n
\end{equation}
where $g_1,\ldots,g_n \overset{iid}{\sim} \Lambda$ as before, and
$\eps_1,\ldots,\eps_n \overset{iid}{\sim} \N(0,\Id_{\td \times \td})$ are
Gaussian noise vectors in the projected dimension $\td$. Our goal is again
to estimate $\orbit_{\theta_*}$ from $Y_1,\ldots,Y_n$.

The unprojected and projected orbit recovery models are both Gaussian mixture
models, where the distribution of mixture centers is the law of $g
\cdot \theta_* \in \R^d$ or of $\Pi(g \cdot \theta_*) \in \R^{\td}$ induced by
the uniform law $g \sim \Lambda$ over $\G$. This mixture distribution may be
continuous if $\G \subseteq \O(d)$ is a continuous subgroup.
In both models, we denote the negative sample log-likelihood as
\[
R_n(\theta)=-\frac{1}{n}\sum_{i=1}^n \log p_\theta(Y_i),
\]
where $p_\theta(Y_i)$ is the Gaussian mixture density for $Y_i$,
marginalizing over the unknown rotation $g_i \sim \Lambda$. This density is
given in the projected setting by
\begin{equation}\label{eq:likelihood}
p_\theta(y)=\int_{\G} \frac{1}{(2\pi\sigma^2)^{\td/2}}
\exp\left(-\frac{\|y-\Pi(g \cdot \theta)\|^2}{2\sigma^2}\right)\der
\Lambda(g),
\end{equation}
and in the unprojected setting by the same expression with
$\Pi=\Id$ and $\td=d$. The maximum likelihood estimator (MLE) of $\theta_*$
is $\hat{\theta}_n=\argmin_{\theta \in \R^d} R_n(\theta)$.
Since $R_n$ satisfies the invariance
$R_n(\theta)=R_n(g \cdot \theta)$ for all $g \in
\G$, the MLE is also only defined up to its orbit $\orbit_{\hat{\theta}_n}$.

\begin{remark}[Identifiability of the orbit]
The parameter $\theta_*$ is identifiable up to the distribution of the mixture
centers $g \cdot \theta_*$ or $\Pi(g \cdot \theta_*)$.
In the unprojected model, the equality in law
$g \cdot \theta\overset{L}{=}g \cdot \theta'$ over $g \sim \Lambda$
holds if and only if $\orbit_{\theta}=\orbit_{\theta'}$, so $\theta_*$ is
identifiable exactly up to its orbit.

In projected models, there may be further
non-identifiability. For instance, under the tomographic projection arising in
cryo-EM, we have $\Pi(g \cdot \theta)\overset{L}{=}\Pi(g \cdot \theta')$
when $\theta'$ represents the mirror reflection of $\theta$
\cite{bendory2020single}. Thus in this setting there may be two distinct
orbits which cannot be further identified, and $\theta_*$ is recovered
only up to chirality.

In general, the number of distinct orbits with the same image under $\Pi$ depends
on the interaction between the structures of $\G$ and $\Pi$, and can be
infinite. For example, for the trivial group $\G=\{\Id\}$ and the
projection $\Pi:\R^d \to \R^{d-k}$ that removes the last $k$ coordinates
of $\theta$, $\orbit_\theta=\{\theta\}$ and
$\Pi(\orbit_\theta)=\Pi(\orbit_{\theta_*})$ for any $\theta$ sharing
the same first $d-k$ coordinates as $\theta_*$.
\end{remark}

We use the equivalence notation
\[\Pi(\orbit_\theta) \equiv \Pi(\orbit_{\theta_*})\]
to mean that $\Pi(\orbit_\theta)=\Pi(\orbit_{\theta_*})$ as subsets of
$\R^{\td}$, and in addition,
$\Pi(g \cdot \theta) \overset{L}{=} \Pi(g \cdot \theta_*)$ under
the Haar-uniform law $g \sim \Lambda$. Thus $\theta_*$ is
identifiable up to this equivalence. 
We will restrict attention to projected models where  
\begin{equation}\label{eq:losslessPi}
\text{there are a finite number of orbits }
\orbit_{\theta} \text{ such that }
\Pi(\orbit_{\theta}) \equiv
\Pi(\orbit_{\theta_*}), \text{ for generic } \theta_* \in \R^d.
\end{equation}
An equivalent algebraic characterization is provided in
Proposition \ref{prop:Kdef}(b) below.

We denote the negative \emph{population} log-likelihood function by
\begin{equation}\label{eq:R}
R(\theta)=\E[R_n(\theta)]=-\E[\log p_\theta(Y)],
\end{equation}
where the expectation is taken under the true model
$Y \sim p_{\theta_*}$. $R(\theta)$ depends implicitly on $\theta_*$,
but we will omit this dependence in the notation. This population
log-likelihood is minimized at $\theta \in \orbit_{\theta_*}$ in the
unprojected model, and at
$\{\theta:\Pi(\orbit_\theta) \equiv \Pi(\orbit_{\theta_*})\}$ in projected
models.

\subsection{Invariant polynomials and the high-noise
expansion}\label{subsec:expansion}

For sufficiently high noise $\sigma^2$, it is informative to study $R(\theta)$
via a series expansion of the Gaussian density of (\ref{eq:likelihood})
in powers of $\sigma^{-1}$, as developed in
\cite{fan2020likelihood,katsevich2020likelihood}. We review this expansion in
this section.

Let $\cR^\G$ be the (real) algebra of all $\G$-invariant 
polynomial functions $p:\R^d \to \R$. These are the polynomials $p$ that satisfy
\[p(\theta)=p(g \cdot \theta) \text{ for all }
\theta \in \R^d \text{ and } g \in \G.\]
For each integer $k \geq 0$, let $\cR_{\leq k}^\G$ be the
subalgebra generated by the $\G$-invariant polynomials having
total degree at most $k$. This subalgebra consists of the polynomials $p \in \cR^\G$ that may be
expressed as $p(\theta)=q(p_1(\theta),\ldots,p_j(\theta))$ for some polynomial
$q$ and some $p_1,\ldots,p_j \in \cR^\G$ each having degree $\leq k$
(where $p$ itself may have degree larger than $k$).

Examples of polynomials in $\cR_{\leq k}^\G$ include
the entries of the symmetric moment tensors
\begin{equation}\label{eq:Tk}
T_k(\theta)=\int_\G (g \cdot \theta)^{\otimes k}\,\der \Lambda(g)
\in \R^{d \times \ldots \times d}
\end{equation}
where $T_k(\theta)$ is a tensor of order $k$. The entries of
$T_k(\theta)$ are
the $k^\text{th}$-order mixed moments of the distribution of Gaussian
mixture centers $g \cdot \theta$.
Conversely, any $\G$-invariant polynomial $p(\theta)$ of degree $\leq k$ 
satisfies the identity
\[p(\theta)=\int_\G p(g \cdot \theta)\,\der \Lambda(g),\]
and decomposing $p$ on the right side into a sum of monomials shows that
$p(\theta)$ is an affine linear combination of entries of $T_1,\ldots,T_k$.
Hence $\cR_{\leq k}^\G$ is generated by $T_1,\ldots,T_k$, and the
subalgebra $\cR_{\leq k}^\G$ may be intuitively understood as containing all
information in the moments of orders 1 to $k$ for the Gaussian mixture
defined by $\theta$.

For the projected model with projection $\Pi$, we define analogously the
projected moment tensors
\[\tT_k(\theta)=\int_\G (\Pi \cdot g \cdot \theta)^{\otimes k} \der \Lambda(g)
\in \R^{\td \times \ldots \times \td},\]
which are again the mixed moments of the Gaussian mixture centers $\Pi \cdot g
\cdot \theta$. We then define
\[\tcR_{\leq k}^\G=\text{subalgebra of } \cR^\G \text{ generated by the
entries of } \tT_1,\ldots,\tT_k.\]
Since each entry of $\tT_k$ is a $\G$-invariant polynomial of degree
$k$, we have $\tcR_{\leq k}^\G \subseteq \cR_{\leq k}^\G$, but equality does not
necessarily hold.

We denote by $\langle \cdot,\cdot \rangle$ the Euclidean inner-product
in the vectorization of these tensor spaces $\R^{d \times \ldots \times d}$
and $\R^{\td \times \ldots \times \td}$, and by $\|\cdot\|_\HS^2$ the
corresponding squared Euclidean norm. We will use the following general form
of the large-$\sigma$ series expansion of the population log-likelihood
$R(\theta)$. We explain how the results of \cite{katsevich2020likelihood} yield
this form in Appendix \ref{appendix:general}.

\begin{theorem}\label{thm:seriesexpansion}
Let $\G \subseteq \O(d)$ be any compact subgroup. Fix any $\theta_* \in \R^d$
and any integer $K \geq 0$.
\begin{enumerate}[(a)]
\item In the unprojected orbit recovery model, $R(\theta)$
admits an expansion
\begin{equation}\label{eq:seriesexpansionunprojected}
R(\theta)=C_0+\sum_{k=1}^K \frac{1}{\sigma^{2k}}\,
\Big(s_k(\theta)+q_k(\theta)\Big)+q(\theta).
\end{equation}
Here $C_0 \in \R$, $q_k \in \cR_{\leq k-1}^\G$ is a polynomial of degree
at most $2k$, and $s_k \in \cR_{\leq k}^\G$ is the polynomial
\begin{equation}\label{eq:sk}
s_k(\theta)=\frac{1}{2(k!)}\|T_k(\theta)-T_k(\theta_*)\|_\HS^2.
\end{equation}
The remainder $q(\theta)$ is $\G$-invariant and satisfies,
for all $\theta \in \R^d$ with $\|\theta\| \leq \sigma$,
\begin{equation}\label{eq:remainderbound}
|q(\theta)| \leq \frac{C_K(1 \vee \|\theta\|)^{2K+2}}{\sigma^{2K+2}},\;\;
\|\nabla q(\theta)\| \leq \frac{C_K(1 \vee\|\theta\|)^{2K+1}}
{\sigma^{2K+2}},\;\;
\|\nabla^2 q(\theta)\| \leq \frac{C_K(1 \vee\|\theta\|)^{2K}}
{\sigma^{2K+2}}.
\end{equation}
\item In the projected orbit recovery model, $R(\theta)$ admits an expansion
\begin{equation}\label{eq:seriesexpansionprojected}
R(\theta)=C_0+\sum_{k=1}^K \frac{1}{\sigma^{2k}}\,
\Big(\ts_k(\theta)+\big\langle \tT_k(\theta),P_k(\theta)\big\rangle
+q_k(\theta)\Big)+q(\theta).
\end{equation}
Here $C_0 \in \R$, $q_k \in \tcR_{\leq k-1}^\G$ is a polynomial of
degree at most $2k$, all entries of $P_k$ are polynomials of degree at most $k$
belonging to $\tcR_{\leq k-1}^\G$, $P_k$ satisfies $P_k(\theta_*)=0$,
and $\ts_k \in \tcR_{\leq k}^\G$ is the polynomial
\begin{equation}\label{eq:tildesk}
\ts_k(\theta)=\frac{1}{2(k!)}\|\tT_k(\theta)-\tT_k(\theta_*)\|_\HS^2.
\end{equation}
The remainder $q(\theta)$ is
$\G$-invariant and satisfies (\ref{eq:remainderbound}) for all 
$\theta \in \R^d$ with $\|\theta\| \leq \sigma$.
\end{enumerate}
The above constants $C_0,C_K$, the coefficients of the polynomials
$q_k(\theta)$ and $P_k(\theta)$, and the forms of the functions $q(\theta)$
may all depend on $\theta_*,\G,d,\td$, and the projection $\Pi$.
\end{theorem}

The exact forms of $q_k(\theta)$ and $P_k(\theta)$ can be
explicitly derived---see \cite[Section 4.2]{fan2020likelihood} for
these derivations in the unprojected setting---but we will
not require them in what follows. Our arguments will only require the forms of
the ``leading'' terms $s_k(\theta)$ and $\tilde{s}_k(\theta)$
defined in (\ref{eq:sk}) and (\ref{eq:tildesk}).

\subsection{Fisher information in high noise}

Consider the Fisher information matrix
\[I(\theta_*)=\nabla^2 R(\theta)\big|_{\theta=\theta_*}.\]
In this section, we characterize the eigenvalues and eigenvectors of
$I(\theta_*)$ for high noise and
generic $\theta_* \in \R^d$. This generalizes 
\cite[Theorem 4.14]{fan2020likelihood} for the unprojected model
and a discrete group.

\begin{definition}\label{def:generic}
A subset $S \subseteq \R^d$ is \emph{generic} if $\R^d \setminus S$ is contained
in the zero set of some non-zero analytic function $\psi:\R^d \to \R^k$,
for some $k \geq 1$.
\end{definition}
\noindent If $S \subseteq \R^d$ is generic, then $\R^d \setminus S$ has
zero Lebesgue measure \cite{mityagin2020zero}. We say that a statement
holds for generic $\theta_* \in \R^d$ if it holds for all $\theta_*$ in
some generic subset of $\R^d$.

Our characterization of $I(\theta_*)$ is in terms of the number of
distinct ``degrees-of-freedom'' captured by the moments of the Gaussian mixture
model up to each order $k$. This is formalized by the notion of the
transcendence degrees of the subalgebras $\cR_{\leq k}^\G$ and
$\tcR_{\leq k}^\G$.

\begin{definition}
Polynomials $p_1,\ldots,p_k:\R^d\to\R$ are \emph{algebraically
independent} (over $\R$) if there is no non-zero polynomial $q:\R^k\to\R$ such
that $q(p_1(\theta),\ldots,p_k(\theta))$ is identically $0$ for all $\theta\in\R^d$.

For any $\cA \subseteq \cR^\G$, its \emph{transcendence degree}
$\trdeg(\cA)$ is the maximum cardinality of any algebraically independent
subset $A \subseteq \cA$. Any maximal such subset $A \subseteq \cA$
is a \emph{transcendence basis} for $\cA$.
\end{definition}

\noindent Geometrically, by the Jacobian criterion for algebraic
independence (c.f.\ Lemma \ref{lemma:jaccrit}), the transcendence degree
coincides with the maximum number of linearly
independent gradient vectors of the polynomials in $\cA$, evaluated at any
generic point $\theta \in \R^d$.

As a simple example, if $\G$ is the symmetric group of all
permutations of $d$ coordinates, then $\cR^\G$ is the algebra of all symmetric
polynomials in $d$ variables. Each
subalgebra $\cR_{\leq k}^\G$ for $k \leq d$ has transcendence
degree exactly equal to $k$, and one choice of a transcendence basis for
$\cR_{\leq k}^\G$ is the set of
symmetric power sums $\{\theta_1^j+\ldots+\theta_d^j:j=1,\ldots,k\}$.

For the full invariant algebra $\cR^\G$, if $\G \subset \O(d)$ is any
discrete subgroup as studied in \cite{fan2020likelihood},
then $\trdeg(\cR^\G)=d$. More generally, we have the following.

\begin{proposition}\label{prop:trdegorbitdim}
Let $\G$ be a compact subgroup of $\O(d)$. Then
\[\trdeg(\cR^\G)=d-\max_{\theta \in \R^d} \dim(\orbit_\theta)\]
where $\dim(\orbit_\theta)$ is the dimension of the orbit $\orbit_\theta$ as a
submanifold of $\R^d$. Here, the maximum orbit dimension
$\max_{\theta \in \R^d} \dim(\orbit_\theta)$ is also the orbit dimension
for generic points $\theta \in \R^d$.
\end{proposition}

We will mostly consider group actions where this generic orbit dimension
equals the group dimension $\dim(\G)$,
so that $\trdeg(\cR^\G)=d-\dim(\G)$. In particular,
for the function estimation examples to be discussed in
Sections \ref{sec:func-est-so2} and \ref{sec:func-est-so3},
we will have $\trdeg(\cR^\G)=d-1$ for an action of $\G$ that is isomorphic to
$\SO(2)$, and $\trdeg(\cR^\G)=d-3$ for an action of $\G$ that is isomorphic to
$\SO(3)$.

It was shown in \cite[Theorem 4.9]{bandeira2017estimation}, for generic
signals $\theta_* \in \R^d$, that the values of the moment tensors
$T_1(\theta_*),\ldots,T_k(\theta_*)$ are sufficient to identify $\theta_*$ up to
a finite list of possible orbits if and only if $\trdeg(\cR_{\leq k}^\G)
=\trdeg(\cR^\G)$. More informally, the order of moments needed to ``locally''
identify the orbit of $\theta_*$
coincides with the order of moments needed to capture all $\trdeg(\cR^\G)$
degrees-of-freedom of the invariant algebra.
Throughout this paper, we will denote this number as $K$ in the unprojected model
and as $\tK$ in the projected model, which are well-defined by the following
proposition. We defer proofs of Propositions \ref{prop:trdegorbitdim}
and \ref{prop:Kdef} to Appendix \ref{appendix:general}.
\begin{proposition}\label{prop:Kdef}
For any compact subgroup $\G \subseteq \O(d)$,
\begin{enumerate}[(a)]
\item There is a smallest integer $K<\infty$ for which
$\trdeg(\cR_{\leq K}^\G)=\trdeg(\cR^\G)$.
\item $\Pi$ satisfies (\ref{eq:losslessPi}) if and only if
there is a smallest integer $\tK<\infty$ for which
$\trdeg(\tcR_{\leq \tK}^\G)=\trdeg(\cR^\G)$.
\end{enumerate}
\end{proposition}

In the unprojected model, let us now denote
\begin{equation}\label{eq:dk}
d_0=\max_{\theta \in \R^d} \dim(\orbit_\theta), \qquad
d_k=\trdeg \cR_{\leq k}^\G-\trdeg \cR_{\leq k-1}^\G \text{ for } k=1,\ldots,K
\end{equation}
to decompose the total dimension of $\theta_*$ as $d=d_0+d_1+\ldots+d_K$.
In the projected model, assuming the condition (\ref{eq:losslessPi}), let us 
similarly denote
\begin{equation}\label{eq:tdk}
\td_0=\max_{\theta \in \R^d} \dim(\orbit_\theta), \qquad
\td_k=\trdeg \tcR_{\leq k}^\G-\trdeg \tcR_{\leq k-1}^\G \text{ for }
k=1,\ldots,\tK
\end{equation}
to decompose the total dimension as $d=\td_0+\td_1+\ldots+\td_{\tK}$.
The following result expresses the spectral properties of the Fisher
information matrix in terms of these decompositions.

\begin{theorem}\label{thm:FI}
For generic $\theta_* \in \R^d$, some
$(\theta_*,\G,\Pi)$-dependent constants $\sigma_0,C,c>0$ and
function $\eps(\sigma)$ satisfying
$\eps(\sigma) \to 0$ as $\sigma \to \infty$, and all $\sigma>\sigma_0$:
\begin{enumerate}[(a)]
\item In the unprojected orbit recovery model,
\begin{enumerate}[1.]
\item The Fisher information matrix
$I(\theta_*)$ has rank exactly $\trdeg(\cR^\G)=d-d_0$. Defining $K$ by
Proposition \ref{prop:Kdef}(a), for each $k=1,\ldots,K$,
\[\text{exactly } d_k \text{ eigenvalues of } I(\theta_*) \text{ belong to }
[c\sigma^{-2k},C\sigma^{-2k}].\]
\item For each $k=1,\ldots,K$, let $V_k$ be the subspace spanned by the
leading $d_1+\ldots+d_k$ eigenvectors of
$I(\theta_*)$, and let $W_k$ be the subspace spanned by the gradient vectors
$\{\nabla p(\theta_*):p \in \cR_{\leq k}^\G\}$. Then
the sin-theta distance between $V_k$ and $W_k$ is bounded as
\[\|\sin \Theta(V_k,W_k)\|<\eps(\sigma).\]
\item For any $k=1,\ldots,K$ and any polynomial $p \in \cR_{\leq k}^\G$,
the gradient $\nabla p(\theta_*) \in \R^d$ is orthogonal to the null space of
$I(\theta_*)$ and satisfies
\[\nabla p(\theta_*)^\top I(\theta_*)^\dagger \nabla p(\theta_*)
\leq C\sigma^{2k}\]
where $I(\theta_*)^\dagger$ is the Moore-Penrose pseudo-inverse.
\end{enumerate}
\item In the projected orbit recovery model satisfying condition
(\ref{eq:losslessPi}), the same statements hold
with $\cR_{\leq k}^\G$, $K$, and $d_k$ replaced by $\tcR_{\leq k}^\G$, $\tK$,
and $\td_k$, where $\tK$ is defined by Proposition \ref{prop:Kdef}(b).
\end{enumerate}
\end{theorem}

\begin{remark*}
Theorem \ref{thm:FI}(a1) states that $I(\theta_*)$ has eigenvalues on differing
scales of $\sigma^{-2}$ in high noise, with $d_k$ such eigenvalues
scaling as $\sigma^{-2k}$, and $d_0=\dim(\orbit_{\theta_*})$ eigenvalues of 0
representing the non-identifiable degrees-of-freedom tangent to 
$\orbit_{\theta_*}$. Thus there are $d_k$ degrees-of-freedom in $\theta_*$ that
are estimated with asymptotic variance $O(\sigma^{2k}/n)$ by the MLE.
The largest such variance is $O(\sigma^{2K}/n)$, which is in accordance with 
results about list-recovery of generic signals in \cite{bandeira2017estimation}
and with the $\sigma^6$ sample complexity established in
\cite{perry2019sample} for multi-reference alignment, where $K=3$.
\end{remark*}

\begin{remark*}
Theorem \ref{thm:FI}(a2) describes also the associated spaces of
eigenvectors of $I(\theta_*)$, where the eigenspaces corresponding
to eigenvalues at scales $\sigma^{-2},\ldots,\sigma^{-2k}$ coincide
approximately with the span of the gradients of $\G$-invariant polynomials up
to degree $k$. Theorem \ref{thm:FI}(a3) then implies that
the functional $p(\theta_*)$ for any $p \in \cR_{\leq k}^\G$ is
estimated by the plug-in MLE $p(\hat{\theta}_n)$ with asymptotic
variance $O(\sigma^{2k}/n)$. Similar statements hold for projected models
by Theorem \ref{thm:FI}(b).
\end{remark*}

The following result connects the above sequences of transcendence degrees 
and gradients $\{\nabla p(\theta_*):p \in \cR_{\leq k}^\G\}$ to
the terms $s_k(\theta)$ and $\ts_k(\theta)$ in the series expansions of
$R(\theta)$ in Theorem \ref{thm:seriesexpansion}. We will use this to
deduce the values of these transcendence degrees for the function estimation
examples of Sections \ref{sec:func-est-so2} and \ref{sec:func-est-so3}.

\begin{lemma}\label{lem:trdeg}
\begin{enumerate}[(a)]
\item In the unprojected orbit recovery model, let $s_k(\theta)$ be defined by
(\ref{eq:sk}). Then each matrix $\nabla^2 s_k(\theta)|_{\theta=\theta_*}$ is
positive semidefinite. For any $k \geq 1$, at generic $\theta_* \in \R^d$,
\begin{equation}\label{eq:trdeg}
\trdeg(\cR_{\leq k}^\G)=\rank\Big(\nabla^2 s_1(\theta)+\ldots+
\nabla^2 s_k(\theta)\Big|_{\theta=\theta_*}\Big),
\end{equation}
and the span of $\{\nabla p(\theta_*):p \in \cR_{\leq k}^\G\}$ is
the column span of $\nabla^2 s_1(\theta)+\ldots+
\nabla^2 s_k(\theta)|_{\theta=\theta_*}$.
\item In the projected orbit recovery model, the same holds for
$\tcR_{\leq k}^\G$ and $\ts_k(\theta)$ as defined by (\ref{eq:tildesk}).
\end{enumerate}
\end{lemma}

\begin{remark}
We restrict attention to generic signals $\theta_* \in \R^d$ in this work.
The specific condition for $\theta_*$ that we use in
Theorem \ref{thm:FI}, and in Theorems \ref{thm:benignlandscape}
and \ref{thm:landscape} to follow, is that the gradient vectors
$\{\nabla p(\theta_*):p \in \cR_{\leq K}^\G\}$
or $\{\nabla p(\theta_*):p \in \tcR_{\leq \tK}^\G\}$ span a
subspace of dimension $\trdeg(\cR^\G)$ or $\trdeg(\tcR^\G)$,
respectively.

Different behavior may be observed for
non-generic signals: For $\G=\{+\Id,-\Id\}$, which has been studied
in \cite{xu2016global,wu2019randomly}, the Fisher information
$I(\theta_*)$ is singular at $\theta_*=0$ (even though $d_0=0$, as the group is
discrete). This leads to a $n^{-1/4}$ rate of estimation error near
$\theta_*=0$, instead of the $n^{-1/2}$ parametric rate.
This $n^{-1/4}$ rate holds more generally for any discrete group $\G$ at
signals $\theta_*$ whose orbit points are not pairwise distinct, which are
precisely those signals where the Fisher information
$I(\theta_*)$ is singular \cite{brunel2019learning}.

A different distinction between generic and non-generic signals was
highlighted in \cite{perry2019sample} when
$\G$ is the group of cyclic rotations of coordinates in $\R^d$.
There, orbits of generic signals
are uniquely identified by moments up to the order $K=3$, but identification of
non-generic signals having zero power in certain Fourier frequencies may require
moments up to the order $d-1$. For such non-generic signals, we expect
$I(\theta_*)$ to be non-singular and the MLE to attain the
parametric rate, but with asymptotic variance scaling as
$\sigma^{2(d-1)}/n$ rather than $\sigma^{2K}/n=\sigma^6/n$. In a related model
of continuous MRA, this asymptotic scaling
is implied by the results of \cite{bandeira2020optimal}.
\end{remark}

\subsection{Global likelihood landscape}

In this section, we establish correspondences between global and local
minimizers of the population negative log-likelihood $R(\theta)$
with those of a sequence of moment optimization problems. These results are
similar to results of \cite[Sections 4.3 and 4.5]{fan2020likelihood} for discrete
groups $\G$, with a distinction that when $\G$ is continuous, these 
minimizers are not isolated points but rather manifolds of positive dimension.

We recall the following structural property for smooth non-convex optimization
landscapes, under which convergence to the global optimum from a random
initialization is guaranteed for various descent-based optimization 
algorithms \cite{ge2015escaping,lee2016gradient,jin2017escape}.

\begin{definition}
The problem of minimizing a twice-continuously differentiable function
$f:\cV \to \R$ over a smooth manifold $\cV$ is \emph{globally benign}
if each point $x \in \cV$ where $\nabla f(x)|_{\cV}=0$ is either a
global minimizer of $f$ over $\cV$, or has a direction of strictly negative
curvature, $\lambda_{\min}(\nabla^2 f(x)|_{\cV})<0$.
\end{definition}

\noindent Here $\nabla f(x)|_{\cV}$ and $\nabla^2 f(x)|_{\cV}$ denote the
gradient and Hessian of $f$ on $\cV$, which may be taken in any choice of a
smooth local chart around $x \in \cV$.

Minimizing $R(\theta)$ in high noise may be viewed as successively solving a
sequence of moment optimizations defined by the terms of its expansion in
Theorem \ref{thm:seriesexpansion}. To ease notation, let us collect the
vectorized moment tensors up to order $k$ as
\begin{align}
M_k(\theta)&=\vectorize\Big(T_1(\theta),\ldots,T_k(\theta)\Big) \in
\R^{d+d^2+\ldots+d^k},\label{eq:Mk}\\
\tM_k(\theta)&=\vectorize\Big(\tT_1(\theta),\ldots,\tT_k(\theta)\Big) \in
\R^{\td+\td^2+\ldots+\td^k}.\label{eq:tMk}
\end{align}
Fixing the true signal $\theta_* \in \R^d$, we define the moment varieties
\begin{align}
\cV_k(\theta_*)&=\Big\{\theta \in \R^d:\,M_k(\theta)=M_k(\theta_*)\Big\},
\qquad \cV_0(\theta_*)=\R^d,\label{eq:Vk}\\
\tcV_k(\theta_*)&=\Big\{\theta \in \R^d:\,\tM_k(\theta)=\tM_k(\theta_*)\Big\},
\qquad \tcV_0(\theta_*)=\R^d.\label{eq:tVk}
\end{align}
These are the points $\theta \in \R^d$ for which the mixed moments of the
Gaussian mixture model defined by $\theta$ match those of the true signal
$\theta_*$ up to order $k$.

We state a general result on the optimization landscape,
assuming that the Jacobian matrices $\der M_k$ and $\der \tM_k$
have constant rank over $\cV_k(\theta_*)$ and $\tcV_k(\theta_*)$, so that
$\cV_k(\theta_*)$ and $\tcV_k(\theta_*)$ are smooth manifolds.
Then, recalling $s_k(\theta)$ and
$\ts_k(\theta)$ from (\ref{eq:sk}) and (\ref{eq:tildesk}),
we consider the optimization problem
\begin{equation}\label{eq:momentoptunprojected}
\text{minimize } s_k(\theta) \text{ over } \theta \in \cV_{k-1}(\theta_*)
\end{equation}
in the unprojected setting, and
\begin{equation}\label{eq:momentoptprojected}
\text{minimize } \ts_k(\theta) \text{ over } \theta \in \tcV_{k-1}(\theta_*)
\end{equation}
in the projected setting. These are polynomial optimization problems in
$\theta$ that are defined independently of the noise level $\sigma^2$.
The following theorem guarantees that the landscape
of $R(\theta)$ is globally benign in high noise, as long as the
landscape of each problem (\ref{eq:momentoptunprojected}) or
(\ref{eq:momentoptprojected}) is globally benign, and the final moment variety
$\cV_K(\theta_*)$ or $\tcV_{\tK}(\theta_*)$ contains only the points which
globally minimizer $R(\theta)$.
We illustrate part (a) of this result using a simple example of
orthogonal Procrustes alignment at the conclusion of this section.

\begin{theorem}\label{thm:benignlandscape}
For generic $\theta_* \in \R^d$:
\begin{enumerate}[(a)]
\item In the unprojected model, define $K$ by Proposition \ref{prop:Kdef}(a).
Suppose that $\cV_K(\theta_*)=\orbit_{\theta_*}$. Suppose also that 
for each $k=1,\ldots,K$, the derivative matrix
$\der M_k(\theta)$ has constant rank over $\cV_k(\theta_*)$, and
the minimization of $s_k(\theta)$ over $\cV_{k-1}(\theta_*)$ is globally benign.
Then for some $\sigma_0 \equiv \sigma_0(\theta_*,\G)$ and any
$\sigma>\sigma_0$, the minimization of $R(\theta)$ is also globally benign.
\item In the projected model satisfying (\ref{eq:losslessPi}), define $\tK$ by
Proposition \ref{prop:Kdef}(b).
Suppose that $\tcV_{\tK}(\theta_*)=\{\theta:\Pi(\orbit_\theta) \equiv
\Pi(\orbit_{\theta_*})\}$. Suppose also that for each $k=1,\ldots,\tK$,
the derivative matrix
$\der \tM_k(\theta)$ has constant rank over $\tcV_k(\theta_*)$, and
the minimization of $\ts_k(\theta)$ over $\tcV_{k-1}(\theta_*)$ is globally
benign. Then for any constant $B>0$, some $\sigma_0 \equiv
\sigma_0(\theta_*,\G,\Pi,B)$, and any $\sigma>\sigma_0$,
the minimization of $R(\theta)$ is globally benign over the domain
$\{\theta \in \R^d:\|\theta\|<B(\|\theta_*\|+\sigma)\}$.
\end{enumerate}
\end{theorem}

In Theorem \ref{thm:benignlandscape}(b), we have restricted to a ball
$\{\theta \in \R^d:\|\theta\|<B(\|\theta_*\|+\sigma)\}$, as the
landscape of $R(\theta)$ outside this ball may depend on 
the specific interaction between $\G$ and $\Pi$. In practice, such a bound
for $\|\theta\|$ may be known a priori, so that optimization may indeed
be restricted to this ball. (In unprojected models,
we show that $R(\theta)$ cannot have critical points outside this
ball for any group $\G$, allowing us to remove such a restriction in
part (a).)

Whether the conditions of Theorem \ref{thm:benignlandscape} hold
depends on the specific model, and both positive and negative examples for
discrete groups were exhibited in \cite{fan2020likelihood}.
In models where they do not hold,
$R(\theta)$ may in fact have spurious local minimizers in high noise, and Theorem
\ref{thm:seriesexpansion} can be used to further establish a correspondence 
between the local minimizers of $R(\theta)$ and those of the above moment
optimizations. We formalize one such result---not fully general, but
sufficient to study many examples of interest---as follows.

\begin{definition}\label{def:nondegenerate}
Suppose $\cV_{K-1}(\theta_*)$ is a smooth manifold.
A critical point $\theta$ of
$s_K(\theta)|_{\cV_{K-1}(\theta_*)}$ is \emph{non-degenerate up to orbit}
if $\orbit_\theta$ is a smooth manifold of dimension $d_0$ in a local
neighborhood of $\theta$, and
\[\rank\big(\nabla^2 s_K(\theta)|_{\cV_{K-1}(\theta_*)}\big)
=\dim(\cV_{K-1}(\theta_*))-d_0.\]
\end{definition}

Note that $\nabla^2 s_K(\theta)|_{\cV_{K-1}(\theta_*)}$ is a symmetric matrix of
dimension $\dim(\cV_{K-1}(\theta_*))$. For any critical point $\theta$ of
$s_K|_{\cV_{K-1}(\theta_*)}$, the null space of this Hessian must contain the
tangent space to $\orbit_\theta$, and Definition \ref{def:nondegenerate}
ensures that this Hessian has no further rank degeneracy.

\begin{theorem}\label{thm:landscape}
For generic $\theta_* \in \R^d$:
\begin{enumerate}[(a)]
\item In the unprojected model, suppose that 
$\der M_k(\theta)$ has constant rank over $\cV_k(\theta_*)$, and the
minimization of $s_k(\theta)$ over $\cV_{k-1}(\theta_*)$ is globally benign for
each $k=1,\ldots,K-1$. Then
for some $(\theta_*,\G)$-dependent constant $\sigma_0>0$ and function
$\eps(\sigma)$ satisfying $\eps(\sigma) \to 0$ as $\sigma \to \infty$,
and for all $\sigma>\sigma_0$:
\begin{enumerate}[1.]
\item Let $\theta_+$ be any \emph{local} minimizer
of $s_K(\theta)$ over $\cV_{K-1}(\theta_*)$ that is non-degenerate up to orbit.
Then there exists a local minimizer $\theta_+'$ of $R(\theta)$ where
$\|\theta_+-\theta_+'\|<\eps(\sigma)$.
\item Conversely, suppose that all critical points of
$s_K(\theta)$ over $\cV_{K-1}(\theta_*)$ are non-degenerate up to orbit.
Let $\theta_+$ be any local minimizer of $R(\theta)$.
Then there exists a local
minimizer $\theta_+'$ of $s_K(\theta)$ over $\cV_{K-1}(\theta_*)$ where
$\|\theta_+-\theta_+'\|<\eps(\sigma)$.
\end{enumerate}
\item In the projected model satisfying (\ref{eq:losslessPi}), 
statement (1.) holds with $K$, $M_k$, $\cV_k$, and $s_k$ replaced by
$\tK$, $\tM_k$, $\tcV_k$, and $\ts_k$, where $\sigma_0$ and $\eps(\sigma)$
may depend also on the projection $\Pi$.
Statement (2.) holds for local minimizers $\theta_+$
of $R(\theta)$ satisfying $\|\theta_+\|<B(\|\theta_*\|+\sigma)$ for any constant
$B>0$, where $\sigma_0$ and $\eps(\sigma)$ may depend also on $\Pi$ and $B$.
\end{enumerate}
\end{theorem}

The guarantees of Theorems \ref{thm:benignlandscape} and \ref{thm:landscape}
may be translated to the sample log-likelihood $R_n(\theta)$ by establishing
concentration of $\nabla R_n(\theta)$ and $\nabla^2 R_n(\theta)$
around $\nabla R(\theta)$ and $\nabla^2 R(\theta)$
\cite{mei2018landscape}. For orbit recovery models in the
high-noise regime, we believe that it may be possible to obtain
sharp concentration bounds by deriving a series expansion also of the
\emph{empirical} log-likelihood function $R_n(\theta)$ in powers of
$\sigma^{-1}$, and analyzing
the concentration term-by-term. Some results of this form were obtained for
models without linear projection in \cite[Lemma 4.11 and Corollary
4.18]{fan2020likelihood}, and we leave the analysis of the empirical
log-likelihood function and landscape
for more general models as an open problem for future work.

\begin{example}[Landscape of orthogonal Procrustes alignment]\label{ex:procrustes}
We illustrate Theorems \ref{thm:benignlandscape} and \ref{thm:landscape} using
a simple example of \emph{orthogonal Procrustes alignment}
\cite{gower1975generalized,goodall1991procrustes,pumir2019generalized}.

In this problem,
samples of an object consisting of $m \geq 3$ atoms in $\R^3$ are observed
under random orthogonal rotations and reflections.
We represent the object as $\theta_* \in \R^{3 \times
m} \cong \R^d$ where $d=3m$. The rotational group is
$\G=\O(3) \otimes \Id_m \subset \O(d)$, where a common orthogonal matrix in
3-dimensions is applied to all $m$ atoms. Assuming the generic condition that
$\rank(\theta_*)=3$, i.e.\ these $m$ atoms do not lie on a common
2-dimensional subspace, we study the likelihood landscape
for estimating $\theta_*$ from many independently rotated samples.

In this model, we check in Appendix \ref{appendix:procrustes} that
$K=2$, $(d_0,d_1,d_2)=(3,0,d-3)$, $\cV_1(\theta_*)=\R^d$,
and $\cV_2(\theta_*)=\{g \cdot \theta_*:g \in \G\}=\orbit_{\theta_*}$.
The first two moment tensors $T_1(\theta)$ and $T_2(\theta)$ are given
by $T_1(\theta)=0$ and $T_2(\theta)=\frac{1}{3}\Id_{3 \times 3} \otimes
(\theta^\top \theta) \in \R^{d \times d}$, and
the terms $s_1(\theta)$ and $s_2(\theta)$ in
(\ref{eq:seriesexpansionunprojected}) are given by $s_1(\theta)=0$ and
\[s_2(\theta)=\frac{1}{12}\|\theta^\top \theta-\theta_*^\top \theta_*\|_\HS^2,\]
where $\theta^\top \theta,\theta_*^\top \theta_* \in \R^{m \times m}$.
The minimization of $s_1(\theta)$ over $\cV_0(\theta_*)=\R^d$ is
trivially globally benign. We show in Appendix \ref{appendix:procrustes} that
$\der M_2(\theta)$ has constant rank over $\cV_2(\theta_*)$, and that
the minimization of $s_2(\theta)$ over 
$\cV_1(\theta_*)=\R^d$ is also globally benign, with minimizers given exactly by
$\cV_2(\theta_*)=\orbit_{\theta_*}$. Thus, Theorem
\ref{thm:benignlandscape}(a) implies that the landscape of $R(\theta)$ is also
globally benign for sufficiently high noise, and the only local minimizers of
$R(\theta)$ are rotations and reflections of the true object.

A variation of this problem is the rotation-only variant,
where we observe 3-dimensional rotations (but not reflections) of
the object. Then the rotational group is instead
$\G=\SO(3) \otimes \Id_m \subset \O(d)$.
We show in Appendix \ref{appendix:procrustes} that still
$K=2$, $(d_0,d_1,d_2)=(3,0,d-3)$, and the forms of
$T_1(\theta),T_2(\theta),\cV_1(\theta_*),\cV_2(\theta_*),s_1(\theta),s_2(\theta)$ are identical to the
above (even though the full log-likelihood $R(\theta)$ is not).
Thus the minimization of $s_2(\theta)$ over
$\cV_1(\theta_*)=\R^d$ is still globally benign, with minimizers
$\cV_2(\theta_*)$. However, this set of minimizers is now written as
\[\cV_2(\theta_*)
=\{g \cdot \theta_*:g \in \G\} \cup \{-g \cdot \theta_*:g \in \G\}
=\orbit_{\theta_*} \cup \orbit_{-\theta_*}\]
constituting two distinct orbits under this more restrictive group action.
The first orbit $\orbit_{\theta_*}$ are the global
minimizers of $R(\theta)$. The second orbit corresponds to the mirror reflection
$-\theta_*$, which does not
globally minimize $R(\theta)$, but the difference between $R(\theta_*)$ and
$R(-\theta_*)$ lies in the remainder term of the expansion
(\ref{eq:seriesexpansionunprojected}). Theorem \ref{thm:landscape}(a) shows
that for high noise, $R(\theta)$ will have spurious local minimizers near
(but not exactly equal to) this second orbit $\orbit_{-\theta_*}$.
\end{example}

\section{Continuous multi-reference alignment}\label{sec:func-est-so2}

We now specialize the preceding general results to the problem
of estimating a periodic function on the circle, observed under $\SO(2)$
rotations of its domain. We will refer to this as the \emph{continuous
MRA} model. We study the unprojected model in this section,
and a version with a two-fold projection in
Appendix \ref{appendix:projectedSO2}.
These provide simpler 1-dimensional analogues
of the 2-dimensional and 3-dimensional problems that we will discuss in
Section \ref{sec:func-est-so3}.

To describe the model, let $f:\sS^1 \to \R$ be a periodic function on
the unit circle $\sS^1 \cong [0,1)$. We identify the rotational group $\SO(2)$
also with $[0,1)$, and represent the rotation of $f$ by an element
$\frakg \in \SO(2) \cong [0,1)$ as $f_\frakg(t)=f(t+\frakg \bmod 1)$.
Each sample is an observation of the rotated function $f_\frakg$
with additive white noise,
\[f_\frakg(t)\der t+\sigma\,\der W(t)\]
where $\frakg \sim \Unif([0,1))$ and
$\der W(t)$ denotes a standard Gaussian white noise process on $\sS^1$.
This may be understood as observing a realization of the
Gaussian process $\{F(h)\}_{h \in L_2(\sS^1)}=\{\int h(t)[f_\frakg(t)\der t+
\sigma \der W(t)]\}_{h \in L_2(\sS^1)}$ with mean and covariance functions
\begin{equation}\label{eq:GPwhitenoise}
\E[F(h)]=\int_0^1 h(t)f_\frakg(t)\der t,
\qquad \Cov[F(h_1),F(h_2)]=\sigma^2 \int_0^1 h_1(t)h_2(t)\,\der t,
\end{equation}
or equivalently as observing all coefficients of $f_\frakg$
in a complete orthonormal
basis $\{h_j(t)\}_{j=1}^\infty$ of $L_2(\cS^1)$ with independent $\N(0,\sigma^2)$
noise for each basis coefficient.

We consider the real Fourier basis on $\sS^1$, given by
\begin{equation}\label{eq:Fourierbasis}
h_0(t)=1, \quad h_{l1}(t)=\sqrt{2} \cos 2\pi lt, \quad
h_{l2}(t)=\sqrt{2} \sin 2\pi lt \quad \text{for } l=1,2,3,\ldots.
\end{equation}
We then restrict our model to the finite-dimensional space of functions
$f:\sS^1 \to \R$ that have finite bandlimit $L \geq 1$ in this basis,
i.e.\ $f$ admits a representation
\begin{equation}\label{eq:fourierbasis}
f(t)=\theta^{(0)}h_0(t)+\sum_{l=1}^L \theta_1^{(l)}h_{l1}(t)+
\sum_{l=1}^L\theta_2^{(l)}h_{l2}(t).
\end{equation}
Importantly, the space of such bandlimited functions is closed under rotations
of $\sS^1$. Writing
\[\theta=(\theta^{(0)},\theta_1^{(1)},\theta_2^{(1)},\ldots,
\theta_1^{(L)},\theta_2^{(L)}) \in \R^d, \qquad d=2L+1\]
for the vector of Fourier coefficients, the rotation $f \mapsto f_\frakg$
corresponds to $\theta \mapsto g \cdot \theta$, where $g$ belongs to
the block-diagonal representation
\begin{equation}\label{eq:MRAG}
\G=\left\{\diag\left(1,\begin{pmatrix} \cos 2\pi \frakg & \sin 2\pi \frakg \\
-\sin 2\pi \frakg & \cos 2\pi \frakg \end{pmatrix},
\ldots,\begin{pmatrix} \cos 2\pi L\frakg & \sin 2\pi L\frakg \\
-\sin 2\pi L\frakg & \cos 2\pi L\frakg \end{pmatrix}\right):
\;\frakg \in [0,1)\right\}
\end{equation}
of $\SO(2)$. The observation model for the Fourier coefficients of $f$
then takes the form of (\ref{eq:unprojectedmodel}), where we observe
coordinates of $g \cdot \theta$ with entrywise i.i.d.\ $\N(0,\sigma^2)$ noise.

Theorem \ref{thm:MRA} below first characterizes, for this model,
the decomposition of total dimension described in Theorem \ref{thm:FI}.
As a direct consequence of this result, we state Corollary \ref{cor:MRA-cor}
which summarizes the implications for identifying $\theta_*$ based on its
low-order moments, and for the spectral structure of the Fisher information
matrix $I(\theta_*)$.

\begin{theorem}\label{thm:MRA}
For any $L \geq 1$, we have
\[\trdeg(\cR_{\leq 1}^\G)=1, \quad
\trdeg(\cR_{\leq 2}^\G)=L+1, \quad
\trdeg(\cR_{\leq 3}^\G)=\trdeg(\cR^\G)=2L=d-1.\]
\end{theorem}

\begin{corollary} \label{cor:MRA-cor}
A generic signal $\theta_* \in \R^d$ in this
continuous MRA model has the following properties:
\begin{enumerate}
\item[(a)] $\theta_*$ is identified up to a finite list
of orbits by the moments of $g \cdot \theta_*$ up to order $K = 3$ when $L \geq
2$, and order $K=2$ when $L=1$.

\item[(b)] For $(\theta_*, \G)$-dependent constants $C, c > 0$ independent of
$\sigma$, the Fisher information $I(\theta_*)$ has
$d_0=1$ eigenvalue of 0 and $d_k$
eigenvalues in $[c \sigma^{-2k}, C \sigma^{-2k}]$ for
$k = 1,2,3$ and $(d_1,d_2,d_3) = (1, L, L-1)$.
\end{enumerate}
\end{corollary}

\noindent
Part (a) of this corollary follows immediately from
\cite[Theorem 4.9]{bandeira2017estimation} (which we review in Appendix
\ref{appendix:trdeg}), and part (b) follows from Theorem \ref{thm:FI}.

In Appendix \ref{sec: proof SO2 unprojected}, we provide a proof of
Theorem \ref{thm:MRA} using our general result of
Lemma \ref{lem:trdeg}, as a warm-up for our analyses of the
$\SO(3)$-rotational models to follow. We note that for a similar
observation model of continuous MRA studied in \cite{bandeira2020optimal}, a
stronger form of Corollary \ref{cor:MRA-cor}(a) is already known,
namely that $3^\text{rd}$-order
moments are sufficient to identify generic signals $\theta_* \in \R^d$ up
to a single unique orbit.

Next, we study the moment optimization problems of
(\ref{eq:momentoptunprojected}), and we describe more
explicit forms for these optimization problems in this continuous MRA example.
Denote the Fourier coefficients of the true function $f$
by $\theta_* \in \R^d$. Define the complex Fourier coefficients
\[u^{(0)}(\theta)=\theta^{(0)} \in \R, \qquad
u^{(l)}(\theta)=\theta_1^{(l)}+\i \theta_2^{(l)}=r_l(\theta)e^{\i
\lambda_l(\theta)} \in \C,\]
where $(r_l(\theta),\lambda_l(\theta))$ for $l \geq 1$
are the magnitude and phase of $u^{(l)}(\theta)$.
Write as shorthand
\[r_{l,l',l''}(\theta)=r_l(\theta)r_{l'}(\theta)r_{l''}(\theta),
\qquad \lambda_{l,l',l''}(\theta)=\lambda_l(\theta)-\lambda_{l'}(\theta)
-\lambda_{l''}(\theta).\]
Here $\lambda_{l,l',l''}(\theta)$ are the elements of the Fourier bispectrum
of $\theta$.

\begin{theorem}\label{thm:MRA-mom}
For any $L \geq 1$,
{\footnotesize
\begin{align*}
s_1(\theta)&=\frac{1}{2}\Big(\theta^{(0)}-\theta_*^{(0)}\Big)^2\\
s_2(\theta)&
=\frac{1}{4}\Big((\theta^{(0)})^2-(\theta_*^{(0)})^2\Big)^2
+\frac{1}{8}\sum_{l=1}^L \Big(r_l(\theta)^2-r_l(\theta_*)^2\Big)^2\\
s_3(\theta)&=\frac{1}{48}\Big((u^{(0)}(\theta))^3-(u^{(0)}(\theta_*))^3\Big)^2
+\frac{1}{16}\mathop{\sum_{l,l',l''=0}^L}_{l=l'+l''}
\Big|u^{(l)}(\theta)\overline{u^{(l')}(\theta) u^{(l'')}(\theta)} -
u^{(l)}(\theta_*)\overline{u^{(l')}(\theta_*)u^{(l'')}(\theta_*)} \Big|^2\\
&=\frac{1}{12}\Big((\theta^{(0)})^3-(\theta_*^{(0)})^3\Big)^2
+\frac{1}{8}\sum_{l=1}^L
\Big(\theta^{(0)} \cdot r_l(\theta)^2-\theta_*^{(0)} \cdot
r_l(\theta_*)^2\Big)^2\\
&\hspace{0.2in}+\frac{1}{16}\mathop{\sum_{l,l',l''=1}^L}_{l=l'+l''}\bigg(
r_{l,l',l''}(\theta)^2+r_{l,l',l''}(\theta_*)^2
-2r_{l,l',l''}(\theta)r_{l,l',l''}(\theta_*)
\cos\big(\lambda_{l,l',l''}(\theta_*)-\lambda_{l,l',l''}(\theta)\big)\bigg).
\end{align*}}
\end{theorem}

Since each moment variety $\cV_k(\theta_*)$ in (\ref{eq:Vk}) is
precisely the set of points $\{\theta \in \R^d:
s_1(\theta)=0,\ldots,s_k(\theta)=0\}$, this implies also that
\[\cV_0(\theta_*)=\R^d, \qquad
\cV_1(\theta_*)=\{\theta:\theta^{(0)}=\theta_*^{(0)}\},\]
\[\cV_2(\theta_*)=\{\theta:\theta^{(0)}=\theta_*^{(0)} \text{ and }
r_l(\theta)=r_l(\theta_*) \text{ for each } l=1,\ldots,L\}.\]
Thus the minimization of $s_1(\theta)$ on $\cV_0(\theta_*)$ is over the
global function mean $\theta^{(0)}$, the minimization of $s_2(\theta)$ on
$\cV_1(\theta_*)$ is over the Fourier power spectrum
$\{r_l(\theta):l=1,\ldots,L\}$, and the minimization of $s_3(\theta)$ on
$\cV_2(\theta_*)$ is over the Fourier bispectrum
$\{\lambda_{l,l',l''}(\theta):l=l'+l''\}$.

In high noise, minimizing the
population log-likelihood function $R(\theta)$ becomes similar to successively
minimizing $s_1(\theta)$, $s_2(\theta)$, and $s_3(\theta)$.
The following result describes the nature of these three
optimization landscapes.

\begin{theorem}\label{thm:MRAlandscape}
For any $L \geq 1$ and generic $\theta_* \in \R^d$, the minimizations 
of $s_1(\theta)$ over $\cV_0(\theta_*)$ and of $s_2(\theta)$ over
$\cV_1(\theta_*)$ are globally benign. However, for any $L \geq 30$, there
exists a non-empty open subset $U \subset \R^d$ such that
for any $\theta_* \in U$, the minimization of $s_3(\theta)$ over
$\cV_2(\theta_*)$ has a local minimizer outside $\orbit_{\theta_*}$ that
is non-degenerate up to orbit.
\end{theorem}

The correspondence between optimization landscapes shown in
Theorem \ref{thm:landscape}(a) then implies that,
for the class of signals $\theta_* \in U$ described in Theorem
\ref{thm:MRAlandscape} and in sufficiently high noise,
the landscape of the population negative log-likelihood function $R(\theta)$
must also have spurious local minimizers near those of $s_3(\theta)$.
The particular local minimizers of $s_3(\theta)$ that we
exhibit in the proof of Theorem \ref{thm:MRAlandscape} correspond to
certain Fourier phase shifts of the true signal. This example is somewhat
analogous to the spurious local minimizers discovered in
dimensions $d \geq 53$ for the log-likelihood landscape of
discrete MRA in \cite[Section 4.6]{fan2020likelihood}.

We conjecture, based on the algebraic similarities between these
models, that spurious local minimizers of $R(\theta)$ may also exist for
generic $\theta_* \in \R^d$ in the $\SO(3)$-rotational models to be discussed in
Section \ref{sec:func-est-so3}, and we leave this as an open question.

\section{Spherical registration and cryo-EM} \label{sec:func-est-so3}

We now describe examples of estimating a function in 2 or 3
dimensions, observed under $\SO(3)$ rotations of its domain.
Section \ref{subsec:sphericalregistration}
studies estimation on the sphere, Section \ref{subsec:cryoEM} studies
estimation in $\R^3$, and Section \ref{subsec:projectedcryoEM} studies a
simplified ``cryo-EM model'' of estimation in $\R^3$ with a tomographic
projection onto a 2-dimensional plane.

\subsection{Spherical registration}\label{subsec:sphericalregistration}

Let $\sS^2 \subset \R^3$ be the unit sphere, and let $f:\sS^2 \to \R$ be a
function on this sphere. We parametrize $\sS^2$ by the latitude $\phi_1 \in
[0,\pi]$ and longitude $\phi_2 \in [0,2\pi)$. Writing
$f_\frakg(\phi_1,\phi_2)=f(\frakg^{-1} \cdot
(\phi_1,\phi_2))$ for the rotation of the function $f$, we
consider the observation model with samples
\[f_\frakg(\phi_1,\phi_2)\,\der(\phi_1,\phi_2)+\sigma\,\der W(\phi_1,\phi_2)\]
where $\frakg \in \SO(3)$ is a uniform random rotation for each sample,
$\der(\phi_1,\phi_2)=\sin\phi_1\,\der \phi_1\,\der \phi_2$ denotes
the surface area measure on $\sS^2$, and $\der W(\phi_1,\phi_2)$ is
a standard Gaussian white noise process on $\sS^2$. This observation model may
be understood as observing a realization of the Gaussian process
$\{\int h(\phi_1,\phi_2)[f_\frakg(\phi_1,\phi_2)\der(\phi_1,\phi_2)
+\sigma\,\der W(\phi_1,\phi_2)]\}_{h \in L_2(\cS^2)}$
defined analogously to (\ref{eq:GPwhitenoise}),
or equivalently, as observing each coefficient of $f_\frakg$ in an
orthonormal basis of $L_2(\cS^2)$ with i.i.d.\ $\N(0,\sigma^2)$ noise.

We choose as our orthonormal basis the real spherical harmonics
\[h_{lm}(\phi_1, \phi_2) \quad \text{ for } l=0,1,2,\ldots \text{ and }
m=-l,-l+1,\ldots,l-1,l.\]
We assume that
$f:\sS^2 \to \R$ has a finite bandlimit $L \geq 1$ in this basis, i.e.\ it
takes the form
\begin{equation}\label{eq:bandlimitedspherical}
f(\phi_1,\phi_2)=\sum_{l=0}^L \sum_{m=-l}^l \theta_m^{(l)}h_{lm}(\phi_1,\phi_2).
\end{equation}
We may then represent $f$ by its vector of real spherical harmonic coefficients
\[\theta=(\theta_m^{(l)}:\,l=0,\ldots,L \text{ and } m=-l,\ldots,l) \in \R^d,
\qquad d=(L+1)^2.\]
This subspace of bandlimited functions is closed under
$\SO(3)$-rotations of $\cS^2$, and we review the forms of $h_{lm}$ and
of the rotational action on the basis coefficients in Appendix \ref{sec:sr}.

The following result describes the decomposition of total dimension in Theorem
\ref{thm:FI}(a) for bandlimits $L \geq 10$.

\begin{theorem} \label{thm:S2registration}
For any $L \geq 10$, we have
\begin{align*}
\trdeg(\cR_{\leq 1}^\G)=1, \quad
\trdeg(\cR_{\leq 2}^\G)=L+1, \quad
\trdeg(\cR_{\leq 3}^\G)=\trdeg(\cR^\G)=d-3.
\end{align*}
\end{theorem}

\begin{corollary} \label{cor:S2registration}
A generic signal $\theta_* \in \R^d$ in
this spherical registration model for $L \geq 10$ has the following properties:
\begin{enumerate}
\item[(a)] $\theta_*$ may be identified up to a finite list
of orbits by the moments of $g \cdot \theta_*$ up to order $K = 3$.

\item[(b)] For $(\theta_*, \G)$-dependent constants $C, c > 0$ independent of
$\sigma$, the Fisher information $I(\theta_*)$ has $d_0 = 3$
eigenvalues of 0 and $d_k$
eigenvalues in $[c \sigma^{-2k}, C \sigma^{-2k}]$ for
$k = 1, 2, 3$ and $(d_1, d_2, d_3) = (1, L, L(L + 1) - 3)$.
\end{enumerate}
\end{corollary}
%\begin{proof}
%This follows from Theorem \ref{thm:S2registration} by applying
%\cite[Theorem 4.9]{bandeira2017estimation} and Theorem \ref{thm:FI}(a).
%\end{proof}

\begin{remark}
The result of Theorem \ref{thm:S2registration} was conjectured for all
bandlimits $L \geq 10$ in \cite[Conjecture 5.6]{bandeira2017estimation}, and
it was verified numerically in exact-precision arithmetic for $L \in \{10,
\ldots, 16\}$. Our result resolves this conjecture for all $L \geq 10$.
Conversely, for low bandlimits $L \leq 9$, it was shown
in \cite[Section 5.4]{bandeira2017estimation} that $K > 3$ strictly, meaning
that moments up to $3^\text{rd}$ order are insufficient to locally identify
$\theta_*$ up to its orbit.
\end{remark}

Turning to the forms of $s_k(\theta)$ in (\ref{eq:momentoptunprojected}),
let us denote the real spherical harmonic coefficients of the true function
by $\theta_* \in \R^d$. We write as shorthand
\[u^{(l)}(\theta)=(u_m^{(l)}(\theta):m=-l,\dots,l) \in \C^{2l+1}\]
for the complex spherical harmonic coefficients at frequency $l$, which
are defined from the real coefficients $(\theta_m^{(l)}:m=-l,\ldots,l)$ by a
unitary transform described in (\ref{eq:u-theta}). We denote
\begin{equation}\label{eq:Bl}
B_{l,l',l''}(\theta)=
\mathop{\sum_{m=-l}^l \sum_{m'=-l'}^{l'}
\sum_{m''=-l''}^{l''}}_{m''=m+m'}
\langle l,m;l',m'|l'',m'' \rangle
\overline{u_m^{(l)}(\theta)u_{m'}^{(l')}(\theta)} u_{m''}^{(l'')}(\theta)
\end{equation}
where $\langle l,m;l',m'|l'',m'' \rangle \in \R$ is the Clebsch-Gordan
coefficient. These quantities express the
integrals of three-fold products of spherical harmonics over $\cS^2$ and
arise naturally in the computation of $3^\text{rd}$-order
moments of $g \cdot \theta$.
We review their definition in Appendix \ref{sec:rc-harmonic}. The functions
$B_{l,l',l''}(\theta)$ are analogous to the scaled components
$r_{l,l',l''}(\theta)\lambda_{l,l',l''}(\theta)$ of the Fourier bispectrum
that appeared in the 1-dimensional MRA example of Section \ref{sec:func-est-so2}.
The minimizations of $s_1(\theta)$, $s_2(\theta)$, and $s_3(\theta)$
described in Theorem \ref{thm:landscape} may then
be analogously understood as minimizing the global function mean, the
power in each spherical harmonic frequency, and certain ``bispectrum'' variables
for each frequency.

\begin{theorem} \label{thm:S2registration-mom}
For any $L \geq 1$,
\begin{align*}
s_1(\theta)&=\frac{1}{2}\Big(u^{(0)}(\theta)-u^{(0)}(\theta_*)\Big)^2\\
s_2(\theta)&=\frac{1}{4}\sum_{l=0}^L \frac{1}{2l+1}\Big(\|u^{(l)}(\theta)\|^2
-\|u^{(l)}(\theta_*)\|^2\Big)^2\\
s_3(\theta)&=\frac{1}{6}\mathop{\sum_{l,l',l''=0}^L}_{|l-l'| \leq l''
\leq l+l'} \frac{1}{2l''+1}\Big|B_{l,l',l''}(\theta)-B_{l,l',l''}(\theta_*)\Big|^2.
\end{align*}
\end{theorem}

We prove Theorems \ref{thm:S2registration} and
\ref{thm:S2registration-mom} in Appendix \ref{sec:sr}.
Here, let us describe the high-level proof idea for Theorem
\ref{thm:S2registration}, which is used
also in our analyses of the cryo-EM models to follow. By Lemma
\ref{lem:trdeg}, it suffices to analyze the ranks of the Hessians $\nabla^2
s_1(\theta_*)$, $\nabla^2 s_2(\theta_*)$, and $\nabla^2 s_3(\theta_*)$ at a
generic point $\theta_* \in \R^d$. This analysis is straightforward for
$s_1,s_2$, and the core of
the proof is to show that $\nabla^2 s_3(\theta_*)$ has full rank $d-3$
(which accounts for the 3-dimensional orbit of $\theta_*$)
when $L \geq 10$.

Importantly, for any matrix $M(\theta)$ that is analytic in $\theta$,
we have $\rank(M(\theta))<k$ if and only if
every $k \times k$ submatrix of $M(\theta)$ has determinant 0. Because the
$k \times k$ minors are themselves analytic in $\theta$, 
this holds either for all $\theta \in \R^d$, or only for $\theta$ outside a
generic subset of $\R^d$. This implies the following fact.

\begin{fact}\label{fact:fullrank}
For any $k \geq 1$ and matrix $M(\theta)$ whose entries are
analytic in $\theta$, we have $\rank M(\theta_*) \geq k$ for generic points
$\theta_* \in \R^d$ if and only if there exists at least one point
$\theta_* \in \R^d$ for which this inequality holds.
\end{fact}

\noindent Thus, to show that
$\rank(\nabla^2 s_3(\theta_*)) \geq d-3$ for generic $\theta_* \in \R^d$,
it suffices to construct a single point
$\theta_* \in \R^d$ where this holds. We do this by analyzing the explicit
form of $\nabla^2 s_3(\theta_*)$ derived from
Theorem \ref{thm:S2registration-mom}. For $L=10$, we exhibit such a point
$\theta_*$ numerically. We then use this as a base case to inductively
construct $\theta_*$ for all $L \geq 10$, by carefully choosing certain
coordinates of $\theta_*$ to be 0 so that
$\nabla^2 s_3(\theta_*)$ has a sparse structure and its rank
may be explicitly deduced from the ranks of $2 \times 2$ submatrices.

\subsection{Unprojected cryo-EM}\label{subsec:cryoEM}

Consider now a function $f:\R^3 \to \R$, and the action of $\SO(3)$ on 
$\R^3$ given by rotation about the origin. Write 
$f_\frakg(x)=f(\frakg^{-1} \cdot x)$ for the rotated function.
We consider the observation model with samples
\[f_\frakg(x)\,\der x+\sigma\,\der W(x)\]
where $\frakg \in \SO(3)$ is uniformly random for each sample, and $\der W(x)$
is a standard Gaussian white noise process on $\R^3$.
This is an unprojected model of the
single-particle reconstruction problem in cryo-EM, to which we will add a
tomographic projection in the next section. This model may be of
independent interest for applications to cryo-ET, described in Appendix
\ref{sec:cryoET}.

We model $f$ using a basis representation for its Fourier transform
$\hat{f}:\R^3 \to \C$, similar to the approach of
\cite[Section 5.5]{bandeira2017estimation}.
We parametrize the Fourier domain $\R^3$ by
spherical coordinates $(\rho,\phi_1,\phi_2)$ with radius
$\rho\geq 0$, latitude $\phi_1\in[0,\pi]$ and longitude $\phi_2\in[0,2\pi)$, and decompose
$\hat{f}(\rho,\phi_1,\phi_2)$ in a complex basis $\{\hat{j}_{lsm}\}$ given by
the product of the complex spherical harmonics $y_{lm}(\phi_1,\phi_2)$
(reviewed in Appendix \ref{sec:rc-harmonic})
with radial functions $z_s(\rho)$:
\begin{equation}\label{eq:C3basis}
\hat{j}_{lsm}(\rho,\phi_1,\phi_2)=z_s(\rho)y_{lm}(\phi_1,\phi_2)
\quad \text{ for } s \geq 1, \quad l \geq 0, \quad m \in \{-l,\ldots,l\}.
\end{equation}
Here $\{z_s:s \geq 1\}$ may be any system of radial
basis functions $z_s:[0,\infty) \to \R$ satisfying the orthogonality relation
\begin{equation}\label{eq:zorthogonality}
\int_0^\infty \rho^2 z_s(\rho) z_{s'}(\rho)d\rho=\1\{s=s'\},
\end{equation}
so that $\{\hat{j}_{lsm}\}$ are orthonormal over $L_2(\R^3,\C)$.
The inverse Fourier
transforms $\{j_{lsm}\}$ of $\{\hat{j}_{lsm}\}$ then provide
a complex orthonormal basis in the original signal domain of $f$.

Fixing integer bandlimits $L \geq 1$ and $S_0,\ldots,S_L \geq 1$, we define the 
index set
\begin{equation}\label{eq:indices}
\cI=\Big\{(l,s,m):0 \leq l \leq L,\; 1 \leq s \leq S_l,\;-l \leq m \leq l
\Big\}, \qquad d=|\cI|=\sum_{l=0}^L (2l+1)S_l
\end{equation}
and assume that $f$ is \emph{$(L, S_0, \ldots, S_L)$-bandlimited} in the sense
of admitting the finite basis representation
\begin{equation}\label{eq:bandlimitedS2}
f=\sum_{(l,s,m) \in \cI} u_m^{(ls)} \cdot j_{lsm}, \qquad
u=\big(u_m^{(ls)}:(l,s,m) \in \cI\big) \in \C^d.
\end{equation}
This corresponds to modeling the Fourier transform $\hat{f}$ up to the
spherical frequency $L$, and up to the radial frequency $S_l$ for each spherical
component $l=0,1,\ldots,L$. For real-valued functions $f$, writing
$u=\hat{V}^*\theta$ for a
unitary transform $\hat{V} \in \C^{d \times d}$ defined explicitly in
(\ref{eq:hatv-trans}), we then obtain a real sequence representation
\begin{equation}\label{eq:cryoEMbandlimited}
f=\sum_{(l,s,m) \in \cI} \theta_m^{(ls)} \cdot h_{lsm}, \qquad
\theta=\big(\theta_m^{(ls)}:(l,s,m) \in \cI\big) \in \R^d
\end{equation}
for a real-valued orthonormal basis $\{h_{lsm}\}$.
We describe the forms of $h_{lsm}$ and the rotational action on the basis
coefficients $\theta \in \R^d$ in Appendix \ref{sec:unprojCEM}.

The following result describes the decomposition of total dimension in Theorem
\ref{thm:FI}(a), assuming $L \geq 1$ and $S_l \geq 2$ for each
$l=0,\ldots,L$.
%In particular, this verifies that
%$\trdeg \cR_{\leq K}^\G=\trdeg \cR^\G$ for $K=3$.
(Note that the case of
$S_0=\ldots=S_L=1$ would be similar to the spherical registration example of
Section \ref{subsec:sphericalregistration}, and a lower bound of $L \geq 10$
would be needed in this case to ensure $K=3$.) 

\begin{theorem} \label{thm:cryoEM}
For any $L \geq 1$ and $S_0,\ldots,S_L \geq 2$, we have
\begin{align*}
\trdeg(\cR_{\leq 1}^\G)&=S_0\\
\trdeg(\cR_{\leq 2}^\G)&=\sum_{l=0}^L d(S_l),
\qquad d(S_l) \equiv
\begin{cases} \frac{S_l(S_l+1)}{2} & \text{ for } S_l<2l+1\\
(2l+1)(S_l-l) & \text{ for } S_l \geq 2l+1 \end{cases}\\
\trdeg(\cR_{\leq 3}^\G)&=\trdeg(\cR^\G)=d-3.
\end{align*}
\end{theorem}

\begin{corollary} \label{cor:cryoEM}
In this unprojected cryo-EM model with $S_0, \ldots, S_L \geq 2$,
a generic signal $\theta_* \in \R^d$ may be identified up to a finite list
of orbits by the moments of $g \cdot \theta_*$ up to order $K = 3$ if $L \geq
2$, and up to order $K = 2$ if $L = 1$.
\end{corollary}

\begin{remark}
In \cite[Conjecture B.1]{bandeira2017estimation}, the authors conjectured
that a generic signal $\theta_*$ in this model may be identified by
$3^\text{rd}$-order moments up to a \emph{single unique} orbit when $L \geq 1$ and $S_0 = \cdots
= S_L \geq 3$. As discussed in \cite{bandeira2017estimation}, this would
hold if the fraction field of all $\G$-invariant rational functions coincides
with that generated by $\cR_{\leq 3}^\G$. Theorem \ref{thm:cryoEM} shows the
weaker statement that these fraction fields have the same transcendence
degree. Thus Corollary \ref{cor:cryoEM}(a) guarantees only that
$\theta_*$ may be identified up to a finite list of orbits, and we show this
under a slightly weaker requirement that $S_0,\ldots,S_L \geq 2$.
\end{remark}

Turning to the forms of $s_k(\theta)$ that define the moment optimization
(\ref{eq:momentoptunprojected}), write $\theta_* \in \R^d$ for the true
coefficients in the above real basis $\{h_{lsm}\}$. Let
\begin{equation}\label{eq:uls}
u^{(ls)}(\theta)=(u_m^{(ls)}(\theta):m=-l,\ldots,l) \in \C^{2l+1}
\end{equation}
be the components of the \emph{complex} coefficients $u=\hat{V}^*\theta$
for the frequency pair $(l,s)$, and define analogously to (\ref{eq:Bl})
\begin{equation}\label{eq:Bls}
B_{(l,s),(l',s'),(l'',s'')}(\theta)=
\mathop{\sum_{m=-l}^l \sum_{m'=-l'}^{l'} \sum_{m''=-l''}^{l''}}_{m''=m+m'}
\langle l,m;l',m'|l'',m'' \rangle
\overline{u_m^{(ls)}(\theta)u_{m'}^{(l's')}(\theta)}
u_{m''}^{(l''s'')}(\theta).
\end{equation}
When the original function $f:\R^3 \to \R$ is real-valued,
we verify in the proof of Theorem \ref{thm:cryoEM-mom} below that each
$B_{(l,s),(l',s'),(l'',s'')}(\theta)$ is also real-valued.

\begin{theorem}\label{thm:cryoEM-mom}
For any $L \geq 1$ and $S_0,\ldots,S_L \geq 1$,
\begin{align*}
s_1(\theta)&=\frac{1}{2}\sum_{s=1}^{S_0} \Big(u^{(0s)}(\theta)-
u^{(0s)}(\theta_*)\Big)^2\\
s_2(\theta)&=\frac{1}{4}\sum_{l=0}^L \frac{1}{2l+1}
\sum_{s,s'=1}^{S_l} \Big(\langle u^{(ls)}(\theta),u^{(ls')}(\theta) \rangle
-\langle u^{(ls)}(\theta_*),u^{(ls')}(\theta_*) \rangle\Big)^2\\
s_3(\theta)&=\frac{1}{12}\mathop{\sum_{l,l',l''=0}^L}_{|l-l'| \leq l'' \leq
l+l'} \frac{1}{2l''+1}\sum_{s=1}^{S_l}\sum_{s'=1}^{S_{l'}}\sum_{s''=1}^{S_{l''}}
\Big(B_{(l,s),(l',s'),(l'',s'')}(\theta)-B_{(l,s),(l',s'),(l'',s'')}(\theta_*)\Big)^2.\label{eq:S3cryoEM}
\end{align*}
\end{theorem}

In this model, the optimization of $s_1(\theta)$ is over the mean
component $u^{(0s)}(\theta)$ corresponding to each radial frequency $s$.
The optimization of $s_2(\theta)$ is over not just the
power $\|u^{(ls)}(\theta)\|^2$ within each frequency pair $(l,s)$, but also
the cross-correlations between $u^{(ls)}$ and $u^{(ls')}$ for different radial
frequencies $s$ and $s'$.

The proofs of Theorems \ref{thm:cryoEM} and \ref{thm:cryoEM-mom} are
deferred to Appendix \ref{sec:unprojCEM}. The argument for Theorem \ref{thm:cryoEM} is similar to
that of Theorem \ref{thm:S2registration}: When $S_0,\ldots,S_L \geq 2$, the
claim that $\rank(\nabla^2 s_3(\theta_*)) \geq d-3$ may be established by
induction on $L$ down to the base case of $L=1$ rather than $L=10$, using a
different construction of the point $\theta_* \in \R^d$ that induces a sparse
structure in $\nabla^2 s_3(\theta_*)$.

\subsection{Projected cryo-EM}\label{subsec:projectedcryoEM}

We now extend the model of the preceding section to include the
tomographic projection arising in cryo-EM. In this projected model, we
observe samples
\begin{equation}\label{eq:projectedobservations}
(\Pi \cdot f_\frakg)(x)\der x+\sigma\,\der W(x)
\end{equation}
on $\R^2$ where, for $x=(x_1,x_2) \in \R^2$, the tomographic projection $\Pi$ is
defined by
\begin{equation}\label{eq:cryoEMmodel}
(\Pi \cdot f_\frakg)(x_1,x_2)=
\int_{-\infty}^\infty f_{\frakg}(x_1,x_2,x_3)\der x_3,
\end{equation}
and $\der W(x)$ in (\ref{eq:projectedobservations}) is a
standard Gaussian white noise process on the projected domain $\R^2$.

Our model setup is similar to \cite[Section 5.5 and Appendix A.4]{bandeira2017estimation}.
We again model the Fourier transform of $f$ in a basis $\{\hat{j}_{lsm}\}$
given by the product of complex spherical harmonics with radial functions.
We restrict $f$ to a space of $(L,S_0,\ldots,S_L)$-bandlimited functions
with representation
\begin{equation}\label{eq:projcryoEMbandlimited}
f=\sum_{(l,s,m) \in \cI} u_m^{(ls)} \cdot j_{lsm}
=\sum_{(l,s,m) \in \cI} \theta_m^{(ls)} \cdot h_{lsm}
\end{equation}
for the index set $\cI$ defined in (\ref{eq:indices}), where
$\{j_{lsm}\}$ are the inverse Fourier transforms of $\{\hat{j}_{lsm}\}$, and
the second equality describes a parametrization by an equivalent real
orthonormal basis $\{h_{lsm}\}$ as before.
In Appendix \ref{sec:projCEM}, we apply the Fourier slice theorem to derive
basis representations for the tomographic projection $\Pi \cdot f$. These take
the forms
\begin{equation}\label{eq:projbandlimited}
\Pi \cdot f=\sum_{(s,m) \in \tcI} \tu_m^{(s)} j_{sm}
=\sum_{(s,m) \in \tcI} \ttheta_m^{(s)} h_{sm},
\end{equation}
where $\{j_{sm}\}$ and $\{h_{sm}\}$ are (complex and real, resp.) basis
functions over $\R^2$, and $\Pi \cdot f$ is bandlimited to an index set
\begin{equation}\label{eq:indicesproj}
\tcI=\Big\{(s,m):1 \leq s \leq S,\; -L \leq m \leq L\Big\},
\qquad \td=|\tcI|=S(2L+1)
\end{equation}
for $S=\max(S_0,\ldots,S_L)$. This expresses $\Pi$ as a linear map from
$\theta \in \R^d$ to $\ttheta \in \R^{\td}$, and we give its
explicit form in (\ref{eq:cryoEMPi}).
We choose radial functions to ensure that the basis $\{h_{sm}\}$ is
orthonormal in $L_2(\R^2)$, so that
(\ref{eq:projectedobservations}) is equivalent to observing the
coefficients of $\Pi \cdot f$ in this basis
with i.i.d.\ $\N(0,\sigma^2)$ noise.
Further details of the setup are described in Appendix~\ref{sec:projCEM}.

The following result verifies that when the bandlimits satisfy $L \geq 1$ and
$S_1,\ldots,S_L \geq 4$, we have also
$\trdeg \tcR_{\leq \tK}^\G=\trdeg \cR^\G$ for $\tK=3$.

\begin{theorem} \label{thm:projectedcryoEM}
For any $L \geq 1$ and $S_0,\ldots,S_L \geq 4$, we have
\begin{align*}
\trdeg(\tcR_{\leq 1}^\G)&=S_0\\
\trdeg(\tcR_{\leq 2}^\G)&=\sum_{l=0}^L d(S_l),
\qquad d(S_l) \equiv
\begin{cases} \frac{S_l(S_l+1)}{2} & \text{ for } S_l<2l+1\\
(2l+1)(S_l-l) & \text{ for } S_l \geq 2l+1 \end{cases}\\
\trdeg(\tcR_{\leq 3}^\G)&=\trdeg(\cR^\G)=d-3,
\end{align*}
which matches the values of $\trdeg(\cR_{\leq 1}^\G)$, $\trdeg(\cR_{\leq 2}^\G)$, 
and $\trdeg(\cR_{\leq 3}^\G)$ in the unprojected setting of Theorem
\ref{thm:cryoEM}.
\end{theorem}

\begin{corollary} \label{cor:proj-cryoEM}
In this projected cryo-EM model with $S_0, \ldots, S_L \geq 4$,
a generic signal $\theta_* \in \R^d$ may be identified up to a finite list
of orbits by the moments of $\Pi(g \cdot \theta_*)$ up to order $\tK=3$
if $L \geq 2$, and order $\tK = 2$ if $L = 1$.
\end{corollary}

We prove Theorem \ref{thm:projectedcryoEM} in Appendix
\ref{sec:projCEM}, where we also
state an analogue of Theorem \ref{thm:cryoEM-mom}
that describes the explicit forms of $\ts_k(\theta)$ for $k=1,2,3$
in this projected model.

Our proof of Theorem \ref{thm:projectedcryoEM} again
constructs a point $\theta_* \in \R^d$ where
$\rank(\nabla^2 \ts_3(\theta_*)) \geq d-3$.
However, the form of $\nabla^2 \ts_3(\theta_*)$
now involves the precise form of the projection $\Pi$, and our choice of
$\theta_*$ does not induce sparsity in this Hessian.
Instead, we choose $\theta_*$ to have many coordinates equal to 0,
and track the dependence of minors of $\nabla^2 \ts_3(\theta_*)$ on
the non-zero coordinates of $\theta_*$ to show they are generically
non-vanishing. We give this argument in the proof of
Lemma \ref{lemma:SO3projinduction} in Appendix \ref{sec:projCEM}.

\begin{remark}
Taking $S_0=\ldots=S_L=S$ yields a model equivalent to the projected cryo-EM
model with $S$ spherical shells in \cite[Section 5.5]{bandeira2017estimation}. In
\cite[Conjecture 5.11]{bandeira2017estimation}, the authors conjectured that
a generic signal $\theta_*$ may be identified up to a finite list of orbits by
$3^\text{rd}$-order moments
if $S \geq 2$.  Corollary \ref{cor:proj-cryoEM}(a) thus resolves this
conjecture positively when $S \geq 4$.  The constraint $S \geq 4$ is technical, and
we believe that the conjecture holds as stated for $S \in \{2, 3\}$ as well, but we
do not pursue these cases in this work.
\end{remark}

\section{Numerical evaluations of the Fisher information}\label{sec:simulations}

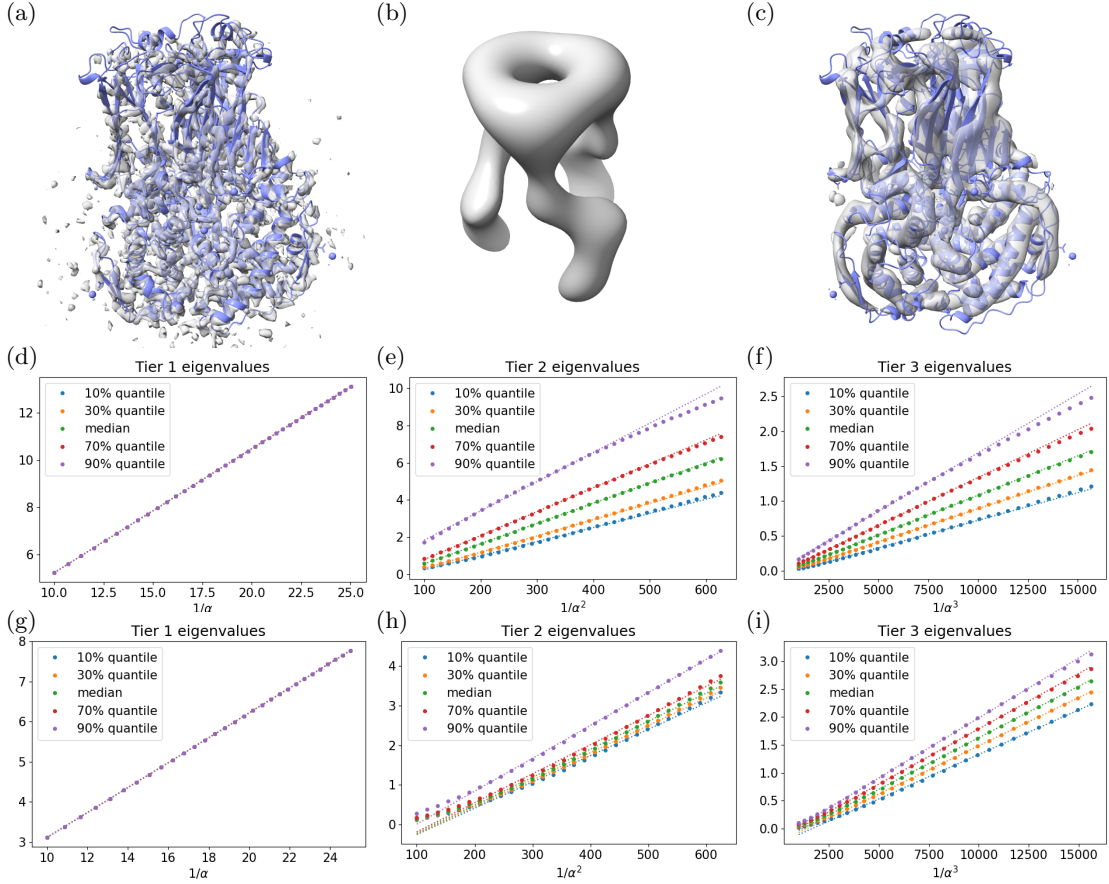
\begin{figure}
\begin{picture}(420,340)
\put(0,200){\includegraphics[width=0.3\textwidth]{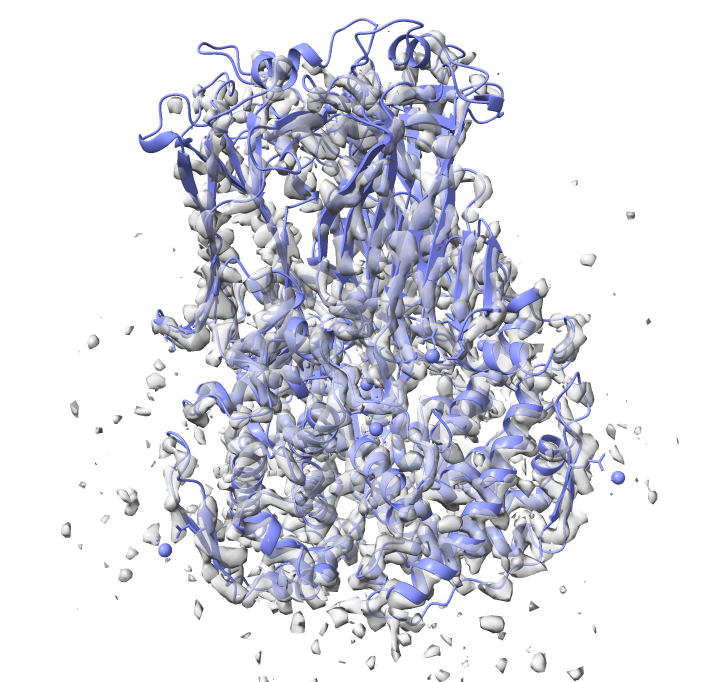}}
\put(140,200){\includegraphics[width=0.3\textwidth]{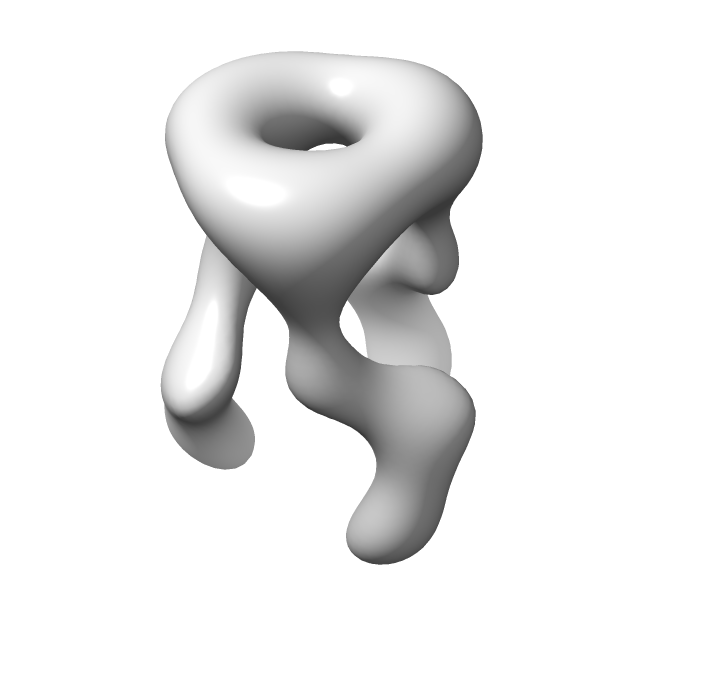}}
\put(280,200){\includegraphics[width=0.3\textwidth]{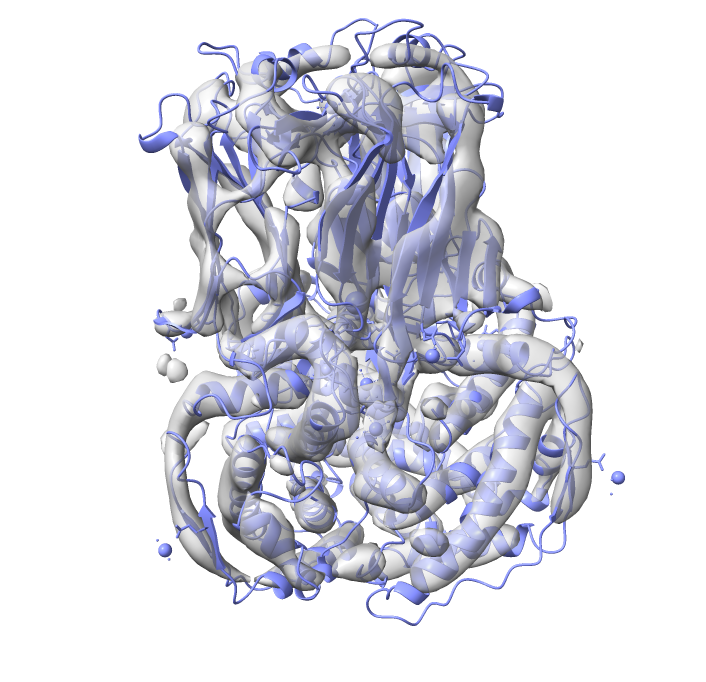}}
\put(0,100){\includegraphics[width=0.3\textwidth]{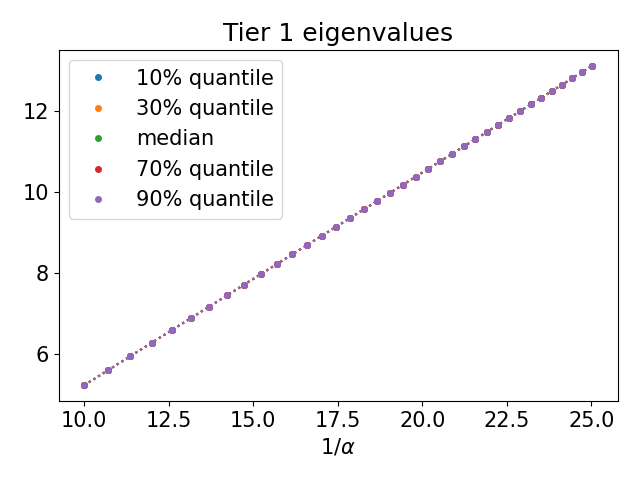}}
\put(140,100){\includegraphics[width=0.3\textwidth]{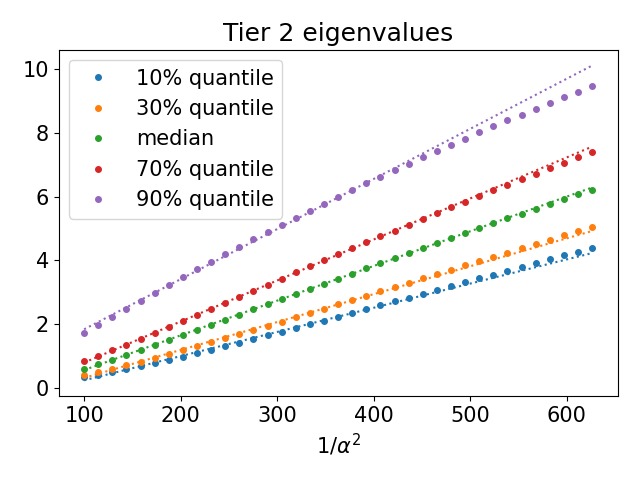}}
\put(280,100){\includegraphics[width=0.3\textwidth]{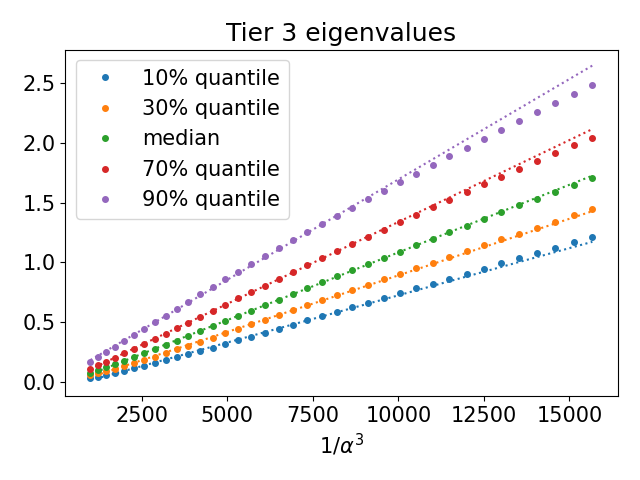}}
\put(0,0){\includegraphics[width=0.3\textwidth]{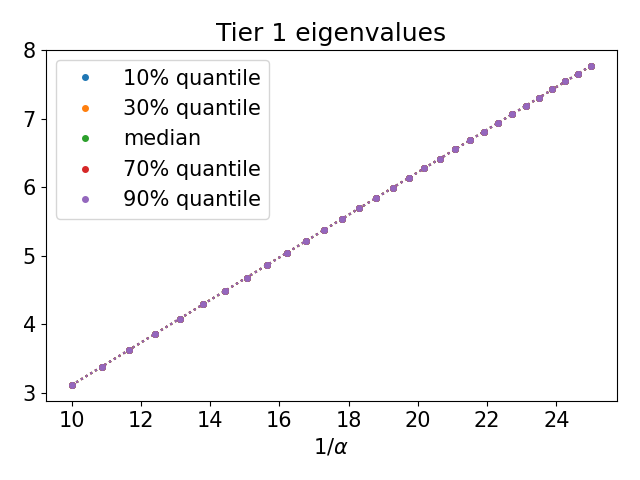}}
\put(140,0){\includegraphics[width=0.3\textwidth]{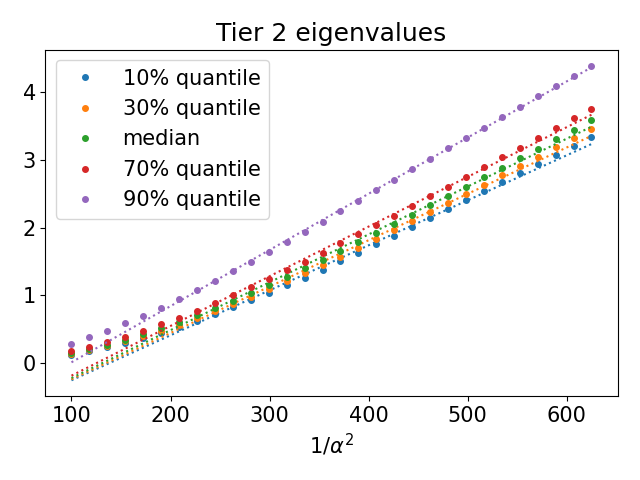}}
\put(280,0){\includegraphics[width=0.3\textwidth]{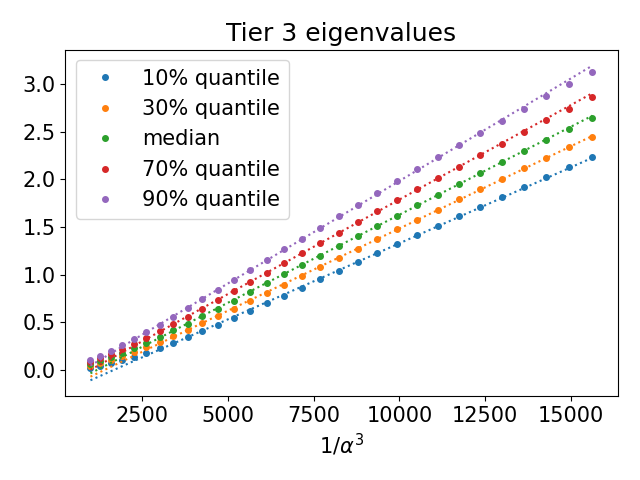}}
\put(0,330){{\small(a)}}
\put(140,330){{\small(b)}}
\put(280,330){{\small(c)}}
\put(0,200){{\small(d)}}
\put(140,200){{\small(e)}}
\put(280,200){{\small(f)}}
\put(0,100){{\small(g)}}
\put(140,100){{\small(h)}}
\put(280,100){{\small(i)}}
\end{picture}
\caption{(a) 3.8\AA-resolution cryo-EM map of the
rotavirus VP6 trimer, overlaid with the atomic structure. (b) A
finite-dimensional approximation using 405 basis functions at 24.6\AA-resolution
(displayed in a rotated orientation for clarity). (c) An
approximation using 4410 basis functions at 8.2\AA-resolution. (d--f)
We stratify the eigenvalues of the 405-dimensional observed Fisher information
corresponding to (b) into three ``eigenvalue tiers'' according to
Theorem \ref{thm:cryoEM}, and plot the scalings of the $10^\text{th}$,
$30^\text{th}$, $50^\text{th}$, $70^\text{th}$, and $90^\text{th}$
percentiles of eigenvalues in each tier against
$1/\alpha \propto \sigma^{-2}$, $1/\alpha^2 \propto \sigma^{-4}$, and
$1/\alpha^3 \propto \sigma^{-6}$. (These quantiles nearly overlap for Tier 1.)
Linear trends fitted using least squares are shown as dashed lines. 
(g--i) The same for the 4410-dimensional Fisher information matrix
corresponding to (c).}\label{fig:rotavirus}
\end{figure}

We conclude with an empirical investigation of the spectrum of the
Fisher information matrix in two simulated examples of the unprojected cryo-EM
model described in Section \ref{subsec:cryoEM}.

In each example, we begin with a near-atomic-resolution electric potential map
estimated from a cryo-EM experiment. We obtain a
finite-dimensional approximation to this map by applying a low-pass filter to
its Fourier transform, followed by a basis approximation for the filtered map.
We simulate noisy and rotated samples 
using this finite-dimensional approximation as the underlying true signal,
for various inverse-SNR parameters
\[\alpha \equiv \sigma^2/\|\theta_*\|^2.\]
We then study the dependence of eigenvalues of the
observed information matrix $\nabla^2 R_n(\theta_*)$ on $\alpha$.\\

\noindent {\bf Rotavirus VP6 trimer.}
We consider a map of the VP6 trimer in bovine rotavirus,
reported in \cite{zhang2008near} (EMDB:1461). A contour plot
of this map is overlaid with the atomic structure previously obtained
by \cite{mathieu2001atomic} (PDB:1QHD), in Figure \ref{fig:rotavirus}(a).
We applied low-pass filters in the Fourier domain at two different cutoff
frequencies, a ``low-resolution'' frequency of $(24.6\text{\AA})^{-1}$ and a
``medium-resolution'' frequency of $(8.2\text{\AA})^{-1}$. The corresponding
smoothed maps in the spatial domain are depicted in Figure \ref{fig:lowpass}
of Appendix \ref{appendix:simulations}.

We approximated each smoothed map using a finite basis of the form
(\ref{eq:C3basis}), with an adaptive construction of the radial functions
$\{z_s\}$ to maximize the power captured by
each successive radial frequency. Details of our numerical
procedures are described in Appendix \ref{appendix:simulations}.
Choosing bandlimits $(S,L)=(5,8)$ and total dimension
$d=L(S+1)^2=405$ gave an accurate approximation to the 
$24.6\text{\AA}$-resolution map that reveals the trimer composition of the VP6
complex, as depicted in Figure \ref{fig:rotavirus}(b).
Choosing bandlimits $(S,L)=(10,20)$ and total dimension
$d=L(S+1)^2=4410$ gave an accurate approximation of the
$8.2\text{\AA}$-resolution map that captures interesting aspects of the
tertiary and secondary structure, as shown in Figure \ref{fig:rotavirus}(c). 
We denote the basis coefficients of these approximated
maps as $\theta_* \in \R^d$.

We computed the Hessians $\nabla^2 R_n(\theta_*)$ of the empirical negative
log-likelihood functions from $n=500{,}000$ simulated samples, with inverse-SNR
$\alpha=\sigma^2/\|\theta_*\|^2 \in [0.04,0.10]$. We then
separated the largest $d-3$ eigenvalues of $\nabla^2 R_n(\theta_*)$ into three
``tiers'' with cardinalities $(d_1,d_2,d_3)$ as implied by
Theorem \ref{thm:cryoEM}.
Figure \ref{fig:rotavirus}(d--f) depicts representative eigenvalues in each
tier, plotted against $1/\alpha \propto \sigma^{-2}$, $1/\alpha^2 \propto
\sigma^{-4}$, and $1/\alpha^3 \propto \sigma^{-6}$.
A linear trend is observed in all
settings, in agreement with the prediction of Theorem \ref{thm:FI}.
This may be contrasted with Figure \ref{fig:alltiers} in Appendix
\ref{appendix:simulations}, which instead plots eigenvalues in all three
tiers against $1/\alpha \propto \sigma^{-2}$, and where non-linearity of the
scaling is visually apparent for Tiers 2 and 3.\\

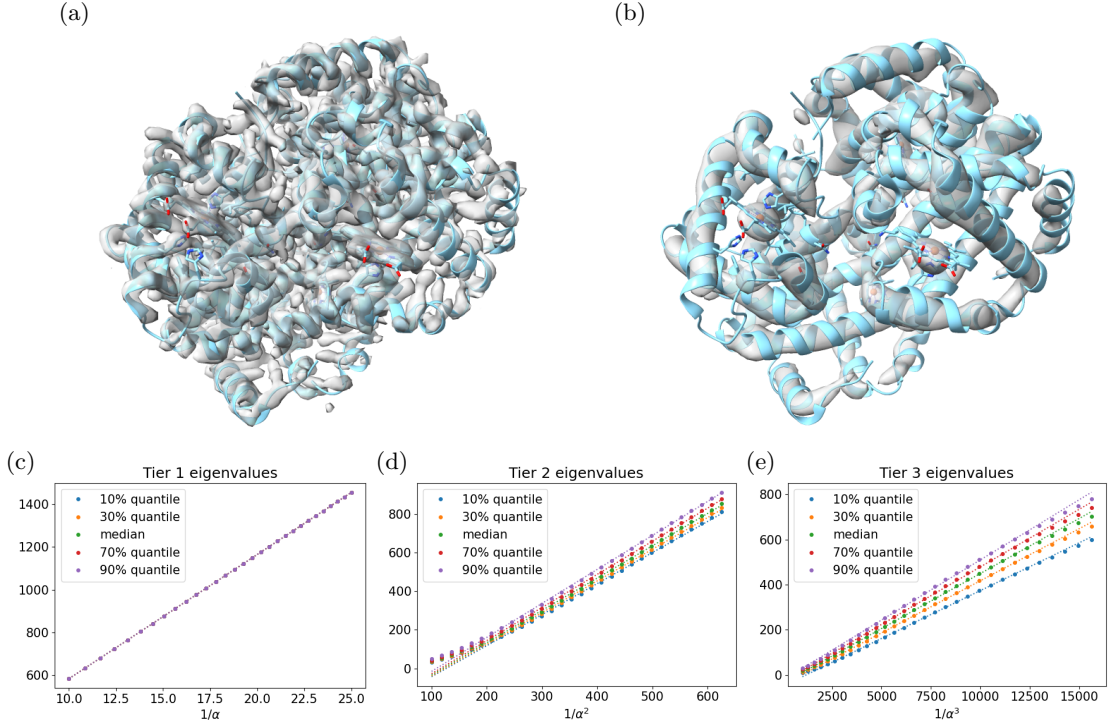
\begin{figure}
\begin{picture}(420,280)
\put(20,100){\includegraphics[width=0.4\textwidth]{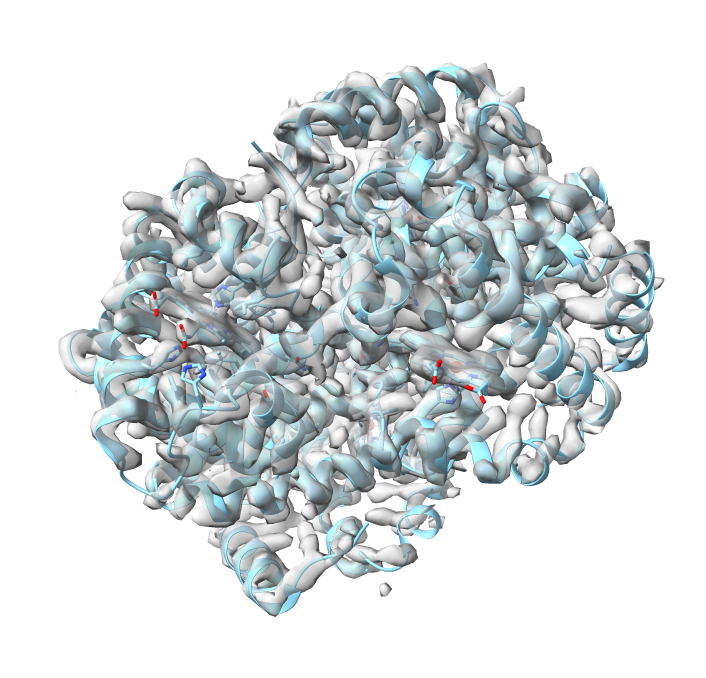}}
\put(230,100){\includegraphics[width=0.4\textwidth]{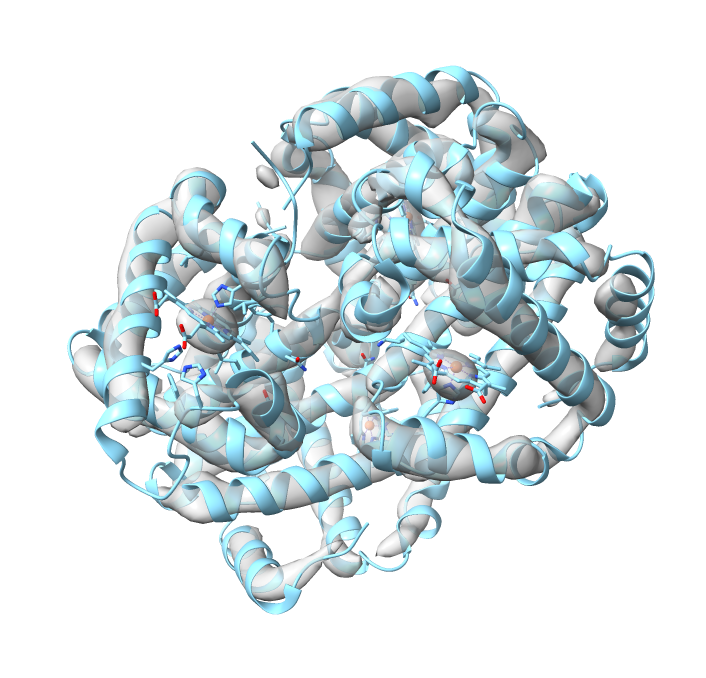}}
\put(0,0){\includegraphics[width=0.3\textwidth]{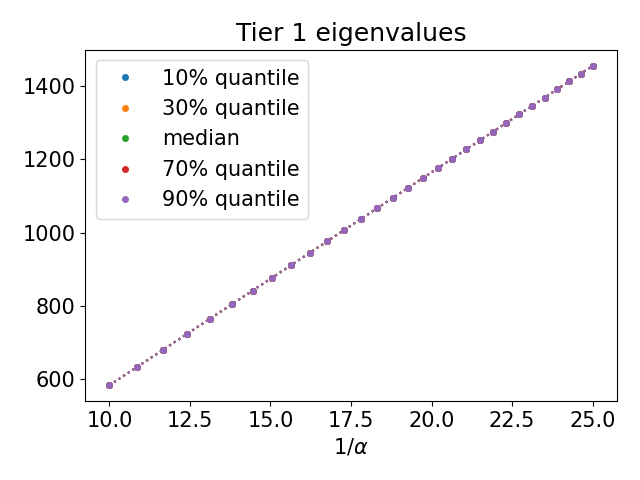}}
\put(140,0){\includegraphics[width=0.3\textwidth]{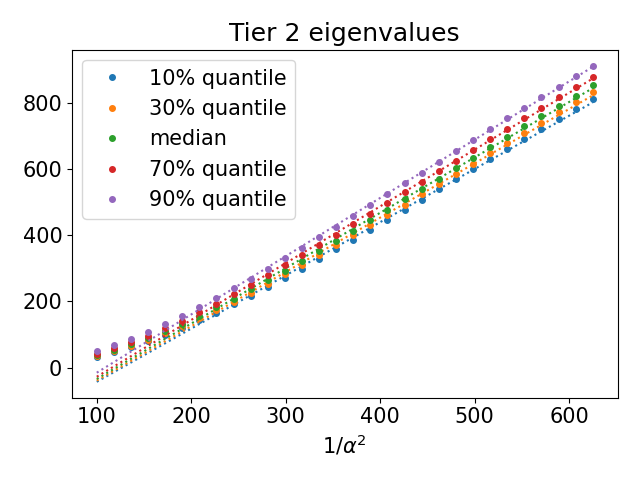}}
\put(280,0){\includegraphics[width=0.3\textwidth]{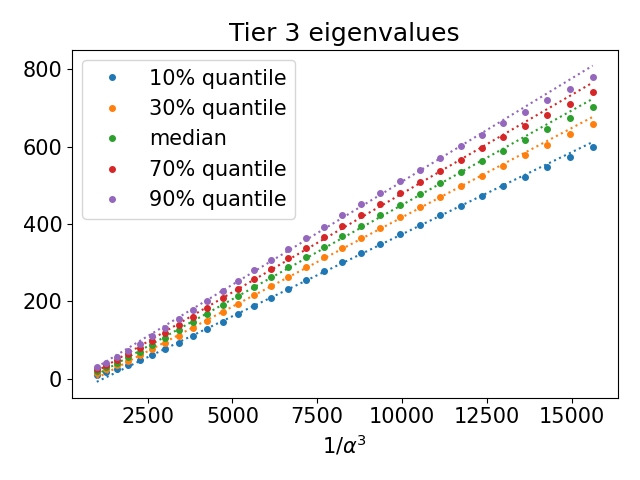}}
\put(20,270){{\small(a)}}
\put(230,270){{\small(b)}}
\put(0,100){{\small(c)}}
\put(140,100){{\small(d)}}
\put(280,100){{\small(e)}}
\end{picture}
\caption{(a) 3.4\AA-resolution cryo-EM map of hemoglobin, overlaid with the
atomic structure. (b) A finite-dimensional approximation using 3528 basis
functions at 7.0\AA-resolution. (c--e) The $10^\text{th}$, $30^\text{th}$,
$50^\text{th}$, $70^\text{th}$, and $90^\text{th}$ percentiles
of eigenvalues within each ``eigenvalue tier'' of the 3528-dimensional
observed Fisher information,
plotted against $1/\alpha \propto \sigma^{-2}$, $1/\alpha^2 \propto
\sigma^{-4}$, $1/\alpha^3 \propto \sigma^{-6}$ as in Figure
\ref{fig:rotavirus}.}\label{fig:hemoglobin}
\end{figure}

\noindent {\bf Hemoglobin.}
We consider a map of hemoglobin, reported in \cite{khoshouei2017cryo} (EMDB:3650, PDB:5NI1). A contour
plot overlaid with the atomic structure is presented in Figure
\ref{fig:hemoglobin}(a). We applied a low-pass filter with cutoff
frequency $(7.0\text{\AA})^{-1}$ in the Fourier domain, depicted
in Figure \ref{fig:lowpass}. We then applied a basis approximation
with bandlimits $(S,L)=(8,20)$ and total dimension $d=3528$.
The approximated map is shown in Figure \ref{fig:hemoglobin}(b), and captures
important aspects of the secondary structure including the locations of the
$\alpha$-helices and embedded prosthetic heme groups.
We denote the basis coefficients of this approximation as $\theta_*$.

Figure \ref{fig:hemoglobin}(c--e) again depicts the leading $d-3$ eigenvalues
of $\nabla^2 R_n(\theta_*)$ computed from $n=500{,}000$ simulated samples,
stratified into three tiers of sizes
$(d_1,d_2,d_3)$. Linear trends with $1/\alpha \propto \sigma^{-2}$,
$1/\alpha^2 \propto \sigma^{-4}$, and $1/\alpha^3 \propto \sigma^{-6}$ are
again observed, and may be contrasted with the non-linear scalings of
eigenvalues in Tiers 2 and 3 with $1/\alpha$ as depicted in Figure
\ref{fig:alltiers}.\\

We note that although the eigenvalues of $\nabla^2 R_n(\theta_*)$
do scale with powers of the SNR $1/\alpha$
according to our theoretical predictions, at any fixed SNR and for basis
dimensions exceeding $d \approx 100$, we do not observe
a clear separation between the eigenvalues of Tier 2 and of Tier 3,
due to the variation in magnitude of eigenvalues corresponding to differing
radial frequencies within each tier.

In these examples, we also begin to observe some deviations from the
predicted eigenvalue scalings at the higher and lower ends of tested SNR.
Deviations in higher basis dimensions $d$ and at lower SNR $1/\alpha$
(seen in Figures \ref{fig:rotavirus}(h--i) and \ref{fig:hemoglobin}(d))
are likely finite-sample effects due to differences
between the observed information matrix $\nabla^2 R_n(\theta_*)$ and the
(population) Fisher information $I(\theta_*)=\nabla^2 R(\theta_*)$. We believe
that deviations at higher SNR $1/\alpha$ (seen in
Figures \ref{fig:rotavirus}(e--f) and \ref{fig:hemoglobin}(e))
reflect a departure of the behavior of the population Fisher
information $I(\theta_*)$ from the predictions of the large-$\sigma$
theoretical regime. Our largest tested SNR $1/\alpha=25$ yields a spectral SNR
(average power of signal / average power of noise at a fixed Fourier radius)
of 0.2--0.4 near the origin of the Fourier domain, which we believe reflects a
level of noise that may be slightly higher than that of modern cryo-EM experiments.

\section{Conclusion}

In this work, we characterized properties of the Fisher information
matrix and log-likelihood function landscape for continuous group orbit
estimation problems in a high noise regime, showing that they are related to
the structure of the invariant algebra of the rotational
group. We applied these results to study several models of
function estimation in finite-dimensional function spaces, in particular
establishing that $3^\text{rd}$-order moment information is sufficient to
locally identify generic signals in these models.

In many interesting applications including single-particle cryo-EM, the target 
function at full spatial resolution may not admit an accurate low-dimensional
approximation. In such settings, our theoretical results may have relevance to
estimating lower-dimensional smoothed approximations of the function. We
demonstrated in simulation that this theory can accurately predict the noise
scalings of the Fisher information eigenvalues for two small protein molecules
over a range of sufficiently high noise, or low SNR. We highlight the
theoretical understanding of likelihood-based estimation in high-dimensional
and infinite-dimensional settings and over a broader range of 
SNR as a question for future work.

\appendix

\section{Proofs for general results on the orbit recovery model}\label{appendix:general}

\subsection{High-noise expansion}

We first provide a form for $s_k(\theta)$ and $\ts_k(\theta)$ that will be more
convenient for later computations. We then
prove Theorem \ref{thm:seriesexpansion}.

\begin{lemma} \label{lem:skform}
The expressions $s_k(\theta)$ and $\ts_k(\theta)$ from (\ref{eq:sk})
and (\ref{eq:tildesk}) have the equivalent forms
\begin{align}
s_k(\theta) &= \frac{1}{2 (k!)} \E_g[\langle \theta, g \cdot \theta\rangle^k
- 2 \langle \theta, g \cdot \theta_* \rangle^k + \langle \theta_*, g \cdot \theta_*\rangle^k]\label{eq:skequiv}\\
\ts_k(\theta) &= \frac{1}{2 (k!)} \E_{g, h}[\langle \Pi\cdot  g\cdot  \theta, \Pi\cdot  h \cdot \theta\rangle^k
- 2 \langle \Pi\cdot  g \cdot \theta, \Pi\cdot  h \cdot \theta_* \rangle^k + \langle \Pi\cdot  g \cdot \theta_*, \Pi \cdot h \cdot \theta_*\rangle^k]\label{eq:tskequiv}
\end{align}
where the expectations are over independent Haar-uniform random
elements $g,h \in \G$.
\end{lemma}
\begin{proof}
For the first statement, expanding the square in
the definition of $s_k$ from (\ref{eq:sk}), we have
\begin{align*}
s_k(\theta)&=\frac{1}{2(k!)}\|T_k(\theta)-T_k(\theta_*)\|_\HS^2\notag\\
&=\frac{1}{2(k!)}\big\|\E_g[(g\cdot\theta)^{\otimes k}]-\E_g[(g\cdot\theta_*)^{\otimes k}]\big\|_\HS^2\notag\\
&=\frac{1}{2(k!)}\E_{g,h}\left[\big\langle (g\cdot\theta)^{\otimes k},
(h\cdot\theta)^{\otimes k}\big\rangle
-2\big\langle (g\cdot\theta)^{\otimes k},(h\cdot\theta_*)^{\otimes k}\big\rangle
+\big\langle (g\cdot\theta_*)^{\otimes k},(h\cdot\theta_*)^{\otimes k}\big\rangle\right]\notag\\
&=\frac{1}{2(k!)}\E_{g,h}\left[\langle g\cdot\theta,h\cdot\theta \rangle^k
-2\langle g\cdot\theta,h\cdot\theta_*\rangle^k+\langle g\cdot\theta_*,h\cdot\theta_*\rangle^k
\right]\notag\\
&=\frac{1}{2(k!)}\E_g\left[\langle \theta,g\cdot\theta \rangle^k
-2\langle \theta,g\cdot\theta_*\rangle^k+\langle \theta_*,g\cdot\theta_*\rangle^k\right].
\end{align*} 
The last step above applies $\langle g \cdot u,h \cdot v \rangle
=\langle u,(g^\top h) \cdot v \rangle$ and the equality in law
$g^\top h \overset{L}{=} g$. The second statement follows similarly from
(\ref{eq:tildesk}), omitting this last step.
\end{proof}

\begin{proof}[Proof of Theorem \ref{thm:seriesexpansion}]
Part (a) follows from specializing part (b) to $\td=d$ and $\Pi=\Id$,
and observing that in this case, the term
$\langle \tilde{T}_k(\theta),P_k(\theta) \rangle$ in
(\ref{eq:seriesexpansionprojected}) also belongs to $\cR_{\leq k-1}^\G$ and
hence may be absorbed into $q_k(\theta)$---see \cite[Proposition
2.3]{katsevich2020likelihood} or \cite[Lemma 4.8]{fan2020likelihood}.

Most of the claims in part (b) follow directly from \cite[Lemma
2.2]{katsevich2020likelihood}: Specializing to our setting (where
$\rho$ and $\rho_*$ in \cite{katsevich2020likelihood} are the distributions of
$\Pi \cdot g \cdot \theta$ and $\Pi \cdot h \cdot \theta_*$ for
$g,h \sim \Lambda$, and where $\delta$ in \cite[Eqs.\ (2.12--2.13)]{katsevich2020likelihood} is
bounded as $\delta \leq C(1 \vee \|\theta\|)$ for a
$(\Pi,\G,\theta_*)$-dependent constant $C>0$ and all $\theta \in \R^d$),
this result guarantees that the expansion
(\ref{eq:seriesexpansionprojected}) holds for
$\tilde{s}_k$, $\tilde{T}_k$, $P_k$, and $q_k$ having all of the stated
properties, and for a remainder $q(\theta)$ that satisfies
\begin{equation}\label{eq:KBremainder}
|q(\theta)| \leq \frac{C_K(1 \vee \|\theta\|)^{2K+2}}{\sigma^{2K+2}}
\end{equation}
when $\|\theta\| \leq \sigma$. This remainder $q(\theta)$ must
also be $\G$-invariant, as all of the other terms in
(\ref{eq:seriesexpansionprojected}) are $\G$-invariant.

It remains to verify the bounds for $\|\nabla q(\theta)\|$ and $\|\nabla^2
q(\theta)\|$ in (\ref{eq:remainderbound}). These types of bounds were shown
in the unprojected setting of $\Pi=\Id$
in \cite[Lemma 4.7]{fan2020likelihood}. They were not stated explicitly in
\cite{katsevich2020likelihood}, but may be deduced from a small extension
of the analysis: Denote by $\E_g,\E_h$ the expectations over $g,h \sim \Lambda$,
and by $\E_\eps,\E_{\eps'}$ those over $\eps,\eps' \sim \N(0,\Id)$. Write
\[t=1/\sigma, \qquad Y=\Pi \cdot h \cdot \theta_*+t^{-1}\eps, \qquad
w=\Pi \cdot g \cdot \theta-\Pi \cdot h \cdot \theta_* \in \R^{\td},\]
and define
\[f(t)=-\log M(t), \qquad M(t)=\E_g\left[\exp\left(-\frac{t^2\|w\|^2}{2}
+tw^\top \eps\right)\right]\]
Comparing with (\ref{eq:likelihood}), this function $f(t)$ is the negative
log-likelihood for the single sample $Y$, up to a $\theta$-independent constant
and viewed as a function of $t=1/\sigma$.
Applying a Taylor expansion of $f(t)$ around
$t=0$, and then taking expectations over $(h,\eps)$ that define $Y$, we have
\begin{equation}\label{eq:KBtaylorexpansion}
R(\theta)=\text{constant}+\sum_{p=1}^{2K+1} \frac{t^p}{p!}
\E_{h,\eps}[f^{(p)}(0)]+\frac{t^{2K+2}}{(2K+2)!}
\E_{h,\eps}[f^{(2K+2)}(\xi(h,\eps))]
\end{equation}
for a random point $\xi(h,\eps)$ between 0 and $t=1/\sigma$. This is a
rewriting of the Taylor expansion
in \cite[Eq.\ (5.7)]{katsevich2020likelihood}.
It is shown in \cite{katsevich2020likelihood} that the leading terms in
(\ref{eq:KBtaylorexpansion}) of orders $t^1,\ldots,t^{2K+1}$ give exactly the
leading terms of (\ref{eq:seriesexpansionprojected}), and the last term of
(\ref{eq:KBtaylorexpansion}) is the remainder $q(\theta)$ in
(\ref{eq:seriesexpansionprojected}). The bound (\ref{eq:KBremainder}) for
$q(\theta)$ follows from \cite[Eq.\ (5.22)]{katsevich2020likelihood}.

To bound $\partial_{\theta_a} q(\theta)$ for any index $a \in \{1,\ldots,d\}$,
we may apply a similar Taylor expansion
for $\partial_{\theta_a} R(\theta)$, and write
\[\partial_{\theta_a} R(\theta)=\sum_{p=1}^{2K+1}
\frac{t^p}{p!}\E_{h,\eps}[\partial_t^p \partial_{\theta_a} f(0)]
+\frac{t^{2K+2}}{(2K+2)!} \E_{h,\eps}[\partial_t^{2K+2} \partial_{\theta_a}
f(\xi(h,\eps))]\]
for a possibly different point $\xi(h,\eps) \in (0,t)$
depending on the index $a$. Then
\begin{equation}\label{eq:derq}
\partial_{\theta_a} q(\theta)=\frac{t^{2K+2}}{(2K+2)!}
\E_{h,\eps}[\partial_t^{2K+2} \partial_{\theta_a} f(\xi(h,\eps))]
\end{equation}
and we wish to bound this term for each $a \in \{1,\ldots,d\}$. The function
$\partial_{\theta_a} f(t)$ is the $\theta_a$-derivative of $f$, given by
\[\partial_{\theta_a} f(t)=-\frac{\partial_{\theta_a} M(t)}{M(t)}.\]
Then differentiating $2K+2$ times in $t$, we see that
$\partial_t^{2K+2}\partial_{\theta_a} f(t)$ is a sum of at most
$C_K$ terms of the form
\[C_{\ell_0,\ldots,\ell_j} \cdot
\frac{\partial_t^{\ell_0} \partial_{\theta_a}M(t)}{M(t)}
\cdot \frac{\partial_t^{\ell_1} M(t)}{M(t)} \cdot \ldots
\cdot \frac{\partial_t^{\ell_j} M(t)}{M(t)}\]
for some integers $j \geq 0$ and $\ell_0,\ldots,\ell_j \geq 0$ such that
$\ell_0+\ldots+\ell_j=2K+2$, and for some universal
constants $C_{\ell_0,\ldots,\ell_j}$ depending only on $\ell_0,\ldots,\ell_j$.

From \cite[Eq.\ (5.20)]{katsevich2020likelihood}, we have
\[\left|\frac{\partial_t^\ell M(\xi)}{M(\xi)}\right| \leq
\delta^\ell\E_{\eps'}\Big[(\|W\|+|\xi|\delta)^\ell\Big],
\qquad W=\eps+\i\eps', \qquad
\delta=\sup_{g,h \in \G} \|\Pi \cdot g \cdot \theta-
\Pi \cdot h \cdot \theta_*\|\]
where $\eps' \sim \N(0,\Id_{\td \times \td})$ is an independent copy
of $\eps$. Applying this to any $\xi \in (0,t)$, and applying $\delta
\leq C(1 \vee \|\theta\|) \leq C\sigma=Ct^{-1}$, we obtain
\begin{equation}\label{eq:Mbound}
\left|\frac{\partial_t^\ell M(\xi)}{M(\xi)}\right| \leq
C_\ell\delta^\ell(1+\|\eps\|)^\ell.
\end{equation}
We may bound the $t$-derivatives of $\partial_{\theta_a} M(t)$ using a similar
argument: Introducing the $a^{\text{th}}$ standard basis vector
$e_a \in \R^d$, observe that
\begin{align*}
\partial_{\theta_a} M(t)&=\E_g\left[e_a^\top g^\top \Pi^\top (t\eps-t^2w)
\cdot \exp\left(-\frac{t^2\|w\|^2}{2}+tw^\top \eps\right)\right]\\
&=\E_g\left[e_a^\top g^\top \Pi^\top (t\eps-t^2w)
\cdot \E_{\eps'}[e^{tw^\top W}]\right], \qquad \text{ where } W=\eps+\i\eps'.
\end{align*}
Then applying Leibniz' rule and the same argument as
\cite[Eqs.\ (5.18--5.19)]{katsevich2020likelihood} to differentiate in $t$,
we obtain
\begin{align*}
\partial_t^\ell \partial_{\theta_a} M(t)
&=\E_g\left[e_a^\top g^\top \Pi^\top (t\eps-t^2w) \cdot \E_{\eps'}[e^{t w^\top
W}] \cdot \E_{\eps'}[(w^\top (W-tw))^\ell]\right]\\
&\hspace{1in}+
\ell \cdot
\E_g\left[e_a^\top g^\top \Pi^\top (\eps-2tw) \cdot \E_{\eps'}[e^{t w^\top
W}] \cdot \E_{\eps'}[(w^\top (W-tw))^{\ell-1}]\right]\\
&\hspace{1in}+
\binom{\ell}{2} \cdot
\E_g\left[e_a^\top g^\top \Pi^\top (-2w) \cdot \E_{\eps'}[e^{t w^\top
W}] \cdot \E_{\eps'}[(w^\top (W-tw))^{\ell-2}]\right].
\end{align*}
Applying also $M(t)=\E_g[\E_{\eps'}[e^{tw^\top W}]]$, so that
$\E_g[(\cdot)\E_{\eps'}[e^{tw^\top W}]]/M(t)$ is a reweighted average over $g
\in \G$, this yields
analogously to \cite[Eq.\ (5.20)]{katsevich2020likelihood} that
\begin{align*}
\left|\frac{\partial_t^\ell \partial_{\theta_a}M(\xi)}{M(\xi)}\right|
&\leq C_\ell \Big[(|\xi|\|\eps\|+\xi^2\delta)
\cdot \delta^\ell\E_{\eps'}[(\|W\|+|\xi|\delta)^\ell]\nonumber\\
&\hspace{0.3in}+(\|\eps\|+|\xi|\delta)
\cdot \delta^{\ell-1}\E_{\eps'}[(\|W\|+|\xi|\delta)^{\ell-1}]
+\delta \cdot \delta^{\ell-2}\E_{\eps'}[(\|W\|+|\xi|\delta)^{\ell-2}]\Big]
\end{align*}
where we have absorbed $\|e_a^\top g^\top \Pi\|$ into the constant $C_\ell$.
Applying this with $\xi \in (0,t)$ and
$\delta \leq C(1 \vee \|\theta\|) \leq Ct^{-1}$, we get
\begin{equation}\label{eq:derMbound}
\left|\frac{\partial_t^\ell \partial_{\theta_a}M(\xi)}{M(\xi)}\right|
\leq C_\ell\delta^{\ell-1}(1+\|\eps\|)^{\ell+1}.
\end{equation}
Then combining with (\ref{eq:Mbound}) and applying to the previously
stated form of $\partial_t^{2K+2}\partial_{\theta_a} f$,
\[|\partial_t^{2K+2}\partial_{\theta_a} f(\xi)| \leq C_K\!\!\!
\mathop{\sum_{\ell_0,\ldots,\ell_j \geq 0}}_{\ell_0+\ldots+\ell_j=2K+2}
\delta^{\ell_0+\ldots+\ell_j-1}(1+\|\eps\|)^{\ell_0+\ldots+\ell_j+1}
\leq C_K' (1 \vee \|\theta\|)^{2K+1}(1+\|\eps\|)^{2K+3}.\]
Finally, taking the expectation over $\eps$ and applying this
to (\ref{eq:derq}) for each $a \in \{1,\ldots,d\}$, we obtain the desired bound
\[\|\nabla q(\theta)\| \leq
\frac{C_K(1 \vee \|\theta\|)^{2K+1}}{\sigma^{2K+2}}.\]

The argument to bound $\|\nabla^2 q(\theta)\|$ is similar: For any $a,b \in
\{1,\ldots,d\}$, applying a Taylor
expansion of $\partial_{\theta_a,\theta_b}^2 R(\theta)$, we wish to bound
\[\partial_{\theta_a,\theta_b}^2 q(\theta)
=\frac{t^{2K+2}}{(2K+2)!}\E_{h,\eps}[\partial_t^{2K+2}
\partial_{\theta_a,\theta_b}^2 f(\xi(h,\eps))]\]
for some $\xi(h,\eps) \in (0,t)$ depending on $a,b$. We may compute
\[\partial_{\theta_a,\theta_b}^2 f(t)
=-\frac{\partial_{\theta_a,\theta_b}^2 M(t)}{M(t)}
+\frac{\partial_{\theta_a} M(t)}{M(t)}\frac{\partial_{\theta_b} M(t)}{M(t)},\]
differentiate this $2K+2$ times in $t$, and apply 
(\ref{eq:Mbound}), (\ref{eq:derMbound}), and the analogous bound
\[\left|\frac{\partial_t^\ell \partial_{\theta_a,\theta_b}^2
M(\xi)}{M(\xi)}\right| \leq C_\ell\delta^{\ell-2}(1+\|\eps\|)^{\ell+2}\]
which is derived similarly. This yields
$|\partial_t^{2K+2}\partial_{\theta_a} f(\xi)| \leq 
C_K (1 \vee \|\theta\|)^{2K}(1+\|\eps\|)^{2K+4}$,
and taking the expectation over $\eps$ gives the desired bound for $\|\nabla^2
q(\theta)\|$.
\end{proof}

\subsection{Identifiability and transcendence degree}\label{appendix:trdeg}

We prove Propositions \ref{prop:trdegorbitdim} and \ref{prop:Kdef}.
An analogue of Proposition \ref{prop:trdegorbitdim} for algebraic groups
over algebraically closed fields may be found in
e.g.\ \cite[Section 2.3]{popov1994invariant}; we provide here an argument in
our context of a compact subgroup $\G \subseteq \O(d)$ acting on $\R^d$.

\begin{proof}[Proof of Proposition \ref{prop:Kdef}(a)]
The algebra of all polynomials $\R[\theta_1,\ldots,\theta_d]$ has transcendence
degree $d$ over $\R$, so also $\trdeg \cR^\G \leq d<\infty$ for the subalgebra
$\cR^\G$. Taking any finite transcendence basis $\varphi$
of $\cR^\G$ and letting $k$ be the maximum degree of polynomials constituting
$\varphi$, we have $\trdeg \cR_{\leq k}^\G=|\varphi|=\trdeg \cR^\G$. Then there
must be a smallest integer $k$ for which this holds.
\end{proof}

For the remaining statements of
Propositions \ref{prop:trdegorbitdim} and \ref{prop:Kdef}, we will use
the following Jacobian criterion for algebraic independence, and a relation between
generic list recovery of $\theta_*$ and transcendence degree.

\begin{lemma}[Jacobian criterion, \cite{beecken2013algebraic} Theorem 8]
\label{lemma:jaccrit}
For any real polynomials $p_1,\ldots,p_m$ in $\theta_1,\ldots,\theta_d$,
\[\trdeg(\{p_1,\ldots,p_m\})=\operatorname{rank}_{\R(\theta_1,\ldots,\theta_d)}
[\nabla p_1,\ldots,\nabla p_m]^\top\]
where the right side denotes the rank of the Jacobian matrix
$[\nabla p_1,\ldots,\nabla p_m]^\top \in \R^{m \times d}$ over
the field of rational functions $\R(\theta_1,\ldots,\theta_d)$.
\end{lemma}

Note that this implies $\trdeg(\{p_1,\ldots,p_m\}) \geq \rank[\nabla
p_1(\theta),\ldots,\nabla p_m(\theta)]^\top$ (the rank over $\R$)
evaluated at any $\theta \in \R^d$, with equality holding at generic points
$\theta \in \R^d$.

\begin{lemma}[Generic list recovery, \cite{bandeira2017estimation} Theorem 4.9]
\label{lemma:genericlist}
Let $\G \subseteq \O(d)$ be a compact subgroup, and let $U$ be a
finite-dimensional subspace of $\cR^\G$. If $\trdeg U=\trdeg \cR^G$,
then for generic $\theta_* \in \R^d$, the set of points
$\{\theta \in \R^d:P(\theta)=P(\theta_*) \text{ for all } P \in U\}$
is a union of a finite number of orbits.
\end{lemma}
\noindent (Theorem 4.9 of \cite{bandeira2017estimation} shows also a converse of
this statement, but we will not directly use this converse.)

Recall from (\ref{eq:dk}--\ref{eq:tdk}) the values
\[d_k=\trdeg \cR_{\leq k}^\G-\trdeg \cR_{\leq k-1}^\G,
\qquad \td_k=\trdeg \tcR_{\leq k}^\G-\trdeg \tcR_{\leq k-1}^\G\]
and from (\ref{eq:Mk}--\ref{eq:tMk})
the combined moment functions $M_k(\theta)$ and $\tM_k(\theta)$.
The following structural lemma is an important ingredient for our proofs, and provides
explicit coordinate systems in local neighborhoods of points $\theta \in
\R^d$ using $\G$-invariant polynomials. Versions of these statements were shown
in \cite[Section 4]{fan2020likelihood} when $\G$ is a discrete
group and generic orbits have dimension 0, and the following lemma
provides an extension to models where orbits have positive dimension.

\begin{lemma}\label{lemma:transcendencebasis}
Let $K$ be the smallest integer for which $\trdeg \cR_{\leq K}^\G=\trdeg
\cR^\G$.
In the unprojected model, for generic $\theta \in \R^d$ and every
$k=1,\ldots,K$, the rank of $\der M_k(\theta)$ equals $d_1+\ldots+d_k$.

Furthermore, for any $k \in \{1,\ldots,K\}$ and any
$\theta \in \R^d$ where $\rank \der M_k(\theta)=d_1+\ldots+d_k$, there exist
functions $\varphi^j:\R^d \to \R^{d_j}$ for each $j=1,\ldots,k$
and $\bar\varphi:\R^d \to \R^{d-d_1-\ldots-d_k}$, such that:
\begin{enumerate}[(a)]
\item For each $j=1,\ldots,k$, the $d_j$ coordinates of $\varphi^j$ are
entries of the moment tensor $T_j$.
\item The combined map $\varphi=(\varphi^1,\ldots,\varphi^k,\bar\varphi):
\R^d \to \R^d$ has non-singular derivative $\der \varphi(\theta) \in
\R^{d \times d}$ at this point $\theta$, and hence an analytic inverse function $\theta(\varphi)$ over a neighborhood $U$ of $\theta$.
\item For any $j=1,\ldots,k$, any polynomial $p \in \cR_{\leq j}^\G$, and
sufficiently small such neighborhood $U$, $\varphi \mapsto p(\theta(\varphi))$
is a function only of the $d_1+\ldots+d_j$ coordinates
$\varphi^1,\ldots,\varphi^j$, over $\varphi(U)$.
\item Suppose $k=K$. Then for any $\G$-invariant continuous function $f:\R^d \to
\R$ and sufficiently small such neighborhood $U$, 
$\varphi \mapsto f(\theta(\varphi))$ is a function only of the
$d_1+\ldots+d_K$ coordinates $\varphi^1,\ldots,\varphi^K$, over $\varphi(U)$.
Also,
\begin{equation}\label{eq:barphiorbit}
\Big\{\theta' \in U:(\varphi^1(\theta'),\ldots,\varphi^K(\theta'))
=(\varphi^1(\theta),\ldots,\varphi^K(\theta))\Big\}
=U \cap \orbit_\theta.
\end{equation}
\end{enumerate}
In the projected model, suppose there exists a smallest integer
$\tK<\infty$ for which $\trdeg \tcR_{\leq \tK}^\G=\trdeg \cR^G$. Then
the same statements hold with $\tK$, $\tM_j$, $\tT_j$,
$\td_j$, and $\tcR_{\leq j}^\G$ in place of $K$, $M_j$, $T_j$, $d_j$, and
$\cR_{\leq j}^\G$.
\end{lemma}

\begin{proof}
In the unprojected model, $\cR_{\leq k}^\G$ is generated by
the entries of $T_1,\ldots,T_k$, and $\trdeg \cR_{\leq k}^\G=d_1+\ldots+d_k$.
Thus there are exactly $d_1+\ldots+d_k$ algebraically independent entries of
$T_1,\ldots,T_k$ (c.f.\ \cite[Chapter 8, Theorem 1.1]{lang2002algebra}).
Then by Lemma \ref{lemma:jaccrit},
$\der M_k(\theta)$ has rank $d_1+\ldots+d_k$ at generic $\theta \in \R^d$,
and rank at most $d_1+\ldots+d_k$ at every $\theta \in \R^d$.

If $\theta \in \R^d$ has $\rank \der M_k(\theta)=d_1+\ldots+d_k$, then this
implies that also
$\rank \der M_j(\theta)=d_1+\ldots+d_j$ for each $j=1,\ldots,k$. Then for each
$j=1,\ldots,k$, we may pick $d_j$
entries of $T_j$ to be $\varphi^j$, such that $(\varphi^1,\ldots,\varphi^k)$
have linearly independent gradients at $\theta$. We may arbitrarily
pick $d-d_1-\ldots-d_k$ additional analytic functions to be $\bar\varphi$, so
that $\varphi=(\varphi^1,\ldots,\varphi^k,\bar\varphi)$
has non-singular derivative $\der \varphi(\theta)$. This shows properties
(a) and (b), where the existence of an analytic inverse $\theta(\varphi)$ on
$U$ follows from the inverse function theorem.

For (c), denote $q(\varphi)=p(\theta(\varphi))$. Applying the
chain rule, for any $x \in U$,
\[\nabla p(x)=\der_x \varphi(x)^\top \nabla_\varphi q(\varphi(x)).\]
Since $p$ is a function of $T_1,\ldots,T_j$, its gradient is a linear
combination of the gradients of the entries of $T_1,\ldots,T_j$,
and hence linearly dependent with the gradients of $\varphi^1,\ldots,\varphi^j$.
Thus $\nabla p(x)$ belongs to the
span of the columns of $\der_x \varphi(x)^\top$ corresponding to
$\varphi^1,\ldots,\varphi^j$, implying that $\nabla_\varphi q(\varphi(x))$
is 0 in the remaining coordinates $\varphi^{j+1},\ldots,\varphi^k,\bar\varphi$. 
This holds at every $x \in U$, so $q$ is a function only of
$\varphi^1,\ldots,\varphi^j$.

For (d), denote $h(\varphi)=f(\theta(\varphi))$.
Suppose first that $f \in \cR^\G$ is a $\G$-invariant polynomial.
Since $d_1+\ldots+d_K=\trdeg \cR_{\leq K}^\G=\trdeg \cR^\G$, we must have that
$(\varphi^1,\ldots,\varphi^K,f)$ are algebraically dependent. Then their
gradients are linearly dependent at every $x \in U$.
Then the same argument as in (c) shows that $h(\varphi)$ depends only on
$\varphi^1,\ldots,\varphi^K$. For a general $\G$-invariant continuous
function $f$, let $r>0$ be large enough such that $U \subset
\overline{B_r}=\{x \in \R^d:\|x\| \leq r\}$. For any $\eps>0$,
by the Stone-Weierstrass theorem, there is a polynomial
$p$ such that $|p(x)-f(x)|<\eps$ for all $x \in
\overline{B_r}$. Applying the Reynolds operator
$\bar{p}(x)=\int p(g \cdot x) \der\Lambda(g)$,
we then have $\bar{p} \in \cR^\G$, and also
\[|\bar{p}(x)-f(x)|
=\bigg|\int \Big(p(g \cdot x)-f(g \cdot x)\Big)\der\Lambda(g)\bigg|<\eps
\quad \text{ for all } x \in \overline{B_r}\]
because $g \cdot x \in \overline{B_r}$ for any orthogonal matrix
$g$. Writing $\bar{q}(\varphi)=\bar{p}(\theta(\varphi))$, we have shown that
$\bar{q}$ depends only on $\varphi^1,\ldots,\varphi^K$. So for any
$\varphi,\varphi' \in \varphi(U)$ differing in only the coordinates of
$\bar\varphi$, we have $\bar{q}(\varphi)=\bar{q}(\varphi')$ and
$\theta(\varphi),\theta(\varphi') \in U \subset \overline{B_r}$, hence
$|h(\varphi)-h(\varphi')| \leq |h(\varphi)-\bar{q}(\varphi)|
+|h(\varphi')-\bar{q}(\varphi')|<2\eps$.
Here $\eps>0$ is arbitrary, so in fact $h(\varphi)=h(\varphi')$. Thus $h$
depends only on $\varphi^1,\ldots,\varphi^K$.

To show (\ref{eq:barphiorbit}), clearly if $\orbit_{\theta'}=\orbit_\theta$,
then $(\varphi^1(\theta'),\ldots,\varphi^K(\theta'))=
(\varphi^1(\theta),\ldots,\varphi^K(\theta))$. For the converse direction,
if $\orbit_{\theta'}$ and $\orbit_{\theta}$ are distinct, then they are
disjoint compact subsets of $\R^d$. Then there is a continuous function
$f:\R^d \to \R$ taking value 1 on $\orbit_{\theta'}$ and 0 on $\orbit_\theta$.
Then $\bar{f}(x)=\int f(g \cdot x) \der \Lambda(g)$ is a $\G$-invariant
continuous function with the same property. Thus $\bar{f}$ depends only on
$\varphi^1,\ldots,\varphi^K$, implying that
$(\varphi^1(\theta'),\ldots,\varphi^K(\theta')) \neq
(\varphi^1(\theta),\ldots,\varphi^K(\theta))$. This shows
(\ref{eq:barphiorbit}), and concludes the proof in the unprojected setting.

The proof in the projected setting is the same, where the given
condition for $\tK$ is used in part (d) to show
$\td_1+\ldots+\td_\tK=\trdeg \cR^\G$,
and hence $(\varphi^1,\ldots,\varphi^{\tK},f)$ are algebraically
dependent for any $f \in \cR^\G$.
\end{proof}

\begin{proof}[Proof of Proposition \ref{prop:trdegorbitdim}]
Since $\G$ is a compact Lie group acting smoothly on $\R^d$, its action is
proper. Then each orbit $\orbit_\theta$ is an embedded submanifold of $\R^d$,
and the tangent space to $\orbit_\theta$ at $\theta$ is given by
$T_\theta \orbit_\theta=\{\mathfrak{g} \cdot \theta:\mathfrak{g} \in T_{\Id}
\G\}$ where $T_{\Id} \G$ is the Lie algebra, i.e.\ the tangent space to $\G$
at $g=\Id$ (c.f.\ \cite[Section I.1.b and Corollary I.1.2]{audin2004torus}).
Parametrizing $\G$ around $g=\Id$ by any local
chart $x$ such that $g(x)=\Id$ at $x=0$, we then have
$\dim(\orbit_\theta)=\dim(\{\mathfrak{g} \cdot \theta:\mathfrak{g} \in T_{\Id}
G\})=\rank(\der_x[g(x) \cdot \theta]_{x=0})$. Since
$\der_x[g(x) \cdot \theta]_{x=0}$ is an analytic matrix in $\theta$, this
implies that $\max_{\theta \in \R^d} \dim(\orbit_\theta)$ is
attained at generic points $\theta \in \R^d$ (c.f.\ Fact \ref{fact:fullrank}). On the other hand,
for generic $\theta \in \R^d$, the statement (\ref{eq:barphiorbit}) of
Lemma \ref{lemma:transcendencebasis}(d) shows that
$\dim(\orbit_\theta)=d-(d_1+\ldots+d_K)
=d-\trdeg \cR_{\leq K}^\G=d-\trdeg \cR^\G$. Hence $\trdeg \cR^\G=d-\max_{\theta
\in \R^d} \dim(\orbit_\theta)$.
\end{proof}

\begin{proof}[Proof of Proposition \ref{prop:Kdef}(b)]

$\Pi(\orbit_\theta)$ is a continuous image of the compact group $\G$, and
hence is also compact. Then the distribution
of $\Pi \cdot g \cdot \theta$ is uniquely determined by its sequence of mixed
moments. Thus
\begin{equation}\label{eq:allpolynomialsequal}
\Big\{\theta:\Pi(\orbit_\theta) \equiv \Pi(\orbit_{\theta_*})\Big\}
=\Big\{\theta:\tT_k(\theta)=\tT_k(\theta_*) \text{ for all } k \geq 1\Big\}.
\end{equation}
Suppose that $\trdeg \tcR_{\leq k}^\G=\trdeg \cR^\G$ for some integer $k$. 
Let $U_k$ be the linear subspace of $\cR^\G$ spanned by 1 and all entries of
$\tT_1,\ldots,\tT_k$. Since $U_k$ generates $\tcR_{\leq k}^\G$, we have
$\trdeg U_k=\trdeg \tcR_{\leq k}^\G$
(c.f.\ \cite[Chapter 8, Theorem 1.1]{lang2002algebra}).
Then Lemma \ref{lemma:genericlist} implies that
for generic $\theta_* \in \R^d$, there are only finitely many orbits
$\orbit_\theta$ such that $P(\theta)=P(\theta_*)$ for all
$P \in U_k$. This condition must hold for all $\theta$ belonging to
(\ref{eq:allpolynomialsequal}), so (\ref{eq:allpolynomialsequal})
also consists of finitely many orbits.

Conversely, suppose $\trdeg \tcR_{\leq k}^\G<\trdeg \cR^\G$ for
all $k \geq 1$. Consider $\tcR^\G=\bigcup_{k \geq 1} \tcR_{\leq k}^\G$
(the subalgebra generated by entries of $\tT_k$ for all $k \geq 1$). By the
same argument as in the proof of Proposition \ref{prop:Kdef}(a), 
we have $\trdeg \tcR^\G \leq d<\infty$, so
$\trdeg \tcR^\G=\trdeg \tcR_{\leq k}^\G$ for some integer
$k$. Then also $\trdeg \tcR^\G<\trdeg \cR^\G$. We now apply an argument similar
to that of the proof of Lemma \ref{lemma:transcendencebasis}(c):
Fix any transcendence
basis $\varphi^0$ of $\tcR^\G$. Then the gradient vectors of $\varphi^0$ are
linearly independent at generic $\theta_* \in \R^d$ by Lemma
\ref{lemma:jaccrit}. Fix any such $\theta_*$. In a sufficiently small
open neighborhood $O$ of $\theta_*$, we claim that
\begin{equation}\label{eq:localparamclaim}
\Big\{\theta \in O:\tT_k(\theta)=\tT_k(\theta_*) \text{ for all } k \geq
1\Big\}=\Big\{\theta \in O:\varphi^0(\theta)=\varphi^0(\theta_*)\Big\}.
\end{equation}
To see this, choose
any $d-|\varphi^0|$ additional functions $\bar{\varphi}$ for which
$\varphi=(\varphi^0,\bar{\varphi})$ has non-singular derivative at $\theta_*$,
and hence forms an invertible local reparametrization
over a sufficiently small such neighborhood
$O$, by the inverse function theorem. Let $p(\theta)$ be any entry of
$\tT_k$ for any $k \geq 1$, and write $q(\varphi)=p(\theta(\varphi))$ for its
reparametrization by the local coordinates $\varphi$ on $O$. By the chain rule,
\[\nabla p(\theta)=\der_\theta \varphi(\theta)^\top \nabla_\varphi
q(\varphi(\theta)).\]
Since $p \in \tcR^\G$, and $\varphi^0$ is a transcendence basis for $\tcR^\G$,
we have that $(p,\varphi^0)$ is algebraically dependent.
Then the gradients of $p$ and
$\varphi^0$ are linearly dependent at every $\theta \in O$
by Lemma \ref{lemma:jaccrit}, so $\nabla p(\theta)$
belongs to the span of columns of $\der_\theta \varphi(\theta)^\top$
corresponding to only the coordinates of $\varphi^0$. Then $\nabla_\varphi
q(\varphi(\theta))$ must be 0 in the remaining coordinates $\bar{\varphi}$.
This holds for all $\theta \in O$, so $q(\varphi)$ is a function only of
$\varphi^0$ in this local parametrization over $O$. This shows our claim
(\ref{eq:localparamclaim}). The set (\ref{eq:localparamclaim}) forms a
manifold of dimension $d-|\varphi^0|=d-\trdeg \tcR^\G$. On the other hand,
Proposition \ref{prop:trdegorbitdim}
shows that every orbit $\orbit_\theta$ has dimension at most
$d-\trdeg \cR^\G$, which is
strictly smaller when $\trdeg \tcR^\G<\trdeg \cR^\G$. Thus
(\ref{eq:allpolynomialsequal}) must contain infinitely many orbits corresponding
to $\theta \in O$.
\end{proof}

\subsection{Fisher information}

We prove Theorem \ref{thm:FI} and Lemma \ref{lem:trdeg}.
Throughout, we assume that
(\ref{eq:losslessPi}) holds in the projected setting, and we denote by
$K$ and $\tK$ the (smallest) integers satisfying Proposition \ref{prop:Kdef}. 
All constants $C,C',c,c',\sigma_0>0$ in the proofs may depend
implicitly on $(\theta_*,\G,\Pi)$.

The proofs are analogous to the arguments of \cite[Section
4.4]{fan2020likelihood}: Locally around any generic point $\theta_* \in \R^d$,
we reparametrize $\theta$ by a transcendence basis $\varphi^1,\ldots,\varphi^K$
for $\cR^\G$ having full-rank Jacobian, and then
compute the Hessian of $R(\theta(\varphi))$ in $\varphi$ by taking derivatives
of the series expansion (\ref{eq:seriesexpansionunprojected}) or
(\ref{eq:seriesexpansionprojected}) term-by-term. When $\G$ is a continuous
group, we extend this transcendence basis using the additional analytic
functions $\bar\varphi \in \R^{d_0}$ provided in
Lemma \ref{lemma:transcendencebasis} to obtain a complete system of coordinates.
The properties stated in Theorem \ref{thm:seriesexpansion}
will guarantee that the Hessian of each term of order $\sigma^{-2k}$
depends only on $\varphi^1,\ldots,\varphi^k$ in this system of
coordinates, and that at the true parameter $\theta=\theta_*$,
the block of this Hessian
corresponding to the coordinates of $\varphi^k$ is strictly
positive definite. In the projected setting, this latter property uses the
condition given in Theorem \ref{thm:seriesexpansion}(b) that $P_k(\theta_*)=0$.
Then Theorem \ref{thm:FI} will follow from the chain rule and some linear algebra.

We recall here the following definition and elementary
linear-algebraic result from \cite{fan2020likelihood}.

\begin{figure}
\[\begin{pmatrix} \sigma^{-2} & \sigma^{-4} & \sigma^{-6} & \sigma^{-8} &
\cdots & \sigma^{-2K} \\
\sigma^{-4} & \sigma^{-4} & \sigma^{-6} & \sigma^{-8} & \cdots & \sigma^{-2K} \\
\sigma^{-6} & \sigma^{-6} & \sigma^{-6} & \sigma^{-8} & \cdots & \sigma^{-2K} \\
\sigma^{-8} & \sigma^{-8} & \sigma^{-8} & \sigma^{-8} & \cdots & \sigma^{-2K} \\
\vdots & \vdots & \vdots & \vdots & \ddots & \vdots \\
\sigma^{-2K} & \sigma^{-2K} & \sigma^{-2K} & \sigma^{-2K} &
\cdots & \sigma^{-2K}
\end{pmatrix}\]
\caption{An illustration of the block scalings with $\sigma^{-2}$ for
  matrices $H$ with graded block structure, where each $(k,\ell)$ entry
represents the scaling for a single block $H_{k\ell}$ of $H$.}\label{fig:gradedblocks}
\end{figure}

\begin{definition}[\cite{fan2020likelihood} Definition 4.14]
\label{def:gradedblock}
Let $(\varphi^1,\ldots,\varphi^K)$ be a partition of coordinates for
$\R^{d'}$. Let $H \equiv H(\sigma) \in \R^{d' \times d'}$ be a symmetric matrix,
and write its $K \times K$ block decomposition with respect to this partition as
\[H=\begin{pmatrix} H_{11} & \cdots & H_{1K} \\
\vdots & \ddots & \vdots \\ H_{K1} & \cdots & H_{KK} \end{pmatrix}.\]
The matrix $H(\sigma)$ has a {\bf graded block structure} with respect to this
partition if there are constants $C,c,\sigma_0>0$ such that for all
$\sigma>\sigma_0$ and all $k,\ell \in \{1,\ldots,K\}$ where $\varphi^k$ and
$\varphi^\ell$ have non-zero dimension,
\[C\sigma^{-2k} \geq \lambda_{\max}(H_{kk}) \geq \lambda_{\min}(H_{kk})
\geq c\sigma^{-2k} \qquad \text{ and } \qquad \|H_{k\ell}\| \leq
C\sigma^{-2\max(k,\ell)}\]
where $\lambda_{\max},\lambda_{\min}$ denote the largest and smallest
eigenvalues.
\end{definition}

A visual illustration of this structure is depicted in Figure
\ref{fig:gradedblocks}.

\begin{lemma}[\cite{fan2020likelihood} Lemma 4.17]
\label{lemma:gradedblock}
Suppose $H \equiv H(\sigma) \in \R^{d' \times d'}$ has a graded block structure
with respect to $(\varphi^1,\ldots,\varphi^K)$. Let $d_k \geq 0$ be the
dimension of $\varphi^k$. Let $H_{:k,:k}$ and $(H^{-1})_{:k,:k}$ be the submatrices
of upper-left $k \times k$ blocks of $H$ and $H^{-1}$. Then for
some constants $C,c,\sigma_0>0$ and all $\sigma>\sigma_0$:
\begin{enumerate}[(a)]
\item $H$ has $d_k$ eigenvalues belonging to $[c\sigma^{-2k},C\sigma^{-2k}]$ for
each $k=1,\ldots,K$.
\item For each $k$ where $d_1+\ldots+d_k>0$, $\lambda_{\min}(H_{:k,:k}) \geq
c\sigma^{-2k}$.
\item For each $k$ where $d_1+\ldots+d_k>0$, $\lambda_{\max}((H^{-1})_{:k,:k})
\leq C\sigma^{2k}$.
\end{enumerate}
\end{lemma}

Let us now fix $d_k$ and $\td_k$ as the constants defined by
(\ref{eq:dk}--\ref{eq:tdk}). Recall
the combined moment functions $M_k(\theta)$ and
$\tM_k(\theta)$ from (\ref{eq:Mk}--\ref{eq:tMk}), and the moment varieties
$\cV_k(\theta_*)$ and $\tcV_k(\theta_*)$ from (\ref{eq:Vk}--\ref{eq:tVk}).

\begin{lemma}\label{lemma:skfullrank}
In the unprojected model, fix any $\theta_* \in \R^d$ and any $k \in
\{1,\ldots,K\}$. Let $\ttheta \in \cV_k(\theta_*)$ be such that $\rank \der
M_k(\ttheta)=d_1+\ldots+d_k$, and let
$\varphi=(\varphi^1,\ldots,\varphi^k,\bar\varphi)$ be the map defined by
Lemma \ref{lemma:transcendencebasis} with inverse $\theta(\varphi)$
in a neighborhood of $\ttheta$.
Let $s_k(\theta)$ be as defined in (\ref{eq:sk}). Then in the parametrization by
$\varphi$,
\[\nabla_{\varphi^k}^2 s_k(\theta(\varphi))\Big|_{\varphi=\varphi(\ttheta)}
\text{ has full rank } d_k \text{ and is positive definite}.\]
In the projected model, the same statements hold for $\tcV_k$, $\tM_k$, $\td_k$,
and $\ts_k$ in place of $\cV_k$, $M_k$, $d_k$, and $s_k$.
\end{lemma}
\begin{proof}
We focus on the unprojected model; the proof in the projected model is
the same.

Since $s_k$ is globally minimized at all points of $\cV_k(\theta_*)$,
we must have
$\nabla_{\varphi^k}^2 s_k(\theta(\varphi))|_{\varphi=\varphi(\ttheta)}
\succeq 0$. To show this has full rank $d_k$,
observe that $\varphi^k$ consists of a subset of entries of $T_k$. Thus the
corresponding $d_k \times d_k$ submatrix of $\der_{\varphi^k} T_k$ is the
identity, so $\der_{\varphi^k} T_k$ has full column rank $d_k$.
Applying the chain rule and the observation
$T_k(\ttheta)-T_k(\theta_*)=0$ because $\ttheta \in \cV_k(\theta_*)$,
we may differentiate
$s_k(\theta(\varphi))$ twice in $\varphi^k$ to obtain
\[\nabla_{\varphi^k}^2
s_k(\theta(\varphi))\Big|_{\varphi=\varphi(\ttheta)}
=\frac{1}{k!} \cdot \der_{\varphi^k} T_k(\theta(\varphi))^\top
\der_{\varphi^k} T_k(\theta(\varphi))\Big|_{\varphi=\varphi(\ttheta)}.\]
Thus this matrix has full rank $d_k$.
\end{proof}

\begin{proof}[Proof of Theorem \ref{thm:FI}]
Consider the unprojected setting of part (a). For generic $\theta_* \in \R^d$,
by Lemma \ref{lemma:transcendencebasis},
we have $\rank\der M_K(\theta_*)=d_1+\ldots+d_K$. Let
$\varphi=(\varphi^1,\ldots,\varphi^K,\bar\varphi)$ be the map defined by
Lemma \ref{lemma:transcendencebasis}, with inverse $\theta(\varphi)$ in
a neighborhood $U$ of $\theta_*$. We may assume without loss of
generality that
\begin{equation}\label{eq:orthogradients}
A=\der \varphi(\theta_*) \text{ is orthogonal}
\end{equation}
upon replacing each function in $\varphi$ by a ($\theta_*$-dependent)
linear combination of itself
and its preceding functions. Note that statements (c) and (d)
of Lemma \ref{lemma:transcendencebasis} continue to hold after such a
replacement. We denote
$\varphi_*=\varphi(\theta_*)$. With slight abuse of notation, we write as
shorthand $f(\varphi)$ for $f(\theta(\varphi))$. In particular,
recalling the expansion (\ref{eq:seriesexpansionunprojected}), we denote by
$q_k(\varphi),s_k(\varphi),q(\varphi)$ the terms of this expansion parametrized
by $\varphi \in \varphi(U)$. 

For (a1), observe that Theorem \ref{thm:seriesexpansion}(a)
guarantees $s_k \in \cR_{\leq k}^\G$, so Lemma \ref{lemma:transcendencebasis}(c)
shows that $s_k(\varphi)$ depends only on $\varphi^1,\ldots,\varphi^k$ in
the reparametrization by $\varphi$. Similarly, $q_k \in \cR_{\leq k-1}^\G$, so
$q_k(\varphi)$ depends only on
$\varphi^1,\ldots,\varphi^{k-1}$, and $q$ is a continuous $\G$-invariant
function, so $q(\varphi)$ depends only on
$\varphi^1,\ldots,\varphi^K$ by Lemma \ref{lemma:transcendencebasis}(d).
Let us decompose $\nabla^2_{\varphi}R(\varphi_*)$
into $(K+1)\times (K+1)$ blocks according to the partition
$(\varphi^1,\ldots,\varphi^K,\bar\varphi)$, of sizes $(d_1,\ldots,d_K,d_0)$.
Differentiating the expansion
(\ref{eq:seriesexpansionunprojected}) term-by-term,
it then follows that the entries of
$\nabla^2_{\varphi}R(\varphi_*)$ are non-zero only in the upper-left $K\times
K$ blocks, and that the $(k,\ell)$ block corresponding to
$\nabla^2_{\varphi^k,\varphi^\ell}R(\varphi_*)$ has operator norm bounded above
by $C\sigma^{-2\max(k,\ell)}$ for a constant $C>0$.
Furthermore, as $\nabla_{\varphi^k}^2 q_k(\varphi)=0$
and $\nabla_{\varphi^k}^2 s_k(\varphi_*)$ is strictly positive definite
by Lemma \ref{lemma:skfullrank} (applied with $\ttheta=\theta_*$),
there are constants $c,\sigma_0>0$ such that for all $\sigma>\sigma_0$,
\begin{equation}\label{eq:gradedblocks2}
\lambda_{\min}\Big(\nabla_{\varphi^k}^2 R(\varphi_*)\Big) \geq
c\sigma^{-2k}>0 \text{ for all } k=1,\ldots,K.
\end{equation}
So the upper-left $K \times K$ blocks of
$\nabla^2_{\varphi}R(\varphi_*)$ have the graded block structure
of Definition \ref{def:gradedblock} with $d'=d-d_0$.
Then by Lemma \ref{lemma:gradedblock}(a),
$\nabla^2_{\varphi}R(\varphi_*)$ has $d_0$ eigenvalues equal to 0, and $d_k$
eigenvalues in $[c\sigma^{-2k},C\sigma^{-2K}]$ for each $k=1,\ldots,K$.
For the Hessian in $\theta$ rather than in $\varphi$,
since $\nabla_{\theta}R(\theta_*)=0$, we have by the chain rule
\begin{equation}\label{eq:chainruleI}
I(\theta_*)=\nabla^2_{\theta}R(\theta_*)=A^\top\cdot\nabla^2_{\varphi}R(\varphi_*)\cdot
A, \qquad A=\der \varphi(\theta_*).
\end{equation}
By the orthogonality of $A$ assumed in (\ref{eq:orthogradients}), the
eigenvalues of $I(\theta_*)$ are the same as those of
$\nabla^2_{\varphi}R(\varphi_*)$, and this shows (a1).

For (a2), observe from (\ref{eq:chainruleI})
that the subspace $V_k$ spanned by the
$d_1+\ldots+d_k$ leading eigenvectors of $I(\theta_*)$ is given by $V_k=A^\top
\cdot V_{k,\varphi}$, where $V_{k,\varphi}$ is the subspace spanned by the
$d_1+\ldots+d_k$ leading eigenvectors of 
$\nabla^2_{\varphi}R(\varphi_*)$. By Lemma \ref{lemma:gradedblock}(b),
the submatrix of upper-left $k \times k$ blocks 
of $\nabla^2_{\varphi}R(\varphi_*)$ has smallest eigenvalue at least
$c\sigma^{-2k}$, while the remaining blocks have operator norm
at most $C\sigma^{-2(k+1)}$ by Definition \ref{def:gradedblock}. Let
$W_{k,\varphi}$ be the subspace of vectors having only first $d_1+\ldots+d_k$
coordinates non-zero. Then the Davis-Kahan theorem implies that, for a
function $\eps(\sigma)$ satisfying $\eps(\sigma) \to 0$ as $\sigma \to \infty$,
\[\|\sin \Theta(V_{k,\varphi},W_{k,\varphi})\|<\eps(\sigma).\]
For any polynomial $p \in \cR_{\leq k}^\G$,
we have $\nabla_{\theta}p(\theta_*)=A^\top \nabla_{\varphi}p(\varphi_*)$ by the
chain rule. Lemma \ref{lemma:transcendencebasis}(d) shows that $p$ depends only
on $\varphi^1,\ldots,\varphi^k$ in the parametrization by $\varphi$,
so $\nabla_{\varphi}p(\varphi_*) \in W_{k,\varphi}$, and hence
$\nabla_{\theta}p(\theta_*) \in A^\top \cdot W_{k,\varphi}$.
For generic $\theta_* \in \R^d$, the linear span $W_k$ of all
such gradient vectors $\nabla_{\theta}p(\theta_*)$ has dimension exactly
$d_1+\ldots+d_k$ by Lemma \ref{lemma:jaccrit}, so this shows $W_k=A^\top
\cdot W_{k,\varphi}$. Thus also
\[\|\sin \Theta(V_k,W_k)\|=\|\sin
\Theta(V_{k,\varphi},W_{k,\varphi})\|<\eps(\sigma).\]

For (a3), observe that since $A$ is orthogonal, we have from
(\ref{eq:chainruleI}) that $I(\theta_*)^\dagger=A^\top \cdot \nabla_\varphi^2
R(\varphi_*)^\dagger \cdot A$ for the Moore-Penrose pseudo-inverse.
Combining this with $\nabla_{\theta}p(\theta_*)=A^\top
\nabla_{\varphi}p(\varphi_*)$, we have
\[\nabla_\theta p(\theta_*)^\top I(\theta_*)^\dagger \nabla_\theta p(\theta_*)
=\nabla_\varphi p(\varphi_*)^\top \cdot
\nabla_\varphi^2 R(\varphi_*)^\dagger \cdot \nabla_\varphi p(\varphi_*).\]
Lemma \ref{lemma:gradedblock}(c) shows that the maximum eigenvalue of 
the upper-left $k \times k$ blocks of $\nabla_\varphi^2 R(\varphi_*)^\dagger$
is at most $C\sigma^{2k}$. Since $\nabla_{\varphi}p(\varphi_*)$ is non-zero
only in its first $k$ blocks, this implies
\[\nabla_\theta p(\theta_*)^\top I(\theta_*)^\dagger \nabla_\theta p(\theta_*)
\leq C\sigma^{2k}.\]
Finally,
if $w$ is in the null space of $I(\theta_*)$, then $Aw$ is in the null space of
$\nabla_\varphi^2 R(\varphi_*)$, i.e.\ $Aw$
is non-zero only in the last block corresponding to $\bar\varphi$. Then
$\nabla_\theta p(\theta_*)^\top w=\nabla_\varphi p(\varphi_*)^\top Aw=0$,
so $\nabla_\theta p(\theta_*)$ is orthogonal to the null space of $I(\theta_*)$.
This shows (a3).

The proof of part (b) in the projected setting is similar: We let
$\varphi=(\varphi^1,\ldots,\varphi^{\tK},\bar\varphi)$ be the map defined by
Lemma \ref{lemma:transcendencebasis}, and compute the Hessian of
(\ref{eq:seriesexpansionprojected}) in the parametrization by $\varphi$
term-by-term. In this computation, there is an additional contribution
from each term $\sigma^{-2k} \langle \tT_k(\varphi),P_k(\varphi) \rangle$.
This term depends only on $\varphi^1,\ldots,\varphi^k$, so its
Hessian lies only in the upper-left $k \times k$ blocks of the $(\tK+1) \times
(\tK+1)$ block decomposition of $\nabla_\varphi^2 R(\varphi_*)$.
Hence each block
$\nabla_{\varphi^k,\varphi^\ell}^2 R(\varphi_*)$ still has operator norm bounded
above by $C\sigma^{-2\max(k,\ell)}$.
Furthermore, the Hessian of $\sigma^{-2k} \langle \tT_k(\varphi),P_k(\varphi)
\rangle$ is the sum of three terms, corresponding to
differentiating twice $\tT_k(\varphi)$, twice $P_k(\varphi)$, and
once each $\tT_k(\varphi)$ and $P_k(\varphi)$. The first term vanishes upon
evaluating at $\varphi=\varphi_*$, because $P_k(\varphi_*)=0$ by its
characterization in Theorem \ref{thm:seriesexpansion}(b). The remaining two
terms are 0 on the $(k,k)$ block, because $P_k(\varphi)$
depends only on $\varphi^1,\ldots,\varphi^{k-1}$. Thus the Hessian of
$\sigma^{-2k} \langle \tT_k(\varphi),P_k(\varphi) \rangle$ at
$\varphi=\varphi_*$ is 0 in the $(k,k)$ block, so we still have
$\nabla_{\varphi^k}^2 R(\varphi_*) \succeq c\sigma^{-2k}$.
Then the upper-left $\tK \times \tK$
blocks of $\nabla_\varphi^2 R(\varphi_*)$ still have the graded block structure
of Definition \ref{def:gradedblock}, and the remainder of the
proof is the same as in the unprojected setting of part (a).
\end{proof}

\begin{proof}[Proof of Lemma \ref{lem:trdeg}]
We focus on the unprojected setting; the proof in the projected setting is
the same.

Note that $\theta_*$ is a global minimizer of $s_k(\theta)$, so $\nabla_\theta
s_k(\theta_*)=0$ and $\nabla_\theta^2 s_k(\theta_*) \succeq 0$.
For generic $\theta_*$, we have $\rank \der M_K(\theta_*)=d_1+\ldots+d_K$
by Lemma \ref{lemma:transcendencebasis}.
Let $\varphi=(\varphi^1,\ldots,\varphi^K,\bar\varphi)$ be the map
defined by Lemma \ref{lemma:transcendencebasis}, with inverse $\theta(\varphi)$
in a neighborhood $U$ of $\theta_*$.
Let $\varphi_*=\varphi(\theta_*)$. Then by the chain rule,
\begin{equation}\label{eq:skchainrule}
\nabla_\varphi^2 s_k(\theta(\varphi))\big|_{\varphi=\varphi_*}
=\der_\varphi \theta(\varphi_*)^\top \cdot \nabla_\theta^2 s_k(\theta_*) \cdot
\der_\varphi \theta(\varphi_*).
\end{equation}
Since $\der_\varphi \theta(\varphi_*)$ is non-singular, this yields
\[\rank\Big(\nabla_\theta^2 s_1(\theta_*)+\ldots+\nabla_\theta^2 s_k(\theta_*)
\Big)=\rank\Big(\nabla_\varphi^2 s_1(\theta(\varphi))+\ldots
+\nabla_\varphi^2 s_k(\theta(\varphi))\Big|_{\varphi=\varphi_*}\Big).\]
Lemma \ref{lemma:transcendencebasis}(c) ensures that
$s_1(\theta(\varphi)),\ldots,s_k(\theta(\varphi))$ depend only on
$\varphi^1,\ldots,\varphi^k$, so
\[\rank\Big(\nabla_\varphi^2 s_1(\theta(\varphi))+\ldots
+\nabla_\varphi^2 s_k(\theta(\varphi))\Big|_{\varphi=\varphi_*}\Big)
\leq d_1+\ldots+d_k=\trdeg(\cR_{\leq k}^\G).\]
To show that this holds with equality, consider any non-zero vector
$v=(v_1,\ldots,v_k,0,\ldots,0) \in \R^d$, where $v_j$ is the subvector
corresponding to the coordinates of $\varphi^j$. Let $j \in
\{1,\ldots,k\}$ be the smallest index for which $v_j \neq 0$.
Lemma \ref{lemma:skfullrank} applied with $\ttheta=\theta_*$
shows that $\nabla_{\varphi^j}^2
s_j(\theta(\varphi))|_{\varphi=\varphi_*} \succ 0$ strictly, so
\[v^\top \Big[\nabla_\varphi^2
s_j(\theta(\varphi))\Big|_{\varphi=\varphi_*}\Big] v
=v_j^\top \Big[\nabla_{\varphi^j}^2
s_j(\theta(\varphi))\Big|_{\varphi=\varphi_*}\Big] v_j>0\]
where the first equality holds because $v_1=\ldots=v_{j-1}=0$, whereas
$s_j(\theta(\varphi))$ depends only on $\varphi^1,\ldots,\varphi^j$.
Furthermore, $v^\top [\nabla_\varphi^2
s_i(\theta(\varphi))|_{\varphi=\varphi_*}] v \geq 0$ for all
$i \neq j$, because (\ref{eq:skchainrule}) and the condition $\nabla_\theta^2
s_i(\theta_*) \succeq 0$ imply that
$\nabla_\varphi^2 s_i(\theta(\varphi))|_{\varphi=\varphi_*}$
is positive semidefinite. Then
$v^\top[\nabla_\varphi^2 s_1(\theta(\varphi))+\ldots
+\nabla_\varphi^2 s_k(\theta(\varphi))|_{\varphi=\varphi_*}]v>0$ strictly.
This holds for every non-zero vector $v=(v_1,\ldots,v_k,0,\ldots,0) \in \R^d$,
so in fact
\[\rank\Big(\nabla_\varphi^2 s_1(\theta(\varphi))+\ldots
+\nabla_\varphi^2 s_k(\theta(\varphi))\Big|_{\varphi=\varphi_*}\Big)
=d_1+\ldots+d_k=\trdeg(\cR_{\leq k}^\G).\]

This shows also that the column span of 
$\nabla_\varphi^2 s_1(\theta(\varphi))+\ldots
+\nabla_\varphi^2 s_k(\theta(\varphi))|_{\varphi=\varphi_*}$ is exactly the
space of vectors $v$ with only its first $d_1+\ldots+d_k$ coordinates non-zero.
Then by (\ref{eq:skchainrule}),
the column span of $\nabla_\theta^2 s_1(\theta_*)+\ldots+\nabla_\theta^2
s_k(\theta_*)$ is the span of the first $d_1+\ldots+d_k$ rows of
$\der_\varphi \theta(\varphi_*)^{-1}=\der_\theta \varphi(\theta_*)$,
which are the gradients of $\varphi^1(\theta_*),\ldots,\varphi^k(\theta_*)$. By Lemma
\ref{lemma:transcendencebasis}, the span of these gradients is exactly the span
of $\{\nabla p(\theta_*):p \in \cR_{\leq k}^\G\}$, concluding the proof.
\end{proof}

\subsection{Global landscape}

We prove Theorems \ref{thm:benignlandscape} and \ref{thm:landscape}.
The following lemma first shows that all critical points of $R(\theta)$
described by Theorems \ref{thm:benignlandscape} and \ref{thm:landscape} in fact
belong to a ball of constant radius $M>0$, independent of $\sigma$.
In unprojected models, this result was proven in
\cite[Lemmas 2.10 and 4.19]{fan2020likelihood}. The argument is reviewed
and extended in the proof below, to the domain
$\{\theta:\|\theta\|<B(\|\theta_*\|+\sigma)\}$ for projected models
under the assumption (\ref{eq:losslessPi}) for the projection $\Pi$.

\begin{lemma}\label{lemma:localizecrit}
In the unprojected model,
for some constants $M,c,\sigma_0>0$ depending on $\theta_*,\G$ and 
for all $\sigma>\sigma_0$,
\[\|\nabla R(\theta)\| \geq c\sigma^{-4} \quad \text{ for all } \theta \text{
satisfying } \|\theta\|>M.\]
In the projected model with projection $\Pi$, for any $B>0$, some
constants $M,c,\sigma_0>0$ depending on $\theta_*,\G,\Pi,B$, and 
all $\sigma>\sigma_0$,
\[\|\nabla R(\theta)\| \geq c\sigma^{-4} \quad \text{ for all } \theta \text{
satisfying } B(\|\theta_*\|+\sigma)>\|\theta\|>M.\]
\end{lemma}
\begin{proof}

{\bf Step 1: Forms of $\nabla R(\theta)$.} 
We consider the projected model, which will reduce to the unprojected model
when $\Pi = \Id$. 
Write $\E_g$ and $\E_{g,g'}$ for expectations over independent group elements
$g,g' \sim \Lambda$, and $\E_Y$ for that over the sample
$Y \sim p_{\theta_*}$. Introduce the weight
\[p(g,Y)=\frac{\exp\left(-\frac{1}{2\sigma^2}\|Y-\Pi g \theta\|^2\right)}
{\E_{g'}\left[\exp\left(-\frac{1}{2\sigma^2}\|Y-\Pi g'\theta\|^2\right)\right]}
=\frac{\exp\left(\frac{1}{\sigma^2}Y^\top \Pi g\theta
-\frac{1}{2\sigma^2}\|\Pi g\theta\|^2\right)}
{\E_{g'}\left[\exp\left(\frac{1}{\sigma^2}Y^\top \Pi
g'\theta-\frac{1}{2\sigma^2}\|\Pi g'\theta\|^2\right)\right]}.\]
Then
\begin{align}
\sigma^2 \nabla R(\theta)=\sigma^2 \nabla_\theta \E_Y[-\log
p_\theta(Y)]&=-\sigma^2 \nabla_\theta \E_Y\left[\log \E_g\left[
\exp\left(-\frac{\|Y-\Pi g \theta\|^2}{2\sigma^2}\right)\right]\right]\nonumber\\
&=-\E_Y\left[\E_g\left[p(g,Y)
g^\top \Pi^\top(Y-\Pi g\theta) \right]\right].\label{eq:gradRtmp}
\end{align}

We derive a second alternative form for $\sigma^2 \nabla R(\theta)$ using
Gaussian integration by parts.
Let us represent $Y=\Pi h \theta_*+\sigma \eps$, where $h \sim \Lambda$ and
$\eps \sim \N(0,\Id)$, and write $\E_Y=\E_{h,\eps}$. It follows
from (\ref{eq:gradRtmp}) that 
\begin{align}\label{eq:gradRtmp2}
\sigma^2 \nabla R(\theta)=&-\E_{h,\eps}\left[\E_g\left[p(g,Y)
g^\top \Pi^\top\Pi h\theta_* \right]\right]+\E_{h,\eps}\left[\E_g\left[p(g,Y)
g^\top \Pi^\top\Pi g\theta \right]\right]\notag\\
&\hspace{1in}-\E_{h,\eps}\left[\E_g\left[p(g,Y) g^\top \Pi^\top (\sigma \eps)
\right]\right].
\end{align}
For the third term above, applying the integration by parts
identity $\E_{\xi \sim \N(0,1)}[\xi f(\xi)]=\E_{\xi \sim \N(0,1)}[f'(\xi)]$ to
each coordinate of $\eps$, we have for $Y=\Pi h\theta_*+\sigma \eps$ and any fixed $h$,
\begin{align*}
\E_\eps\left[\E_g\left[p(g,Y) g^\top \Pi^\top (\sigma \eps)
\right]\right]
=\sigma \, \E_\eps\left[\E_g\left[g^\top \Pi^\top \nabla_\eps
p(g,Y)\right]\right].
\end{align*}
Explicitly differentiating $p(g,Y)=p(g,\Pi h\theta_*+\sigma \eps)$ in $\eps$ gives
\begin{align*}
\nabla_\eps p(g,Y)=
\sigma\,\nabla_Y p(g,Y)=\frac{1}{\sigma}\Big(
p(g,Y) \Pi g\theta - p(g,Y) \E_{g'}\left[ p(g',Y) \Pi g'\theta\right]\Big).
\end{align*}
Thus,
\begin{align*}
\E_\eps\left[\E_g\left[p(g,Y) g^\top \Pi^\top (\sigma \eps)
\right]\right]
&=\E_\eps\left[\E_g\left[p(g,Y) g^\top \Pi^\top \Pi g \theta\right]
-\E_{g,g'}\left[p(g,Y)p(g',Y) g^\top \Pi^\top \Pi g' \theta\right]\right].
\end{align*}
Then, taking the expectation also over $h \sim \Lambda$ and
substituting this for the third term in (\ref{eq:gradRtmp2}),
\begin{equation}\label{eq:gradR}
\sigma^2 \nabla R(\theta)
=\E_{h,\eps}\left[\E_{g,g'}\left[p(g,Y)p(g',Y)g^\top \Pi^\top
\Pi g'\theta\right]
-\E_g\left[p(g,Y)g^\top \Pi^\top \Pi h\theta_*\right]\right].
\end{equation}
The expressions (\ref{eq:gradRtmp}) and (\ref{eq:gradR}) hold also in
the unprojected model upon substituting $\Pi=\Id$, where they may be
further reduced to \cite[Eqs.\ (2.8--2.9)]{fan2020likelihood}.\\

{\bf Step 2: Gradient bound for $\|\theta\| \geq B(\|\theta_*\|+\sigma)$.}
In the unprojected model, fixing a sufficiently large constant $B>0$, let us
first derive a bound $\|\nabla R(\theta)\| \geq c\sigma^{-1}$ for
$\|\theta\| \geq B(\|\theta_*\|+\sigma)$. 
(This is the same argument as in \cite[Lemma 2.9]{fan2020likelihood},
which for convenience we reproduce here.)
Restricting to $\Pi=\Id$ and taking
the inner-product of (\ref{eq:gradRtmp}) with $\theta$,
\[\sigma^2\|\theta\| \cdot \|\nabla R(\theta)\|
\geq \sigma^2 \theta^\top \nabla R(\theta)
=\|\theta\|^2-\E_Y\Big[\E_g\Big[p(g,Y)\theta^\top g^\top Y\Big]\Big]
\geq \|\theta\|^2-C(\|\theta_*\|+\sigma)\|\theta\|\]
for a constant $C=C(\theta_*,\G)>0$. Then for sufficiently large $B>0$ and large
$\sigma$, this shows $\|\nabla R(\theta)\| \geq c\sigma^{-1}$ as claimed.
\\

{\bf Step 3: Gradient bound for
$\|\theta\| \geq C_0\sigma^{2/3}$.} We now show
the lower bound $\|\nabla R(\theta)\| \geq c\sigma^{-2}$
when $B(\|\theta_*\|+\sigma) \geq \|\theta\| \geq C_0\sigma^{2/3}$, for a
large enough constant $C_0>0$ and $\sigma>\sigma_0(\theta_*,\G,\Pi,B)$.
The bound in the unprojected model follows from specializing to $\Pi=\Id$.

Define unit vectors
$\bar{\theta}=\theta/\|\theta\|$ and $\bar{Y}=Y/\|Y\|$. Now taking the
inner-product of (\ref{eq:gradR}) with $\bar{\theta}$,
\begin{align*}
\sigma^2\|\nabla R(\theta)\|
&\geq \sigma^2 \cdot \bar{\theta}^\top \nabla R(\theta)\\
&=\|\theta\| \cdot
\E_{h,\eps}\left[\left\|\E_g\left[p(g,Y)\Pi g\bar{\theta}\right]
\right\|^2\right]-\bar{\theta}^\top \E_{h,\eps}\left[
\E_g\left[p(g,Y)g^\top \Pi^\top \Pi h\right]\right] \cdot \theta_*\\
&\geq \|\theta\| \cdot \E_{h,\eps}\left[
\left(\bar{Y}^\top \E_g\left[p(g,Y) \Pi g\bar{\theta}\right]\right)^2\right]
-\|\Pi\|^2 \cdot \|\theta_*\|.
\end{align*}
For fixed $\theta$ and $Y$, define
\[K(t)=\log \E_g\left[\exp\left(t \bar{Y}^\top \Pi g\bar{\theta}
-\frac{1}{2\sigma^2}\|\Pi g\theta\|^2\right)\right],
\qquad t(Y,\theta)=\frac{\|Y\|\cdot\|\theta\|}{\sigma^2}.\]
Then $\bar{Y}^\top\E_g[p(g,Y) \Pi g\bar{\theta}]=\E_g[p(g,Y) \bar{Y}^\top \Pi
g\bar{\theta}]=K'(t(Y,\theta))$, so
\begin{equation}\label{eq:gradlowertmp}
\sigma^2 \|\nabla R(\theta)\|
\geq \|\theta\| \cdot \E_{h,\eps}\left[K'(t(Y,\theta))^2\right]
-\|\Pi\|^2 \cdot \|\theta_*\|.
\end{equation}

Define a tilted probability distribution $\bar{\Lambda}$ for $g \in \G$,
having density $\der \bar{\Lambda}(g)
\propto \exp(-\frac{\|\Pi g\theta\|^2}{2\sigma^2}) \der \Lambda(g)$
with respect to the Haar measure $\Lambda$. 
Observe that $K(t)-K(0)$ is the cumulant generating function for the law of
$\bar{Y}^\top \Pi g\bar{\theta}$ (fixing $\bar{Y}$) that is induced by $g \sim
\bar{\Lambda}$. We proceed to analyze the cumulants of $\bar{Y}^\top \Pi
g\bar{\theta}$ under this law. This is simpler in the unprojected setting
of $\Pi=\Id$, where $\bar{\Lambda}=\Lambda$ and $\bar{Y}^\top \Pi
g\bar{\theta}=\bar{Y}^\top g\bar{\theta}$; in this setting,
upper and lower bounds for the
cumulants were established in the proof of
\cite[Lemma 2.10]{fan2020likelihood}. Here, we extend
these bounds to the setting of a general projection
$\Pi$ that satisfies (\ref{eq:losslessPi}).

Note that for $\|\theta\|<B(\|\theta_*\|+\sigma)$ and $\sigma>\sigma_0$, we have
\begin{equation}\label{eq:tiltingbound}
\frac{\der \bar{\Lambda}}{\der \Lambda}(g) \in [c,C]
\end{equation}
for some $(\theta_*,\G,\Pi,B)$-dependent constants $C,c,\sigma_0>0$.
The random variable $\bar{Y}^\top \Pi g \bar{\theta}$ is bounded as
$|\bar{Y}^\top \Pi g \bar{\theta}| \leq \|\Pi\|$,
so for $|t|<1/(\|\Pi\|e)$, $K(t)-K(0)$ is defined equivalently by the convergent
cumulant series
\[K(t)-K(0)=\sum_{\ell \geq 1} \kappa_\ell (\bar{Y}^\top \Pi g\bar{\theta})
\frac{t^\ell}{\ell!}.\]
Here, $\kappa_\ell=\kappa_\ell(\bar{Y}^\top \Pi g\bar{\theta})$
is the $\ell^\text{th}$ cumulant of
$\bar{Y}^\top \Pi g \bar{\theta}$ under its law induced by $g \sim \bar{\Lambda}$,
satisfying $|\kappa_\ell| \leq (\|\Pi\|\ell)^\ell$
(c.f.\ \cite[Lemma A.1]{fan2020likelihood}). In particular, $\kappa_1$ is the
mean and $\kappa_2$ is the variance. Then for any $0<t<1/(\|\Pi\|e)$,
applying also $\ell! \geq (\ell/e)^\ell$ and 
convexity of the cumulant generating function $K(t)$,
\begin{equation}\label{eq:CGFlower}
K'(t) \geq \frac{K(t)-K(0)}{t}
=\sum_{\ell \geq 1} \kappa_\ell \frac{t^{\ell-1}}{\ell!}
\geq \kappa_1+\frac{t}{2}\kappa_2-\sum_{\ell \geq 3} (\|\Pi\| e)^\ell
t^{\ell-1}.
\end{equation}

We now lower-bound the mean and variance $\kappa_1,\kappa_2$ when
$\bar{Y}$ belongs to some ``good'' subset $U$ of the unit sphere:
First note that for any non-zero $\theta \in \R^d$, $\Pi g\theta$ cannot be
identically 0 over all $g \in \G$. This is because otherwise, 
$\Pi(\orbit_{\theta_*}) \equiv \Pi(\orbit_{\theta_*+c\theta})$ for any $\theta_*
\in \R^d$ and any $c \in \R$, where $\{\orbit_{\theta_*+c\theta}:c \in \R\}$ is
an infinite family of distinct orbits, violating (\ref{eq:losslessPi}). Thus,
denoting by $\cS^{d-1}$ the unit sphere in $\R^d$,
\[\sup_{\bar{y}:\|\bar{y}\|=1} \sup_{g \in \G} \bar{y}^\top
\Pi g\bar{\theta}>0 \text{ for all } \bar{\theta} \in \cS^{d-1}\]
By continuity of the left side as a function of $\bar{\theta}$ and by
compactness of $\cS^{d-1}$, there is then a constant $c=c(\Pi,\G)>0$
such that
\[\sup_{\bar{y}:\|\bar{y}\|=1} \sup_{g \in \G} \bar{y}^\top
\Pi g\bar{\theta}>c \text{ for all } \bar{\theta} \in \cS^{d-1}.\]
Let $\Gamma$ denote the uniform probability measure on $\cS^{d-1}$. 
Since $\{(\bar{y},g):\bar{y}^\top \Pi g\bar{\theta}>c\}$ is an open
subset of $\cS^{d-1} \times \G$, the above implies
\[\Gamma \times \Lambda\Big((\bar{y},g):\bar{y}^\top
\Pi g\bar{\theta}>c\Big)>0 \text{ for all } \bar{\theta} \in \cS^{d-1}.\]
By the bounded convergence theorem and lower-semicontinuity of $x
\mapsto \1\{x>c\}$, if $\bar{\theta}_k \in \cS^{d-1}$
is a sequence converging to $\bar{\theta} \in \cS^{d-1}$, then
\[\liminf_{k \to \infty} \Gamma \times \Lambda\Big((\bar{y},g):\bar{y}^\top \Pi
g\bar{\theta}_k>c\Big)
=\E_{\bar{y},g \sim \Gamma \times \Lambda}
\Big[\liminf_{k \to \infty} \1\{\bar{y}^\top \Pi
g\bar{\theta}_k>c\}\Big]
\geq \Gamma \times \Lambda\Big((\bar{y},g):\bar{y}^\top \Pi
g\bar{\theta}>c\Big).\]
Thus $\bar{\theta} \mapsto \Gamma \times \Lambda((\bar{y},g):\bar{y}^\top \Pi
g\bar{\theta}>c)$ is lower-semicontinuous on $\cS^{d-1}$, so again by
compactness of $\cS^{d-1}$, there is a constant $\delta=\delta(\Pi,\G)>0$
such that
\[\Gamma \times \Lambda\Big((\bar{y},g):\bar{y}^\top
\Pi g\bar{\theta}>c\Big)>\delta \text{ for all } \bar{\theta} \in \cS^{d-1}.\]
Then by (\ref{eq:tiltingbound}), for a constant
$\delta'=\delta'(\theta_*,\Pi,\G,B)$, we get
$\Gamma \times \bar{\Lambda}((\bar{y},g):\bar{y}^\top
\Pi g\bar{\theta}>c)>\delta'$.
Define the $\bar{\theta}$-dependent subset of the unit sphere
\[U'=\Big\{\bar{y}:\bar{\Lambda}(g \in \G:\bar{y}^\top \Pi
g\bar{\theta}>c)>\delta'/2\Big\}.\]
Then the above implies $\Gamma(U')+(\delta'/2)(1-\Gamma(U'))>\delta'$,
so $\Gamma(U')>\delta'/2$.
If any random variable $X \in \R$ satisfies $\P[X>c]>\delta'/2$ for 
constants $c,\delta'>0$, then
$\max(\E[X],\Var[X])>c'$ for a constant $c'=c'(c,\delta')>0$, and
furthermore either $\E[X] \geq 0$ or $\E[-X] \geq 0$. Thus, defining $U$ from
$U'$ by multiplying each element $\bar{y} \in U'$ by an appropriate choice of
$\pm$ sign, we obtain $\Gamma(U)>\delta'/4$ and
\begin{equation}\label{eq:kappa12bound}
\max\Big(\kappa_1(\bar{Y}^\top \Pi g\bar{\theta}),
\kappa_2(\bar{Y}^\top \Pi g\bar{\theta})\Big)>c' \text{ and }
\kappa_1(\bar{Y}^\top \Pi g\bar{\theta}),
\kappa_2(\bar{Y}^\top \Pi g\bar{\theta}) \geq 0 \text{ whenever } \bar{Y} \in
U.
\end{equation}

Recalling that $Y=\Pi h \theta_*+\sigma \eps$ and $\bar{Y}=Y/\|Y\|$, the law of
$\bar{Y}$ converges to the uniform measure $\Gamma$ on the sphere
as $\sigma \to \infty$. Thus, for $\sigma>\sigma_0(\Pi,\G,\theta_*,B)$, we have
$\P[\bar{Y} \in U] \geq \Gamma(U)/2>\delta'/8$.
For a constant $C_0=C_0(\Pi,\theta_*,\delta',c')>0$ large enough and to be
determined, if $\|\theta\| \geq C_0\sigma^{2/3}$, then
\[t(Y,\theta)=\frac{\|Y\| \cdot \|\theta\|}{\sigma^2}
=\frac{\|\Pi h\theta_*+\sigma \eps\| \cdot \|\theta\|}{\sigma^2}
\geq \sigma^{-1/3}\]
with probability at least $1-\delta'/16$. Then,
on an event of probability at least $\delta'/16$ where both $\bar{Y} \in U$
and $t(Y,\theta) \geq \sigma^{-1/3}$, and for
$\sigma>\sigma_0(\Pi,\G,\theta_*,B)$, we have
\[K'(t(Y,\theta)) \geq K'(\sigma^{-1/3}) \geq (c'/3)\sigma^{-1/3},\]
the first inequality applying convexity of $K(t)$ and the second applying
(\ref{eq:CGFlower}) with the bound (\ref{eq:kappa12bound}). Then, applying this
to (\ref{eq:gradlowertmp}),
\[\sigma^2\|\nabla R(\theta)\|^2
\geq C_0\sigma^{2/3} \cdot (\delta'/16) \cdot (c'/3)^2 \sigma^{-2/3}
-\|\Pi\|^2 \cdot \|\theta_*\|.\]
Here, the constants $c',\delta'>0$ are as defined in the argument leading to
(\ref{eq:kappa12bound}), and do not depend on $C_0$.
Then taking $C_0=C_0(\Pi,\theta_*,\delta',c')$ large enough ensures that
$\|\nabla R(\theta)\|^2 \geq c\sigma^{-2}$ as desired.\\

{\bf Step 4: Gradient bound for $\|\theta\|>M$.}
Finally, we show $\|\nabla R(\theta)\| \geq c\sigma^{-4}$ for 
$C_0\sigma^{2/3} \geq \|\theta\|>M$ and a sufficiently large constant $M>0$.
This is again simpler in the unprojected setting of $\Pi=\Id$, and
was shown in \cite[Lemma 4.19]{fan2020likelihood}. Here, we extend the
argument to the projected model using the series expansion of Theorem
\ref{thm:seriesexpansion}(b), and specializing to $\Pi=\Id$ again
recovers the result in the unprojected setting.

Define $v=(\theta-\theta_*)/\|\theta-\theta_*\|$, and suppose first that
$\|\E_g[\Pi g v]\| \geq c_0$ for any $(\Pi,\G,\theta_*)$-dependent
constant $c_0>0$.
We apply the expansion (\ref{eq:seriesexpansionprojected}) to the order $K=1$.
Noting
that $P_1(\theta)=0$ (because it is a constant that is 0 at $\theta=\theta_*$)
and $q_1(\theta)$ is a constant,
\[\nabla R(\theta)=\frac{1}{\sigma^2}\nabla \ts_1(\theta)+\nabla q(\theta),
\qquad \|\nabla q(\theta)\| \leq \frac{C\|\theta\|^3}{\sigma^4}.\]
Applying the form of $\ts_1(\theta)$ in Lemma \ref{lem:skform},
\[\nabla R(\theta)=\frac{1}{\sigma^2}\E_{g,h}[g^\top \Pi^\top \Pi
h](\theta-\theta_*)+\nabla q(\theta).\]
Then
\[\sigma^2\|\nabla R(\theta)\| \geq \sigma^2\,v^\top \nabla R(\theta)
\geq \|\theta-\theta_*\| \cdot \|\E_g[\Pi g v]\|^2-C\|\theta\|^3/\sigma^2.\]
When $C_0\sigma^{2/3} \geq \|\theta\|>M$, $\|\E_g[\Pi g v]\| \geq c_0$,
and $M=M(C_0,c_0,\Pi,\theta_*)$ is
large enough, this is lower-bounded by a positive constant, so $\|\nabla
R(\theta)\| \geq c\sigma^{-2}$.

Now suppose $\|\E_g[\Pi g v]\|<c_0$, where we will choose this
$(\Pi,\G,\theta_*)$-dependent constant $c_0 \in (0,1)$ sufficiently small and
to be determined. Constants $C,C',c,c'>0$ below
are independent of $c_0$, and we will track explicitly the dependence of the
argument on $c_0$. We apply the expansion (\ref{eq:seriesexpansionprojected})
to the order $K=2$. Then similarly,
\begin{equation}\label{eq:sigma4grad}
\sigma^4\|\nabla R(\theta)\| \geq
\sigma^4\,v^\top \nabla R(\theta) \geq 
v^\top \nabla \ts_2(\theta)+v^\top
\nabla [\langle \tT_2(\theta),P_2(\theta) \rangle]
+v^\top \nabla q_2(\theta)-C\|\theta\|^5/\sigma^2,
\end{equation}
where we have applied $v^\top \nabla \ts_1(\theta) \geq 0$ from the form of
$\nabla \ts_1(\theta)$ above to
drop the contribution from the $k=1$ term. We bound each expression on the
right side of (\ref{eq:sigma4grad}):
First, applying the form of $\ts_2(\theta)$ in Lemma \ref{lem:skform},
\begin{align*}
v^\top \nabla \ts_2(\theta)
&=v^\top \E_{g,h}\Big[g^\top \Pi^\top \Pi h \theta \cdot
\theta^\top g^\top \Pi^\top \Pi h \theta
-g^\top \Pi^\top \Pi h \theta_* \cdot
\theta^\top g^\top \Pi^\top \Pi h \theta_*\Big]\\
& \geq \|\theta-\theta_*\|^3 \cdot \E_{g,h}[(v^\top g^\top \Pi^\top \Pi h v)^2]
-C\|\theta\|^2,
\end{align*}
where the second line is obtained by writing each $\theta$ as
$(\theta-\theta_*)+\theta_*$ and absorbing all but the term with
cubic dependence on $(\theta-\theta_*)$ into the $C\|\theta\|^2$ remainder.
As noted in Step 2 above, $\Pi g v$ is not identically 0 over $g \in \G$,
for any unit vector $v$. Then $(v^\top g^\top \Pi^\top \Pi h v)^2$ is the
squared-inner product between two i.i.d.\ non-zero vectors, and hence is
strictly positive with positive probability. So
$\E_{g,h}[(v^\top g^\top \Pi^\top \Pi h v)^2]>0$. Then by compactness of the
unit sphere, $\E_{g,h}[(v^\top g^\top \Pi^\top \Pi h v)^2]>c>0$ for every
unit vector $v$ and some constant $c=c(\Pi,\G)>0$. So for $\|\theta\|>M$ and
large enough $M=M(\Pi,G,\theta_*)$, this shows
\begin{equation}\label{eq:sigma4grad1}
v^\top \nabla \ts_2(\theta) \geq c'\|\theta\|^3.
\end{equation}
Next, consider $v^\top \nabla q_2(\theta)$. Since $q_2 \in \tcR_{\leq 1}^\G$
which is generated by $\tT_1(\theta)=\E_g[\Pi g \theta]$,
$q_2(\theta)$ is a quartic polynomial of the entries of
$\E_g[\Pi g \theta]$, whose
specific form depends only on $\Pi,\G,\theta_*$. Then, applying
the chain rule to differentiate $q_2(\theta)$, we have
\[\|\nabla q_2(\theta)\| \leq C(\|\E_g[\Pi g\theta]\|^3+1).\]
Now applying
$\theta=(\theta-\theta_*)+\theta_*=\|\theta-\theta_*\|v+\theta_*$
and $\|\E_g[\Pi g v]\|<c_0$, we have
\[\|\E_g[\Pi g\theta]\| \leq c_0\|\theta-\theta_*\|+\|\Pi\|\|\theta_*\|
\leq C'(c_0\|\theta\|+1).\]
Then for a $(\theta_*,\G,\Pi)$-dependent constant $C>0$ independent of $c_0$,
we get
\begin{equation}\label{eq:sigma4grad2}
\big|v^\top \nabla q_2(\theta)\big| \leq \|\nabla q_2(\theta)\|
\leq C(c_0^3\|\theta\|^3+1).
\end{equation}
Similarly, consider
$v^\top \nabla[\langle \tT_2(\theta),P_2(\theta) \rangle]$. 
Each entry of $P_2(\theta)$ is a quadratic polynomial of
$\E_g[\Pi g \theta]$. Noting that $\tT_2(\theta)$ is also quadratic in $\theta$,
by a similar argument as above,
\begin{align}
\big|v^\top \nabla[\langle \tT_2(\theta),P_2(\theta) \rangle]\big|
&\leq C\Big((\|\E_g[\Pi g \theta]\|+1) \cdot \|\theta\|^2
+(\|\E_g[\Pi g \theta]\|+1)^2 \cdot \|\theta\|\Big)\nonumber\\
&\leq C'\big(c_0\|\theta\|^3+\|\theta\|^2\big).\label{eq:sigma4grad3}
\end{align}
Finally, we may apply the condition $\|\theta\| \leq C_0\sigma^{2/3}$ to bound
\begin{equation}\label{eq:sigma4grad4}
\|\theta\|^5/\sigma^2 \leq \|\theta\|^3 \cdot C_0^2\sigma^{-2/3}.
\end{equation}
Applying (\ref{eq:sigma4grad1}), (\ref{eq:sigma4grad2}), (\ref{eq:sigma4grad3}),
and (\ref{eq:sigma4grad4}) to (\ref{eq:sigma4grad}),
for sufficiently small $c_0=c_0(\Pi,G,\theta_*) \in (0,1)$, sufficiently
large $M=M(\Pi,G,\theta_*)>0$, any $\theta$ satisfying $C_0\sigma^{2/3} \geq
\|\theta\|>M$, and all sufficiently large $\sigma>\sigma_0$,
we obtain that $\sigma^4\|\nabla R(\theta)\|$ is
lower-bounded by a $(\Pi,\G,\theta_*)$-dependent constant,
so $\|\nabla R(\theta)\| \geq c\sigma^{-4}$ as desired.
\end{proof}

To complete the proofs of Theorems \ref{thm:benignlandscape} and
\ref{thm:landscape}, it remains to analyze the optimization landscape
of $R(\theta)$ over the ball $\{\theta:\|\theta\| \leq M\}$. Our arguments are similar to those of \cite[Sections 4.3 and
4.5]{fan2020likelihood}: Fixing any $\ttheta$ in this ball, we may apply
Lemma \ref{lemma:transcendencebasis} to reparametrize $\theta$ by $\G$-invariant
polynomials in a sufficiently small
neighborhood $U_{\ttheta}$ of $\ttheta$, and then deduce statements about
the landscape of $R(\theta)$ within $U_{\ttheta}$ by sequentially analyzing
the landscapes of the terms $s_1,s_2,s_3,\ldots$
in this new system of coordinates. The conclusions about the landscape of
$R(\theta)$ over the full ball $\{\theta:\|\theta\| \leq M\}$ then follow from
patching together these analyses for local neighborhoods $U_{\ttheta}$ that
form a finite cover of this ball. Importantly, this argument requires both the radius
$M$ and the local neighborhoods $U_{\ttheta}$ to be independent of $\sigma$.

To describe the full landscape of $R(\theta)$ in this ball
$\{\theta:\|\theta\| \leq M\}$,
we must consider non-generic points $\theta$ (even if $\theta_*$ is generic).
Thus it may be necessary to use
non-generic points $\ttheta$ in constructing this finite cover, where the
system of coordinates given by Lemma \ref{lemma:transcendencebasis} locally
around $\ttheta$ may not contain a complete transcendence basis for $\cR^\G$.
Instead, we separate the cases for $\ttheta$ by the largest index
$k$ for which $\ttheta \in \cV_{k-1}(\theta_*)$ but $\ttheta \notin
\cV_k(\theta_*)$, and we provide a separate analysis for each $k$,
using the given condition that $\der M_k(\theta)$ has constant rank on $\cV_k(\theta_*)$.
Our argument in the final case $\ttheta \in \cV_K(\theta_*)$ extends the
analyses of \cite{fan2020likelihood} to handle orbits of positive dimension
arising in the setting of a continuous group,
and we then explain how these arguments may be adapted to use the expansion
of (\ref{eq:seriesexpansionprojected}) in models with projection.

\begin{proof}[Proof of Theorem \ref{thm:benignlandscape}]
Consider the unprojected setting of part (a). Fixing a generic point $\theta_*
\in \R^d$, Lemma \ref{lemma:transcendencebasis} shows that for
any $k=1,\ldots,K$, the rank of $\der M_k(\theta_*)$ is $d_1+\ldots+d_k$.
Then by the assumption that $\der M_k(\theta)$ has constant rank over
$\cV_k(\theta_*)$, since $\theta_* \in \cV_k(\theta_*)$,
this constant rank must be $d_1+\ldots+d_k$.

We now consider two cases for a (possibly non-generic) point $\ttheta \in \R^d$:
\begin{itemize}
\item[\bf Case 1:] $\ttheta \in \cV_{k-1}(\theta_*) \subseteq \ldots \subseteq
\cV_0(\theta_*)=\R^d$, but $\ttheta \notin \cV_k(\theta_*)$, for some
$k \in \{1,\ldots,K\}$. The argument in this case is the same as that of
\cite[Theorem 4.27]{fan2020likelihood}, and we reproduce it here for the
reader's convenience. By the constant rank assumption,
$\der M_{k-1}(\ttheta)$ has rank $d_1+\ldots+d_{k-1}$. Let
$\varphi=(\varphi^1,\ldots,\varphi^{k-1},\bar\varphi)$ be the map of Lemma
\ref{lemma:transcendencebasis}, with inverse $\theta(\varphi)$ in
a neighborhood $U_{\ttheta}$ of $\ttheta$. (If $k=1$, we take
$\varphi=\bar\varphi:\R^d \to \R^d$ to be an arbitrary invertible map, say the
identity map.) We write $f(\varphi)$ as shorthand for $f(\theta(\varphi))$.

In the parametrization by $\varphi$, each entry of $T_1,\ldots,T_{k-1}$
depends only on the coordinates
$\varphi^1,\ldots,\varphi^{k-1}$, by Lemma \ref{lemma:transcendencebasis}(c).
Thus
\[\cV_{k-1}(\theta_*) \cap U_{\ttheta}=\Big\{\theta \in
U_{\ttheta}:\varphi^1(\theta)=\varphi^1(\theta_*),\ldots,
\varphi^{k-1}(\theta)=\varphi^{k-1}(\theta_*)\Big\},\]
and the remaining coordinates $\bar\varphi$ form a local chart for the
manifold $\cV_{k-1}(\theta_*)$ over $U_{\ttheta}$. This holds trivially
also for $k=1$.

Consider now the minimization of $s_k$ over $\cV_{k-1}(\theta_*)$.
By the form of $s_k$ in (\ref{eq:sk}), its global minimizers over
$\cV_{k-1}(\theta_*)$ are exactly the points of $\cV_k(\theta_*)$. Since
$\ttheta \notin \cV_k(\theta_*)$, and the minimization of $s_k$ over
$\cV_{k-1}(\theta_*)$ is globally benign by assumption, this implies that
\[\text{either } \qquad
\nabla_{\bar\varphi} s_k(\varphi(\ttheta)) \neq 0 \qquad \text{ or } \qquad 
\lambda_{\min}\Big(\nabla_{\bar\varphi}^2 s_k(\varphi(\ttheta))\Big)<0.\]
Applying continuity of $s_k$ and its derivatives, and reducing the size of
$U_{\ttheta}$ as necessary, we may then ensure
\[\text{either } \quad
\|\nabla_{\bar\varphi} s_k(\varphi)\|>c \quad \text{ or } \quad 
\lambda_{\min}\Big(\nabla_{\bar\varphi}^2 s_k(\varphi)\Big)<-c \quad \text{ for
all } \varphi \in \varphi(U_{\ttheta}).\]
Here, the size of the neighborhood $U_{\ttheta}$ and the constant $c>0$ are
independent of $\sigma$, as $s_k$ does not depend on $\sigma$.
Now applying the expansion (\ref{eq:seriesexpansionunprojected}) to the order
$k$ and differentiating term-by-term in
$\varphi=(\varphi^1,\ldots,\varphi^{k-1},\bar\varphi)$,
observe that $\{s_j:j \leq k-1\}$ and $\{q_j:j \leq k\}$ depend only on
$(\varphi^1,\ldots,\varphi^{k-1})$ and not on $\bar\varphi$.
Then for some constant $\sigma_0=\sigma_0(\ttheta)>0$,
all $\sigma>\sigma_0$, and all
$\varphi \in \varphi(U_{\ttheta})$,
\begin{equation}\label{eq:typeItmp}
\text{either } \quad
\|\nabla_\varphi R(\varphi)\|>(c/2)\sigma^{-2k} \quad \text{ or } \quad 
\lambda_{\min}\Big(\nabla_\varphi^2 R(\varphi)\Big)<-(c/2)\sigma^{-2k}.
\end{equation}
Finally, changing variables back to $\theta$ by the chain rule, this implies
\begin{equation}\label{eq:typeIpoint}
\text{either } \quad
\nabla_\theta R(\theta) \neq 0 \quad \text{ or } \quad 
\lambda_{\min}\Big(\nabla_\theta^2 R(\theta)\Big)<0
\qquad \text{ for all } \theta \in U_{\ttheta}.
\end{equation}

\item[\bf Case 2:]  $\ttheta \in \cV_K(\theta_*)$. By the constant rank assumption, $\der
M_K(\ttheta)=d_1+\ldots+d_K$. Let
$\varphi=(\varphi^1,\ldots,\varphi^K,\bar\varphi)$ be the map of
Lemma \ref{lemma:transcendencebasis} in a neighborhood $U_{\ttheta}$ of
$\ttheta$, with inverse $\theta(\varphi)$. We again write as shorthand
$f(\varphi)=f(\theta(\varphi))$.

Let us write $\nabla_\varphi^2 R(\varphi)$ in
the $(K+1) \times (K+1)$ block decomposition corresponding to
$(\varphi^1,\ldots,\varphi^K,\bar\varphi)$. Applying the expansion
(\ref{eq:seriesexpansionunprojected}) now to order $K$, each
$s_k(\varphi)$ depends only on $\varphi^1,\ldots,\varphi^k$, each $q_k(\varphi)$
depends only on $\varphi^1,\ldots,\varphi^{k-1}$, and $q(\varphi)$ and
$R(\varphi)$ depend only on $\varphi^1,\ldots,\varphi^K$.
Furthermore, Lemma \ref{lemma:skfullrank} shows
$\nabla_{\varphi^k}^2 s_k(\varphi(\ttheta)) \succ 0$ strictly for each 
$k=1,\ldots,K$, so $\nabla_{\varphi^k}^2 s_k(\varphi) \succ c\Id$ for all
$\varphi \in \varphi(U_{\ttheta})$ by continuity, for a sufficiently small
neighborhood $U_{\ttheta}$ and constant $c>0$. Then differentiating
(\ref{eq:seriesexpansionunprojected}) term-by-term,
$\nabla_\varphi^2 R(\varphi)$ is zero outside the upper-left $K \times K$
blocks, and these $K \times K$ blocks have a graded block structure
in the sense of Definition \ref{def:gradedblock} with $d'=d-d_0$,
for any $\sigma>\sigma_0=\sigma_0(\ttheta)$ and all
points $\varphi \in \varphi(U_{\ttheta})$.
Then, applying Lemma \ref{lemma:gradedblock}(b),
the upper-left $K \times K$ blocks of $\nabla_\varphi^2 R(\varphi)$ form
a strictly positive-definite matrix. Recalling that $R(\varphi)$
depends only on $\varphi^1,\ldots,\varphi^K$ and not on $\bar\varphi$, 
let us write
$\bar{R}(\varphi^1,\ldots,\varphi^K)=R(\varphi)$, and also reduce to a smaller
neighborhood $U_{\ttheta}$ such that $\varphi(U_{\ttheta})$ has a product form
$V \times W$, where $V \subset \R^{d_1+\ldots+d_K}$ and $W \subset \R^{d_0}$.
Then this shows that
\begin{equation}\label{eq:barRconvex}
\bar{R}(\varphi^1,\ldots,\varphi^K) \text{ is strictly convex on } V.
\end{equation}

Now applying the assumption that $\cV_K(\theta_*)=\orbit_{\theta_*}$, we have
that $\ttheta \in \orbit_{\theta_*}$ is a global minimizer of $R(\theta)$. Thus
$(\varphi^1(\ttheta),\ldots,\varphi^K(\ttheta))$ is the global minimizer and
unique critical point of $\bar{R}$ on $V$. Changing coordinates back to
$\theta$ by the chain rule, the
critical points of $R(\theta)$ on $U_{\ttheta}$ are then given exactly by
\[\Big\{\theta \in U_{\ttheta}:(\varphi^1(\theta),\ldots,\varphi^K(\theta))
=(\varphi^1(\ttheta),\ldots,\varphi^K(\ttheta))\Big\}.\]
Applying (\ref{eq:barphiorbit}), this shows
\begin{equation}\label{eq:typeIIpoint}
\Big\{\theta \in U_{\ttheta}:\nabla R(\theta)=0\Big\}
=U_{\ttheta} \cap \orbit_{\ttheta}=U_{\ttheta} \cap \orbit_{\theta_*}
\end{equation}
\end{itemize}

Finally, we combine these two cases using a compactness argument: By Lemma
\ref{lemma:localizecrit}, there are no critical points of $R(\theta)$ outside
a sufficiently large
ball $\overline{B_M}=\{\theta:\|\theta\| \leq M\}$. For each point $\ttheta
\in \overline{B_M}$, construct the neighborhood $U_{\ttheta}$ as above, and take
a finite set $S$ of such points $\ttheta$ for which $\bigcup_{\ttheta \in S}
U_{\ttheta}$ covers $\overline{B_M}$. Set $\sigma_0=\max_{\ttheta \in S}
\sigma_0(\ttheta)$, where $\sigma_0(\ttheta)$ is as defined in the two cases
above. Then for any $\sigma>\sigma_0$, the conditions (\ref{eq:typeIpoint}) and
(\ref{eq:typeIIpoint}) combine to show that any critical point of $R(\theta)$
inside $\overline{B_M}$ either belongs to the locus $\orbit_{\theta_*}$ of
global minimizers, or has $\lambda_{\min}(\nabla^2 R(\theta))<0$. Thus the
minimization of $R(\theta)$ is globally benign, concluding the proof of part
(a).

The proof in the projected setting of part (b) is similar, with the following
modifications: For the first case where $\ttheta \in \tcV_{k-1}(\theta_*)
\subseteq \ldots \subseteq \tcV_0(\theta_*)=\R^d$ but $\ttheta \notin 
\tcV_k(\theta_*)$, Lemma \ref{lemma:transcendencebasis} still yields a local
parametrization
$\varphi=(\varphi^1,\ldots,\varphi^{k-1},\bar\varphi)$ where $\bar\varphi$ forms
a local chart for $\tV_{k-1}(\theta_*)$. Differentiating
(\ref{eq:seriesexpansionprojected}) applied to the order $k$
term-by-term in
$\varphi=(\varphi^1,\ldots,\varphi^{k-1},\bar\varphi)$, the gradient and Hessian
of $R(\varphi)$ in $\bar\varphi$ have an additional
contribution from $\sigma^{-2k}\langle \tT_k(\varphi),P_k(\varphi)\rangle$. Since $P_k$
depends only on $\varphi^1,\ldots,\varphi^{k-1}$ and not on $\bar\varphi$, the
gradient and Hessian in $\bar\varphi$ are obtained by differentiating only
$\tT_k$. Then both $\nabla_{\bar\varphi}[\langle
\tT_k(\varphi),P_k(\varphi)\rangle]|_{\varphi=\varphi(\ttheta)}=0$ and
$\nabla_{\bar\varphi}^2[\langle
\tT_k(\varphi),P_k(\varphi)\rangle]|_{\varphi=\varphi(\ttheta)}=0$, because
$P_k(\ttheta)=P_k(\theta_*)=0$ for any $\ttheta \in \cV_{k-1}(\theta_*)$.
Then for a sufficiently small neighborhood $U_{\ttheta}$, 
we still obtain (\ref{eq:typeItmp}) for all $\varphi \in \varphi(U_{\ttheta})$,
and hence (\ref{eq:typeIpoint}) still holds.

For the second case where $\ttheta \in \tcV_{\tK}(\theta_*)$, similarly when
computing the Hessian of (\ref{eq:seriesexpansionprojected}) applied to the
order $\tK$ term-by-term in
$\varphi=(\varphi^1,\ldots,\varphi^{\tK},\bar\varphi)$, we have
additional contributions from the terms $\sigma^{-2k}\langle
\tT_k(\varphi),P_k(\varphi) \rangle$. In each $(k,k)$ block, we again have
$\nabla_{\varphi^k}^2 [\langle \tT_k(\varphi),P_k(\varphi)
\rangle]|_{\varphi=\varphi(\ttheta)}=0$, because only $\tT_k$ depends on
$\varphi^k$ whereas $P_k(\ttheta)=P_k(\theta_*)=0$.
Then $\nabla_\varphi^2 R(\varphi)$ still has a graded block structure in
its upper-left $\tK \times \tK$ blocks, implying (\ref{eq:barRconvex}) for
$\varphi(U_{\ttheta})=V \times W$ and a sufficiently small neighborhood
$U_{\ttheta}$ of $\ttheta$. Now by the assumption given in part (b) of the theorem, we have
$\Pi(\orbit_{\ttheta}) \equiv \Pi(\orbit_{\theta_*})$ for $\ttheta \in
\tcV_{\tK}(\theta_*)$. Then $\ttheta$ is a global minimizer of $R(\theta)$.
The convexity of (\ref{eq:barRconvex}) then implies, by the same argument as in
the unprojected setting, that
$\{\theta \in U_{\ttheta}:\nabla R(\theta)=0\}=U_{\ttheta} \cap
 \orbit_{\ttheta}$.
Over the given domain $\{\theta:\|\theta\|<B(\|\theta_*\|+\sigma)\}$,
Lemma \ref{lemma:localizecrit} ensures that there are no critical points of
$R(\theta)$ outside the smaller ball $\{\theta:\|\theta\| \leq M\}$, which is
independent of $\sigma$. We conclude the proof by
applying the same compactness argument over $\{\theta:\|\theta\| \leq M\}$ as in
the unprojected setting.
\end{proof}

Finally, we show Theorem \ref{thm:landscape} on a correspondence
between local minimizers of $R(\theta)$ and $s_K(\theta)|_{\cV_{K-1}(\theta_*)}$.
We remark that in the preceding proof of Theorem \ref{thm:benignlandscape},
\emph{global} minimizers of $R(\theta)$ must also minimize each function
$s_k(\theta)$ over $\cV_{k-1}(\theta_*)$, and in particular, they are also
exactly the minimizers of $s_K(\theta)|_{\cV_{K-1}(\theta_*)}$.
Such a statement is not true for local minimizers, and we will instead show
that local minimizers of $R(\theta)$ and $s_K(\theta)|_{\cV_{K-1}(\theta_*)}$
are close for large $\sigma$. We will use the following elementary lemma from
\cite{fan2020likelihood}, which ensures that minimizers of convex functions are
close if the functions are pointwise close to each other.

\begin{lemma}[\cite{fan2020likelihood} Lemma 2.8]
\label{lemma:convexcompare}
Let $B_\eps(\theta_0)$ be the ball of radius $\eps>0$ around $\theta_0 \in
\R^d$. Let $f_1,f_2:B_\eps(\theta_0) \to \R$ be two functions which are twice
continuously differentiable. Suppose that $\theta_0$ is a critical point of $f_1$,
and $\lambda_{\min}(\nabla^2 f_1(\theta_0)) \geq c_0$ for a constant $c_0>0$ and
all $\theta \in B_\eps(\theta_0)$. If
\[|f_1(\theta)-f_2(\theta)| \leq \delta \qquad \text{ and } 
\qquad \|\nabla^2 f_1(\theta)-\nabla^2 f_2(\theta)\| \leq \delta\]
for some $\delta<\min(c_0,c_0\eps^2/4)$ and all $\theta \in B_\eps(\theta_0)$,
then $f_2$ has a unique critical point in $B_\eps(\theta_0)$, which is a local
minimizer of $f_2$.
\end{lemma}

\begin{proof}[Proof of Theorem \ref{thm:landscape}]
Consider the unprojected setting of part (a). Fix a generic point $\theta_* \in
\R^d$. Lemma \ref{lemma:transcendencebasis} shows that $\rank \der
M_k(\theta_*)=d_1+\ldots+d_k$ for each $k=1,\ldots,K$. Then the given constant
rank assumption ensures that $\rank \der M_k(\theta)=d_1+\ldots+d_k$ for all
$\theta \in \cV_k(\theta_*)$ and $k=1,\ldots,K-1$.

Statement (a2) is established by a small extension of the argument in
Theorem \ref{thm:benignlandscape}, using the given condition that
critical points of $s_K(\theta)|_{\cV_{K-1}(\theta_*)}$
are non-degenerate up to orbit: Lemma \ref{lemma:localizecrit} ensures that
all critical points of $R(\theta)$ belong to the ball
$\overline{B_M}=\{\theta:\|\theta\| \leq M\}$. Fix any constant
$\eps>0$, and let $N_{\eps,M}$ be the points in $\overline{B_M}$ at distance
$\geq \eps$ from all critical points of $s_K(\theta)|_{\cV_{K-1}(\theta_*)}$.
We consider two cases for a point $\ttheta \in N_{\eps,M}$:
\begin{itemize}
\item[\bf Case 1:]  $\ttheta \in \cV_{k-1}(\theta_*) \subseteq \ldots \subseteq
\cV_0(\theta_*)$, but $\ttheta \notin \cV_k(\theta_*)$, for some $k \in
\{1,\ldots,K-1\}$. Then we have
\[\text{either } \quad
\nabla_\theta R(\theta) \neq 0 \quad \text{ or } \quad 
\lambda_{\min}\Big(\nabla_\theta^2 R(\theta)\Big)<0
\qquad \text{ for all } \theta \in U_{\ttheta}\]
by the same argument as leading to (\ref{eq:typeIpoint})
in Theorem \ref{thm:benignlandscape}.
\item[\bf Case 2:] $\ttheta \in \cV_{K-1}(\theta_*)$. Then Lemma
\ref{lemma:transcendencebasis} provides a local reparametrization
$\varphi=(\varphi^1,\ldots,\varphi^{K-1},\bar\varphi)$ on a neighborhood
$U_{\ttheta}$ of $\ttheta$, where $\varphi^j:\R^d \to \R^{d_j}$ and $T_j$
depends only on $(\varphi^1,\ldots,\varphi^j)$ for each $j=1,\ldots,K-1$. Then
$\bar\varphi$ forms a local chart for $\cV_{K-1}(\theta_*)$ at $\ttheta$.
Let $\tvarphi=\varphi(\ttheta)$.

If $\nabla_{\bar\varphi} s_K(\tvarphi)=0$, then $\ttheta$ is a critical
point of $s_K|_{\cV_{K-1}(\theta_*)}$, which by the given assumption must be
non-degenerate up to orbit. Hence $\orbit_{\ttheta}$ is locally a manifold of
dimension $d_0$ at $\ttheta$, so we may choose the
parametrization $\bar{\varphi}$ above to have a decomposition
$\bar{\varphi}=(\varphi^K,\varphi^0)$, where $\varphi^0$ has $d_0$ coordinates
forming a local chart for $\orbit_{\ttheta}$, and $\varphi^K$ has $d_K$
remaining coordinates. Since $s_K$ is constant over $\orbit_{\ttheta}$, we must 
have $\nabla_{\varphi^0} s_K=0$, so $s_K$ depends only on $\varphi^K$ and not on
$\varphi^0$ in this chart $\bar\varphi=(\varphi^K,\varphi^0)$ for
$\cV_{K-1}(\theta_*)$. Then non-degeneracy of $\ttheta$ up to orbit further
implies that $\nabla_{\varphi^K}^2 s_K(\tvarphi)$ is a
$d_K \times d_K$ matrix of full rank $d_K$.
If this were positive definite, then $\ttheta$ would be a local minimizer
of $s_K$ on $\cV_{K-1}(\theta_*)$, but we have assumed $\ttheta \in N_{\eps,M}$
which does not include such local minimizers. Therefore
$\nabla_{\varphi^K}^2 s_K(\tvarphi)$ must have a negative
eigenvalue. This shows that
\[\text{either } \quad \nabla_{\bar\varphi} s_K(\tvarphi) \neq 0
\quad \text{ or } \quad \lambda_{\min}\Big(\nabla_{\bar\varphi}^2
s_K(\tvarphi)\Big)<0.\]
Thus, differentiating (\ref{eq:seriesexpansionunprojected}) applied to the order
$K-1$ term-by-term in
$\varphi=(\varphi^1,\ldots,\varphi^{K-1},\bar\varphi)$, also in this case
\[\text{either } \quad
\nabla_\theta R(\theta) \neq 0 \quad \text{ or } \quad 
\lambda_{\min}\Big(\nabla_\theta^2 R(\theta)\Big)<0
\qquad \text{ for all } \theta \in U_{\ttheta}.\]
\end{itemize}

Combining these two cases,
taking a finite cover of $N_{\eps,M}$ by such neighborhoods $U_{\ttheta}$,
this shows that for any $\eps>0$ and all $\sigma>\sigma_0(\eps)$, all local
minimizers of $R(\theta)$ must be $\eps$-close to some local minimizer of
$s_K(\theta)$ on $\cV_{K-1}(\theta_*)$. Then there exists a slowly
decreasing sequence $\eps(\sigma) \to 0$ as $\sigma \to \infty$, for which
each local minimizer of $R(\theta)$ is $\eps(\sigma)$-close to a local
minimizer of $s_K(\theta)$. This establishes (a2). The proof
of (b2) in the projected setting is the same, where the gradients and
Hessians in $\bar{\varphi}$ of the additional terms
$\sigma^{-2k}\langle \widetilde{T}_k(\varphi),P_k(\varphi) \rangle$
from (\ref{eq:seriesexpansionprojected}) are handled in the same way as 
in the proof of Theorem \ref{thm:benignlandscape}.

We now show the converse direction (a1). Let $\theta_+ \in
\cV_{K-1}(\theta_*)$ be a local
minimizer of $s_K|_{\cV_{K-1}(\theta_*)}$ that is non-degenerate
up to orbit. By Lemma \ref{lemma:transcendencebasis} and the same
argument as above, there is a local reparametrization
$\varphi=(\varphi^1,\ldots,\varphi^{K-1},\varphi^K,\varphi^0)$ on a neighborhood
$U_{\theta_+}$ of $\theta_+$ such that $T_k(\varphi)$ depends only on $\varphi^1,\ldots,\varphi^k$ for
each $k=1,\ldots,K-1$, and $s_K(\varphi)$ and $R(\varphi)$ depend only
on $\varphi^1,\ldots,\varphi^K$. Let $\varphi_+=\varphi(\theta_+)$.
For $k=1,\ldots,K-1$, some constant $c>0$, and all $\varphi
\in \varphi(U_{\theta_+})$,
\begin{equation}\label{eq:sksecondorder}
\lambda_{\min}\Big(\nabla_{\varphi^k}^2 s_k(\varphi)\Big)>c
\end{equation}
by Lemma \ref{lemma:skfullrank} applied with $\ttheta=\theta_+$ and by
continuity of this Hessian. This holds also for $k=K$, by the non-degeneracy of
$\theta_+$ up to orbit. Then, writing the Hessian $\nabla_\varphi^2
R(\varphi)$ in the $(K+1) \times (K+1)$ block structure corresponding to
$(\varphi^1,\ldots,\varphi^K,\varphi^0)$, we obtain as in Theorem
\ref{thm:benignlandscape} that the upper-left $K \times K$ blocks have a graded
block structure, in a sufficiently small neighborhood $U_{\theta_+}$ where
$\varphi(U_{\theta_+})=V \times W$ has a product form. So, defining
$\bar{R}(\varphi^1,\ldots,\varphi^K)=R(\varphi)$, $\bar{R}$ is strictly convex
over $V$.

However, in contrast to Theorem \ref{thm:benignlandscape}, $\theta_+$ is not
necessarily a global (or local) minimizer of $R(\theta)$, so the existence of a
local minimizer of $\bar{R}(\varphi^1,\ldots,\varphi^K)$
in $V$ is less immediate. By further reducing $U_{\theta_+}$, we may
assume $V$ takes a product form $V=V_1 \times \ldots \times V_K$ where $V_k$
corresponds to the coordinates of $\varphi^k$. Let $\bar{V},\bar{V}_k$ be the
closures of $V,V_k$, which are compact.
Let $\hat\varphi=(\hat\varphi^1,\ldots,\hat\varphi^K)$ be
a point which minimizes $\bar{R}$ over $\bar{V}$. We aim to show that
$\hat\varphi$ in fact belongs to the interior of $\bar{V}$, and hence is a critical
point and local minimizer of $\bar{R}$ in $V$.
To show this, we will inductively show that each subvector $\hat\varphi^k$
belongs to the interior of $\bar{V}_k$, by using Lemma
\ref{lemma:convexcompare} to argue that it is close to $\varphi_+^k$
where $(\varphi_+^1,\ldots,\varphi_+^k)$ minimizes $s_k(\varphi)$.
The argument is similar to that of \cite[Lemma 4.15]{fan2020likelihood} (fixing
a minor error therein), and we reproduce this argument here.

Let $C,C',c>0$ denote $(\theta_*,\G)$-dependent
constants changing from instance to
instance. Let us write $s_k(\varphi^1,\ldots,\varphi^k)=s_k(\varphi)$,
as this does not depend on the remaining coordinates of $\varphi$.
Because $\theta_+ \in \cV_1(\theta_*)$, $\theta_+$ is a global minimizer of
$s_1$ over $\bar{V}_1$. Then, applying (\ref{eq:sksecondorder}) for $k=1$, we
get for some constant $c>0$ and sufficiently small neighborhood $V_1$ that
\begin{align*}
s_1(\hat\varphi^1)-s_1(\varphi_+^1)\geq c\|\hat\varphi^1-\varphi_+^1\|^2.
\end{align*}
Applying the series expansion (\ref{eq:seriesexpansionunprojected})
to order $K=1$, and
noting that $q_1 \in \cR_{\leq 0}^\G$ must be a constant, this implies that
\begin{equation}\label{eq:R1comparison}
\bar{R}(\hat\varphi^1,\ldots,\hat\varphi^K)-\bar{R}(\varphi_+^1,\ldots,\varphi_+^K)\geq
c\sigma^{-2}\|\hat\varphi^1-\varphi_+^1\|^2-C\sigma^{-4},
\end{equation}
for all $\sigma>\sigma_0$ and large enough $\sigma_0$. Since $\hat\varphi$
minimizes $\bar{R}(\varphi)$ over $\bar{V}$, the left side is non-positive,
so we obtain
$c\sigma^{-2}\|\hat\varphi^1-\varphi_+^1\|^2-C\sigma^{-4}\leq 0$. This
shows
\begin{equation}\label{eq:varphi1close}
\|\hat\varphi^1-\varphi_+^1\|\leq C'\sigma^{-1}<\sigma^{-\eta_1}
\end{equation}
for large enough $\sigma$ and, say, $\eta_1=0.1$.

Suppose inductively that we have shown
\begin{equation}\label{eq:varphiclose}
\|\hat\varphi^1-\varphi_+^1\|,\ldots,\|\hat\varphi^{k-1}-\varphi_+^{k-1}\|
<\sigma^{-\eta_{k-1}}
\end{equation}
for some constant $\eta_{k-1}>0$. Consider the functions
$h_+(\varphi^k)=s_k(\varphi_+^1,\ldots,\varphi_+^{k-1},\varphi^k)$
and $\hat{h}(\varphi^k)=s_k(\hat\varphi^1,\ldots,\hat\varphi^{k-1},\varphi^k)$
on $\bar{V}_k$. Since $s_k$ is a polynomial function of its arguments, both
$s_k$ and its Hessian are 
Lipschitz over the bounded domain $\bar{V}_1 \times \ldots \times \bar{V}_k$.
Then applying (\ref{eq:varphiclose}), for a constant $C>0$
depending on $\bar{V}_1 \times \ldots \times \bar{V}_k$ and $s_k$ but not on
$\sigma$, we have
\begin{equation}\label{eq:comparison}
\sup_{\varphi^k \in \bar{V}_k} |h_+(\varphi^k)-
\hat{h}(\varphi^k)|<C\sigma^{-\eta_{k-1}},
\qquad \sup_{\varphi^k \in \bar{V}_k} \|\nabla^2 h_+(\varphi^k)-
\nabla^2 \hat{h}(\varphi^k)\|<C\sigma^{-\eta_{k-1}}.
\end{equation}
Both $h_+$ and $\hat{h}$ are strongly convex over $\bar{V}_k$, with Hessians
lower bounded as (\ref{eq:sksecondorder}). For $k \leq K-1$, since $\theta_+
\in \cV_k(\theta_*)$, we know that $\varphi_+^k$ minimizes $h_+$ on $\bar{V}_k$.
For $k=K$, since $\theta_+$ is a local minimizer of
$s_K|_{\cV_{K-1}(\theta_*)}$ by assumption, we know also that $\varphi_+^K$
minimizes $h_+$ on a sufficiently small neighborhood $\bar{V}_K$.
Then (\ref{eq:comparison}) and Lemma \ref{lemma:convexcompare}
applied with $\eps=C'\sqrt{\sigma^{-\eta_{k-1}}}$ for a large enough constant
$C'>0$ guarantee that the global minimizer
$\bar{\varphi}^k$ of $\hat{h}$ over $\bar{V}_k$ satisfies
\begin{equation}\label{eq:bartildephi2close}
\|\bar{\varphi}^k-\varphi_+^k\|<\eps=C'\sqrt{\sigma^{-\eta_{k-1}}}.
\end{equation}
In particular, $\bar{\varphi}^k$ must be in the interior of $\bar{V}_k$ and is a
critical point of $\hat{h}$, for sufficiently large $\sigma>\sigma_0$. So
(\ref{eq:sksecondorder}) implies for a constant $c>0$ that
\begin{align*}
s_k(\hat\varphi^1,\ldots,\hat\varphi^{k-1},\hat\varphi^k)
-s_k(\hat\varphi^1,\ldots,\hat\varphi^{k-1},\bar\varphi^k) \geq
c\|\hat\varphi^k-\bar\varphi^k\|^2.
\end{align*}
Applying the series expansion (\ref{eq:seriesexpansionunprojected}) to the
order $k$, and recalling that $s_1,\ldots,s_{k-1}$ and $q_1,\ldots,q_k$ depend only on
$\varphi^1,\ldots,\varphi^{k-1}$, we get
\begin{equation}\label{eq:R2comparison}
\bar{R}(\hat\varphi^1,\ldots,\hat\varphi^K)-\bar{R}(\hat\varphi^1,\ldots,
\hat\varphi^{k-1},\bar\varphi^k,\varphi_+^{k+1},\ldots,\varphi_+^K)
\geq c\sigma^{-2k}\|\hat\varphi^k-\bar\varphi^k\|^2-C\sigma^{-2(k+1)}.
\end{equation}
This is again non-positive because $\hat{\varphi}$ minimizes $\bar{R}$ over
$\bar{V}$. So $\|\hat\varphi^k-\bar\varphi^k\|\leq C'\sigma^{-1}$. Combining
with (\ref{eq:bartildephi2close}) and (\ref{eq:varphiclose}), we obtain
\begin{equation}\label{eq:varphiclosedone}
\|\hat\varphi^1-\varphi_+^1\|,\ldots,
\|\hat\varphi^k-\varphi_+^k\| <\sigma^{-\eta_k}
\end{equation}
for large enough $\sigma$ and, say, $\eta_k=\eta_{k-1}/3$. This
completes the induction, showing that (\ref{eq:varphiclosedone}) holds up to
$k=K$, and hence that $\hat\varphi=(\hat\varphi^1,\ldots,\hat\varphi^K)$
is in the interior of $\bar{V}$ for $\sigma>\sigma_0$ sufficiently large.
Hence $\hat\varphi$ is a
critical point and local minimizer of $\bar{R}$. Then the
point $\theta(\hat\varphi^1,\ldots,\hat\varphi^K,\varphi_+^0)$ is a local
minimizer of $R(\theta)$ on $\R^d$, which is $\eps(\sigma)$-close to $\theta_+$
for $\eps(\sigma)=C\sigma^{-\eta_K}$. This shows (a1).

The proof of (b1) in the projected setting is similar, applying
(\ref{eq:seriesexpansionprojected}) in place of
(\ref{eq:seriesexpansionunprojected}). In the first step for $k=1$ we observe,
in addition to $q_1 \in \cR_{\leq 0}^\G$ being constant, that $P_1=0$ because
it is also constant and equals 0 at $\theta_*$. Thus we obtain
(\ref{eq:R1comparison}) and (\ref{eq:varphi1close}) without modification.
In the inductive step, in place of (\ref{eq:R2comparison}), we have
\begin{align*}
&\bar{R}(\hat\varphi^1,\ldots,\hat\varphi^K)
-\bar{R}(\hat\varphi^1,\ldots,\hat\varphi^{k-1},\bar\varphi^k,\varphi_+^{k+1},\ldots,\varphi_+^K)\\
&\geq \sigma^{-2k}\Big(c\|\hat\varphi^k-\bar\varphi^k\|^2
+\Big\langle
\tT_k(\hat\varphi^1,\ldots,\hat\varphi^k)-\tT_k(\hat\varphi^1,\ldots,\hat\varphi^{k-1},\bar\varphi^k),
P_k(\hat\varphi^1,\ldots,\hat\varphi^{k-1}) \Big\rangle\Big)-C\sigma^{-2(k+1)}.
\end{align*}
We may bound
\[\|\tT_k(\hat\varphi^1,\ldots,\hat\varphi^k)-\tT_k(\hat\varphi^1,\ldots,\hat\varphi^{k-1},\bar\varphi^k)\|_{\HS}
\leq C\|\hat\varphi^k-\bar\varphi^k\|\]
and
\begin{align*}
\|P_k(\hat\varphi^1,\ldots,\hat\varphi^{k-1})\|_{\HS}
&=\|P_k(\hat\varphi^1,\ldots,\hat\varphi^{k-1})-P_k(\varphi_+^1,\ldots,\varphi_+^{k-1})\|_{\HS}\\
&\leq C\Big(\|\hat\varphi^1-\varphi_+^1\|
+\ldots+\|\hat\varphi^{k-1}-\varphi_+^{k-1}\|\Big) \leq C'\sigma^{-\eta_{k-1}},
\end{align*}
where the first equality holds because $\theta_+ \in \cV_{k-1}(\theta_*)$ so
$P_k(\varphi_+^1,\ldots,\varphi_+^{k-1})=P_k(\varphi_*^1,\ldots,\varphi_*^{k-1})=0$. This yields
\begin{align*}
0 &\geq \bar{R}(\hat\varphi^1,\ldots,\hat\varphi^K)
-\bar{R}(\hat\varphi^1,\ldots,\hat\varphi^{k-1},\bar\varphi^k,\varphi_+^{k+1},\ldots,\varphi_+^K)\\
&\geq \sigma^{-2k}\Big(c\|\hat\varphi^k-\bar\varphi^k\|^2-
C\sigma^{-\eta_{k-1}}\|\hat\varphi^k-\bar\varphi^k\|\Big)-C\sigma^{-2(k+1)}.
\end{align*}
Viewing the right side as a quadratic function in
$\|\hat\varphi^k-\bar\varphi^k\|$, this implies that
$\|\hat\varphi^k-\bar\varphi^k\|$ is at most the larger of the two roots of this
quadratic function, which still gives (\ref{eq:varphiclosedone}).
This completes the induction,
and the proof is concluded as in the unprojected setting of (a1).
\end{proof}

\section{Analysis of orthogonal Procrustes alignment}\label{appendix:procrustes}

We provide the details for Example \ref{ex:procrustes} on the Procrustes
alignment model, either with or without reflections. We consider
a generic point $\theta_* \in \R^{3 \times m}$ satisfying $\rank(\theta_*)=3$
(where this is a generic condition because $\rank(\theta_*)<3$ is equivalent to
$\det \theta_* \theta_*^\top=0$). Recall that the group is
either $\G=\O(3) \otimes \Id_m$ or $\G=\SO(3) \otimes \Id_m$, acting on $\R^{3
\times m} \cong \R^d$. We represent an element of this group as $g \otimes
\Id_m$, where $g \in \O(3)$ or $g \in \SO(3)$ is a $3 \times 3$ matrix. For the
Haar-uniform law on both $\O(3)$ and $\SO(3)$, we have the moment identities
\begin{equation}\label{eq:O3SO3moments}
\E_g[g_{ij}]=0, \qquad \E_g[g_{ij}g_{i'j'}]=\frac{1}{3} \cdot
\1\{(i,j)=(i',j')\}.
\end{equation}
The first identity and the second for $(i,j) \neq (i',j')$ follow from
the fact that $g$ is invariant in law under negation of any two rows or two
columns. The second identity for $(i,j)=(i',j')$ follows from $\Tr g^\top
g=\sum_{i,j=1}^3 g_{ij}^2=3$, and the equality in law of the entries
$g_{ij}$. The action of this element $g \otimes \Id_m$ on
$\theta \in \R^{3 \times m}$ is given by the matrix product
\[\theta \mapsto g\theta,\]
and the Euclidean inner-product may be written as
$\langle \theta,\theta' \rangle=\Tr \theta^\top \theta'=\Tr \theta'\theta^\top$.

Let us first compute $d_0=\max_\theta \dim(\orbit_\theta)$. Note that
$\dim(\orbit_\theta)=\dim(\G)-\dim(\G_\theta)$, where $\G_\theta=\{g \in \G:g
\theta=\theta\}$ is the stabilizer subgroup of $\theta$ \cite[Section
I.1.b]{audin2004torus}.
For both $\G=\O(3) \otimes \Id_m$ and $\G=\SO(3) \otimes
\Id_m$, we have $\dim(\G)=3$. For any $\theta \in \R^{3 \times m}$ having full
rank 3, there is a right inverse $\theta^\dagger \in \R^{m \times 3}$ for
which $\theta\theta^\dagger=\Id$. Thus $g\theta=\theta$ requires $g=\Id$, so
that $\dim(\G_\theta)=0$. Thus, we obtain
\[d_0=3, \qquad \trdeg \cR^\G=d-d_0=d-3.\]

We now verify the values of $d_1,d_2$ and the forms of
$T_1(\theta)$, $T_2(\theta)$, $\cV_1(\theta_*)$,
$\cV_2(\theta_*)$, $s_1(\theta)$, and $s_2(\theta)$ as stated in Example
\ref{ex:procrustes}. For $T_1(\theta)$, $\cV_1(\theta_*)$, and
$s_1(\theta)$, by the first identity of
(\ref{eq:O3SO3moments}), $T_1(\theta)=\E_g[g\theta]=0$. So by the definitions of
$s_1(\theta)$ and $\cV_1(\theta_*)$ and Lemma \ref{lem:trdeg},
\[s_1(\theta)=0, \qquad d_1=0, \qquad \cV_1(\theta_*)=\R^d.\]

For $T_2(\theta)$, $\cV_2(\theta_*)$, and $s_2(\theta)$,
let us write $[\theta]_i,[g\theta]_i \in \R^m$ for the
$i^\text{th}$ rows of $\theta,g\theta \in \R^{3 \times m}$,
and $[T_2(\theta)]_{ii'} \in \R^{m \times m}$
for the $(i,i')$ block of $T_2(\theta)$ in the $3 \times 3$ block
decomposition of $\R^{d \times d} \cong \R^{3m \times 3m}$, where $i,i' \in
\{1,2,3\}$. Then
\[[T_2(\theta)]_{ii'}=\E_g\big[[g\theta]_i \cdot [g\theta]_{i'}^\top\big].\]
When $i \neq i'$, the second identity of (\ref{eq:O3SO3moments})
yields $[T_2(\theta)]_{ii'}=0$. When $i=i'$, it yields
\[[T_2(\theta)]_{ii}=\sum_{j,j'=1}^3 \E_g\big[g_{ij}[\theta]_j \cdot
g_{ij'}[\theta]_{j'}^\top\big]
=\frac{1}{3}\big([\theta]_1[\theta]_1^\top
+[\theta]_2[\theta]_2^\top+[\theta]_3[\theta]_3^\top\big)
=\frac{1}{3}\theta^\top \theta.\]
Thus we get $T_2(\theta)=\frac{1}{3}\Id_{3 \times 3} \otimes
(\theta^\top \theta)$. Then also
\begin{align}
s_2(\theta)&=\frac{1}{4}\|T_2(\theta)-T_2(\theta_*)\|_{\HS}^2
=\frac{1}{4} \cdot \frac{1}{9} \|\Id_{3 \times 3} \otimes (\theta^\top \theta)
-\Id_{3 \times 3} \otimes (\theta_*^\top \theta_*)\|_{\HS}^2\nonumber\\
&=\frac{1}{12}\|\theta^\top \theta-\theta_*^\top \theta_*\|_{\HS}^2.
\label{eq:procrustess2}
\end{align}
This shows that $\cV_2(\theta_*)=\{\theta:s_2(\theta)=0\}
=\{\theta:\theta^\top\theta=\theta_*^\top \theta_*\}$. For such $\theta$,
the row span of $\theta$ coincides with that of $\theta_*$. Since
$\rank(\theta_*)=3$, this
implies $\theta=A\theta_*$ for some invertible matrix $A \in \R^{3 \times 3}$.
Then $\theta^\top \theta=\theta_*^\top \theta_*$ requires
$0=\theta_*^\top(A^\top A-\Id)\theta_*$, so $A^\top A-\Id=0$.
Then $A$ is orthogonal, and we obtain that
$\theta \in \{g\theta_*:g \in \O(3)\}$. This verifies
\[\cV_2(\theta_*)=\{g\theta_*:g \in \O(3)\},\]
which is exactly $\orbit_{\theta_*}$ if $\G=\O(3) \otimes \Id_m$, and
$\orbit_{\theta_*} \cup \orbit_{-\theta_*}$ if $\G=\SO(3) \otimes \Id_m$.

To compute $d_2$, we apply Lemma \ref{lem:trdeg}. Differentiating
the expression (\ref{eq:procrustess2}) twice at
$\theta=\theta_*$ by the chain rule,
\[\nabla^2 s_2(\theta_*)=\frac{1}{6}\der_\theta[\theta^\top
\theta]^\top \cdot \der_\theta[\theta^\top \theta]\Big|_{\theta=\theta_*}\]
where $\der_\theta[\theta^\top \theta]$ denotes the Jacobian of the
vectorization of $\theta^\top \theta$ as a function of $\theta$.
For generic $\theta_* \in \R^{3 \times m}$,
specializing Lemma \ref{lem:s2-rank} to follow with $l=1$ and $S=m \geq 3$,
we then get $\rank[\nabla^2 s_2(\theta_*)]=\rank(\der_\theta[\theta^\top
\theta]|_{\theta=\theta_*})=3(m-1)=d-3$. Thus, recalling $d_1=0$,
Lemma \ref{lem:trdeg} shows
\[d_2=\trdeg \cR_{\leq 2}^\G=d-3.\]
This coincides with $\trdeg \cR^\G$, so also $K=2$ is the
smallest integer satisfying Proposition \ref{prop:Kdef}.

Finally, we analyze the optimization landscape of $s_2(\theta)$ over
$\cV_1(\theta_*)=\R^d$, and show that its only critical points are strict
saddles or the global minimizers $\cV_2(\theta_*)$.
Computing the gradient of (\ref{eq:procrustess2}), we have
\begin{equation}\label{eq:procrustesgrads2}
\nabla s_2(\theta)=\frac{1}{3}\theta\Big(\theta^\top\theta-\theta_*^\top
\theta_*\Big) \in \R^{3 \times m}.
\end{equation}
The row span of $\theta \theta^\top \theta$
is the same as that of $\theta$ (regardless of the rank of $\theta$),
while the row span of $\theta \theta_*^\top \theta_*$ is
contained in the row span of $\theta_*$. Thus, at any critical point $\theta$
satisfying $\nabla s_2(\theta)=0$,
the row span of $\theta$ is contained in that of $\theta_*$, i.e.\ we have
\begin{equation}\label{eq:procrustestheta}
\theta=A\theta_*
\end{equation}
for some (possibly singular) matrix $A \in \R^{3 \times 3}$. Applying
this form to (\ref{eq:procrustesgrads2}), we see that $\nabla s_2(\theta)=0$
implies
\[0=A\theta_*\theta_*^\top A^\top A\theta_*-A\theta_*\theta_*^\top \theta_*
=A\theta_*\theta_*^\top(A^\top A-\Id)\theta_*.\]
When $\rank(\theta_*)=3$, the rows of $\theta_*$ are linearly independent, so
this implies
\begin{equation}\label{eq:procrustesstationary}
0=A\theta_*\theta_*^\top(A^\top A-\Id).
\end{equation}

We now consider two cases for a critical point $\theta$ of $s_2(\theta)$:\\

{\bf Case 1:} $\theta$ has full rank 3. Then by
(\ref{eq:procrustestheta}), $A$ must be nonsingular.
Multiplying (\ref{eq:procrustesstationary})
by $(\theta_*\theta_*^\top)^{-1}A^{-1}$, we get $0=A^\top A-\Id$,
so $A^\top A=\Id$. Thus, $A$ is an orthogonal matrix, so $\theta \in
\cV_2(\theta_*)$, and this is a global minimizer of $s_2(\theta)$.

{\bf Case 2:}
$\theta$ has some rank $k \leq 2$. Then by (\ref{eq:procrustestheta}),
also $\rank(A)=\rank(A^\top A)=k$.
Then $P=\Id-A^\top A$ has $3-k$ eigenvalues equal to 1, and in particular,
$\rank P \geq 3-k$. On the other hand, since
$\rank(A\theta_*\theta_*^\top)=k$ and $0=A\theta_*\theta_*^\top P$ by
(\ref{eq:procrustesstationary}),
the kernel of $P$ has dimension at least $k$. Then also
$\rank P \leq 3-k$, so $\rank P=3-k$ exactly. Then $P$
has $3-k$ eigenvalues 1 and $k$ eigenvalues 0, so it is an orthogonal
projection onto a subspace of dimension $3-k$ in $\R^3$.

Using this observation, we now exhibit a direction of negative curvature in
$\nabla^2 s_2(\theta)$: Let $\Delta=uv^\top \in \R^{3 \times m}$ be any
rank-one non-zero matrix where $A^\top u=0$ 
and $P\theta_*v \neq 0$. Such vectors $u$ and $v$ exist
because $\rank A<3$, $\rank P>0$, and $\theta_*$ has full rank 3. Consider
\[\theta_t=\theta+t\Delta=A\theta_*+t\Delta,\]
and the Hessian $\nabla^2 s_2(\theta)$ in the direction of $\Delta$, given by
$\partial_t^2 s_2(\theta_t)|_{t=0}$. Applying
$A^\top \Delta=A^\top uv^\top=0$ by the condition $A^\top u=0$, observe that
$\theta_t^\top \theta_t=\theta_*^\top A^\top A\theta_*+t^2\Delta^\top \Delta$,
so from (\ref{eq:procrustess2}),
\[s_2(\theta_t)=\frac{1}{12}\Big\|\theta_t^\top\theta_t-\theta_*^\top\theta_*\|_\HS^2=\frac{1}{12}\Big\|t^2\Delta^\top \Delta
-\theta_*^\top P\theta_*\Big\|_{\HS}^2.\]
This is a polynomial in $t$, whose quadratic term is
\[[t^2]s_2(\theta_t)=-\frac{1}{6}\Tr \theta_*^\top P
\theta_*\Delta^\top\Delta.\]
So
\[\partial_t^2 s_2(\theta_t)\Big|_{t=0}
=-\frac{1}{3}\Tr \theta_*^\top P\theta_*\Delta^\top\Delta
=-\frac{1}{3}\|P\theta_*\Delta^\top\|_\HS^2.\]
Finally, note that $P\theta_*\Delta^\top=(P\theta_*v)u^\top \neq 0$
because $P\theta_*v \neq 0$ and $u \neq 0$. Then this is strictly
negative, so $\lambda_{\min}(\nabla^2 s_2(\theta))<0$.\\

Combining these two cases, this verifies the claim in Example
\ref{ex:procrustes} that for generic $\theta_*$, the minimization of
$s_2(\theta)$ over $\cV_1(\theta_*)=\R^d$ is globally benign.
For the claims about the landscape of $R(\theta)$,
observe that for any $\theta=g\theta_*$ where $g \in \O(3)$, we have
$s_2(\theta)=0$, so $T_2(\theta)=T_2(\theta_*)$. Then
applying the chain rule to differentiate twice
$s_2(\theta)=\|T_2(\theta)-T_2(\theta_*)\|_{\HS}^2/4$,
and applying also $s_1(\theta)=0$, we obtain
$\rank \nabla^2 s_2(\theta)=\rank \der T_2(\theta)=\rank \der M_2(\theta)$.
On the other hand, for any such $\theta$, the preceding computation shows also
$\rank \nabla^2 s_2(\theta)=\rank \der_\theta[\theta^\top \theta]=d-3$, as this
rank is the same at $\theta=g\theta_*$ as at $\theta_*$.
Therefore $\rank \der M_2(\theta)=d-3$.
Thus $\der M_2(\theta)$ has constant rank on $\cV_2(\theta_*)$, which is a
manifold of dimension $d_0=3$, so
the minimizers $\cV_2(\theta_*)=\{g\theta_*:g \in \O(3)\}$ of $s_2(\theta_*)$
are non-degenerate up to orbit.
Then the claims about the landscape of $R(\theta)$ for large $\sigma$
follow from Theorems \ref{thm:benignlandscape} and \ref{thm:landscape}.

\section{Analysis of continuous multi-reference alignment}\label{appendix:SO2}

\subsection{Unprojected continuous MRA}\label{sec: proof SO2 unprojected}
In this section we prove Theorems \ref{thm:MRA} and \ref{thm:MRA-mom}.

\begin{proof}[Proof of Theorem \ref{thm:MRA-mom}]
Recall by Lemma \ref{lem:skform} that
\begin{equation}\label{eq: sk}
s_k(\theta) = \frac{1}{2 (k!)} \E_g \sbr{\inner{\theta}{g\cdot\theta}^k - 2\inner{\theta}{g\cdot\theta_*}^k + \inner{\theta_*}{g\cdot\theta_*}^k}.
\end{equation}
Let $\theta,\vartheta \in \R^d$. Computing the form of
$\E_g[\langle\theta,g\cdot\vartheta\rangle^k]$ will give the form of
$s_k(\theta)$.

{\bf Case $k=1$:} By (\ref{eq:MRAG}), $\E_g[g] = \diag(1, 0, \ldots, 0)$. It follows from \eqref{eq: sk} that
  \[s_1(\theta) 
  = \frac{1}{2} \rbr{(\theta^{(0)})^2 - 2\theta^{(0)}\theta_*^{(0)} + (\theta_*^{(0)})^2}
  = \frac{1}{2} \rbr{ \theta^{(0)} - \theta_*^{(0)}}^2.\]

{\bf Case $k=2$:} Recall $u^{(0)}(\theta)=\theta^{(0)}$ and
$u^{(l)}(\theta)=\theta_1^{(l)}+\i \theta_2^{(l)}$ for $l=1,\ldots,L$.
From (\ref{eq:MRAG}), we may check that
\begin{equation}\label{eq:complexrepr}
u^{(l)}(g \cdot \theta)=e^{-\i 2\pi l\mathfrak{g}} u^{(l)}(\theta).
\end{equation}
Then, applying the identity $\Re a=(a+\bar{a})/2$,
\begin{align}\label{eq: inner prod SO2}
  \inner{\theta}{g\cdot\vartheta}
=\sum_{l=0}^L \Re \Big[u^{(l)}(\theta) \cdot \overline{u^{(l)}(g \cdot
\vartheta)}\Big]
=\frac{1}{2} \sum_{l=0}^L \sbr{u^{(l)}(\theta) \overline{u^{(l)}(\vartheta)} e^{\i 2\pi l\frakg} + \overline{u^{(l)}(\theta)} u^{(l)}(\vartheta) e^{-\i 2\pi  l \frakg}}.
\end{align}
Taking the expected square on both sides gives
\begin{align*}
  \E_g[\inner{\theta}{g\cdot\vartheta}^2] =
  & \frac{1}{4} \E_\frakg \sbr{\sum_{l_1,l_2=0}^L u^{(l_1)}(\theta) u^{(l_2)}(\theta) \overline{u^{(l_1)}(\vartheta) u^{(l_2)}(\vartheta)} e^{\i2\pi (l_1+l_2)\frakg}}\notag\\
  &+ \frac{1}{4} \E_\frakg \sbr{\sum_{l_1,l_2=0}^L \overline{u^{(l_1)}(\theta) u^{(l_2)}(\theta)} u^{(l_1)}(\vartheta) u^{(l_2)}(\vartheta) e^{-\i 2\pi (l_1+l_2)\frakg}}\notag\\
  & + \frac{1}{4} \E_\frakg \sbr{\sum_{l_1,l_2=0}^L u^{(l_1)}(\theta) \overline{u^{(l_2)}(\theta) u^{(l_1)}(\vartheta)} u^{(l_2)}(\vartheta) e^{\i2\pi (l_1-l_2)\frakg}}\notag\\
  & + \frac{1}{4} \E_\frakg \sbr{\sum_{l_1,l_2=0}^L \overline{u^{(l_1)}(\theta)} u^{(l_2)}(\theta) u^{(l_1)}(\vartheta) \overline{u^{(l_2)}(\vartheta)} e^{-\i2\pi (l_1-l_2)\frakg}},
\end{align*}  
where $\frakg$ is uniformly distributed over $[0,1)$. Applying the property
\begin{align}\label{eq:gexp}
\E_\frakg[e^{\i 2\pi l\frakg}]=\begin{cases}
1 & ~~~{\rm for}~~~ l=0,\\
0 & ~~~{\rm for}~~~ l\not=0,
\end{cases}
\end{align}
gives
\begin{align*}
\E_g[\inner{\theta}{g\cdot\vartheta}^2] &= (u^{(0)}(\theta))^2
(u^{(0)}(\vartheta))^2+\frac{1}{4}\sum_{l=1}^L \sbr{u^{(l)}(\theta) \overline{u^{(l)}(\theta) u^{(l)}(\vartheta)} u^{(l)}(\vartheta) + \overline{u^{(l)}(\theta)} u^{(l)}(\theta) u^{(l)}(\vartheta) \overline{u^{(l)}(\vartheta)}}\\
  &= (\theta^{(0)})^2 (\vartheta^{(0)})^2+\frac{1}{2} \sum_{l=1}^L r_l(\theta)^2
r_l(\vartheta)^2.
\end{align*}
Then from \eqref{eq: sk},
  \[s_2(\theta) = \frac{1}{4} \rbr{(\theta^{(0)})^2 - (\theta_*^{(0)})^2}^2
+\frac{1}{8} \sum_{l=1}^L \Big(r_l(\theta)^2 - r_l(\theta_*)^2\Big)^2.\]

{\bf Case $k=3$:} We now take the expected cube on both sides of (\ref{eq: inner
prod SO2}). Applying (\ref{eq:gexp}),
\begin{align*}
  \E_g[\inner{\theta}{g\cdot\vartheta}^3] &=\frac{1}{4}(u^{(0)}(\theta))^3(u^{(0)}(\vartheta))^3
 + \frac{3}{8} \mathop{\sum_{l,l',l''=0}^L}_{l=l'+l''} u^{(l)}(\theta) \overline{u^{(l')}(\theta) u^{(l'')}(\theta) u^{(l)}(\vartheta)} u^{(l')}(\vartheta) u^{(l'')}(\vartheta)\\
  &\quad+ \frac{3}{8} \mathop{\sum_{l,l',l''=0}^L}_{l=l'+l''} \overline{u^{(l)}(\theta)} u^{(l')}(\theta) u^{(l'')}(\theta) u^{(l)}(\vartheta) \overline{u^{(l')}(\vartheta) u^{(l'')}(\vartheta)}.
\end{align*}
Let us write as shorthand $u^{(l)}=u^{(l)}(\theta)$,
$u_*^{(l)}=u^{(l)}(\theta_*)$, and similarly for
$r_l,\lambda_l,r_{l,l',l''},\lambda_{l,l',l''}$ and
$r_{*,l},\lambda_{*,l},r_{*,l,l',l''},\lambda_{*,l,l',l''}$.
Then from \eqref{eq: sk},
\begin{align}\label{eq: s3 SO2}
  s_3(\theta)
&= \frac{1}{12} \E_g \sbr{\inner{\theta}{g\cdot\theta}^3 -2\inner{\theta}{g\cdot\theta_*}^3 + \inner{\theta_*}{g\cdot\theta_*}^3}\notag\\
  &= \frac{1}{48}\Big((u^{(0)})^3-(u_*^{(0)})^3\Big)^2
  +\frac{1}{16} \mathop{\sum_{l,l',l''=0}^L}_{l=l'+l''} \abr{u^{(l)}u^{(l')}u^{(l'')}}^2+\abr{u_*^{(l)}u_*^{(l')}u_*^{(l'')}}^2\notag\\
  &\hspace{1in} - \frac{1}{16} \mathop{\sum_{l,l',l''=0}^L}_{l=l'+l''} u^{(l)}\overline{u^{(l')}u^{(l'')}u_*^{(l)}} u_*^{(l')} u_*^{(l'')}
+\overline{u^{(l)}} u^{(l')}u^{(l'')}u_*^{(l)} \overline{u_*^{(l')} u_*^{(l'')}}\notag\\
  &= \frac{1}{48}\Big((u^{(0)})^3-(u_*^{(0)})^3\Big)^2
  +\frac{1}{16} \mathop{\sum_{l,l',l''=0}^L}_{l=l'+l''}
\abr{u^{(l)}\overline{u^{(l')}u^{(l'')}}-u_*^{(l)}\overline{u_*^{(l')}u_*^{(l'')}}}^2.
\end{align}
This verifies the first expression for $s_3(\theta)$. For the second expression,
we split the second sum of (\ref{eq: s3 SO2}) into the cases
$l=l'=l''=0$, only $l'=0$ and $l=l'' \geq 1$, only $l''=0$ and $l=l' \geq 1$,
and all $l,l',l'' \geq 1$. The first three cases are easily rewritten in terms
of $u^{(0)},r_l,u_*^{(0)},r_{*,l}$. Each term of the
last case $l,l',l'' \geq 1$ may be written as
\[\abr{u^{(l)}\overline{u^{(l')}u^{(l'')}}-u_*^{(l)}
\overline{u_*^{(l')}u_*^{(l'')}}}^2
=r_{l,l',l''}^2+r_{*,l,l',l''}^2-2r_{l,l',l''}r_{*,l,l',l''}
\cos\big(\lambda_{*,l,l',l''}-\lambda_{l,l',l''}\big),\]
and this yields the second expression for $s_3(\theta)$.
\end{proof}

\begin{proof}[Proof of Theorem \ref{thm:MRA}]
Note that for generic $\theta \in \R^d$, for example having
$(\theta_1^{(1)},\theta_2^{(1)}) \neq (0,0)$,
its stabilizer $\G_\theta=\{g \in \G:g \cdot
\theta=\theta\}=\{\Id\}$ is trivial. Thus $\dim \orbit_\theta=\dim
\G=1$, so $\trdeg \cR^\G=d-1=2L$.

We compute $\trdeg(\cR_{\leq k}^\G)$ for $k=1,2,3$ by applying Lemma
\ref{lem:trdeg} at any generic point $\theta_* \in \R^d$ where
$r_l(\theta_*)>0$ for each $l=1,\ldots,L$.
Write as shorthand $u^{(l)}(\theta)=r_l e^{\i \lambda_l}$,
$u^{(l)}(\theta_*)=r_{*,l}e^{\i \lambda_{*,l}}$, and define
$t_l=\lambda_l-\lambda_{*,l} \in [-\pi,\pi)$ for each $l=1,\ldots,L$. Setting 
\begin{align*}
\zeta(\theta)=(\theta_0,r_1,\ldots,r_L,t_1,\ldots,t_L),
\end{align*}
this map $\zeta(\theta)$ has non-singular derivative at $\theta_*$.
Then by the inverse function theorem, the
coordinates $\zeta\in\R^{2L+1}$ provide an invertible reparametrization of
$\theta$ in a local neighborhood of $\theta_*$, with some inverse function
$\theta(\zeta)$.
Let $\zeta_*=\zeta(\theta_*)$. Note that $\theta_*$ is a global minimizer and
hence critical point of
$s_k(\theta)$ for each $k \geq 1$. Then by the chain rule,
\[\nabla_\zeta^2 s_k(\theta(\zeta))\big|_{\zeta=\zeta_*}
=\der_\zeta \theta(\zeta_*)^\top \cdot \nabla_\theta^2 s_k(\theta_*) \cdot
\der_\zeta \theta(\zeta_*).\]
Applying Lemma \ref{lem:trdeg} and the fact that
$\der_\zeta \theta(\zeta_*)$ is non-singular, this gives
\[\trdeg(\cR_{\leq k}^\G)
=\rank\Big(\nabla_\theta^2 s_1(\theta_*)+\ldots+\nabla_\theta^2 s_k(\theta_*)
\Big)=\rank\Big(\nabla_\zeta^2 s_1(\theta(\zeta))+\ldots
+\nabla_\zeta^2 s_k(\theta(\zeta))\Big|_{\zeta=\zeta_*}\Big).\]

For $k=1$ and $k=2$, by Theorem \ref{thm:MRA-mom},
\[s_1(\theta(\zeta))=\frac{1}{2}(\theta_0-\theta_{*,0})^2, \qquad
s_2(\theta(\zeta))=\frac{1}{4}(\theta_0^2-\theta_{*,0}^2)^2
+\frac{1}{8}\sum_{l=1}^L(r_l^2-r_{*,l}^2)^2.\]
Then taking the Hessians yields
\begin{align*}
\trdeg(\cR_{\leq 1}^\G) = \rank\Big(\nabla_\zeta^2 s_1(\theta(\zeta))\Big|_{\zeta=\zeta_*}\Big)=\rank\Big(\diag(1,0,\ldots,0)\Big) = 1,
\end{align*}
and 
\begin{align*}
\trdeg(\cR_{\leq 2}^\G)&=\rank\Big(\nabla_\zeta^2 s_1(\theta(\zeta))
+\nabla_\zeta^2 s_2(\theta(\zeta))\Big|_{\zeta=\zeta_*}\Big)\\
&=\rank\Big(\diag(1,0,\ldots,0)
+\diag(2\theta_{*,0}^2,r_{*,1}^2,\ldots,r_{*,L}^2)\Big)= L+1.
\end{align*}

For $k=3$, noting that $\trdeg(\cR_{\leq 3}^\G)\leq
\trdeg(\cR^\G)=2L$, it remains to show
$\trdeg(\cR_{\leq 3}^\G)\geq 2L$. Denote 
\begin{align*}
H(\zeta_*)=\nabla_\zeta^2 s_1(\theta(\zeta))
+\nabla_\zeta^2 s_2(\theta(\zeta))+\nabla_\zeta^2
s_3(\theta(\zeta))\Big|_{\zeta=\zeta_*},
\end{align*}
so that $\trdeg(\cR_{\leq 3}^\G)=\rank(H(\zeta_*))$.
Group the coordinates of $\zeta$ as $r=(\theta_0,r_1,\ldots,r_L) \in \R^{L+1}$
and $t=(t_1,\ldots,t_L) \in \R^L$, and define the corresponding block
decomposition
\begin{align*}
H(\zeta_*)=\begin{pmatrix} H_{rr}(\zeta_*) & H_{rt}(\zeta_*) \\
H_{tr}(\zeta_*) & H_{tt}(\zeta_*) \end{pmatrix}.
\end{align*}
Since $s_1(\theta(\zeta))$ and $s_2(\theta(\zeta))$ are functions only of $r$
and not of $t$, we have
\begin{align*}
H_{rr}(\zeta_*)&=\nabla_r^2 s_1(\theta(\zeta))
+\nabla_r^2 s_2(\theta(\zeta))+\nabla_r^2 s_3(\theta(\zeta))\Big|_{\zeta=\zeta_*},\\
H_{rt}(\zeta_*)&=\nabla_{rt}^2 s_3(\theta(\zeta))\Big|_{\zeta=\zeta_*},\\
H_{tr}(\zeta_*)&=\nabla_{tr}^2 s_3(\theta(\zeta))\Big|_{\zeta=\zeta_*},\\
H_{tt}(\zeta_*)&=\nabla_t^2 s_3(\theta(\zeta))\Big|_{\zeta=\zeta_*}.
\end{align*}

For the upper-left block $H_{rr}(\zeta_*)$ of size $(L+1) \times (L+1)$,
Lemma \ref{lem:trdeg} ensures that
each matrix $\nabla_\zeta^2 s_k(\theta(\zeta))$ is positive semidefinite, and
hence so is each submatrix $\nabla_r^2 s_k(\theta(\zeta))$.
Then from the analysis for $\trdeg(\cR_{\leq 2}^\G)$,
\begin{align}\label{eq:SO2H11}
\rank(H_{rr}(\zeta_*)) \geq \rank\Big(\nabla_r^2 s_1(\theta(\zeta))
+\nabla_r^2 s_2(\theta(\zeta))\Big|_{\zeta=\zeta_*}\Big)=L+1,
\end{align}
and equality must hold because $H_{rr}(\zeta_*)$ has dimension $L+1$.
For the blocks $H_{rt}(\zeta_*)$ and $H_{tr}(\zeta_*)$,
recall from the form of $s_3(\theta)$ in Theorem \ref{thm:MRA-mom} that
\[s_3(\theta(\zeta))=f(r)-\frac{1}{8} \mathop{\sum_{l,l',l''=1}^L}_{l=l'+l''}
r_{l,l',l''}r_{*,l,l',l''}\cos(t_l-t_{l'}-t_{l''}),\]
where $f(r)$ is a function depending only on $r$ and not on $t$.
Then, noting that $t_l=0$ for all $l=1,\ldots,L$ at $\zeta=\zeta_*$, we get
\begin{align*}
\nabla_{r\lambda}^2 s_3(\theta(\zeta))\Big|_{\zeta=\zeta_*}=0~~~~~~~{\rm and}~~~~~~\nabla_{\lambda r}^2 s_3(\theta(\zeta))\Big|_{\zeta=\zeta_*}=0,
\end{align*}
which further implies
\begin{align}\label{eq:SO2H12H21}
H_{rt}(\zeta_*)=0~~~~~~~{\rm and}~~~~~~H_{tr}(\zeta_*)=0.
\end{align}

Finally, the lower-right block $H_{tt}(\zeta_*)$ of size $L \times L$ is given
explicitly by
\begin{align*}
H_{tt}(\zeta_*)=-\frac{1}{8} \mathop{\sum_{l,l',l''=1}^L}_{l=l'+l''}
r_{*,l,l',l''}^2 \cdot \nabla^2_t \Big[\cos( t_l - t_{l'} -
t_{l''})\Big]\Big|_{t=0}.
\end{align*}
To show that its rank is at least $L-1$ for generic $\theta_* \in \R^d$, it
suffices to exhibit a single such point $\theta_*$. For simplicity,
we pick $\theta_*$ such that $\theta_*^{(0)}=1$ and
$r_{*,l}=1$ for all $l=1,\ldots,L$.
Let $e_l$ be the $l^\text{th}$ standard basis vector, and define the vector
\begin{align}\label{eq:wl}
w_{l,l',l''}=\nabla_t [t_{l}-t_{l'}-t_{l''}]
=e_{l}-e_{l'}-e_{l''} \in \R^L.
\end{align}
By the chain rule,
\[\nabla^2_t \Big[\cos( t_l - t_{l'} - t_{l''})\Big]\Big|_{t=0}
=-\cos(t_{l}-t_{l'}-t_{l''})\Big|_{t=0} \cdot w_{l,l',l''} w_{l,l',l''}^\top
=-w_{l,l',l''} w_{l,l',l''}^\top.\]
Define the index set
$\mathcal{L}=\{l,l',l'' \in \{1,\ldots,L\}:l=l'+l''\}$, and let
$W \in \R^{L \times |\mathcal{L}|}$ be the matrix with the vectors
$w_{l,l',l''}$ as columns. Then
\[H_{tt}(\zeta_*)=\frac{1}{8}WW^\top.\]
Note that, in particular, $W$ has a subset of $L-1$ columns
corresponding to $l'=1$, $l'' \in \{1,\ldots,L-1\}$, and $l=1+l''$. These
columns (in $\R^L$) are given by
\[\begin{pmatrix} 2 \\ -1 \\ 0 \\ 0 \\ \vdots \\ 0 \\ 0 \end{pmatrix},
\qquad \begin{pmatrix} 1 \\ 1 \\ -1 \\ 0 \\ \vdots \\ 0 \\ 0 \end{pmatrix},
\qquad \begin{pmatrix} 1 \\ 0 \\ 1 \\ -1 \\ \vdots \\ 0 \\ 0 \end{pmatrix},
\ldots,
\qquad \begin{pmatrix} 1 \\ 0 \\ 0 \\ 0 \\ \vdots \\ 1 \\ -1 \end{pmatrix}.\]
The bottom $L-1$ rows of these columns form an upper-triangular matrix
with non-zero diagonal, and hence these columns are linearly independent.
Thus $\rank(W) \geq L-1$, so also
\begin{align}\label{eq:SO2H22}
\rank(H_{tt}(\zeta_*)) \geq L-1.
\end{align}
Combining (\ref{eq:SO2H11}), (\ref{eq:SO2H12H21}), and (\ref{eq:SO2H22}), we
obtain $\trdeg \cR_{\leq 3}^\G=\rank(H(\zeta_*)) \geq 2L$. Hence
$\trdeg \cR_{\leq 3}^\G=2L$, as desired. This proves Theorem \ref{thm:MRA}.

The statement $(d_0,d_1,d_2,d_3)=(1,1,L,L-1)$ in Corollary \ref{cor:MRA-cor}
now follows from these
transcendence degrees. The smallest $K$ for which $\trdeg \cR_{\leq
K}^\G=\trdeg \cR^\G=2L$ is $K=3$, except in the case $L=1$ where it is instead
$K=2$. This proves Corollary \ref{cor:MRA-cor}, in light of
Theorem \ref{thm:FI} and \cite[Theorem 4.9]{bandeira2017estimation}
(as reviewed in Lemma \ref{lemma:genericlist}).
\end{proof}

\subsection{Spurious local minimizers for
continuous MRA}\label{appendix:MRAspurious}

In this section, we prove Theorem \ref{thm:MRAlandscape} showing that
in the continuous MRA model where $L \geq 30$, for some open subset of true
signal vectors $\theta_* \in \R^d$ and for sufficiently high noise $\sigma^2>0$,
there must exist spurious local minimizers of the negative population
log-likelihood function $R(\theta)$.

\begin{proof}[Proof of Theorem \ref{thm:MRAlandscape}]
Recall the forms of $s_1(\theta)$, $s_2(\theta)$, and $s_3(\theta)$ from Theorem
\ref{thm:MRA-mom}. By definition of $s_k(\theta)$ in (\ref{eq:sk}), each moment
variety $\cV_k(\theta_*)$ is the intersection of the global minimizers of
$s_1(\theta),\ldots,s_k(\theta)$. Then
\[\cV_1(\theta_*)=\{\theta:\theta^{(0)}=\theta_*^{(0)}\},
\quad \cV_2(\theta_*)=\{\theta:\theta^{(0)}=\theta_*^{(0)},\;
r_l(\theta)=r_l(\theta_*) \text{ for each } l=1,\ldots,L\}.\]
On $\cV_0(\theta_*)=\R^d$, we have
$\nabla s_1(\theta)|_{\cV_0(\theta_*)}=\theta^{(0)}-\theta_*^{(0)}$,
and this vanishes exactly when $\theta \in \cV_1(\theta_*)$. Differentiating 
in the coordinates $\{(\theta_1^{(l)},\theta_2^{(l)}):l=1,\ldots,L\}$ that
parametrize $\cV_1(\theta_*)$, and applying
$r_l(\theta)^2=(\theta_1^{(l)})^2+(\theta_2^{(l)})^2$, we have
\[\nabla s_2(\theta)|_{\cV_1(\theta_*)}
=\frac{1}{2}\Big((r_l(\theta)^2-r_l(\theta_*)^2) \cdot
(\theta_1^{(l)},\theta_2^{(l)}):l=1,\ldots,L\Big).\]
Suppose $\theta_*$ satisfies the generic condition $r_l(\theta_*)>0$ for each
$l=1,\ldots,L$. Then $\nabla s_2(\theta)|_{\cV_1(\theta_*)}=0$ if and only if,
for each $l=1,\ldots,L$, either $r_l(\theta)=r_l(\theta_*)$ or
$(\theta_1^{(l)},\theta_2^{(l)})=(0,0)$. If the latter holds for any
$l=1,\ldots,L$, then differentiating in $(\theta_1^{(l)},\theta_2^{(l)})$
a second time shows that the Hessian of $s_2(\theta)$ in
$(\theta_1^{(l)},\theta_2^{(l)})$ is negative-definite at 
$(\theta_1^{(l)},\theta_2^{(l)})=(0,0)$, and hence $\lambda_{\min}(\nabla^2
s_2(\theta)|_{\cV_1(\theta_*)})<0$. On the other hand, if
$r_l(\theta)=r_l(\theta_*)$ for every $l=1,\ldots,L$, then $\theta \in
\cV_2(\theta_*)$ and $\theta$ is a global minimizer of $s_2(\theta)$.
Thus, the minimizations of $s_1(\theta)$ and
$s_2(\theta)$ on $\cV_0(\theta_*)$ and $\cV_1(\theta_*)$ are globally benign.

We now take $L \geq 30$, and
construct the example of $\theta_*$ where $s_3(\theta)$ has a spurious
local minimizer in $\cV_2(\theta_*)$ that is nondegenerate up to orbit.
Consider $\theta_*$ such that
$r_{*,l}:=r_l(\theta_*)>0$ for each $l=1,\ldots,L$.
Then $\cV_2(\theta_*)$ may be smoothly parametrized by the coordinates
$t=(t_1,\ldots,t_L)$ where $t_l=\lambda_l(\theta)-\lambda_l(\theta_*)$.
The function $8s_3(\theta)$ restricted to $\cV_2(\theta_*)$ is given as a
function of $t$ by
\[s(t)=-\mathop{\sum_{l,l',l''=1}}_{l=l'+l''}^L
r_{*,l}^2r_{*,l'}^2r_{*,l''}^2\cos(t_l-t_{l'}-t_{l''})+\text{constant}.\]
The orbit $\orbit_{\theta_*} \cap \cV_2(\theta_*)$ is defined by
\begin{equation}\label{eq:orbitt}
\big\{(t_1,\ldots,t_L):
t_l \equiv \tau \cdot l \bmod 2\pi \text{ for all } l=1,\ldots,L
\text{ and some } \tau \in \R\big\},
\end{equation}
where $\tau=0$ corresponds to the point $\theta_*$ itself.
Thus, our goal is to exhibit a point $\theta_*$ for which
\begin{equation}\label{eq:goods}
\nabla s(\hat{t})=0, \qquad \nabla^2 s(\hat{t}) \succeq 0, \qquad
\text{ and } \qquad \rank(\nabla^2 s(\hat{t}))=L-1
\end{equation}
at some $\hat{t}$ \emph{not} belonging to this orbit (\ref{eq:orbitt}).
Then the corresponding point $\theta \in \R^d$ where
$\theta^{(0)}=\theta_*^{(0)}$, $r_l(\theta)=r_{*,l}$, and
$t_l(\theta)=\hat{t}_l$ for each $l=1,\ldots,L$
is our desired spurious local minimizer for $s_3(\theta)$. Note that the
condition (\ref{eq:goods}) depends on $\theta_*$ only via
$r_*=(r_{*,1},\ldots,r_{*,L})$, so equivalently, our goal is to construct an
appropriate such vector $r_*$.

We split the construction into two steps: First, we construct
$\tilde{r}_*$ for which $\hat{t}=(\pi,0,\ldots,0)$ satisfies
\[\nabla s(\hat{t})=0, \qquad \nabla^2 s(\hat{t}) \succeq 0,
\qquad \rank(\nabla^2 s(\hat{t}))=L-2.\]
This Hessian $\nabla^2 s(\hat{t})$ will have a dimension-2 null space
spanned by the vectors $e_3=(0,0,1,0,\ldots,0)$ and $e=(1,2,\ldots,L)$.
Second, we make a small perturbation of the third coordinate of $\tilde{r}_*$,
to eliminate the null vector $e_3$ while preserving $\nabla s(\hat{t})=0$
and $\nabla^2 s(\hat{t}) \succeq 0$. This yields $r_*$ satisfying
(\ref{eq:goods}).\\

{\bf Step I.} Clearly $\hat{t}=(\pi,0,\ldots,0)$ satisfies $\nabla
s(\hat{t})=0$ for any choice of $r_*$, because $\sin(k\pi)=0$ for any integer
$k$. The Hessian of $s(t)$ is given by
\begin{align*}
  \nabla^2 s(t) &= \sum_{\substack{l,l',l''=1 \\ l=l'+l''}}^L r_{*,l}^2 r_{*,l'}^2 r_{*,l''}^2 \cos(t_l - t_{l'} - t_{l''}) (e_l - e_{l'} - e_{l''}) (e_l - e_{l'} - e_{l''})^\top,
\end{align*}
where $e_l$ is the $l^\text{th}$ standard basis vector. Then for the above
choice of $\hat{t}$ and for any vectors $u,v\in\R^L$,
\begin{align}\label{eq:SO2quad}
u^\top \nabla^2 s(\hat t) v = \sum_{\substack{l,l',l''=1 \\ l=l'+l''}}^L
r_{*,l}^2r_{*,l'}^2r_{*,l''}^2 (u_l-u_{l'}-u_{l''})(v_l-v_{l'}-v_{l''}) \times
\begin{cases} -1 & \text{ if exactly one of } l',l'' \text{ equals }1,\\
1 & \text{ otherwise.} \end{cases}
\end{align}
Consider
$\tilde{r}_*=(\tilde{r}_{*,1},\tilde{r}_{*,2},\tilde{r}_{*,3},\tilde{r}_{*,4},\ldots,\tilde{r}_{*,L})=(1,L^{\kappa/2},0,1,\ldots,1)$
where $\tilde{r}_{*,\ell}=1$ for all $\ell \geq 4$, and for a constant
$\kappa>0$ to be determined later. Then from (\ref{eq:SO2quad}) applied with
$v=e_3$ and $v=e=(1,2,\ldots,L)$,
it is immediate that $u^\top \nabla^2 s(\hat{t})v=0$ for any vector $u
\in \R^L$, so these two vectors $v=e_3$ and $v=e$ belong
to the null space of $\nabla^2 s(\hat t)$.

We now check that for any other unit vector $v \in \R^L$ orthogonal to
both $e_3$ and $e$, we have $v^\top \nabla^2 s(\hat{t}) v>0$ strictly. Observe from
(\ref{eq:SO2quad}) that for $r_*=\tilde{r}_*$,
\begin{align}\label{eq:SO2test}
v^\top \nabla^2 s(\hat t) v =& \sum_{\substack{l,l',l''=4 \\ l=l'+l''}}^L (v_{l} - v_{l'} - v_{l''})^2 + L^\kappa (v_2 - 2 v_1)^2 + L^{2\kappa} (v_4 - 2 v_2)^2 \notag\\
    & + 2 L^\kappa \sum_{l=4}^{L-2} (v_{l+2} - v_l - v_2)^2 - 2 \sum_{l=4}^{L-1} (v_{l+1} - v_l - v_1)^2.
\end{align}
Note that only the last term is negative. Denote
$\epsilon=\sqrt{6}L^{\frac{1-\kappa}{2}}$. We consider two cases:\\

{\bf Case 1:} Any one of $|v_2-2v_1|$, $|v_4-2v_2|$, $\{|v_{l+2}-v_l-v_2|:l\geq
4\}$ is larger than $\epsilon$. Then let us upper bound the last term of
(\ref{eq:SO2test}) by
\begin{align*}
    2 \sum_{l=4}^{L-1} (v_{l+1} - v_l - v_1)^2 &\leq 6 \sum_{l=4}^{L-1} (v_{l+1}^2 + v_l^2 + v_1^2)\leq  6 L,
\end{align*}
where the last inequality follows from $\|v\|_2^2=1$. Then  
\begin{align*}
v^\top \nabla^2 s(\hat t) v > L^\kappa\epsilon^2-2 \sum_{l=4}^{L-1} (v_{l+1} - v_l - v_1)^2\geq 0.
\end{align*}

{\bf Case 2:} We have instead
\begin{equation}\label{eq:SO2case21}
    |v_2 - 2 v_1| \leq \epsilon, \qquad
    |v_4 - 2 v_2| \leq \epsilon, \qquad
    |v_{l+2} - v_l - v_2| \leq \epsilon~~~{\rm for}~~~ l\geq 4.
\end{equation}
In this case, we aim to show that the first term in (\ref{eq:SO2test}) is large
enough to compensate for the negative last term of (\ref{eq:SO2test}).

For all $m \geq 1$, the second and third inequalities of (\ref{eq:SO2case21})
imply $|v_{2m}-mv_2| \leq (m-1)\epsilon$. Similarly, for all
$m \geq 0$, the last inequality implies $|v_{2m+5}-v_5-mv_2| \leq m\epsilon$.
Combining with $|mv_2-2mv_1| \leq m\epsilon$ by the first inequality of
(\ref{eq:SO2case21}), we obtain
\begin{equation}\label{eq:vbounds}
|v_{2m}-2m v_1| \leq L\epsilon ~~~{\rm for}~~~ 1 \leq m \leq L/2, \quad
|v_{2m+5}-v_5-2m v_1| \leq L\epsilon ~~~{\rm for}~~~ 0 \leq m \leq
(L-5)/2.
\end{equation}

For the summands of the first term in (\ref{eq:SO2test}), if $l'=2m'+5$
and $l''=2m''+5$ are both odd where $m',m' \geq 0$,
then $l=2m+10$ for $m=m'+m''$, and we have
\[|v_l-v_{l'}-v_{l''}| \geq |10v_1-2v_5|-3L\epsilon\]
by (\ref{eq:vbounds}) and the triangle inequality
\[|v_l-v_{l'}-v_{l''}|+|{-}v_l+(2m+10)v_1|
+|v_{l'}-v_5-2m'v_1|+|v_{l''}-v_5-2m''v_1|
\geq |10v_1-2v_5|.\]
For any even $l=2m+10$ with $m\geq 0$, the number of odd pairs $l',l'' \geq
5$ where $l'+l''=2m+10$ is $m+1$. Then the total number of tuples $(l,l',l'')$
in the first term of (\ref{eq:SO2test}) where $l',l''$ are odd is
\begin{align*}
\sum_{m=0}^{\lfloor L/2\rfloor-5}(m+1)=\frac{1}{2}(\lfloor L/2\rfloor-3)(\lfloor L/2\rfloor-4)\geq \frac{1}{8}(L-8)(L-10).
\end{align*}
Hence, we may lower bound the first term in (\ref{eq:SO2test}) by
\begin{align}\label{eq:SO2term1}
\sum_{\substack{l,l',l''=4 \\ l=l'+l''}}^L (v_{l} - v_{l'} - v_{l''})^2\geq
\frac{1}{8}(L-8)(L-10)\Big(|10v_1-2v_5|-3L\epsilon\Big)_+^2.
\end{align}
Similarly, the summands of the last term in (\ref{eq:SO2test}) for $l \geq 4$
may be upper bounded as
\begin{align}\label{eq:SO2termlast}
|v_{l+1}-v_l-v_1|\leq |5v_1-v_5|+2L\epsilon
\end{align}
by applying (\ref{eq:vbounds}) to approximate both $v_{l+1}$ and $v_l$ via
linear combinations of $v_1,v_5$, and using the triangle inequality.
(This may be checked
separately in the cases where $l$ is odd and even.) The number of summands in
this last term is $L-4$. Then,
combining (\ref{eq:SO2term1}) and (\ref{eq:SO2termlast}) gives
\begin{align}
v^\top \nabla^2 s(\hat{t})v &\geq 
\frac{1}{8}(L-8)(L-10)\Big(|10v_1-2v_5|-3L\epsilon\Big)_+^2-2(L-4)\Big(|5v_1-v_5|+2L\epsilon\Big)^2\label{eq:vshatvbound}
\end{align}

We now claim that for sufficiently large $\kappa>0$, we must have
\begin{equation}\label{eq:v1v5bound}
|5v_1-v_5|>5L\epsilon.
\end{equation}
To show this claim, since $v\perp e_3$ and $v\perp e$,
\begin{align*}
0=-3v_3&=\sum_{\substack{l=1 \\ l\not=3}}^L l\,v_l
=v_1+5v_5+\sum_{m=1}^{\lfloor L/2\rfloor} 2m\cdot v_{2m}+\sum_{m=1}^{\lfloor (L-5)/2\rfloor} (2m+5)\cdot v_{2m+5}.
\end{align*}
Denote $a_{2m}=v_{2m}-2m\cdot v_1$ and $a_{2m+5}=v_{2m+5}-v_5-2m\cdot v_1$ for
$m\geq1$, where these satisfy $|a_l| \leq L\epsilon$ by (\ref{eq:vbounds}).
Let us define
\begin{align*}
M_1&=1+\sum_{m=1}^{\lfloor L/2 \rfloor} 4m^2
+\sum_{m=1}^{\lfloor (L-5)/2 \rfloor} 2m(2m+5), \qquad
M_5=5+\sum_{m=1}^{\lfloor (L-5)/2\rfloor} (2m+5).
\end{align*}
Then we may write the above as
\begin{align*}
0&=v_1+5v_5+\sum_{m=1}^{\lfloor L/2\rfloor} 2m\cdot (a_{2m}+2m\cdot
v_1)+\sum_{m=1}^{\lfloor (L-5)/2\rfloor} (2m+5)\cdot (a_{2m+5}+v_5+2m\cdot
v_1)\\
&=M_1v_1+M_5v_5+\mathop{\sum_{l=1}^L}_{l \neq 1,3,5} l\,a_l.
\end{align*}
This may be rearranged as
\begin{align*}
v_1=\frac{M_5(5v_1-v_5)-\sum_{l:l \neq 1,3,5}  l\,a_l}{M_1+5M_5}.
\end{align*}
Now suppose by contradiction that $|5v_1-v_5| \leq 5L\epsilon$. Then this implies
\begin{align*}
|v_1|
&\leq \frac{5M_5L+L^2(L+1)/2}{M_1+5M_5}\epsilon<C\epsilon,
\end{align*}
where the second inequality holds for a universal constant $C>0$ and
any $L \geq 1$.
Then $|v_5| \leq |5v_1|+5L\epsilon<5(L+C)\epsilon$, and combining with
(\ref{eq:vbounds}) gives
\[|v_l| \leq C'L\epsilon \text{ for all } l \in \{1,\ldots,L\} \setminus \{3\}\]
and a different universal constant $C'>0$. Recalling $v_3=0$ and
$\epsilon=\sqrt{6}L^{\frac{1-\kappa}{2}}$, this contradicts that $\|v\|_2=1$ for
sufficiently large $\kappa>0$. Thus, (\ref{eq:v1v5bound}) holds.

Finally, this bound (\ref{eq:v1v5bound}) implies $|10v_1-2v_5|-3L\epsilon
>|5v_1-v_5|+2L\epsilon>0$. For $L \geq 30$, we have
$(L-8)(L-10)/8 \geq 2(L-4)$. Applying these to (\ref{eq:vshatvbound})
yields $v^\top s(\hat{t})v>0$ as desired.\\

{\bf Step II.} We now show that making a small positive perturbation to
$\tilde{r}_{*,3}$ yields a point $r_*$ which satisfies (\ref{eq:goods}) at
$\hat{t}=(\pi,0,\ldots,0)$. Denote $q_*=r_{*,3}^2$ and set
$\tilde{q}_*=\tilde{r}_{*,3}^2=0$. Let
$\Pi \in \R^{L \times (L-1)}$ have orthonormal
columns spanning the orthogonal complement of $e=(1,2,\ldots,L)$, and consider
the projected Hessian
\[H(q_*)=\Pi^\top \nabla^2 s(\hat{t})\Pi \in \R^{(L-1) \times (L-1)}\]
now as a function of $q_*$. By the result of Step I, $H(\tilde{q}_*)
\succeq 0$, and $H(\tilde{q}_*)$ has a simple eigenvalue $\mu=0$ with
eigenvector $v=\Pi^\top e_3/\|\Pi^\top e_3\|$. Then this eigenvalue
$\mu=\mu(q_*)$ is differentiable in $q_*$, with derivative given by
$\partial_{q_*} \mu(q_*)=v^\top \partial_{q_*} H(q_*) v$.
Applying $\Pi \cdot \Pi^\top=\Id-ee^\top/\|e\|^2$ and the fact that $e=(1,2,\ldots,L)$
belongs to the null space of $\nabla_t^2 s(\hat{t})$ for any $q_*$, this is
\[\partial_{q_*} \mu(q_*)=\frac{e_3^\top (\Id-ee^\top/\|e\|^2) \cdot
\partial_{q_*} \nabla_t^2 s(\hat{t}) \cdot (\Id-ee^\top/\|e\|^2)e_3}
{e_3^\top(\Id-ee^\top/\|e\|^2)e_3}
=\frac{e_3^\top \cdot \partial_{q_*} \nabla_t^2 s(\hat{t}) \cdot e_3}
{e_3^\top(\Id-ee^\top/\|e\|^2)e_3}.\]
By (\ref{eq:SO2quad}) applied with $u=v=e_3$, for general $r_*$, we have
\begin{align*}
e_3^\top \nabla_t^2 s(\hat t) e_3=-2 r_{*,1}^2 r_{*,2}^2 r_{*,3}^2
-2r_{*,1}^2r_{*,3}^2r_{*,4}^2+2 \sum_{l=2}^{L-3}r_{*,3}^2 r_{*,l}^2 r_{*,l+3}^2 (1 + \1\{l=3\}).
\end{align*}
Then differentiating in $q_*=r_{*,3}^2$ and evaluating at
$\tilde{r}_*=(1,L^\kappa,0,1,\ldots,1)$,
\begin{align*}
e_3^\top \cdot \partial_{q_*}\nabla_t^2 s(\hat t) \cdot
e_3\Big|_{q_*=\tilde{q}_*}
=-2L^{2\kappa}-2+2L^{2\kappa}+2(L-6)>0.
\end{align*}
Thus, for some sufficiently small $\delta>0$, setting $q_*=\delta$ yields
$\mu(\delta)>0$, and hence $H(\delta) \succ 0$ strictly. Then at the point
$r_*=(1,L^\kappa,\delta,1,\ldots,1)$, we obtain that (\ref{eq:goods}) holds.\\

Combining Steps I and II, we have shown that $\theta_*$ given by
$(r_1(\theta_*),\ldots,r_L(\theta_*))=(1,L^\kappa,\delta,1,\ldots,1)$
and (say) $\theta_*^{(0)}=0$ and
$(\lambda_1(\theta_*),\ldots,\lambda_L(\theta_*))=0$ satisfies (\ref{eq:goods}).
Then (\ref{eq:goods}) holds also in a sufficiently small open
neighborhood $U$ of $\theta_*$, by continuity, where $e=(1,2,3,\ldots,L)$ is the
fixed vector in the null space of $\nabla^2 s(\hat{t})$ for every $\theta_* \in
U$. Then for any $\theta_* \in U$, the function
$s_3(\theta)$ has a spurious local minimizer $\theta \in \cV_2(\theta_*)$
that is non-degenerate up to orbit, where
$\theta^{(0)}=\theta_*^{(0)}$, $r_l(\theta)=r_l(\theta_*)$ for all $l \geq 1$,
$\lambda_1(\theta)=\lambda_1(\theta_*)+\pi$, and
$\lambda_l(\theta)=\lambda_l(\theta_*)$ for all $l \geq 2$. This concludes the
proof.
\end{proof}

\subsection{Projected continuous MRA}\label{appendix:projectedSO2}
In this section, we now describe a projected version of the continuous
MRA problem with a two-fold projection onto an interval. We analyze this as a
simpler example of a model with projection, before diving into the projected
cryo-EM model to follow.

Again writing $\sS^1 \cong [0,1)$ for the unit circle and $f_\frakg(t)$ for the
periodic function $f:\sS^1 \to \R$ rotated by $\frakg \in \SO(2) \cong [0,1)$,
we consider the observations
\[(\Pi \cdot f_\frakg)(t)\der t+\sigma\,\der W(t)\]
over $t \in (0,1/2)$ where
\[(\Pi \cdot f_\frakg)(t)=f_\frakg(t)+f_\frakg(1-t)\]
and $\der W(t)$ is a standard Gaussian white noise process
on the interval $(0,1/2)$.
The map $\Pi$ represents a two-fold projection of the circle $\sS^1$ onto the
interval $(0,1/2)$.

To represent this projected model in a Gaussian sequence space,
observe that for the Fourier basis (\ref{eq:Fourierbasis}), we have
$\Pi \cdot h_{l2}=0$ for all $l \geq 1$, while
$(\Pi \cdot h_0)/\sqrt{2}$ and $(\Pi \cdot h_{l1})/\sqrt{2}$
form an orthonormal basis over $(0,1/2)$. Thus, expressing
$\Pi \cdot f$ in this projected basis, $\Pi$ may be represented
as a linear map $\Pi:\R^d \to \R^{\td}$ for $\td=L+1$, where
\begin{equation}\label{eq:PiSO2}
\Pi(\theta)=\sqrt{2}(\theta^{(0)},\theta_1^{(1)},\ldots,\theta_1^{(L)}).
\end{equation}
In this projected basis, the above observation model corresponds to the
projected orbit recovery model
(\ref{eq:projectedmodel}) where $g \in \G$ is a random rotation from the same
group $\G$ as in (\ref{eq:MRAG}).

The following result shows that the decomposition of total dimension in Theorem
\ref{thm:FI} is the same as in the unprojected setting. In particular,
$\trdeg \tcR_{\leq \tK}^\G = \trdeg \cR^\G$ for $\tK=3$. This model is
a continuous analogue of the projected discrete MRA model studied
in \cite[Section 5.3.1]{bandeira2017estimation}, where an analogous conclusion
was described as \cite[Conjecture 5.3]{bandeira2017estimation}.

\begin{theorem}\label{thm:projectedMRA}
For any $L\geq 1$, we have
\[
\trdeg(\tcR_{\leq 1}^\G) = 1, \qquad \trdeg(\tcR_{\leq 2}^\G) = L + 1, \qquad
\trdeg(\tcR_{\leq 3}^\G) = \trdeg(\cR^\G) = 2L,
\]
which match the values of $\trdeg(\cR_{\leq 1}^\G)$, $\trdeg(\cR_{\leq 2}^\G)$, 
and $\trdeg(\cR_{\leq 3}^\G)$ in the unprojected setting of Theorem~\ref{thm:MRA}. 
\end{theorem}

The following result describes the forms of $\ts_k(\theta)$ for $k=1,2,3$,
which are similar to those in the unprojected setting. The
minimizations of $\ts_1(\theta)$, $\ts_2(\theta)$, and $\ts_3(\theta)$ 
are also optimization problems over the signal
mean, Fourier power spectrum, and Fourier bispectrum respectively, although the
specific forms are different from the unprojected counterparts.

\begin{theorem}\label{thm:projectedMRA-mom}
For any $L \geq 1$,
\begin{align*}
\ts_1(\theta)&=\Big(\theta^{(0)}-\theta_*^{(0)}\Big)^2\\
\ts_2(\theta)&=\Big((\theta^{(0)})^2-(\theta_*^{(0)})^2\Big)^2+\frac{1}{4}\sum_{l=1}^L
\Big(r_l(\theta)^2-r_l(\theta_*)^2\Big)^2\\
\ts_3(\theta)
&=\frac{2}{3}\Big((\theta^{(0)})^3-(\theta_*^{(0)})^3\Big)^2
+\frac{1}{2}\sum_{l=1}^L\Big(\theta^{(0)}r_l(\theta)^2-\theta_*^{(0)}r_l(\theta_*)^2\Big)^2\\
&\hspace{0.2in}+\frac{1}{16} \mathop{\sum_{l,l',l''=1}^L}_{l=l'+l''}\bigg(
r_{l,l',l''}(\theta)^2+r_{l,l',l''}(\theta_*)^2
-2r_{l,l',l''}(\theta)r_{l,l',l''}(\theta_*)\cos\big(\lambda_{l,l',l''}(\theta_*)-\lambda_{l,l',l''}(\theta)\big)\\
&\hspace{1in}+r_{l,l',l''}(\theta)^2\cos\big(2\lambda_{l,l',l''}(\theta)\big)
+r_{l,l',l''}(\theta_*)^2\cos\big(2\lambda_{l,l',l''}(\theta_*)\big)\\
&\hspace{1in}-2r_{l,l',l''}(\theta)r_{l,l',l''}(\theta_*)\cos\big(\lambda_{l,l',l''}(\theta_*)+\lambda_{l,l',l''}(\theta_*)\big)\bigg).
\end{align*}
\end{theorem}

In the remainder of this section, we
prove Theorems \ref{thm:projectedMRA-mom} and
\ref{thm:projectedMRA}.

\begin{proof}[Proof of Theorem \ref{thm:projectedMRA-mom}]
Recall by Lemma \ref{lem:skform} that
\begin{equation}\label{eq: sk SO2 projected}
\ts_k(\theta) = \frac{1}{2 (k!)} \E_{g, h}[\langle \Pi\cdot g \cdot \theta, \Pi\cdot h\cdot \theta\rangle^k
- 2 \langle \Pi\cdot g\cdot \theta, \Pi\cdot h\cdot \theta_* \rangle^k + \langle \Pi\cdot g\cdot \theta_*, \Pi\cdot h\cdot \theta_*\rangle^k].
\end{equation}
We compute
$\E_{g,h}[\langle \Pi\cdot g \cdot\theta, \Pi\cdot h\cdot\vartheta\rangle^k]$
for $\theta,\vartheta \in \R^d$.

{\bf Case $k=1$:} Recall $u^{(0)}(\theta)=\theta^{(0)}$ and $u^{(l)}(\theta)=\theta_1^{(l)}+\i \theta_2^{(l)}$ for $l=1,\ldots,L$. Write $g,h \in \G$
corresponding to the rotations $\frakg,\frakh\in [0,1)$. Then, applying
(\ref{eq:PiSO2}), (\ref{eq:complexrepr}), and the identity
$\Re a \cdot \Re b=(ab+\bar{a}\bar{b}+a\bar{b}+\bar{a}b)/4$, we have
\begin{align}\label{eq: inner prod S02 projected}
  \langle \Pi\cdot g\cdot \theta, \Pi\cdot h\cdot\vartheta\rangle
=2\sum_{l=0}^L \Re u^{(l)}(g \cdot \theta) \cdot \Re u^{(l)}(h \cdot
\vartheta)=\frac{1}{2}M_1+\frac{1}{2}M_2
\end{align}
where
\begin{align*}
M_1:=&\sum_{l=0}^L u^{(l)}(\theta) u^{(l)}(\vartheta) e^{-2\i\pi l(\frakg+\frakh)} + \overline{u^{(l)}(\theta) u^{(l)}(\vartheta)} e^{2\i\pi l(\frakg + \frakh)},\\
M_2:=&\sum_{l=0}^L u^{(l)}(\theta) \overline{u^{(l)}(\vartheta)} e^{-2\i\pi l(\frakg - \frakh)} + \overline{u^{(l)}(\theta)} u^{(l)}(\vartheta) e^{2\i\pi l(\frakg - \frakh)}.
\end{align*}
For independent and uniformly random $\frakg,\frakh \in [0,1)$,
taking the expected value on both sides using (\ref{eq:gexp}) gives
$\E_{g,h}[\langle \Pi\cdot g\cdot \theta, \Pi\cdot
h\cdot\vartheta\rangle]=2u^{(0)}(\theta)u^{(0)}(\vartheta)$.
Then from (\ref{eq: sk SO2 projected}), we obtain
\begin{align}\label{eq: s1 SO2 projected}
\ts_1(\theta)=\rbr{u^{(0)}(\theta) - u^{(0)}(\theta_*)}^2
=\rbr{\theta^{(0)}-\theta_*^{(0)}}^2.
\end{align}

{\bf Case $k=2$:} Taking the expected square on both sides of (\ref{eq: inner prod S02 projected}), we have
\begin{align*}
\E_{g,h}[\langle \Pi\cdot g\cdot \theta, \Pi\cdot
h\cdot\vartheta\rangle^2]=\frac{1}{4}\Big\{\E_{\frakg,\frakh}[M_1^2]+2\E_{\frakg,\frakh}[M_1M_2]+\E_{\frakg,\frakh}[M_2^2]\Big\}.
\end{align*}
Applying (\ref{eq:gexp}) and an argument similar to the $k=2$ computation in 
the proof of Theorem \ref{thm:MRA-mom},
\begin{align*}
\E_{\frakg,\frakh}[M_1^2]&=2(u^{(0)}(\theta))^2(u^{(0)}(\vartheta))^2+2\sum_{l=0}^L |u^{(l)}(\theta)|^2|u^{(l)}(\vartheta)|^2,\\
\E_{\frakg,\frakh}[M_1M_2]&=4(u^{(0)}(\theta))^2(u^{(0)}(\vartheta))^2,\\
\E_{\frakg,\frakh}[M_2^2]&=2(u^{(0)}(\theta))^2(u^{(0)}(\vartheta))^2+2\sum_{l=0}^L |u^{(l)}(\theta)|^2|u^{(l)}(\vartheta)|^2.
\end{align*}
Then, separating the $l=0$ terms from these sums,
$\E_{g,h}[\langle \Pi\cdot g\cdot \theta, \Pi\cdot
h\cdot\vartheta\rangle^2]=4(\theta^{(0)})^2(\vartheta^{(0)})^2+\sum_{l=1}^L
r_l(\theta)^2r_l(\vartheta)^2$, so
by (\ref{eq: sk SO2 projected}),
\begin{align}\label{eq: s2 SO2 projected}
\ts_2(\theta)=\Big((\theta^{(0)})^2-(\theta_*^{(0)})^2\Big)^2
+\frac{1}{4}\sum_{l=1}^L \Big(r_l(\theta)^2-r_l(\theta_*)^2\Big)^2.
\end{align}

{\bf Case $k=3$:} Taking the expected cube on both sides of (\ref{eq: inner prod S02 projected}), we have
\begin{align*}
\E_{g,h}[\langle \Pi\cdot g\cdot \theta, \Pi\cdot h\cdot\vartheta\rangle^3]=\frac{1}{8}\Big\{\E_{\frakg,\frakh}[M_1^3]+3\E_{\frakg,\frakh}[M_1^2M_2]+3\E_{\frakg,\frakh}[M_1M_2^2]+\E_{\frakg,\frakh}[M_2^3]\Big\}.
\end{align*}
Applying (\ref{eq:gexp}) and an argument similar to the $k=3$ computation in
the proof of Theorem \ref{thm:MRA-mom},
\begin{align*}
\E_{\frakg,\frakh}[M_1^3]&=2(u^{(0)}(\theta))^3(u^{(0)}(\vartheta))^3
 + 3 \mathop{\sum_{l,l',l''=0}^L}_{l=l'+l''} u^{(l)}(\theta) \overline{u^{(l')}(\theta) u^{(l'')}(\theta)}u^{(l)}(\vartheta) \overline{u^{(l')}(\vartheta) u^{(l'')}(\vartheta)}\\
  &\hspace{0.2in}+ 3 \mathop{\sum_{l,l',l''=0}^L}_{l=l'+l''} \overline{u^{(l)}(\theta)} u^{(l')}(\theta) u^{(l'')}(\theta) \overline{u^{(l)}(\vartheta)} u^{(l')}(\vartheta) u^{(l'')}(\vartheta),\\
\E_{\frakg,\frakh}[M_1^2M_2]&=
\E_{\frakg,\frakh}[M_1M_2^2]=
8(u^{(0)}(\theta))^3(u^{(0)}(\vartheta))^3+4u^{(0)}(\theta)u^{(0)}(\vartheta)\cdot\sum_{l=1}^L |u^{(l)}(\theta)|^2|u^{(l)}(\vartheta)|^2,\\
\E_{\frakg,\frakh}[M_2^3]&=2(u^{(0)}(\theta))^3(u^{(0)}(\vartheta))^3
  + 3 \mathop{\sum_{l,l',l''=0}^L}_{l=l'+l''} u^{(l)}(\theta) \overline{u^{(l')}(\theta) u^{(l'')}(\theta) u^{(l)}(\vartheta)} u^{(l')}(\vartheta) u^{(l'')}(\vartheta)\\
  &\hspace{0.2in}+ 3 \mathop{\sum_{l,l',l''=0}^L}_{l=l'+l''} \overline{u^{(l)}(\theta)} u^{(l')}(\theta) u^{(l'')}(\theta) u^{(l)}(\vartheta) \overline{u^{(l')}(\vartheta) u^{(l'')}(\vartheta)}.
\end{align*}
In these expressions for $\E_{\frakg,\frakh}[M_1^3]$
and $\E_{\frakg,\frakh}[M_2^3]$,
separating out the three cases of $l=l'=l''=0$, only $l'=0$ and $l=l'' \geq 1$,
and only $l''=0$ and $l=l' \geq 1$ , this gives
\begin{align*}
&\E_{g,h}[\langle \Pi\cdot g\cdot \theta, \Pi\cdot h\cdot\vartheta\rangle^3]\\
&=8(u^{(0)}(\theta))^3(u^{(0)}(\vartheta))^3
+6u^{(0)}(\theta)u^{(0)}(\vartheta)\cdot\sum_{l=1}^L |u^{(l)}(\theta)|^2|u^{(l)}(\vartheta)|^2\\
&\quad+\frac{3}{8}\mathop{\sum_{l,l',l''=1}^L}_{l=l'+l''} u^{(l)}(\theta)
\overline{u^{(l')}(\theta) u^{(l'')}(\theta)}u^{(l)}(\vartheta)
\overline{u^{(l')}(\vartheta) u^{(l'')}(\vartheta)}+
\overline{u^{(l)}(\theta)} u^{(l')}(\theta) u^{(l'')}(\theta) \overline{u^{(l)}(\vartheta)} u^{(l')}(\vartheta) u^{(l'')}(\vartheta)\\
&\quad+\frac{3}{8}\mathop{\sum_{l,l',l''=1}^L}_{l=l'+l''} 
u^{(l)}(\theta) \overline{u^{(l')}(\theta) u^{(l'')}(\theta) u^{(l)}(\vartheta)} u^{(l')}(\vartheta) u^{(l'')}(\vartheta)
  +\overline{u^{(l)}(\theta)} u^{(l')}(\theta) u^{(l'')}(\theta) u^{(l)}(\vartheta) \overline{u^{(l')}(\vartheta) u^{(l'')}(\vartheta)}.
\end{align*}
Applying this to (\ref{eq: sk SO2 projected}) using
$(x^2+\bar{x}^2)-2(xx_*+\bar{x}\bar{x}_*)+(x_*^2+\bar{x}_*^2)
=(x-x_*)^2+(\bar{x}-\bar{x}_*)^2=2\Re [(x-x_*)^2]$, and writing as shorthand
$u^{(l)}=u^{(l)}(\theta)$, $u_*^{(l)}=u^{(l)}(\theta_*)$
and similarly for $r_l,\lambda_l,r_{l,l',l''},\lambda_{l,l',l''}$,
\begin{align*}
\ts_3(\theta)&=\frac{2}{3}\Big((u^{(0)})^3-(u_*^{(0)})^3\Big)^2
+\frac{1}{2}\sum_{l=1}^L\Big(u^{(0)}|u^{(l)}|^2-u_*^{(0)}|u_*^{(l)}|^2\Big)^2\\
&\hspace{0.2in}+\frac{1}{16} \mathop{\sum_{l,l',l''=1}^L}_{l=l'+l''}
\Big|u^{(l)}\overline{u^{(l')}u^{(l'')}}-u_*^{(l)}\overline{u_*^{(l')}u_*^{(l'')}}\Big|^2
+\Re\Big[\Big(u^{(l)}\overline{u^{(l')}u^{(l'')}}-u_*^{(l)}\overline{u_*^{(l')}u_*^{(l'')}}\Big)^2\Big]\\
&=\frac{2}{3}\Big((\theta^{(0)})^3-(\theta_*^{(0)})^3\Big)^2
+\frac{1}{2}\sum_{l=1}^L\Big(\theta^{(0)}r_l^2-\theta_*^{(0)}r_{*,l}^2\Big)^2\\
&\hspace{0.2in}+\frac{1}{16} \mathop{\sum_{l,l',l''=1}^L}_{l=l'+l''}
r_{l,l',l''}^2+r_{*,l,l',l''}^2-2r_{l,l',l''}r_{*,l,l',l''}\cos(\lambda_{*,l,l',l''}-\lambda_{l,l',l''})\\
&\hspace{0.5in}+r_{l,l',l''}^2\cos(2\lambda_{l,l',l''})+r_{*,l,l',l''}^2
\cos(2\lambda_{*,l,l',l''})
-2r_{l,l',l''}r_{*,l,l',l''}\cos(\lambda_{*,l,l',l''}+\lambda_{l,l',l''}).
\end{align*}
\end{proof}

\begin{lemma}\label{lem:rank}
Let $H\in\R^{n\times n}$ and $n=n_1+n_2$ with $n_1,n_2\geq 1$. Suppose $H$ can be decomposed as the sum of two positive semidefinite matrices $A$ and $B$ with
\begin{align*}
A=\begin{pmatrix} A_{11} & 0 \\
0 & 0 \end{pmatrix}~~~{\rm and}~~~
B=\begin{pmatrix} B_{11} & B_{12} \\
B_{21} & B_{22} \end{pmatrix},
\end{align*}
where $A_{11},B_{11}\in\R^{n_1\times n_1}, B_{12}\in\R^{n_1\times n_2}, B_{21}\in\R^{n_2\times n_1}$, and $B_{22}\in\R^{n_2\times n_2}$. Then
\begin{align*}
\rank(H)\geq \rank(A_{11})+\rank(B_{22}).
\end{align*}
\end{lemma}
\begin{proof}
Since $A$ and $B$ are positive semidefinite, so are
$A_{11},B_{11},B_{22}$. There are $\rank(A_{11})$ linearly independent
vectors $v \in \R^{n_1}$ where $v^\top A_{11} v>0$ strictly. Then $(v,0)^\top
H(v,0)>0$ strictly for each such vector $v$. There are also $\rank(B_{22})$
linearly independent vectors $w \in \R^{n_2}$ where $w^\top B_{22} w>0$
strictly. Then $(0,w)^\top H(0,w)>0$ strictly for each such vector $w$.
Thus $u^\top Hu>0$ for $\rank(A_{11})+\rank(B_{22})$ linearly independent
vectors $u \in \R^n$, so $\rank(H) \geq \rank(A_{11})+\rank(B_{22})$.
\end{proof}

\begin{proof}[Proof of Theorem \ref{thm:projectedMRA}]
As in Theorem \ref{thm:MRA}, we have $\trdeg \cR^\G=2L$.
We compute $\trdeg(\tcR_{\leq k}^\G)$ for $k=1,2,3$ by applying Lemma
\ref{lem:trdeg} at a generic point $\theta_* \in \R^d$ with $r_l(\theta_*)>0$
for each $l=1,\ldots,L$. Recall from the proof of Theorem \ref{thm:MRA} the map
\begin{align*}
\zeta(\theta)=(\theta_0,r_1,\ldots,r_L,t_1,\ldots,t_L), \qquad
\zeta_*=\zeta(\theta_*),
\end{align*}
with inverse function $\theta(\zeta)$ in a local neighborhood of $\theta_*$.
The forms of $\ts_1(\theta(\zeta))$ and $\ts_2(\theta(\zeta))$ are similar to
those of $s_1(\theta(\zeta))$ and $s_2(\theta(\zeta))$ in Theorem \ref{thm:MRA},
and the same arguments as in the proof of Theorem \ref{thm:MRA} show
\begin{align*}
\trdeg(\tcR_{\leq 1}^\G)&= \rank\Big(\nabla_\zeta^2
\ts_1(\theta(\zeta))\Big|_{\zeta=\zeta_*}\Big)=1,\\
\trdeg(\tcR_{\leq 2}^\G)&=\rank\Big(\nabla_\zeta^2 \ts_1(\theta(\zeta))
+\nabla_\zeta^2 \ts_2(\theta(\zeta))\Big|_{\zeta=\zeta_*}\Big)
=L+1,
\end{align*}
and $\trdeg(\tcR_{\leq 3}^\G)=\rank(\tH(\zeta_*))$ for the Hessian
\begin{align*}
\tH(\zeta_*)=\nabla_\zeta^2 \ts_1(\theta(\zeta))
+\nabla_\zeta^2 \ts_2(\theta(\zeta))+\nabla_\zeta^2
\ts_3(\theta(\zeta))\Big|_{\zeta=\zeta_*}.
\end{align*}
Writing this Hessian in the block decomposition according to
$r=(\theta_0,r_1,\ldots,r_L)$ and $t=(t_1,\ldots,t_L)$,
and noting that $\ts_1,\ts_2$ depend only on $r$ and not on $t$,
we have the decomposition
\begin{align*}
\tH(\zeta_*)&=\begin{pmatrix}
\nabla_r^2 \ts_1(\theta(\zeta))
+\nabla_r^2 \ts_2(\theta(\zeta))\big|_{\zeta=\zeta_*} & 0\\
0 & 0
\end{pmatrix}+\begin{pmatrix} \nabla_r^2 \ts_3(\theta(\zeta))
\big|_{\zeta=\zeta_*}
& \nabla_{rt}^2 \ts_3(\theta(\zeta)) \big|_{\zeta=\zeta_*} \\
\nabla_{tr}^2 \ts_3(\theta(\zeta)) \big|_{\zeta=\zeta_*}
& \nabla_{t}^2 \ts_3(\theta(\zeta)) \big|_{\zeta=\zeta_*}
\end{pmatrix}
\end{align*}
The second matrix is positive semidefinite by Lemma \ref{lem:trdeg},
and the first matrix has
$\nabla_r^2 \ts_1(\theta(\zeta))+\nabla_r^2
\ts_2(\theta(\zeta))|_{\zeta=\zeta_*} \succ 0$ strictly by the analysis of
$\trdeg(\tcR_{\leq 2}^\G)$, with rank exactly $L+1$.
Then by Lemma \ref{lem:rank},
\begin{align}\label{eq:ts3rank}
\rank(\tH(\zeta_*)) \geq L+1+\rank(\nabla_t^2
\ts_3(\theta(\zeta))|_{\zeta=\zeta_*}).
\end{align}
As in the proof of Theorem \ref{thm:MRA},
let us show $\rank(\nabla_t^2 \ts_3(\theta(\zeta))|_{\zeta=\zeta_*}) \geq L-1$
for generic $\theta_* \in \R^d$ by exhibiting a single point $\theta_*$ where
this holds.

We may write the expression for $\ts_3(\theta)$ in
Theorem \ref{thm:projectedMRA-mom} as
\begin{align*}
\ts_3(\theta(\zeta))&=f(r)
+\frac{1}{16} \mathop{\sum_{l,l',l''=1}^L}_{l=l'+l''}
{-}2r_{l,l',l''}r_{*,l,l',l''} \cos(\lambda_{l,l',l''}-\lambda_{*,l,l',l''})
+r_{l,l',l''}^2 \cos( 2\lambda_{l,l',l''})\\
&\hspace{2in}-2r_{l,l',l''}r_{*,l,l',l''}\cos(\lambda_{l,l',l''}+\lambda_{*,l,l',l''}),
\end{align*}
for a function $f(r)$ depending only on $r$ and not $t$.
We pick $\theta_*$ such that $\theta_*^{(0)}=1$ and $r_{*,l}=1$ for
all $l=1,\ldots,L$. Then, recalling $t_l=\lambda_l-\lambda_{*,l}$ and
$\lambda_{l,l',l''}=\lambda_l-\lambda_{l'}-\lambda_{l''}$, and
differentiating twice in $t$ at $(r,t)=(r_*,t_*)=(r_*,0)$,
\begin{align*}
\nabla_t^2 \ts_3(\theta(\zeta))\Big|_{\zeta=\zeta_*}&=\frac{1}{16}\nabla_t^2
\Bigg(\mathop{\sum_{l,l',l''=1}^L}_{l=l'+l''} {-}2\cos(t_l-t_{l'}-t_{l''})
+\cos(2t_l-2t_{l'}-2t_{l''}+2\lambda_{*,l}-2\lambda_{*,l'}-2\lambda_{*,l''})\\
&\hspace{1in}
-2\cos(t_l-t_{l'}-t_{l''}+2\lambda_{*,l}-2\lambda_{*,l'}-2\lambda_{*,l''})
\Bigg)\Bigg|_{t=0}\\
&=\frac{1}{16}\mathop{\sum_{l,l',l''=1}^L}_{l=l'+l''}
\Big(2-4\cos(2\lambda_{*,l,l',l''})+2\cos(2\lambda_{*,l,l',l''})\Big)
\cdot w_{l,l',l''}w_{l,l',l''}^\top\\
&=\frac{1}{8}\mathop{\sum_{l,l',l''=1}^L}_{l=l'+l''}
\Big(1-\cos(2\lambda_{*,l,l',l''})\Big) \cdot w_{l,l',l''}w_{l,l',l''}^\top,
\end{align*}
where $w_{l,l',l''}$ is defined as (\ref{eq:wl}).
Stacking $w_{l,l',l''}$ as the columns of $W \in \R^{L \times |\mathcal{L}|}$
as in the proof of Theorem \ref{thm:MRA}, and defining the diagonal matrix
$D=\diag(1-\cos(2\lambda_{*,l,l',l''})) \in \R^{|\mathcal{L}| \times
|\mathcal{L}|}$, this shows
\[\nabla_t^2 \ts_3(\theta(\zeta))\Big|_{\zeta=\zeta_*}
=\frac{1}{8}WDW^\top.\]
Note that, for generic $\lambda_*=(\lambda_{*,1},\ldots,\lambda_{*,L})$, we have
\begin{align*}
2-2\cos(2\lambda_{*,l}-2\lambda_{*,l'}-2\lambda_{*,l''})>0
\end{align*}
for each fixed tuple $(l,l',l'')\in\cL$. Hence we may pick $\lambda_*$ so that
this holds simultaneously for all tuples $(l,l',l'') \in \cL$. Then
$\rank(WDW^\top)=\rank(WW^\top) \geq L-1$ as shown in Theorem \ref{thm:MRA}.
Applying this back to (\ref{eq:ts3rank}), we have shown $\trdeg \tcR_{\leq 3}^\G
=\rank(\tH(\zeta_*)) \geq 2L$.
Since also $\trdeg \tcR_{\leq 3}^\G \leq \trdeg \cR^\G=2L$,
this shows $\trdeg \tcR_{\leq 3}^\G=2L$.
\end{proof}

\section{Analyses of function estimation under an $\SO(3)$ rotation}\label{appendix:SO3}

This appendix contains further details on the setups of the models and
the proofs of the main results in Section \ref{sec:func-est-so3} on estimating
a function in 2 or 3 dimensions under $\SO(3)$-rotations.

Appendix \ref{sec:rc-harmonic} first reviews the complex spherical harmonics
basis and the associated calculus of Wigner
D-matrices and Clebsch-Gordan coefficients. Appendix \ref{sec:sr} contains
further details and proofs for Section \ref{subsec:sphericalregistration} on spherical registration.
Appendix \ref{sec:unprojCEM} contains further details and proofs for Section
\ref{subsec:cryoEM} on the
unprojected cryo-EM model. Finally, Appendix \ref{sec:unprojCEM} contains
further details and proofs for Section \ref{subsec:projectedcryoEM}
on the projected cryo-EM model.

\subsection{Calculus of spherical harmonics} \label{sec:rc-harmonic}

We first fix notations for some special functions related to the action of $\SO(3)$ and
present some identities between them which will appear in the proofs.

\subsubsection{Complex spherical harmonics}

Let $P_{lm}(x)$ denote the associated Legendre polynomials (without
Cordon-Shortley phase)
\begin{equation}\label{eq:legendre}
P_{lm}(x)=\frac{1}{2^l l!}(1-x^2)^{m/2} \frac{d^{l+m}}{dx^{l+m}}(x^2-1)^l 
\quad \text{ for } m=-l,-l+1,\ldots,l-1,l.
\end{equation}
Let $\sS^2 \subset \R^3$ be the unit sphere, parametrized by the 
latitude $\phi_1 \in [0,\pi]$ and longitude $\phi_2 \in [0,2\pi)$.
The complex spherical harmonics basis on $\sS^2$ is given by (see \cite[Eq.\
(III.20)]{rose1995elementary})
\begin{equation}\label{eq:complexharmonics}
y_{lm}(\phi_1,\phi_2)
=(-1)^m\sqrt{\frac{2l+1}{4\pi} \cdot \frac{(l-m)!}{(l+m)!}} \cdot
P_{lm}(\cos \phi_1)e^{\i m\phi_2}
\text{ for } l \geq 0 \text{ and } m=-l,\ldots,l.
\end{equation}
(We will use interchangeably notations such as $y_{lm}$ and $y_{l,m}$
when the meaning is clear.)
The index $l$ is the frequency, and there are $2l+1$ basis functions at each
frequency $l$. These functions are orthonormal in $L_2(\sS^2,\C)$ with respect
to the surface area measure $\sin \phi_1\,\der \phi_1\,\der \phi_2$,
and satisfy the conjugation symmetry (see \cite[Eq.\
(III.23)]{rose1995elementary})
\begin{equation}\label{eq:conjugationsymmetry}
\overline{y_{l,m}(\phi_1,\phi_2)}=(-1)^m y_{l,-m}(\phi_1,\phi_2).
\end{equation}

\begin{lemma} \label{lem:legendre-value}
For all $m$, the associated Legendre
polynomials in (\ref{eq:legendre}) satisfy
\[P_{lm}(0)=\1\{l+m \text{ is even}\}
\cdot \frac{(-1)^{(l-m)/2}}{2^l l!}\binom{l}{(l+m)/2}(l+m)!\]
\end{lemma}
\begin{proof}
This follows from applying a binomial expansion of $(x^2-1)^l$, and then
differentiating in $x$---see also \cite[Eq.\ (14)]{bandeira2017estimation}.
\end{proof}

\subsubsection{Wigner D-matrices}\label{appendix:wignerd}

Let $f \in L_2(\sS^2,\C)$.
Then $f$ may be decomposed in the complex spherical harmonics basis
(\ref{eq:complexharmonics}) as
\[f=\sum_{l=0}^\infty \sum_{m=-l}^l u_m^{(l)}y_{lm}.\]
Writing $u^{(l)}=\{u_m^{(l)}:-l \leq m \leq l\}$,
the rotation $f \mapsto f_\frakg$ given by
$f_\frakg(\phi_1,\phi_2)=f(\frakg^{-1} \cdot (\phi_1,\phi_2))$
for $\frakg \in \SO(3)$ is described by the map of spherical harmonic coefficients
(see \cite[Eq.\ (4.28a)]{rose1995elementary})
\[u^{(l)}\mapsto D^{(l)}(\frakg) u^{(l)} \text{ for each } l=0,1,2,\ldots,\]
where $D^{(l)}(\frakg)\in\C^{(2l+1)\times (2l+1)}$ is the \emph{complex Wigner
D-matrix} at frequency $l$ corresponding to $\frakg$. We index the rows
and columns of $D^{(l)}(\frakg)$ by $-l,\ldots,l$.

Our computations will not require the explicit forms of $D^{(l)}(\frakg)$, but
only the following moment identities when $\frakg \in \SO(3)$ is 
a Haar-uniform random rotation (see \cite[Section 16]{rose1995elementary} and
\cite[Appendix A.3]{bandeira2017estimation}):
\begin{enumerate}[(1)]
\item Mean identity:
\begin{equation}\label{eq:ED}
D^{(0)}(\frakg)=1, \qquad \E_{\frakg}[D^{(l)}(\frakg)]=0 \text{ for all } l \geq 1.
\end{equation}

\item Orthogonality: for any $l,l'\geq 0$ and $-l\leq q,m\leq l$
and $-l'\leq q',m'\leq l'$,
\begin{equation}\label{eq:WignerDorthog}
\E_\frakg\left[D^{(l)}_{qm}(\frakg)D^{(l')}_{q'm'}(\frakg)\right]
=\frac{(-1)^{m+q}}{2l+1}\1\{l=l',q=-q',m=-m'\}.
\end{equation}

\item Third order identity: for any $l,l',l''\geq 0$ and
$-l\leq q,m\leq l$ and $-l'\leq q',m'\leq l'$ and $-l''\leq q'',m''\leq l''$,
\begin{align}\label{eq:WignerDtriple}
\E_\frakg\left[D_{qm}^{(l)}(\frakg)D_{q'm'}^{(l')}(\frakg)D_{q''m''}^{(l'')}(\frakg)\right]
&=\1\{q+q'=-q''\} \cdot \1\{m+m'=-m''\} \cdot \1\{|l-l'| \leq l'' \leq l+l'\}
\nonumber\\
&\hspace{0.3in}\cdot \frac{(-1)^{m''+q''}}
{2l''+1}\langle l,q;l',q'|l'',-q'' \rangle \langle l,m;l',m'|l'',-m'' \rangle,
\end{align}
where $\langle l,m;l',m'|l'',m'' \rangle$ is a Clebsch-Gordan coefficient,
defined in the following section.
\end{enumerate}

\subsubsection{Clebsch-Gordan coefficients}
The Clebsch-Gordan coefficients $\langle l,m;l',m'|l'',m'' \rangle$
are defined for integer arguments $l,l',l'',m,m',m''$ where
\begin{equation}\label{eq:CGdomain}
|m| \leq l,\quad |m'| \leq l', \quad |m''| \leq l'' \quad \text{ and } \quad
|l-l'| \leq l'' \leq l+l'.
\end{equation}
The latter condition $|l-l'| \leq l'' \leq l+l'$ is equivalent to the three
symmetric triangle inequality conditions $l+l' \geq l''$, $l+l'' \geq l'$, and
$l'+l'' \geq l$.
For such arguments, $\langle l,m;l',m'|l'',m'' \rangle$ is given explicitly by
(see \cite[Eq.\ (2.41)]{bohm2013quantum} and \cite[Appendix
A.2]{bandeira2017estimation})
\begin{align}\label{eq:cg-def}
&\langle l,m;l',m'|l'',m'' \rangle\notag\\
 &=\1\{m''=m+m'\} \times\sqrt{\frac{(2l'' + 1)(l+l'-l'')!(l+l''-l')!(l'+l''-l)!}{(l + l' + l'' + 1)!}}\notag\\
  &\qquad\times\sqrt{(l - m)!(l + m)!(l' - m')!(l' + m')! (l''-m'')! (l''+m'')!
}\notag\\
  &\qquad\times\sum_k \frac{(-1)^k}{k! (l + l' - l'' - k)! (l - m - k)! (l' + m' - k)! (l'' - l' + m + k)! (l'' - l - m' + k)!}
\end{align} 
where the summation is over all integers $k$ for which the argument of every
factorial is nonnegative. We extend the definition to all integer arguments by
\begin{equation}\label{eq:CGconvention}
\langle l,m;l',m'|l'',m' \rangle=0 \quad \text{ if (\ref{eq:CGdomain}) does not
hold}.
\end{equation}
We will use the notational shorthand
\[C_{m,m',m''}^{l,l',l''}=\langle l,m;l',m'|l'',m'' \rangle.\]

These coefficients satisfy the sign symmetry (see
\cite[Eq.\ (3.16a)]{rose1995elementary} and {\cite[Eq.\
(2.47)]{bohm2013quantum})
\begin{equation}\label{eq:CGsymmetry2}
C^{l,l',l''}_{m, m', m''}=(-1)^{l+l'+l''}
C^{l,l',l''}_{-m, -m', -m''}.
\end{equation}
Note that we may have $C_{m,m',m''}^{l,l',l''}=0$ even if
(\ref{eq:CGdomain}) holds and $m''=m+m'$. For example, $C_{2,-1,1}^{3,2,2}=0$.
In our later proofs, we will require that certain Clebsch-Gordan
coefficients are \emph{non-zero}, and the following lemma provides a 
sufficient condition for this to hold.

\begin{lemma}\label{lemma:CGnonvanishing}
Let $l,l',l'',m,m',m''$ satisfy (\ref{eq:CGdomain}), where $m''=m+m'$. In
addition, suppose the following conditions all hold:
\begin{itemize}
\item $l \geq l'$ and $l \geq l''+1$.
\item $|m'| \in \{l'-1,l'\}$ and $|m''| \in \{l''-1,l''\}$.
\item We do not have simultaneously $|m|=l-1$, $|m'|=l'-1$, $|m''|=l''-1$, and $l'=l''$.
\end{itemize}
Then
\[C^{l,l',l''}_{m, m', m''} \neq 0.\]
\end{lemma}
\begin{proof}
Applying the sign symmetry (\ref{eq:CGsymmetry2}), we may assume 
$m' \in \{-l',-l'+1\}$. We consider separately these cases.

{\bf Case I: $m'=-l'$.} Then $m=m''-m'=m''+l'$. The condition that $k$ and
$l'+m'-k$ are both nonnegative in (\ref{eq:cg-def}) requires $k=0$, so the sum
in (\ref{eq:cg-def}) consists of just this single term.
The remaining factorials in (\ref{eq:cg-def}) are also nonnegative, because
$l+l'-l'' \geq 0$, $l-m \geq 0$, $l''-l'+m=l''+m'' \geq 0$, and $l''-l-m'
=l''+l'-l \geq 0$. Thus (\ref{eq:cg-def}) is non-zero.

{\bf Case II: $m'=-l'+1$.} Then $m=m''+l'-1$.
The condition that $k$ and $l'+m'-k$ are both
nonnegative then requires $k \in \{0,1\}$. Substituting $m'=-l'+1$ and
$m=m''+l'-1$, these two terms in (\ref{eq:cg-def}) for $k \in \{0,1\}$ are
\[\tfrac{1}{(l+l'-l'')!(l-m)!(l''+m''-1)!(l'+l''-l-1)!}
-\tfrac{1}{(l+l'-l''-1)!(l-m-1)!(l''+m'')!(l'+l''-l)!}\]
where each term is understood as 0 if an argument to one of its factorials is
negative. Here $l+l'-l'' \geq l+l'-l''-1 \geq 0$ always, because $l \geq l''+1$.
\begin{itemize}
\item If $m=l$, then the second term is 0. Also
$l''+m''-1=l''+(m+m')-1=l''+l-l' \geq 0$ and $l'+l''-l-1=(m-m''+1)+l''-l-1
=l''-m'' \geq 0$, so the first term is non-zero.
\item If $m<l$ but $m''=-l''$ or $l=l'+l''$, then the first term is 0, but the
second term is non-zero.
\end{itemize}
It remains to consider $m<l$, $m'' \in \{-l''+1,l''-1,l''\}$,
and $l<l'+l''$. Then both terms are non-zero, and their sum is
\[\tfrac{1}{(l+l'-l'')!(l-m)!(l''+m'')!(l'+l''-l)!}
\left((l''+m'')(l'+l''-l)-(l+l'-l'')(l-m)\right).\]
We now consider the three cases of $m''$:
\begin{itemize}
\item $m''=l''$ is not possible, because this would imply $m=l'+l''-1
\geq l$, contradicting $m<l$.
\item If $m''=l''-1$, then $m=l'+l''-2 \geq l-1$. Hence we must have the
equalities $m=l-1$ and $l'+l''-1=l$. Then
\[(l''+m'')(l'+l''-l)-(l+l'-l'')(l-m)
=(2l''-1)-(2l'-1),\]
which is non-zero because our third given condition implies $l' \neq l''$ when
$|m|=l-1$, $|m'|=l'-1$, and $|m''|=l''-1$.
\item If $m''=-l''+1$, then
\[(l''+m'')(l'+l''-l)-(l+l'-l'')(l-m)
=(l'+l''-l)-(l+l'-l'')(l-m).\]
This is non-zero because $l-m \geq 1$ and $l+l'-l''>l'+l''-l$ strictly.
\end{itemize}
Thus we obtain that (\ref{eq:cg-def}) is non-zero in all cases.
\end{proof}

\subsection{Spherical registration} \label{sec:sr}

\subsubsection{Function basis}
We review the real spherical harmonics basis that we use for this
example, and the action of $\SO(3)$ on the coefficients in this basis that is
induced by rotation of the function domain $\cS^2$. The setup is the same as
that of the spherical registration model 
discussed in \cite[Sections 5.4 and A.1]{bandeira2017estimation}.

We define the real spherical harmonics basis
$\{h_{lm}:l \geq 0,\,m \in \{-l,\ldots,l\}\}$
from the complex spherical harmonics basis (\ref{eq:complexharmonics}) by
\[h_{l,m}=\begin{cases}
\frac{1}{\sqrt{2}}\big(y_{l,-m}+(-1)^m y_{l,m}\big) & \text{ if } m>0 \\
y_{l,0} & \text{ if } m=0 \\
\frac{\i}{\sqrt{2}}\big(y_{l,m}-(-1)^m y_{l,-m}\big) & \text{ if } m<0.
\end{cases}\]
It may be checked from (\ref{eq:conjugationsymmetry})
that these functions $\{h_{lm}\}$ are real-valued and form an orthonormal basis
for $L_2(\cS^2,\R)$.
For any function $f \in L_2(\sS^2,\C)$, writing its orthogonal
decompositions in the bases $\{y_{lm}\}$ and $\{h_{lm}\}$ as
\[f=\sum_{l=0}^\infty \sum_{m=-l}^l u_m^{(l)} y_{lm}
=\sum_{l=0}^\infty \sum_{m=-l}^l \theta_m^{(l)} h_{lm},\]
its real and complex spherical harmonic coefficients
$\{u_m^{(l)}\}$ and $\{\theta_m^{(l)}\}$ are then related by
\begin{equation} \label{eq:u-theta}
u_m^{(l)}=\begin{cases} \frac{(-1)^m}{\sqrt{2}}(\theta_{|m|}^{(l)}-\i
\theta_{-|m|}^{(l)}) & \text{ if } m>0 \\
\theta_0^{(l)} & \text{ if } m=0 \\
\frac{1}{\sqrt{2}}(\theta_{|m|}^{(l)}+\i\theta_{-|m|}^{(l)}) & \text{ if } m<0.
\end{cases}
\end{equation}
Up to the finite bandlimit $L \geq 1$,
this relation (\ref{eq:u-theta}) is a linear map $u=V^*\theta$ between
$u \in \C^d$ and $\theta \in \C^d$,
where $V \in \C^{d \times d}$ is a unitary matrix.
If $f$ is real-valued, then $\{\theta_m^{(l)}\}$ are real, and hence
$\{u_m^{(l)}\}$ satisfy the sign symmetry
\begin{equation}\label{eq:usymmetryS2}
u_m^{(l)}=(-1)^m\overline{u_{-m}^{(l)}}.
\end{equation}

The space of bandlimited functions (\ref{eq:bandlimitedspherical})
is closed under the action of $\SO(3)$, and the rotation $f \mapsto f_\frakg$
is represented by the following subgroup $\G \subset \O(d)$ acting on
$\theta \in \R^d$.

\begin{lemma} \label{lem:so3-sphere-act}
The action of $\SO(3)$ on the real spherical harmonic coefficients $\theta \in \R^d$ 
admits the representation
\[\G=\Big\{V \cdot D(\frakg) \cdot V^*:\frakg \in \SO(3)\Big\} \subset \O(d),\]
where $V \in \C^{d \times d}$ is the unitary transform describing the map
$u=V^*\theta$ in (\ref{eq:u-theta}), and $D(\frakg)$ is the block-diagonal
matrix
\begin{equation}\label{eq:Dsphericalreg}
D(\frakg)=\bigoplus_{l=0}^L D^{(l)}(\frakg) \in \C^{d \times d}
\end{equation}
with diagonal blocks $D^{(l)}(\frakg) \in \C^{(2l + 1) \times (2l + 1)}$ given
by the complex Wigner D-matrices at frequencies $l=0,\ldots,L$
(defined in Appendix \ref{sec:rc-harmonic}).
\end{lemma}
\begin{proof}
As described in Appendix \ref{appendix:wignerd}, the rotation by
$\frakg \in \SO(3)$ acts on the complex spherical harmonic coefficients
$u \in \C^d$ by
\[u \mapsto D(\frakg)u,\]
where $D(\frakg)$ is the block-diagonal matrix defined in
(\ref{eq:Dsphericalreg}).
Since $u$ and $\theta$ are related by the unitary transformations
$u=V^*\theta$ and $\theta=Vu$, the action on $\theta$ is then given by $\theta
\mapsto V \cdot D(\frakg) \cdot V^*\theta$.
\end{proof}

\subsubsection{Terms of the high-noise series expansion}

We prove Theorem \ref{thm:S2registration-mom} on the forms of $s_1(\theta)$,
$s_2(\theta)$, and $s_3(\theta)$.

\begin{proof}[Proof of Theorem \ref{thm:S2registration-mom}]
Recall by Lemma \ref{lem:skform} that
\begin{equation}\label{eq:Skgeneral}
s_k(\theta) = \frac{1}{2(k!)}\E_g\left[\langle \theta,g\cdot\theta \rangle^k
-2\langle \theta,g\cdot\theta_*\rangle^k+\langle \theta_*,g\cdot\theta_*\rangle^k\right].
\end{equation} 
Consider two different real spherical harmonic coefficient vectors
$\theta,\vartheta \in \R^d$, and the corresponding complex coefficients
$u=V^*\theta \in \C^d$ and $v=V^*\vartheta \in \C^d$.
We compute $\E_g[\langle \theta,g\cdot\vartheta \rangle^k]$ for $k=1,2,3$.\\

\noindent {\bf Case $k=1$:}
By Lemma \ref{lem:so3-sphere-act}, for any $g\in\G$,
\[\langle \theta,g\cdot\vartheta \rangle
=\langle \theta,V D(\frakg)  V^*\vartheta \rangle=\langle u,D(\frakg)v \rangle.\]
From the block-diagonal form for $D(\frakg)$ in (\ref{eq:Dsphericalreg}), we
obtain
\begin{equation}\label{eq:S2ugv}
\langle \theta,g\cdot\vartheta \rangle
=\sum_{l=0}^L  \langle u^{(l)},D^{(l)}(\frakg)v^{(l)} \rangle
=\sum_{l=0}^L 
\sum_{q,m=-l}^l \overline{u_q^{(l)}} D^{(l)}_{qm}(\frakg)v_m^{(l)}.
\end{equation}
Applying the identities (\ref{eq:ED}) yields
\[\E_g[\langle \theta,g\cdot\vartheta \rangle]
=\overline{u_0^{(0)}}v_0^{(0)}.\]
Write the shorthands $u^{(0)}=u_0^{(0)}$ and
$v^{(0)}=v_0^{(0)}$, and recall from (\ref{eq:usymmetryS2}) that
$u^{(0)},v^{(0)}$ are real. Then applying this to (\ref{eq:Skgeneral}),
\begin{align*}
s_1(\theta)&=\frac{1}{2}\E_g[\langle \theta,g\cdot\theta \rangle]
-\E_g[\langle \theta_*,g\cdot\theta \rangle]+\frac{1}{2}\E_g[\langle \theta_*,g\cdot\theta_* \rangle]
=\frac{1}{2}\Big(u^{(0)}(\theta)-u^{(0)}(\theta_*)\Big)^2.
\end{align*}

\noindent {\bf Case $k=2$:}
We take the expected square on both sides of (\ref{eq:S2ugv}),
applying (\ref{eq:usymmetryS2}) and the relation (\ref{eq:WignerDorthog}).
Then
\begin{align*}
\E_g[\langle \theta,g\cdot\vartheta \rangle^2]
&=\sum_{l=0}^L  \sum_{q,m=-l}^l \frac{(-1)^{m+q}}{2l+1}
\overline{u_q^{(l)}u_{-q}^{(l)}}v_m^{(l)}v_{-m}^{(l)}\\
&=\sum_{l=0}^L  \sum_{q,m=-l}^l \frac{1}{2l+1}
\overline{u_q^{(l)}}u_q^{(l)}v_m^{(l)}\overline{v_m^{(l)}}
=\sum_{l=0}^L \frac{1}{2l+1} \|u^{(l)}\|^2 \cdot \|v^{(l)}\|^2.
\end{align*}
Applying this to (\ref{eq:Skgeneral}),
\begin{align*}
s_2(\theta)
&=\frac{1}{4}\E_g[\langle \theta,g\cdot\theta \rangle^2]
-\frac{1}{2}\E_g[\langle \theta_*,g\cdot\theta \rangle^2]
+\frac{1}{4}\E_g[\langle \theta_*,g\cdot\theta_* \rangle^2]\\
&=\sum_{l=0}^L \frac{1}{4(2l+1)}
\left(\|u^{(l)}(\theta)\|^2-\|u^{(l)}(\theta_*)\|^2
\right)^2.
\end{align*}

\noindent {\bf Case $k=3$:} We now take the expected cube on both sides of (\ref{eq:S2ugv}), applying the relation (\ref{eq:WignerDtriple}).
Recall the convention (\ref{eq:CGconvention}). Then
\begin{align*}
&\E_g[\langle \theta,g\cdot\vartheta\rangle^3]\\
&=\mathop{\sum_{l,l',l''=0}^L}_{|l-l'| \leq l'' \leq l+l'}
\sum_{q,m=-l}^l\sum_{q',m'=-l'}^{l'}
\frac{(-1)^{m+m'+q+q'}}{2l''+1}
\cdot C_{q,q',q+q'}^{l,l',l''}
C_{m,m',m+m'}^{l,l',l''} \overline{u_q^{(l)}u_{q'}^{(l')}
u_{-q-q'}^{(l'')}}v_m^{(l)}v_{m'}^{(l')}v_{-m-m'}^{(l'')}\\
&=\mathop{\sum_{l,l',l''=0}^L}_{|l-l'| \leq l'' \leq l+l'}
\sum_{q,m=-l}^l \sum_{q',m'=-l'}^{l'} \frac{1}{2l''+1}
\cdot C_{q,q',q+q'}^{l,l',l''}
C_{m,m',m+m'}^{l,l',l''} \overline{u_q^{(l)}u_{q'}^{(l')}}
u_{q+q'}^{(l'')} v_m^{(l)}v_{m'}^{(l')}
\overline{v_{m+m'}^{(l'')}}.
\end{align*}
Recall that
\[B_{l,l',l''}(\theta)=\sum_{m=-l}^l \sum_{m'=-l'}^{l'}
C_{m,m',m+m'}^{l,l',l''} \overline{u_m^{(l)}u_{m'}^{(l')}}
u_{m+m'}^{(l'')}, \qquad u=V^*\theta.\]
Then the above may be written as
\[\E_g[\langle \theta,g\cdot\vartheta \rangle^3]
=\mathop{\sum_{l,l',l''=0}^L}_{|l-l'| \leq l'' \leq l+l'}\frac{1}{2l''+1}
B_{l,l',l''}(\theta)\overline{B_{l,l',l''}(\vartheta)}.\]
Changing indices $(m,m') \mapsto
(-m,-m')$ and applying the symmetries (\ref{eq:CGsymmetry2}) and
(\ref{eq:usymmetryS2}), we have
\begin{align*}
B_{l,l',l''}(\theta)&=\sum_{m=-l}^l \sum_{m'=-l'}^{l'}
C_{-m,-m',-m-m'}^{l,l',l''} \overline{u_{-m}^{(l)}u_{-m'}^{(l')}}
u_{-m-m'}^{(l'')}\\
&=\sum_{m=-l}^l \sum_{m'=-l'}^{l'}
(-1)^{l+l'+l''}C_{m,m',m+m'}^{l,l',l''}
 \cdot (-1)^{m+m'+(m+m')} u_m^{(l)}u_{m'}^{(l')}
\overline{u_{m+m'}^{(l'')}}\\
&=(-1)^{l+l'+l''}\overline{B_{l,l',l''}(\theta)}.
\end{align*}
Thus $B_{l,l',l''}(\theta)$ is real-valued if $l+l'+l''$ is even and pure
imaginary if $l+l'+l''$ is odd. Applying this to (\ref{eq:Skgeneral}),
\begin{align*}
s_3(\theta)&=\frac{1}{12}\E_g[\langle \theta,g\cdot\theta \rangle^3]
-\frac{1}{6}\E_g[\langle \theta_*,g\cdot\theta \rangle^3]+\frac{1}{12}\E_g[\langle \theta_*,g\cdot\theta_* \rangle^3]\\
&=\frac{1}{12}\mathop{\sum_{l,l',l''=0}^L}_{|l-l'| \leq l'' \leq l+l'}
\frac{1}{2l''+1}\left(B_{l,l',l''}(\theta)-B_{l,l',l''}(\theta_*)\right)\left(\overline{B_{l,l',l''}(\theta)}-\overline{B_{l,l',l''}(\theta_*)}\right)\nonumber\\
&=\frac{1}{12}\mathop{\sum_{l,l',l''=0}^L}_{|l-l'| \leq l'' \leq l+l'}
\frac{1}{2l''+1}\Big|B_{l,l',l''}(\theta)-B_{l,l',l''}(\theta_*)\Big|^2. \qedhere
\end{align*}
\end{proof}

\subsubsection{Transcendence degrees} We now prove Theorem
\ref{thm:S2registration} on the sequences of transcendence degrees.

\begin{proof}[Proof of Theorem \ref{thm:S2registration}]

Recall the form of $\G$ in Lemma \ref{lem:so3-sphere-act}.
Denote the diagonal blocks of $g \in \G$ by
$g^{(l)}(\frakg)=V^{(l)} \cdot D^{(l)}(\frakg) \cdot {V^{(l)}}^*$, where
each $V^{(l)} \in \C^{(2l+1) \times (2l+1)}$ represents the unitary map
$u^{(l)} \mapsto \theta^{(l)}$, and $g^{(l)}$ is
the irreducible representation of $\SO(3)$ acting on the subvector
$\theta^{(l)} \in \R^{2l+1}$. For $l=3$, it is known that any generic
point $\theta^{(3)} \in \R^7$ has a trivial stabilizer subgroup $\{\Id\}$
and a 3-dimensional orbit under this action \cite[Proposition 1]{bryant2016second}.
Defining the pre-image $\mathsf{H}:=\{\frakg \in \SO(3):g^{(3)}(\frakg)=\Id\}$,
for $L \geq 3$ and any extension of $\theta^{(3)}$ to $\theta \in \R^d$,
the stabilizer of $\theta$ in $\G$ must then satisfy $\G_\theta \subseteq
\{V \cdot D(\frakg) \cdot V^*:\frakg \in \mathsf{H}\}$. Here $\mathsf{H}$ is a
discrete subgroup of $\SO(3)$, so $\G_\theta$ is a discrete subgroup of
$\G$, and hence $\dim(\G_\theta)=0$ and $d_0=\max_\theta \dim(\orbit_\theta)
=\dim(\G)-\min_\theta \dim(\G_\theta)=3$.

We compute $\trdeg(\cR_{\leq k}^\G)$ for $k=1,2,3$ using Lemma \ref{lem:trdeg}.
Recall the forms of $s_1(\theta)$ and $s_2(\theta)$ in
Theorem \ref{thm:S2registration-mom}, where
$u^{(0)}(\theta)=\theta_0^{(0)}$ and $\|u^{(l)}(\theta)\|^2
=\|\theta^{(l)}\|^2$ for $\theta^{(l)}=(\theta_m^{(l)}:m=-l,\ldots,l)$.
Then we obtain directly that for generic $\theta_* \in \R^d$,
\begin{align*}
\trdeg(\cR_{\leq 1}^\G)&=\rank\big(\nabla^2 s_1(\theta_*)\big)=1\\
\trdeg(\cR_{\leq 2}^\G)&=\rank\big(\nabla^2 s_1(\theta_*)
+\nabla^2 s_2(\theta_*)\big)=L+1.
\end{align*}
It remains to show $\trdeg(\cR_{\leq 3}^\G)=d-3$. Note that $\trdeg(\cR_{\leq
3}^\G)\leq \trdeg(\cR^\G)=d-3$, so it suffices to show the lower bound $\trdeg(\cR_{\leq 3}^\G)\geq d-3$.

By Lemma \ref{lem:trdeg} and the fact that each Hessian
$\nabla^2 s_k(\theta_*)$ is positive semidefinite,
\begin{align*}
\trdeg(\cR_{\leq 3}^\G)\geq \rank\big(\nabla^2 s_3(\theta_*)\big).
\end{align*}
Writing the index set 
\begin{align}\label{eq:indexJ}
\cJ=\Big\{(l,l',l''):0 \leq l,l',l'' \leq L,\;
|l-l'| \leq l'' \leq l+l'\Big\},
\end{align}
we have by Theorem \ref{thm:S2registration-mom}
\begin{align*}
s_3(\theta)=\frac{1}{12} \sum_{(l,l',l'')\in\cJ} 
\frac{1}{2l''+1}\big(B_{l,l',l''}(\theta)-B_{l,l',l''}(\theta_*)\big)\big(\overline{B_{l,l',l''}(\theta)}-\overline{B_{l,l',l''}(\theta_*)}\big).
\end{align*}
Let us denote 
\begin{align*}
B(\theta)=\Big(B_{l,l',l''}(\theta):
\;(l,l',l'') \in \cJ\Big), \qquad B:\R^d \to \C^{|\cJ|},
\end{align*}
and write $\der B(\theta) \in \C^{|\cJ| \times d}$
for its the derivative in $\theta$. Then, applying the
chain rule to differentiate $s_3(\theta)$ twice at $\theta=\theta_*$,
we obtain
\[\nabla^2 s_3(\theta_*)=\der B(\theta_*)^\top \cdot
\diag\left(\frac{1}{6(2l''+1)}:(l,l',l'')
\in \cJ\right) \cdot \overline{\der B(\theta_*)}.\]
The diagonal matrix in the middle has full rank, so
\begin{equation}\label{eq:s3B1}
\rank\Big(\nabla^2 s_3(\theta_*)\Big)
=\rank\Big(\der B(\theta_*)\Big).
\end{equation}

To analyze this rank, recall the complex parametrization
$u=V^* \theta \in \C^d$ from (\ref{eq:u-theta}), satisfying the
symmetry (\ref{eq:usymmetryS2}). Let us write the real and imaginary parts of
$u$ as
\[u_m^{(l)}=v_m^{(l)}+\i w_m^{(l)}\]
so that this symmetry (\ref{eq:usymmetryS2}) is equivalent to
\begin{equation}\label{eq:vwsymmetryS2}
v_{-m}^{(l)}=(-1)^{m}v_m^{(l)}, \qquad
w_{-m}^{(l)}=(-1)^{m+1}w_m^{(l)}.
\end{equation}
Note that for $m=0$, this implies $w_0^{(l)}=0$. Then, setting
\[\eta^{(l)}(\theta)=(v_0^{(l)},v_1^{(l)},w_1^{(l)},v_2^{(l)},w_2^{(l)},
\ldots,v_l^{(l)},w_l^{(l)}),\]
these coordinates $\eta^{(l)} \in \R^{2l+1}$ provide a (linear)
invertible reparametrization of $\theta^{(l)}$. This defines a
reparametrization
\[\eta(\theta)=
\Big(\eta^{(l)}(\theta):0 \leq l \leq L\Big) \in \R^d\]
with inverse function $\theta(\eta)$. Writing as shorthand
$\eta_*=\eta(\theta_*)$ and $B(\eta) \equiv B(\theta(\eta))$,
and denoting by $\der_\eta B(\eta)$ the derivative of $B$
in the new variables $\eta$, \eqref{eq:s3B1} is equivalent to
\[\rank\Big(\nabla^2 s_3(\theta_*)\Big)=\rank\Big(\der_\eta B(\eta_*)\Big).\]

Denote
\begin{align*}
\tilde{B}(\eta)=\Big(B_{l,l',l''}(\eta):(l,l',l'')\in \cJ,\;\max(l,l',l'') \leq
10 \Big).
\end{align*}
Let us group the columns and rows of $\der_\eta B$ into blocks indexed by 
$\{\sim,11,12,\ldots,L\}$ as follows: The column block $\sim$ corresponds to
$\der_{\eta^{(0)},\ldots,\eta^{(10)}}$. The row block $\sim$ corresponds
to $\tilde{B}(\eta)$. For $l \geq 11$, the column block $l$ corresponds to
$\der_{\eta^{(l)}}$, and the row block $l$ corresponds to
$B^{(l)}$ as defined below in Lemma \ref{lem:S2rankincrease}. (These blocks 
$\tilde{B}$ and $B^{(l)}$ for $l\geq 11$ are disjoint by definition, and we may
discard the remaining rows of $\der_\eta B$ not corresponding to any such block
to produce a lower bound for its rank.)
Ordering the blocks by $\sim,11,12,\ldots,L$,
the resulting matrix $\der_\eta B$ is block lower-triangular,
because each $B^{(l)}$ does not depend on the variables
$\eta^{(l+1)},\ldots,\eta^{(L)}$. Thus $\rank(\der_\eta B)$
is lower-bounded by the sum of ranks of all diagonal blocks, i.e.
\[\rank(\der_\eta B(\eta_*)) \geq
\rank\left(\der_{\eta^{(0)},\ldots,\eta^{(10)}}\tilde{B}(\eta_*)\right)
+\sum_{l=11}^L \rank(\der_{\eta^{(l)}} B^{(l)}(\eta_*)).\]

A direct numerical evaluation of the matrix
$\der_{\eta^{(0)},\ldots,\eta^{(10)}} \tilde{B}(\eta_*)$
verifies that for $\eta_* \in \R^d$ with all entries
of $\eta_*^{(0)},\ldots,\eta_*^{(10)}$ equal to 1, we have\footnote{An
equivalent statement was verified in \cite[Theorem 5.5]{bandeira2017estimation}
corresponding to the case $F=10$, using exact-precision numerical arithmetic.}
\[\rank\left(\der_{\eta^{(0)},\ldots,\eta^{(10)}} \tilde{B}(\eta_*)\right)
=\sum_{l=0}^{10}(2l+1)-3=118.\]
Then also for generic $\eta_* \in \R^d$, by Fact \ref{fact:fullrank},
\[\rank\left(\der_{\eta^{(0)},\ldots,\eta^{(10)}}\tilde{B}(\eta_*)\right)
\geq \sum_{l=0}^{10}(2l+1)-3=118.\]
In particular, this establishes the desired result that
$\rank(\nabla^2 s_3(\theta_*))=\rank(\der_\eta B(\eta_*)) \geq d-3$ for
$L=10$. If $L \geq 11$, then by Lemma \ref{lem:S2rankincrease} below,
we also have for generic $\eta_*\in\R^d$,
\[\sum_{l=11}^L \rank(\der_{\eta^{(l)}} B^{(l)}(\eta_*))=\sum_{l=11}^L(2l+1).\]
Combining the above,
\[\rank(\nabla^2 s_3(\theta_*))=\rank(\der_\eta B(\eta_*)) \geq \sum_{l=0}^{10}(2l+1)-3+\sum_{l=11}^L(2l+1)=d-3,\]
which completes the proof that $\trdeg(\cR_{\leq 3}^\G)=d-3$.
\end{proof}

\begin{lemma}\label{lem:S2rankincrease}
Suppose $L\geq 11$. For each $l\in\{11,\ldots,L\}$, let $\cJ^{(l)}$ be the set of tuples $(l,l',l'')\in\cJ$ where $l$ takes this fixed value, and where $l'\leq l$ and $l''\leq l$. Denote
\begin{align*}
B^{(l)}(\eta)=\Big(B_{l,l',l''}(\eta):(l,l',l'')\in\cJ^{(l)}\Big) \in
\C^{|\cJ^{(l)}|}
\end{align*}
and let $\der_{\eta^{(l)}}B^{(l)}\in \C^{|\cJ^{(l)}|\times (2l+1)}$ be the
submatrix of $\der_\eta B$ corresponding to the derivative of $B^{(l)}$ in
$\eta^{(l)}$. Then for all generic $\eta_*\in\R^d$,
\begin{align*}
\rank\Big(\der_{\eta^{(l)}}B^{(l)}(\eta_*)\Big)=2l+1.
\end{align*}
\end{lemma}
\begin{proof}
By Fact \ref{fact:fullrank},
it suffices to show $\rank(\der_{\eta^{(l)}} B^{(l)}(\eta_*))=2l+1$ for a
single point $\eta_* \in \R^d$. Our strategy is to choose $\eta_*$ with many
coordinates equal to 0, such that $\der_{\eta^{(l)}} B^{(l)}(\eta_*)$ has a 
sparse structure and its rank may be explicitly analyzed.
Specifically, we choose $\eta_*$ so that
\begin{align}
&\text{For } l' \in \{l-1,l\}:v_{*,m'}^{(l')},w_{*,m'}^{(l')}=0
\text{ unless } m'=l-1\notag\\
&\text{For } l' \in \{0,1,4,5,\ldots,l-2\}:
v_{*,m'}^{(l')},w_{*,m'}^{(l')}=0 \text{ unless } m'=l'\label{eq:specialeta}\\
&\text{For } l' \in \{2,3\}:
v_{*,m'}^{(l')},w_{*,m'}^{(l')}=0 \text{ unless } m' \in \{l',l'-1\}.\notag
\end{align}
We choose the values of the non-zero coordinates of $\eta_*$ to be generic.
The rest of this proof checks that 
$\rank(\der_{\eta^{(l)}} B^{(l)}(\eta_*))=2l+1$ holds under this choice.

Recall the form of $B_{l,l',l''}$ from (\ref{eq:Bl}).
We first compute $\der_{\eta^{(l)}} B_{l,l',l''}$ in the two cases:
(i) $l',l''<l$ and (ii) $l'=l$ and $l''<l$.\\

{\bf Case I:} $l',l'' < l$. For each $k=0,\ldots,l$,
the derivatives $\partial_{v_k^{(l)}},\partial_{w_k^{(l)}}$ apply only to the 
terms $\overline{u_m^{(l)}}$ in (\ref{eq:Bl}) for $m \in \{+k,-k\}$.
We have $u_m^{(l)}=v_m^{(l)}+\i w_m^{(l)}=(-1)^m v_{-m}^{(l)}-\i \cdot
(-1)^m w_{-m}^{(l)}$ where the second equality applies the sign
symmetry (\ref{eq:vwsymmetry}). Thus (relabeling $k$ by $m$), for $m>0$
strictly,
\begin{align*}
\partial_{v_m^{(l)}} B_{l,l',l''}
&=\sum_{m'=-l'}^{l'} C_{m,m',m+m'}^{l,l',l''}
\overline{u_{m'}^{(l')}}u_{m+m'}^{(l'')}+(-1)^{m}
C_{-m,m',-m+m'}^{l,l',l''}
\overline{u_{m'}^{(l')}}u_{-m+m'}^{(l'')}\\
\partial_{w_m^{(l)}} B_{l,l',l''}
&=\sum_{m'=-l'}^{l'} -\i \cdot C_{m,m',m+m'}^{l,l',l''}
\overline{u_{m'}^{(l')}}u_{m+m'}^{(l'')}+(-1)^{m}\cdot\i
\cdot C_{-m,m',-m+m'}^{l,l',l''}\overline{u_{m'}^{(l')}}u_{-m+m'}^{(l'')}.
\end{align*}
Re-indexing $m' \mapsto -m'$ for the summations of the second terms, and
applying the symmetries (\ref{eq:CGsymmetry2}) and (\ref{eq:usymmetryS2}), we
obtain
\begin{align}
\partial_{v_m^{(l)}} B_{l,l',l''}&=\sum_{m'=-l'}^{l'}
2C_{m,m',m+m'}^{l,l',l''} \times \begin{cases} \Re \overline{u_{m'}^{(l')}}u_{m+m'}^{(l'')} & \text{ if } l+l'+l'' \text{ is even}\\
\i\cdot \Im \overline{u_{m'}^{(l')}}u_{m+m'}^{(l'')} & \text{ if } l+l'+l''
\text{ is odd,} \end{cases}\label{eq:dervmlI}\\
\partial_{w_m^{(l)}} B_{l,l',l''}&=\sum_{m'=-l'}^{l'}
2C_{m,m',m+m'}^{l,l',l''} \times \begin{cases} \Im \overline{u_{m'}^{(l')}}u_{m+m'}^{(l'')} & \text{ if } l+l'+l'' \text{ is even}\\
(-\i)\cdot \Re \overline{u_{m'}^{(l')}}u_{m+m'}^{(l'')} & \text{ if } l+l'+l''
\text{ is odd.} \end{cases}\label{eq:derwmlI}
\end{align}
For $m=0$, we have similarly
\begin{equation}\label{eq:derv0lI}
\partial_{v_0^{(l)}} B_{l,l',l''}=\sum_{m'=-l'}^{l'} C_{0,m',m'}^{l,l',l''}
\overline{u_{m'}^{(l')}}u_{m'}^{(l'')}.
\end{equation}

{\bf Case II:} $l'=l$ and $l''<l$. An additional contribution to
each derivative $\partial_{v_k^{(l)}},\partial_{w_k^{(l)}}$ arises from
differentiating $u_{m'}^{(l')}$ in (\ref{eq:Bl}). By symmetry of (\ref{eq:Bl})
with respect to interchanging $l$ and $l'$, this has the
effect of doubling each of the above expressions. So for $m>0$, we have
\begin{align}
\partial_{v_m^{(l)}} B_{l,l,l''}&=\sum_{m'=-l}^{l}
4C_{m,m',m+m'}^{l,l,l''} \times \begin{cases} \Re \overline{u_{m'}^{(l)}}u_{m+m'}^{(l'')} & \text{ if } l'' \text{ is even}\\
\i\cdot \Im \overline{u_{m'}^{(l)}}u_{m+m'}^{(l'')} & \text{ if } l''
\text{ is odd,} \end{cases}\label{eq:dervmlII}\\
\partial_{w_m^{(l)}} B_{l,l,l''}&=\sum_{m'=-l}^{l}
4C_{m,m',m+m'}^{l,l,l''} \times \begin{cases} \Im \overline{u_{m'}^{(l)}}u_{m+m'}^{(l'')} & \text{ if } l'' \text{ is even}\\
(-\i)\cdot \Re \overline{u_{m'}^{(l)}}u_{m+m'}^{(l'')} & \text{ if } l''
\text{ is odd.} \end{cases}\label{eq:derwmlII}
\end{align}
For $m=0$ we have
\begin{equation}\label{eq:derv0lII}
\partial_{v_0^{(l)}} B_{l,l,l''}=\sum_{m'=-l}^{l} 2C_{0,m',m'}^{l,l,l''}
\overline{u_{m'}^{(l)}}u_{m'}^{(l'')}.
\end{equation}

Now specializing these derivatives to $\eta_*$ of the form
(\ref{eq:specialeta}), we observe for example the following: If
$l' \in \{l-1,l\}$ and $l'' \in \{4,\ldots,l-2\}$,
then $\partial_{v_m^{(l)}} B_{l,l',l''},\partial_{w_m^{(l)}} B_{l,l',l''}$ are 0
unless $|m+m'|=l''$ for either $m'=l-1$ or $m'=-(l-1)$. Since $0 \leq m \leq l$,
this occurs for only the single index $m=(l-1)-l''$. Thus, only two entries
in the row $\der_{\eta^{(l)}} B_{l,l',l''}$ are non-zero, corresponding to
$\partial_{v_m^{(l)}},\partial_{w_m^{(l)}}$ for this $m$.
More generally, let us write the condition (\ref{eq:specialeta}) succinctly as
\[\Type(0),\Type(1),\Type(4),\ldots,\Type(l-2),\Type(l-1)=0,\quad
\Type(l)=1,\quad \Type(2,3)=\{0,1\}\]
where $\Type(l')=i$ means that $v^{(l')}_{*,m'},w^{(l')}_{*,m'}=0$ except for
$m' \in l'-i$. Then for $\partial_{v_m^{(l)}}B_{l,l',l''},\partial_{w_m^{(l)}}
B_{l,l',l''}$ to be non-zero, we require
$|m+m'| \in l''-\Type(l'')$ for some $m'$ satisfying $|m'| \in l'-\Type(l')$.
This occurs for the indices
\begin{equation}\label{eq:spherregmcondition}
m \in \Big\{\big|(l'-\Type(l'))-(l''-\Type(l''))\big|,\;
(l'-\Type(l'))+(l''-\Type(l''))\Big\} \cap \{0,\ldots,l\}
\end{equation}
where we use the set notations $|A|=\{|a|:a \in A\}$,
$A-B=\{a-b:a \in A,b \in B\}$, and $A+B=\{a+b:a \in A,b \in B\}$.

For the given value of $l$, note that
\[\cJ^{(l)}=\Big\{(l,l',l''):0 \leq l' \leq l,\;0 \leq l'' \leq l,\;
l'+l'' \geq l\Big\}.\]
We label rows of $\der_{\eta^{(l)}} B^{(l)}(\eta_*)$ by the pairs
$(l',l'')$ where $(l,l',l'') \in \cJ^{(l)}$.
We now choose $2l+1$ such rows $(l',l'')$
and check that the corresponding square $(2l+1) \times (2l+1)$ submatrix is 
non-singular. These rows are indicated in the left column of the
below table. The right column displays all values of $m$ satisfying
(\ref{eq:spherregmcondition}), i.e.\ for which
$\partial_{v_m^{(l)}},\partial_{w_m^{(l)}}$ are non-zero in that row.

\begin{table}[h]
\caption{ }
\begin{longtable}{c|c}
$(l',l'')$ & Values of $m$ \\
\hline
$(l-1,l-1)$ if $l$ is even or $(l,l-1)$ if $l$ is odd & 0 \\
\hline
$(l,l-2)$ & 1 \\
$(l-1,l-2)$ & 1 \\
$(l,l-3)$ & 2 \\
$(l-1,l-3)$ & 2 \\
$\vdots$ & $\vdots$ \\
$(l,4)$ & $l-5$ \\
$(l-1,4)$ & $l-5$ \\
\hline
$(l-5,5)$ & $l-10$ and $l$ \\
$(l-4,4)$ & $l-8$ and $l$ \\
$(l,1)$ & $l$ and $l-2$ \\
$(l-1,1)$ & $l$ and $l-2$ \\
$(l,2)$ & $l$, $l-2$, and $l-3$ \\
$(l-1,2)$ & $l$, $l-2$, and $l-3$ \\
$(l,3)$ & $l-3$ and $l-4$ \\
$(l-1,3)$ & $l-3$ and $l-4$ \\
$(l-3,3)$ & $l-6$, $l-5$, $l$, and $l-1$ \\
$(l-2,2)$ & $l-4$, $l-3$, $l$, and $l-1$
\end{longtable}
\end{table}

To verify that this selected $(2l+1)\times (2l+1)$ submatrix of
$\der_{\eta^{(l)}} B^{(l)}(\eta_*)$ is non-singular, let us order its
rows in the order of the above table, and its
columns according to the ordering of variables 
\[v_0^{(l)},v_1^{(l)},w_1^{(l)},\ldots,v_{l-5}^{(l)},w_{l-5}^{(l)},v_l^{(l)},w_l^{(l)},v_{l-2}^{(l)},w_{l-2}^{(l)},v_{l-3}^{(l)},w_{l-3}^{(l)},
v_{l-4}^{(l)},w_{l-4}^{(l)},v_{l-1}^{(l)},w_{l-1}^{(l)}\]
as they appear in the right column above. Then the table implies that this
$(2l+1) \times (2l+1)$ submatrix has a block lower-triangular structure
with respect to $2l+1=1+2+\ldots+2$. So it suffices
to check that each $1\times1$ and $2 \times 2$ diagonal block is non-singular.
It is tedious but straightforward to verify this explicitly, by computing their
forms:

{\bf Block corresponding to $v_0^{(l)}$:} For even $l$, we have that
$l+(l-1)+(l-1)$ is even. Then applying
(\ref{eq:derv0lI}) and the symmetries
(\ref{eq:CGsymmetry2}) and (\ref{eq:usymmetryS2}), this $1\times 1$ matrix is
\begin{align*}
\partial_{v_0^{(l)}} B_{l,l-1,l-1}(\eta_*)=
\Big(C_{0,l-1,l-1}^{l,l-1,l-1}
+C_{0,-(l-1),-(l-1)}^{l,l-1,l-1}\Big) |u_{*,l-1}^{(l-1)}|^2
=2C_{0,l-1,l-1}^{l,l-1,l-1}|u_{*,l-1}^{(l-1)}|^2.
\end{align*}
For odd $l$, we have that $l+l+(l-1)$ is even. Then
applying instead (\ref{eq:derv0lII}), this $1 \times 1$ matrix is
\begin{align*}
\partial_{v_0^{(l)}} B_{l,l,l-1}(\eta_*)=
2C_{0,l-1,l-1}^{l,l,l-1}\overline{u^{(l)}_{*,l-1}}u^{(l-1)}_{*,l-1}
+2C_{0,-(l-1),-(l-1)}^{l,l,l-1}u^{(l)}_{*,l-1}\overline{u^{(l-1)}_{*,l-1}}
=4C_{0,l-1,l-1}^{l,l,l-1} \Re\overline{u^{(l)}_{*,l-1}}u^{(l-1)}_{*,l-1}.
\end{align*}
These two coefficients $C_{0,l-1,l-1}^{l,l-1,l-1}$ and
$C_{0,l-1,l-1}^{l,l,l-1}$ are non-zero by Lemma \ref{lemma:CGnonvanishing},
so this block is non-zero for generic values of the non-zero coordinates
$v_{*,l-1}^{(l-1)},w_{*,l-1}^{(l-1)},v_{*,l-1}^{(l)},w_{*,l-1}^{(l)}$ of
$\eta_*$.

{\bf Blocks corresponding to
$(v_1^{(l)},w_1^{(l)}),\ldots,(v_{l-2}^{(l)},w_{l-2}^{(l)})$:} The calculations
for all these blocks are similar. We demonstrate the case
$v_{l-3}^{(l)},w_{l-3}^{(l)}$: Applying
(\ref{eq:dervmlI}--\ref{eq:derwmlI}) and 
(\ref{eq:dervmlII}--\ref{eq:derwmlII}), this $2 \times 2$ block is
\begin{align*}
&\partial_{v_{l-3}^{(l)},w_{l-3}^{(l)}}
\Big(B_{l,l,2},B_{l,l-1,2}\Big)(\eta_*)\\
&=\begin{pmatrix}
4C_{l-3,-(l-1),-2}^{l,l,2} \cdot
\Re \overline{u_{*,-(l-1)}^{(l)}}u_{*,-2}^{(2)} &
4C_{l-3,-(l-1),-2}^{l,l,2} \cdot
\Im \overline{u_{*,-(l-1)}^{(l)}}u_{*,-2}^{(2)} \\
2\i C_{l-3,-(l-1),-2}^{l,l-1,2} \cdot
\Im\overline{u_{*,-(l-1)}^{(l-1)}}u_{*,-2}^{(2)} &
-2\i C_{l-3,-(l-1),-2}^{l,l-1,2} \cdot
\Re\overline{u_{*,-(l-1)}^{(l-1)}}u_{*,-2}^{(2)}
\end{pmatrix}\\
&=\begin{pmatrix}
4C_{l-3,-(l-1),-2}^{l,l,2} & 0 \\ 0 &
2\i C_{l-3,-(l-1),-2}^{l,l-1,2} \end{pmatrix}
\begin{pmatrix}
\Re \overline{u_{*,-(l-1)}^{(l)}}u_{*,-2}^{(2)} &
\Im \overline{u_{*,-(l-1)}^{(l)}}u_{*,-2}^{(2)} \\
\Im\overline{u_{*,-(l-1)}^{(l-1)}}u_{*,-2}^{(2)} &
-\Re\overline{u_{*,-(l-1)}^{(l-1)}}u_{*,-2}^{(2)}
\end{pmatrix}
\end{align*}
The coefficients $C_{l-3,-(l-1),-2}^{l,l,2}$ and
$C_{l-3,-(l-1),-2}^{l,l-1,2}$ are both non-zero by Lemma
\ref{lemma:CGnonvanishing}, so the first matrix of this product is non-singular.
It is direct to check that the determinant
of the second matrix is a non-zero polynomial of the six non-zero coordinates
$v_{*,l-1}^{(l)},w_{*,l-1}^{(l)},v_{*,l-1}^{(l-1)},w_{*,l-1}^{(l-1)},v_{*,2}^{(2)},w_{*,2}^{(2)}$
of $\eta_*$. Then for generic values of these six coordinates,
the determinant is non-zero, and this matrix is also non-singular.

{\bf Blocks corresponding to
$(v_{l-1}^{(l)},w_{l-1}^{(l)}),(v_l^{(l)},w_l^{(l)})$:}
Applying (\ref{eq:dervmlI}--\ref{eq:derwmlI}), these $2\times 2$ matrices are
\begin{align*}
&\partial_{v_{l-1}^{(l)},w_{l-1}^{(l)}} \Big(B_{l,l-3,3},B_{l,l-2,2}\Big)(\eta_*)\\
&=\begin{pmatrix}
2C_{l-1,-(l-3),2}^{l,l-3,3} \cdot
\Re \overline{u_{*,-(l-3)}^{(l-3)}}u_{*,2}^{(3)} &
2C_{l-1,-(l-3),2}^{l,l-3,3} \cdot
\Im \overline{u_{*,-(l-3)}^{(l-3)}}u_{*,2}^{(3)} \\
2C_{l-1,-(l-2),1}^{l,l-2,2} \cdot
\Re \overline{u_{*,-(l-2)}^{(l-2)}}u_{*,1}^{(2)} &
2C_{l-1,-(l-2),1}^{l,l-2,2} \cdot
\Im \overline{u_{*,-(l-2)}^{(l-2)}}u_{*,1}^{(2)} \end{pmatrix},\\
&\partial_{v_{l}^{(l)},w_{l}^{(l)}} \Big(B_{l,l-5,5},B_{l,l-4,4}\Big)(\eta_*)\\
&=\begin{pmatrix}
2C_{l,-(l-5),5}^{l,l-5,5} \cdot
\Re \overline{u_{*,-(l-5)}^{(l-5)}}u_{*,5}^{(5)} &
2C_{l,-(l-5),5}^{l,l-5,5} \cdot
\Im \overline{u_{*,-(l-5)}^{(l-5)}}u_{*,5}^{(5)} \\
2C_{l,-(l-4),4}^{l,l-4,4} \cdot
\Re \overline{u_{*,-(l-4)}^{(l-4)}}u_{*,4}^{(4)} &
2C_{l,-(l-4),4}^{l,l-4,4} \cdot
\Im \overline{u_{*,-(l-4)}^{(l-4)}}u_{*,4}^{(4)} \end{pmatrix}.
\end{align*}
The coefficients $C_{l-1,-(l-3),2}^{l,l-3,3},C_{l-1,-(l-2),1}^{l,l-2,2},
C_{l,-(l-5),5}^{l,l-5,5},C_{l,-(l-4),4}^{l,l-4,4}$ are non-zero by
Lemma \ref{lemma:CGnonvanishing}. Then the determinant of the first matrix
is a non-zero polynomial of the eight coordinates
\[v_{*,l-3}^{(l-3)},w_{*,l-3}^{(l-3)},
v_{*,l-2}^{(l-2)},w_{*,l-2}^{(l-2)},
v_{*,1}^{(2)},w_{*,1}^{(2)},v_{*,2}^{(3)},w_{*,2}^{(3)},\]
and that of the second matrix is a non-zero polynomial of the eight coordinates
\[v_{*,l-5}^{(l-5)},w_{*,l-5}^{(l-5)},
v_{*,l-4}^{(l-4)},w_{*,l-4}^{(l-4)},
v_{*,5}^{(5)},w_{*,5}^{(5)},v_{*,4}^{(4)},w_{*,4}^{(4)}.\]
(Note that these eight coordinates are distinct when $l \geq 11$.) Thus for
generic values of these coordinates, these matrices are non-singular.

Combining these cases, we have shown that each $1 \times 1$ and $2 \times 2$
diagonal block of this $(2l+1) \times (2l+1)$ submatrix is nonsingular
for generic choices of the non-zero coordinates of $\eta_*$. Then they are also
simultaneously nonsingular for generic choices of these coordinates, so in
particular there exists $\eta_* \in \R^d$ where 
$\der_{\eta^{(l)}} B^{(l)}(\eta_*)$ has full column rank $2l+1$.
Then $\der_{\eta^{(l)}} B^{(l)}(\eta_*)$ must also have full column rank $2l+1$
for all generic $\eta_* \in \R^d$, concluding the proof.
\end{proof}

\subsection{Unprojected cryo-EM} \label{sec:unprojCEM}

\subsubsection{Function basis}
We first describe in further detail the function basis and
rotational action of $\SO(3)$ on basis coefficients in the unprojected 
cryo-EM example of Section \ref{subsec:cryoEM}.
This is similar to the setup of the model with
$S$ spherical shells in \cite[Section 5.5]{bandeira2017estimation}.

For $f \in L_2(\R^3,\C)$, denote its Fourier transform
\begin{equation}\label{eq:fouriertransform}
\hat{f}(k_1,k_2,k_3)=\int_{\R^3} e^{-2\pi \i(k_1x_1+k_2x_2+k_3x_3)}
f(x_1,x_2,x_3)\der x_1\,\der x_2\,\der x_3.
\end{equation}
We reparametrize $k=(k_1,k_2,k_3) \in \R^3$ in the Fourier domain by spherical
coordinates $(\rho,\phi_1,\phi_2)$, and write with a slight abuse
of notation $\hat{f}(\rho,\phi_1,\phi_2)$ for this parametrization.

Let $\hat{j}_{lsm}$ be as defined in (\ref{eq:C3basis}), where $y_{lm}$ are the
complex spherical harmonics in (\ref{eq:complexharmonics}) and $\{z_s:s \geq
1\}$ are any functions $z_s:[0,\infty) \to \R$ satisfying the orthogonality
relation (\ref{eq:zorthogonality}). By the spherical change-of-coordinates
$\der k_1\,\der k_2\,\der k_3 =\rho^2 \sin \phi_1\,\der \rho\,\der \phi_1\,\der
\phi_2$, these functions $\{\hat{j}_{lsm}\}$ are orthonormal in $L_2(\R^3,\C)$.
Then so are their inverse Fourier transforms $\{j_{lsm}\}$, by the Parseval
relation.

Recall the space of $(L,S_0,\ldots,S_L)$-bandlimited functions
(\ref{eq:bandlimitedS2}). By linearity of the Fourier transform,
the basis representation (\ref{eq:bandlimitedS2}) is equivalent to
\begin{equation}\label{eq:bandlimitedS2Fourier}
\hat{f}=\sum_{(l,s,m) \in \cI} u_m^{(ls)} \hat{j}_{lsm}
\end{equation}
in the Fourier domain. A function $f \in L_2(\R^3,\C)$ is real-valued if and
only if
\[\hat{f}(\rho,\phi_1,\phi_2)=\overline{\hat{f}(\rho,\pi-\phi_1,\pi+\phi_2)},\]
where $(\rho,\pi-\phi_1,\pi+\phi_2)$ are the coordinates for the reflection of
$(\rho,\phi_1,\phi_2)$ about the origin. Applying $P_{lm}(x)=(-1)^{l+m}P_{lm}(-x)$ by
(\ref{eq:legendre}) and hence $y_{l,m}(\pi-\phi_1,\pi+\phi_2)
=(-1)^ly_{l,m}(\phi_1,\phi_2)=(-1)^{l+m}\overline{y_{l,-m}(\phi_1,\phi_2)}$
by (\ref{eq:complexharmonics}) and (\ref{eq:conjugationsymmetry}),
it may be checked that this condition is equivalent to the sign symmetry
\begin{equation}\label{eq:usymmetry}
u_m^{(ls)}=(-1)^{l+m}\overline{u_{-m}^{(ls)}}
\end{equation}
in the basis representations (\ref{eq:bandlimitedS2}) and
(\ref{eq:bandlimitedS2Fourier}).
We may then define a real basis $\{h_{lsm}\}$ by
\begin{equation}\label{eq:hlsm}
h_{l,s,m}=\begin{cases}
\frac{1}{\sqrt{2}}\big(j_{l,s,-m}+(-1)^{l+m}j_{l,s,m}\big) & \text{ if } m>0 \\
\i^l \cdot j_{l,s,0} & \text{ if } m=0 \\
\frac{\i}{\sqrt{2}}\big(j_{l,s,m}-(-1)^{l+m}j_{l,s,-m}\big) & \text{ if } m<0.
\end{cases}
\end{equation}
Note that by this definition, $h_{lsm}$ satisfies (\ref{eq:usymmetry}) for its
coefficients $u_m^{(ls)}$ in the basis $\{j_{lsm}\}$, and hence is real-valued.
Thus $\{h_{lsm}\}$ forms an orthonormal basis for $L_2(\R^3,\R)$.
For any $(L,S_0,\ldots,S_L)$-bandlimited
function $f \in L_2(\R^3,\R)$, writing its orthogonal decompositions
\[f=\sum_{(l,s,m) \in \cI} u_m^{(ls)} \cdot j_{lsm}=\sum_{(l,s,m) \in \cI}
\theta^{(ls)}_m \cdot h_{lsm},\]
the coefficients $\{u_m^{(ls)}\}$ and $\{\theta_m^{(ls)}\}$ are then
related by a unitary transform $u=\hat{V}^*\theta$ defined as
\begin{equation} \label{eq:hatv-trans}
u_m^{(ls)}=\begin{cases}
\frac{(-1)^{l+m}}{\sqrt{2}}(\theta_{|m|}^{(ls)}-\i\theta_{-|m|}^{(ls)}) 
& \text{ if } m>0\\
\i^l \cdot \theta_0^{(ls)} & \text{ if } m=0\\
\frac{1}{\sqrt{2}}(\theta_{|m|}^{(ls)}+\i\theta_{-|m|}^{(ls)})
& \text{ if } m<0.\end{cases}
\end{equation}
Here, the sign symmetry (\ref{eq:usymmetry}) and transform $\hat{V}$ 
are different from (\ref{eq:usymmetryS2}) and the transform $V$ defined by
(\ref{eq:u-theta}) in the example of spherical registration, because we are
modeling the Fourier transform $\hat{f}$ rather than $f$ in the spherical
harmonics basis, but we assume that $f$ rather than $\hat{f}$ is real-valued.

The rotation of $\R^3$ induces the following rotational action on the
coefficient vector $\theta \in \R^d$.

\begin{lemma} \label{lem:so3-cryo-unproj-act}
  The action of $\SO(3)$ on the space of real-valued $(L, S_0, \ldots, S_L)$-bandlimited
  functions is given by
\begin{equation}\label{eq:cryoEMG}
\G=\Big\{\hat{V} \cdot D(\frakg) \cdot \hat{V}^*:\frakg \in \SO(3)\Big\} \subset \O(d)
\end{equation}
where $\hat{V} \in \C^{d \times d}$ is the unitary transform defined in
(\ref{eq:hatv-trans}) for which $u=\hat{V}^*\theta$, and
$D(\frakg)$ is the block-diagonal matrix
\begin{equation}\label{eq:DcryoEM}
D(\frakg)=\bigoplus_{l=0}^L \bigoplus_{s=1}^{S_l} D^{(l)}(\frakg)
\end{equation}
with diagonal blocks $D^{(l)}(\frakg) \in \C^{(2l+1) \times (2l+1)}$
given by the complex Wigner D-matrices.
\end{lemma}

\begin{proof}
Note that if $f_\frakg(x)=f(\frakg^{-1} \cdot x)$, then its Fourier transform
undergoes the same rotation $\hat{f}_\frakg(k)=\hat{f}(\frakg^{-1} \cdot k)$,
by (\ref{eq:fouriertransform}). Writing
\[\hat{f}(\rho,\phi_1,\phi_2)=\sum_{l=0}^L \sum_{s=1}^{S_l} z_s(\rho) \cdot
\hat{f}_{ls}(\phi_1,\phi_2),
\qquad \hat{f}_{ls}(\phi_1,\phi_2)=\sum_{m=-l}^l u_m^{(ls)}
y_{lm}(\phi_1,\phi_2),\]
each function $\hat{f}_{ls}$ is defined on the unit sphere, and the rotation
by $\frakg$ acts separately on each such function $\hat{f}_{ls}$ via the
map $u^{(ls)} \mapsto D^{(l)}(\frakg) u^{(ls)}$ described in
Appendix \ref{appendix:wignerd}. Thus rotation by $\frakg$ induces the
transformation $u \mapsto D(\frakg)u$ on the complex coefficient vector $u$,
for $D(\frakg)$ as defined in (\ref{eq:DcryoEM}).
Applying the unitary relations $u=\hat{V}^*\theta$ and $\theta=\hat{V}u$, this
rotation then induces the transformation $\theta \mapsto \hat{V} \cdot
D(\frakg)\cdot \hat{V}^*\theta$ on the real coefficients $\theta$.
\end{proof}

\subsubsection{Terms of the high-noise series expansion}

We prove Theorem \ref{thm:cryoEM-mom} on the forms of $s_1(\theta)$,
$s_2(\theta)$, and $s_3(\theta)$.

\begin{proof}[Proof of Theorem \ref{thm:cryoEM-mom}]
Similar to the proof of Theorem \ref{thm:S2registration-mom}, consider two
different real coefficient vectors $\theta,\vartheta \in \R^d$, with
corresponding complex coefficients $u=\hat{V}^*\theta$ and
$v=\hat{V}^*\vartheta$. We compute
$\E_g[\langle \theta,g\cdot\vartheta \rangle^k]$ for $k=1,2,3$.\\

\noindent {\bf Case $k=1$:} By Lemma \ref{lem:so3-cryo-unproj-act}, for any $g\in\G$,
\[\langle \theta,g\cdot\vartheta \rangle
=\langle \theta,\hat{V}D(\frakg)\hat{V}^*\vartheta \rangle=\langle u,D(\frakg)v
\rangle.\]
From the block-diagonal form of $D(\frakg)$ in (\ref{eq:DcryoEM}), we obtain
\begin{equation}\label{eq:ugv}
\langle \theta,g\cdot\vartheta \rangle
=\sum_{l=0}^L \sum_{s=1}^{S_l} \langle u^{(ls)},D^{(l)}(\frakg)v^{(ls)} \rangle
=\sum_{l=0}^L \sum_{s=1}^{S_l}
\sum_{q,m=-l}^l \overline{u_q^{(ls)}} D^{(l)}_{qm}(\frakg)v_m^{(ls)}.
\end{equation}
Applying the identities (\ref{eq:ED}) yields
\[\E_g[\langle \theta,g\cdot\vartheta \rangle]
=\sum_{s=1}^{S_0} \overline{u_0^{(0s)}}v^{(0s)}.\]
Write as shorthand $u^{(0s)}=u_0^{(0s)}$, $v^{(0s)}=v_0^{(0s)}$,
and observe from (\ref{eq:usymmetry}) that these are real-valued.
Then applying this to (\ref{eq:Skgeneral}), we have
\begin{align*}
s_1(\theta)&=\frac{1}{2}\E_g[\langle \theta,g\cdot\theta \rangle]
-\E_g[\langle \theta_*,g\cdot\theta \rangle]+\frac{1}{2}\E_g[\langle \theta_*,g\cdot\theta_* \rangle]
=\frac{1}{2}\sum_{s=1}^{S_0} \Big(u^{(0s)}(\theta)-u^{(0s)}(\theta_*)\Big)^2.
\end{align*}

\noindent {\bf Case $k=2$:}
We take the expected square on both sides of (\ref{eq:ugv}),
applying (\ref{eq:usymmetry}) and the relation (\ref{eq:WignerDorthog}).
Then
\begin{align*}
\E_g[\langle \theta,g\cdot\vartheta \rangle^2]
&=\sum_{l=0}^L \sum_{s,s'=1}^{S_l} \sum_{q,m=-l}^l \frac{(-1)^{m+q}}{2l+1}
\overline{u_q^{(ls)}u_{-q}^{(ls')}}v_m^{(ls)}v_{-m}^{(ls')}\\
&=\sum_{l=0}^L \sum_{s,s'=1}^{S_l} \sum_{q,m=-l}^l \frac{1}{2l+1}
\overline{u_q^{(ls)}}u_q^{(ls')}v_m^{(ls)}\overline{v_m^{(ls')}}\\
&=\sum_{l=0}^L \frac{1}{2l+1} \sum_{s,s'=1}^{S_l}
\langle u^{(ls)},u^{(ls')} \rangle \cdot \overline{\langle v^{(ls)},v^{(ls')}
\rangle}.
\end{align*}
Note that from the isometry $\langle u^{(ls)},u^{(ls')} \rangle
=\langle \theta^{(ls)},\theta^{(ls')} \rangle$, the
inner-products on the last line are real. Then applying this to
(\ref{eq:Skgeneral}),
\begin{align*}
s_2(\theta)&=\frac{1}{4}\E_g[\langle \theta,g\cdot\theta \rangle^2]
-\frac{1}{2}\E_g[\langle \theta_*,g\cdot\theta \rangle^2]
+\frac{1}{4}\E_g[\langle \theta_*,g\cdot\theta_* \rangle^2]\\
&=\sum_{l=0}^L \frac{1}{4(2l+1)}\sum_{s,s'=1}^{S_l}
\left(\langle u^{(ls)}(\theta),u^{(ls')}(\theta) \rangle
-\langle u^{(ls)}(\theta_*),u^{(ls')}(\theta_*) \rangle\right)^2.
\end{align*}

\noindent {\bf Case $k=3$:}
We now take the expected cube on both sides of (\ref{eq:ugv}) and apply the
relations (\ref{eq:WignerDtriple}) and (\ref{eq:usymmetry}). Then
\begin{align*}
\E_g[\langle \theta,g\cdot\vartheta\rangle^3]
&=\mathop{\sum_{l,l',l''=0}^L}_{|l-l'| \leq l'' \leq l+l'} \sum_{s=1}^{S_l}
\sum_{s'=1}^{S_{l'}} \sum_{s''=1}^{S_{l''}}
\sum_{q,m=-l}^l \sum_{q',m'=-l'}^{l'} \frac{1}{2l''+1}\\
&\hspace{0.5in} \cdot  C_{q,q',q+q'}^{l,l',l''}
C_{m,m',m+m'}^{l,l',l''} \overline{u_q^{(ls)}u_{q'}^{(l's')}}
u_{q+q'}^{(l''s'')} v_m^{(ls)}v_{m'}^{(l's')}
\overline{v_{m+m'}^{(l''s'')}}.
\end{align*}
Recall that
\[B_{(l,s),(l',s'),(l'',s'')}(\theta)=\sum_{m=-l}^l \sum_{m'=-l'}^{l'}
C_{m,m',m+m'}^{l,l',l''} \overline{u_m^{(ls)}u_{m'}^{(l's')}}
u_{m+m'}^{(l''s'')}, \qquad u=\hat{V}^*\theta\]
with the convention (\ref{eq:CGconvention}). Changing indices $(m,m') \mapsto
(-m,-m')$ and applying the symmetries (\ref{eq:CGsymmetry2}) and
(\ref{eq:usymmetry}), we have
\begin{align*}
B_{(l,s),(l',s'),(l'',s'')}(\theta)
&=\sum_{m=-l}^l \sum_{m'=-l'}^{l'}
C_{-m,-m',-m-m'}^{l,l',l''}  \overline{u_{-m}^{(ls)}u_{-m'}^{(l's')}}
u_{-m-m'}^{(l''s'')}\\
&=\sum_{m=-l}^l \sum_{m'=-l'}^{l'}
(-1)^{l+l'+l''}C_{m,m',m+m'}^{l,l',l''} 
 \cdot (-1)^{l+l'+l''+m+m'+(m+m')} u_m^{(ls)}u_{m'}^{(l's')}
\overline{u_{m+m'}^{(l''s'')}}\\
&=\overline{B_{(l,s),(l',s'),(l'',s'')}(\theta)}.
\end{align*}
Thus (in contrast to spherical registration in Appendix \ref{sec:sr})
$B_{(l,s),(l',s'),(l'',s'')}(\theta)$ is always real-valued.
Then the above may be written as
\[\E_g[\langle \theta,g\cdot\vartheta \rangle^3]
=\mathop{\sum_{l,l',l''=0}^L}_{|l-l'| \leq l'' \leq l+l'}\frac{1}{2l''+1}
 \sum_{s=1}^{S_l}\sum_{s'=1}^{S_{l'}}\sum_{s''=1}^{S_{l''}}
B_{(l,s),(l',s'),(l'',s'')}(\theta)B_{(l,s),(l',s'),(l'',s'')}(\vartheta).\]
Then applying this to (\ref{eq:Skgeneral}),
\begin{align}
s_3(\theta)&=\frac{1}{12}\E_g[\langle \theta,g\cdot\theta \rangle^3]
-\frac{1}{6}\E_g[\langle \theta_*,g\cdot\theta \rangle^3]+\frac{1}{12}\E_g[\langle \theta_*,g\cdot\theta_* \rangle^3]\nonumber\\
&=\frac{1}{12} \mathop{\sum_{l,l',l''=0}^L}_{|l-l'| \leq l'' \leq l+l'}
\frac{1}{2l''+1}\sum_{s=1}^{S_l}\sum_{s'=1}^{S_{l'}}\sum_{s''=1}^{S_{l''}}
\left(B_{(l,s),(l',s'),(l'',s'')}(\theta)
-B_{(l,s),(l',s'),(l'',s'')}(\theta_*)\right)^2.\label{eq:S3cryoEM}
\end{align}
\end{proof}

\subsubsection{Transcendence degrees}
We now prove Theorem \ref{thm:cryoEM} on the sequences of transcendence degrees.

\begin{proof}[Proof of Theorem \ref{thm:cryoEM}]
Recall the form of $\G$ in Lemma \ref{lem:so3-cryo-unproj-act}, let
$\hat{V}^{(l)} \in \C^{(2l+1) \times (2l+1)}$
represent the unitary map $u^{(ls)} \mapsto \theta^{(ls)}$ (which has the same
form for each $s=1,\ldots,S_l$), and denote by
$g^{(l)}(\frakg)=\hat{V}^{(l)} \cdot D^{(l)}(\frakg) \cdot (\hat{V}^{(l)})^*$ 
the irreducible representation of $\SO(3)$ acting on each subvector
$\theta^{(ls)} \in \R^{2l+1}$. For ${l=1}$, note that this representation
$g^{(1)}$ is isomorphic to $\SO(3)$ acting on $\R^3$ by rotations,
so its action on the pair of subvectors
$(\theta^{(1,1)},\theta^{(1,2)}) \in \R^3
\oplus \R^3$ has the trivial stabilizer subgroup $\{\Id\}$
for generic $(\theta^{(1,1)},\theta^{(1,2)})$. This implies for any $L \geq 1$
and $S_1 \geq 2$ that, as in the
proof of Theorem \ref{thm:S2registration}, the stabilizer subgroup
$\G_\theta$ of any extension of $(\theta^{(1,1)},\theta^{(1,2)})$ to
$\theta \in \R^d$ satisfies $\dim(\G_\theta)=0$, and hence
$d_0=\dim(\G)-\min_\theta \dim(\G_\theta)=3$.

We now compute $\trdeg(\cR_{\leq k}^\G)$ for $k=1,2,3$ using Lemma
\ref{lem:trdeg}. Recall the forms of $s_1(\theta)$ and $s_2(\theta)$
in Theorem \ref{thm:cryoEM-mom}.
For $k=1$, differentiating twice at $\theta=\theta_*$ yields
\[\nabla^2 s_1(\theta_*)=\sum_{s=1}^{S_0}
\nabla u^{(0s)}(\theta_*) \nabla u^{(0s)}(\theta_*)^\top.\]
Recalling $u^{(0s)}(\theta)=\theta_0^{(0s)}$, these vectors
$\{\nabla u^{(0s)}(\theta_*)\}_{s=1}^{S_0}$ are $S_0$ different standard
basis vectors, so
\begin{align*}
\trdeg(\cR_{\leq 1}^\G)&=\rank\Big(\nabla^2 s_1(\theta_*)\Big)=S_0.
\end{align*}
For $k=2$, differentiating twice at $\theta=\theta_*$ yields
\begin{align*}
\nabla^2 s_1(\theta_*)+\nabla^2 s_2(\theta_*)&=\sum_{s=1}^{S_0}\nabla
u^{(0s)}(\theta_*) \nabla u^{(0s)}(\theta_*)^\top\\
&\hspace{0.5in}+\sum_{l=0}^L \frac{1}{2(2l + 1)}
\sum_{s, s' = 1}^{S_l}  \nabla[\langle
u^{(ls)}(\theta),u^{(ls')}(\theta)\rangle]
\nabla [\langle u^{(ls)}(\theta),u^{(ls')}(\theta)\rangle]^\top \Big|_{\theta =
\theta_*}.
\end{align*}
Defining matrices $G^0$ and $G$ with the columns
\begin{align*}
G_s^0&:=\nabla u^{(0s)}(\theta_*) \qquad \text{ for } 1 \leq s \leq S_0\\
G_{lss'}&:=\frac{1}{\sqrt{2(2l+1)}}
\nabla [\langle u^{(ls)}(\theta),u^{(ls')}(\theta)\rangle]
\Big|_{\theta=\theta_*} \qquad \text{ for } 0 \leq l \leq L,\;
1 \leq s,s' \leq S_l,
\end{align*}
this may be written as
\[\nabla^2 s_1(\theta_*)+\nabla^2 s_2(\theta_*)
=G^0(G^0)^\top+GG^\top=[G \mid G^0][G \mid G^0]^\top.\]
For generic $\theta_*$ where $\theta^{(0s)}_0 \neq 0$,
the column span of $G^0$ coincides with the span of
columns $\{G_{0ss}:s=1,\ldots,S_0\}$ of $G$. Thus
\begin{align*}
\trdeg(\cR_{\leq 2}^\G)&=\rank(\nabla^2 s_1(\theta_*)+\nabla^2 s_2(\theta_*))
=\rank([G \mid G^0])=\rank(G).
\end{align*}
Applying the isometry $\langle u^{(ls)}(\theta),u^{(ls')}(\theta) \rangle
=\langle \theta^{(ls)},\theta^{(ls')} \rangle$, Lemma \ref{lem:s2-rank} below
shows that
\begin{align*}
\rank(G)&=\sum_{l=0}^L
\begin{cases} \frac{S_l(S_l+1)}{2} & \text{ if } S_l<2l+1,\\
(2l+1)(S_l-l) & \text{ if } S_l \geq 2l+1,\end{cases}
\end{align*}
establishing the desired form for $k=2$.

For $k=3$, we have $\trdeg(\cR_{\leq 3}^\G) \leq \trdeg(\cR^\G)=d-d_0=d-3$, so it
suffices to show $\rank(\nabla^2 s_3(\theta_*))\geq d-3$.
Writing the index set
\[\cK=\Big\{((l,s),(l',s'),(l'',s'')):0 \leq l,l',l'' \leq L,\;
|l-l'| \leq l'' \leq l+l',\; 1 \leq s \leq S_l,\; 1 \leq s' \leq S_{l'},\;
1 \leq s'' \leq S_{l''}\Big\},\]
by Theorem \ref{thm:cryoEM-mom} we have
\[s_3(\theta)=\sum_{((l,s),(l',s'),(l'',s'')) \in \cK}
\frac{1}{12} \cdot\frac{1}{2l''+1}\cdot 
\big(B_{(l,s),(l',s'),(l'',s'')}(\theta)-B_{(l,s),(l',s'),(l'',s'')}(\theta_*)\big)^2\]
for the function $B_{(l,s),(l',s'),(l'',s'')}(\theta)$ defined in
(\ref{eq:Bls}). Recall from the proof of Theorem \ref{thm:cryoEM-mom} that
$B_{(l,s),(l',s'),(l'',s'')}(\theta)$ is real-valued.
Let us denote
\[B(\theta)=\Big(B_{(l,s),(l',s'),(l'',s'')}(\theta):
\;((l,s),(l',s'),(l'',s'')) \in \cK\Big), \qquad B:\R^d \to \R^{|\cK|},\]
and write
$\der B(\theta) \in \R^{|\cK| \times d}$
for its the derivative in $\theta$. Then, applying the
chain rule to differentiate $s_3(\theta)$ twice at $\theta=\theta_*$,
we obtain
\[\nabla^2 s_3(\theta_*)=\der B(\theta_*)^\top \cdot
\diag\left(\frac{1}{6(2l''+1)}:((l,s),(l',s'),(l'',s''))
\in \cK\right) \cdot \der B(\theta_*).\]
The diagonal matrix in the middle has full rank, so
\begin{equation}\label{eq:S3cryoEMranktmp}
\rank\left(\nabla^2 s_3(\theta_*)\right)
=\rank\left(\der B(\theta_*)\right).
\end{equation}

To analyze this rank, recall the complex parametrization
$u=\hat{V}^* \theta \in \C^d$ from (\ref{eq:hatv-trans}),
where $u$ satisfies the symmetry (\ref{eq:usymmetry}). Let us write the real
and imaginary parts of $u$ as
\[u_m^{(ls)}=v_m^{(ls)}+\i w_m^{(ls)}\]
so that this symmetry (\ref{eq:usymmetry}) is equivalent to
\begin{equation}\label{eq:vwsymmetry}
v_{-m}^{(ls)}=(-1)^{l+m}v_m^{(ls)}, \qquad
w_{-m}^{(ls)}=(-1)^{l+m+1}w_m^{(ls)}.
\end{equation}
Note that for $m=0$, this implies $v_0^{(ls)}=0$ when $l$ is odd and
$w_0^{(ls)}=0$ when $l$ is even. Then, setting
\begin{equation}\label{eq:etals}
\eta^{(ls)}(\theta)=\begin{cases} (v_0^{(ls)},v_1^{(ls)},w_1^{(ls)},
\ldots,v_l^{(ls)},w_l^{(ls)}) & \text { if } l \text{ is even}\\
(w_0^{(ls)},v_1^{(ls)},w_1^{(ls)},
\ldots,v_l^{(ls)},w_l^{(ls)}) & \text { if } l \text{ is odd,} \end{cases}
\end{equation}
these coordinates $\eta^{(ls)} \in \R^{2l+1}$ provide a (linear)
invertible reparametrization of $\theta^{(ls)}$. This defines a
reparametrization
\begin{equation}\label{eq:eta}
\eta(\theta)=
\Big(\eta^{(ls)}(\theta):0 \leq l \leq L,1 \leq s \leq S_l\Big) \in \R^d
\end{equation}
with inverse function $\theta(\eta)$. Writing as shorthand
$\eta_*=\eta(\theta_*)$ and $B(\eta) \equiv B(\theta(\eta))$,
and denoting by $\der_\eta B(\eta)$ the derivative of $B$
in the new variables $\eta$, (\ref{eq:S3cryoEMranktmp}) is equivalent to
\begin{equation}\label{eq:S3cryoEMrank}
\rank\Big(\nabla^2 s_3(\theta_*)\Big)=\rank(\der_\eta B(\eta_*)).
\end{equation}

Let us group the columns and rows of $\der_\eta B$ into blocks indexed by
$(l,s)$, where the $(l,s)$ column block corresponds to $\der_{\eta^{(ls)}}$
and the $(l,s)$ row block corresponds to $B^{(ls)}$ as defined below in
Lemma \ref{lemma:S3cryoEMrankincrease}. (These blocks $B^{(ls)}$ are disjoint by
definition, and we may discard the remaining rows of
$\der_\eta B$ not belonging to any such block to produce a lower bound for
its rank.)
Ordering the pairs $(l,s)$ as in 
Lemma \ref{lemma:S3cryoEMrankincrease}, the resulting matrix $\der_\eta B$ is
block lower-triangular. Thus its rank is lower-bounded by the total rank of all
blocks along the diagonal, i.e.
\[\rank(\der_\eta B(\eta_*)) \geq \sum_{l=0}^L \sum_{s=1}^{S_l}
\rank(\der_{\eta^{(ls)}} B^{(ls)}(\eta_*)).\]
Lemma \ref{lemma:S3cryoEMrankincrease} shows that for generic $\eta_* \in\R^d$, 
\begin{align*}
\sum_{l=2}^L \sum_{s=1}^{S_l}
\rank(\der_{\eta^{(ls)}} B^{(ls)}(\eta_*))
&=\sum_{l=2}^L \sum_{s=1}^{S_l} (2l+1)
=\sum_{l=2}^L (2l+1)S_l,\\
\sum_{s=1}^{S_1} \rank(\der_{\eta^{(1s)}} B^{(1s)}(\eta_*))
& \geq 1+2+\sum_{s=3}^{S_1} 3=3S_1-3\\
\sum_{s=1}^{S_0} \rank(\der_{\eta^{(0s)}} B^{(0s)}(\eta_*))&=S_0.
\end{align*}
Combining these,
\[\trdeg \cR_{\leq 3}^\G \geq \rank(\nabla^2 s_3(\theta_*))
=\rank(\der_\eta B(\eta_*)) \geq \left(\sum_{l=0}^L (2l+1)S_l\right)-3=d-3.\]
Thus $\trdeg \cR_{\leq 3}^\G=d-3$.
\end{proof}

\begin{lemma} \label{lem:s2-rank}
For any $l \geq 0$ and $S \geq 2$, consider
$\theta^{(1)},\ldots,\theta^{(S)} \in \R^{2l+1}$ and the Jacobian matrix of all
their pairwise inner-products with
respect to $\theta=(\theta^{(1)},\ldots,\theta^{(S)}) \in \R^{(2l+1)S}$,
\[\der_\theta [\langle \theta^{(s)}, \theta^{(s')}\rangle:1\leq s \leq s'\leq
S]
\in \R^{\frac{S(S+1)}{2} \times (2l+1)S}.\]
At generic $\theta_* \in \R^{(2l+1)S}$, this matrix has rank
\[
\rank\Big(\der_\theta [\langle \theta^{(s)}, \theta^{(s')}\rangle:1\leq s
\leq s'\leq S]
\Big|_{\theta=\theta_*}\Big)
= \begin{cases} \frac{S(S + 1)}{2} & S<2l + 1 \\  
(2l + 1) (S - l) & S \geq 2l + 1. \end{cases}
\]
\end{lemma}
\begin{proof}
Let $q=\min(2l+1,S)$ and consider the rows of the Jacobian for pairs $(s,s')$
given by
\[(1, 1), (1, 2), \ldots, (1, S), (2, 2), (2, 3), \ldots, (2, S), \ldots
(q, q), \ldots, (q, S).\]
It may be checked that the number of such rows is exactly the desired formula
for the rank. Consider $\theta_*$ 
where $\theta_{*,1}^{(1)} = \theta_{*,2}^{(2)} = \cdots
= \theta^{(q)}_{*,q} = 1$, and all other coordinates are $0$. For this
$\theta_*$, the entries of the Jacobian are given by
\begin{align*}
\partial_{\theta^{(p)}_{m}}[ \langle \theta^{(s)}, \theta^{(s')}\rangle]\Big|_{\theta = \theta_*}
=\1\{p = s\} \cdot\theta^{(s')}_{*,m} + \1\{p = s'\}\cdot \theta^{(s)}_{*,m}
=\1\{p = s, m=s'\} + \1\{p = s', m=s\}.
\end{align*}
Thus for each row $(s,s')$ above where $s \leq q \leq 2l+1$ and $s' \leq S$,
there are either 1 or 2 non-zero entries,
in the column $\partial_{\theta_s^{(s')}}$ and also in the column
$\partial_{\theta_{s'}^{(s)}}$ if $s' \leq 2l+1$. These columns
are distinct for different rows $(s,s')$, so the submatrix of these
columns has full row rank. This shows
\[
\rank\Big(\der_\theta [\langle \theta^{(s)}, \theta^{(s')}\rangle:1\leq s
\leq s'\leq S]
\Big|_{\theta=\theta_*}\Big)
\geq \begin{cases} \frac{S(S + 1)}{2} & S<2l + 1 \\  
(2l + 1) (S - l) & S \geq 2l + 1\end{cases}
\]
for this choice of $\theta_*$, and hence also at any generic $\theta_*$
by Fact \ref{fact:fullrank}.

For the corresponding upper bound, for $S<2l + 1$ this follows because the
Jacobian has only $S(S + 1)/2$ rows.
For $S \geq 2l + 1$, consider the action of $\SO(2l+1)$ on $\theta$ by
simultaneous rotation of the vectors $\theta^{(1)},\ldots,\theta^{(S)}$.
For $S \geq 2l+1$ and at generic $\theta_*$ where
$\theta_*^{(1)},\ldots,\theta_*^{(S)}$ span all of $\R^{2l+1}$, this action
has trivial stabilizer, so
$\dim(\orbit_{\theta_*})=\dim(\SO(2l+1))=(2l+1)l$ where $\orbit_{\theta_*}$ is
the orbit of $\theta_*$ under this action. Since
$\langle\theta^{(s)},\theta^{(s')} \rangle$ is constant on $\orbit_{\theta_*}$,
for each vector $v$ in the dimension-$(2l+1)l$ tangent space to
$\orbit_{\theta_*}$, we have
\[\der_\theta [\langle \theta^{(s)}, \theta^{(s')}\rangle:1\leq s
\leq s'\leq S]\big|_{\theta=\theta_*} \cdot v=0.\]
Thus the dimension of the row span of the Jacobian is at most $(2l+1)S-(2l+1)l=(2l+1)(S-l)$, as
desired.
\end{proof}

\begin{lemma}\label{lemma:S3cryoEMrankincrease}
Suppose $L \geq 1$ and $S_0,\ldots,S_L \geq 2$.
Order the pairs $(l,s)$ by $(l,s)<(l',s')$ if $l<l'$ or if
$l=l'$ and $s<s'$. Fix any $l \in \{1,\ldots,L\}$ and
$s \in \{1,\ldots,S_l\}$, and let $\cK^{(ls)}$ be the set of
tuples $((l,s),(l',s'),(l'',s'')) \in \cK$ where $(l,s)$ take
these fixed values, and $(l',s') \leq (l,s)$ and $(l'',s'') \leq (l,s)$.
Denote
\[B^{(ls)}(\eta)=\Big(B_{(l,s),(l',s'),(l'',s'')}(\eta):
\;((l,s),(l',s'),(l'',s'')) \in \cK^{(ls)}\Big) \in \R^{|\cK^{(ls)}|}\]
and let $\der_{\eta^{(ls)}} B^{(ls)} \in \R^{|\cK^{(ls)}| \times (2l+1)}$
be its Jacobian in $\eta^{(ls)}$. Then for any generic $\eta_* \in \R^d$:
\begin{enumerate}[(a)]
\item If $l \geq 2$, then
$\rank(\der_{\eta^{(ls)}} B^{(ls)}(\eta_*))=2l+1$.
\item If $l \geq 1$ and $s \geq 3$, then also
$\rank(\der_{\eta^{(ls)}} B^{(ls)}(\eta_*))=2l+1=3$. Furthermore for $s=2$,
removing $w_0^{(12)}$ from $\eta^{(12)}$, we have
$\rank(\der_{v_1^{(12)},w_1^{(12)}} B^{(12)}(\eta_*))=2$. For $s=1$,
removing $w_0^{(11)},w_1^{(11)}$ from $\eta^{(11)}$, we have
$\rank(\der_{v_1^{(11)}} B^{(11)}(\eta_*))=1$.
\item If $l=0$, then $\rank(\der_{\eta^{(0s)}} B^{(0s)}(\eta_*))=1$.
\end{enumerate}
\end{lemma}
\begin{proof}
The strategy is similar to the proof of Lemma \ref{lem:S2rankincrease}.
For each statement, by Fact \ref{fact:fullrank},
it suffices to exhibit a single point $\eta_* \in \R^d$
where the rank equality holds. We choose $\eta_*$ having most coordinates 0,
to allow an explicit computation of the rank.

Recall $B_{(l,s),(l',s'),(l'',s'')}$ from
(\ref{eq:Bls}). We first compute $\der_{\eta^{(ls)}}
B_{(l,s),(l',s'),(l'',s'')}$: If $(l',s'),(l'',s'')<(l,s)$ strictly, then the
derivative $\der_{\eta^{(ls)}}$ applies to only the term 
$\overline{u_m^{(ls)}}$ in (\ref{eq:Bls}). A computation analogous
to (\ref{eq:dervmlI}--\ref{eq:derwmlI}) using the sign
symmetry (\ref{eq:usymmetry}) shows, for $m>0$ strictly,
\begin{align}
\partial_{v_m^{(ls)}} B_{(l,s),(l',s'),(l'',s'')}&=\sum_{m'=-l'}^{l'}
2C_{m,m',m+m'}^{l,l',l''} \cdot
\Re \overline{u_{m'}^{(l's')}}u_{m+m'}^{(l''s'')}
\label{eq:dervmlsI}\\
\partial_{w_m^{(ls)}} B_{(l,s),(l',s'),(l'',s'')}&=\sum_{m'=-l'}^{l'}
2C_{m,m',m+m'}^{l,l',l''} \cdot
\Im \overline{u_{m'}^{(l's')}}u_{m+m'}^{(l''s'')}.
\label{eq:derwmlsI}
\end{align}
For $m=0$, recalling that $\eta_0^{(ls)}=v_0^{(ls)}$ if $l$ is even and
$\eta_0^{(ls)}=w_0^{(ls)}$ if $l$ is odd, we also have
\begin{align}
\partial_{\eta_0^{(ls)}} B_{(l,s),(l',s'),(l'',s'')}
&=\sum_{m'=-l'}^{l'} C_{0,m',m'}^{l,l',l''}
\overline{u_{m'}^{(l's')}}u_{m'}^{(l''s'')}
\times \begin{cases} 1 & \text{ if } l \text{ is even}\\
-\i & \text{ if } l \text{ is odd.} \end{cases}
\label{eq:dereta0lsI}
\end{align}

If $(l',s')=(l,s)$ and $(l'',s'')<(l,s)$ strictly, an additional
contribution to the derivatives arise from differentiating $u_{m'}^{(l's')}$.
This doubles the above expressions, and we have
\begin{align}
\partial_{v_m^{(ls)}} B_{(l,s),(l,s),(l'',s'')}&=\sum_{m'=-l}^l
4C_{m,m',m+m'}^{l,l,l''} \cdot
\Re \overline{u_{m'}^{(ls)}}u_{m+m'}^{(l''s'')}
\label{eq:dervmlsII}\\
\partial_{w_m^{(ls)}} B_{(l,s),(l,s),(l'',s'')}&=\sum_{m'=-l}^l
4C_{m,m',m+m'}^{l,l,l''} \cdot
\Im \overline{u_{m'}^{(ls)}}u_{m+m'}^{(l''s'')}
\label{eq:derwmlsII}\\
\partial_{\eta_0^{(ls)}} B_{(l,s),(l,s),(l'',s'')}
&=\sum_{m'=-l}^l 2C_{0,m',m'}^{l,l,l''}
\overline{u_{m'}^{(ls)}}u_{m'}^{(l''s'')}
\times \begin{cases} 1 & \text{ if } l \text{ is even}\\
-\i & \text{ if } l \text{ is odd.} \end{cases}
\label{eq:dereta0lsII}
\end{align}

We now use a different construction of $\eta_*$ for different values of
$(l,s)$:\\

{\bf Part (a), $l \geq 4$:} Let us fix two radial frequencies $(A,B)=(1,2)$.
(We use here the condition $S_0,\ldots,S_L \geq 2$, so that these frequencies
exist for each $l=0,\ldots,L$.) We choose $\eta_*$ such that for all $l' \in
\{1,\ldots,l-1\}$ and $m' \in \{0,\ldots,l'\}$,
\[v_{*,m'}^{(l'A)}=w_{*,m'}^{(l'A)}=0 \text{ unless } m'=l', \qquad
v_{*,m'}^{(l'B)}=w_{*,m'}^{(l'B)}=0 \text{ unless } m'=l'-1.\]
We choose the non-zero coordinates of $\eta_*$ to be generic. Using
similar notation as in Lemma \ref{lem:S2rankincrease}, we write this as
\[\Type(l', A)=0, \qquad \Type(l', B)=1 \qquad \text{ for all } \qquad
l'=1,\ldots,l-1,\]
where $\Type(l',s')=i$ indicates that $v^{(l's')}_{*,m},w^{(l's')}_{*,m}=0$
unless $m=l'-i$. (In contrast to Lemma \ref{lem:S2rankincrease}, here
$\Type(\cdot)$ is a single integer rather than a set.)
For $\partial_{v_m^{(ls)}} B_{(l,s),(l',s'),(l'',s'')},
\partial_{w_m^{(ls)}} B_{(l,s),(l',s'),(l'',s'')}$ to be non-zero,
we require $|m+m'|=l''-\Type(l'',s'')$ and $|m'|=l'-\Type(l',s')$.
This requires analogously to (\ref{eq:spherregmcondition})
\begin{equation}\label{eq:cryoEMmcondition}
m \in \Big\{\big|(l'-\Type(l',s'))-(l''-\Type(l'',s''))\big|,
(l'-\Type(l',s'))+(l''-\Type(l'',s''))\Big\}.
\end{equation}

Rows of $\der_{\eta^{(ls)}} B^{(ls)}(\eta_*)$ may be indexed by
$(l',s'),(l'',s'')$ for which $((l,s),(l',s'),(l'',s'')) \in \cK^{(ls)}$.
We select $2l+1$ such rows, given by the left column of the
following Table \ref{tab:unprojcryo}. Note that when $l \geq 4$, these rows satisfy the requirement
$l'+l'' \geq l$ in the definition of $\cK$.
For each row, the right column indicates the values of $m$
satisfying (\ref{eq:cryoEMmcondition}), for which
$\partial_{v_m^{(ls)}},\partial_{w_m^{(ls)}}$ are non-zero.

\begin{table}[h]
\caption{ }\label{tab:unprojcryo}
\begin{tabular}{c|c}
$(l',s')$ and $(l'',s'')$ & Values of $m$ \\
\hline
$(l-1,A)$ and $(l-1,A)$ if $l$ is even;
$(l-2,A)$ and $(l-1,B)$ if $l$ is odd & 0 \\
\hline
$(l-1,B)$ and $(l-1,A)$ & 1 \\
$(l-2,A)$ and $(l-1,A)$ & 1 \\
$(l-2,B)$ and $(l-1,A)$ & 2 \\
$(l-3,A)$ and $(l-1,A)$ & 2 \\
$\vdots$ & $\vdots$ \\
$(3,B)$ and $(l-1,A)$ & $l-3$ \\
$(2,A)$ and $(l-1,A)$ & $l-3$ \\
\hline
$(3,B)$ and $(l-1,B)$ & $l-4$ and $l$ \\
$(2,A)$ and $(l-1,B)$ & $l-4$ and $l$ \\
$(2,B)$ and $(l-1,B)$ & $l-3$ and $l-1$ \\
$(1,A)$ and $(l-1,B)$ & $l-3$ and $l-1$ \\
$(2,B)$ and $(l-1,A)$ & $l$ and $l-2$ \\
$(1,A)$ and $(l-1,A)$ & $l$ and $l-2$
\end{tabular}
\end{table}

To verify that this selected $(2l+1) \times (2l+1)$ submatrix of
$\der_{\eta^{(ls)}} B^{(ls)}(\eta_*)$ is non-singular, we order its
rows as in Table \ref{tab:unprojcryo}, and its columns by the ordering of variables
\[\eta_0^{(ls)},v_1^{(ls)},w_1^{(ls)},\ldots,
v_{l-3}^{(ls)},w_{l-3}^{(ls)},v_l^{(ls)},w_l^{(ls)},
v_{l-1}^{(ls)},w_{l-1}^{(ls)},v_{l-2}^{(ls)},w_{l-2}^{(ls)}\]
as they appear in the right column above. Then this
$(2l+1) \times (2l+1)$ submatrix is block lower-triangular
in the decomposition $2l+1=1+2+2+\ldots+2$. It suffices
to check that each $1 \times 1$ and $2 \times 2$ diagonal block is
non-singular. 

{\bf Block corresponding to $\eta_0^{(l)}$:} 
Applying (\ref{eq:dereta0lsI}) and the symmetries
(\ref{eq:CGsymmetry2}) and (\ref{eq:usymmetry}), the
first $1 \times 1$ diagonal block is
\begin{align*}
\partial_{\eta_0}^{(ls)} B_{(l,s),(l-1,A),(l-1,A)}(\eta_*)
&=2C_{0,l-1,l-1}^{l,l-1,l-1} \cdot \big|u_{*,l-1}^{(l-1,A)}\big|^2
\text{ for even } l\\
\partial_{\eta_0}^{(ls)} B_{(l,s),(l-2,A),(l-1,B)}(\eta_*)
&=2C_{0,l-2,l-2}^{l,l-2,l-1} \cdot
\Im \overline{u_{*,l-2}^{(l-2,A)}}u_{*,l-2}^{(l-1,B)}
\text{ for odd } l.
\end{align*}
By Lemma \ref{lemma:CGnonvanishing}, these Clebsch-Gordan coefficients are
non-zero, so this block
is generically non-zero.

{\bf Blocks corresponding to $(v_1^{(ls)},w_1^{(ls)}),\ldots,
(v_{l-3}^{(ls)},w_{l-3}^{(ls)})$:} The arguments for these blocks are similar,
so we consider only $v_{l-3}^{(ls)},w_{l-3}^{(ls)}$. Applying
(\ref{eq:dervmlsI}--\ref{eq:derwmlsI}), this $2 \times 2$ block is
\begin{align*}
&\partial_{v_{l-3}^{(ls)},w_{l-3}^{(ls)}} \Big(B_{(l,s),(3,B),(l-1,A)},
B_{(l,s),(2,A),(l-1,A)}\Big)(\eta_*)\\
&=\begin{pmatrix}
2C_{l-3,2,l-1}^{l,3,l-1} \cdot
\Re \overline{u_{*,2}^{(3,B)}}u_{*,l-1}^{(l-1,A)} &
2C_{l-3,2,l-1}^{l,3,l-1} \cdot
\Im \overline{u_{*,2}^{(3,B)}}u_{*,l-1}^{(l-1,A)} \\
2C_{l-3,2,l-1}^{l,2,l-1} \cdot
\Re \overline{u_{*,2}^{(2,A)}}u_{*,l-1}^{(l-1,A)} &
2C_{l-3,2,l-1}^{l,2,l-1} \cdot
\Im \overline{u_{*,2}^{(2,A)}}u_{*,l-1}^{(l-1,A)}
\end{pmatrix}
\end{align*}
These Clebsch-Gordan coefficients
are again non-zero by Lemma \ref{lemma:CGnonvanishing}, so the determinant of
this matrix is a non-zero polynomial in the six coefficients of $\eta_*$
\[v_{*,2}^{(3,B)},w_{*,2}^{(3,B)},v_{*,2}^{(2,A)},w_{*,2}^{(2,A)},
v_{*,l-1}^{(l-1,A)},w_{*,l-1}^{(l-1,A)}.\]
(These coefficients are distinct for $l \geq 4$.) Hence this determinant is
generically non-zero.

{\bf Blocks corresponding to $(v_l^{(ls)},w_l^{(ls)}),
(v_{l-1}^{(ls)},w_{l-1}^{(ls)}),(v_{l-2}^{(ls)},w_{l-2}^{(ls)})$:} 
Applying (\ref{eq:dervmlsI}--\ref{eq:derwmlsI}), these $2 \times 2$ blocks are
\begin{align*}
&\partial_{v_l^{(ls)},w_l^{(ls)}}
\Big(B_{(l,s),(3,B),(l-1,B)},B_{(l,s),(2,A),(l-1,B)}\Big)(\eta_*)\\
&=\begin{pmatrix}
2C_{l,-2,l-2}^{l,3,l-1} \cdot
\Re \overline{u_{*,-2}^{(3,B)}}u_{*,l-2}^{(l-1,B)} &
2C_{l,-2,l-2}^{l,3,l-1} \cdot
\Im \overline{u_{*,-2}^{(3,B)}}u_{*,l-2}^{(l-1,B)} \\
2C_{l,-2,l-2}^{l,2,l-1} \cdot
\Re \overline{u_{*,-2}^{(2,A)}}u_{*,l-2}^{(l-1,B)} &
2C_{l,-2,l-2}^{l,2,l-1} \cdot
\Im\overline{u_{*,-2}^{(2,A)}}u_{*,l-2}^{(l-1,B)}
\end{pmatrix}\\
&\partial_{v_{l-1}^{(ls)},w_{l-1}^{(ls)}}
\Big(B_{(l,s),(2,B),(l-1,B)},B_{(l,s),(1,A),(l-1,B)}\Big)(\eta_*)\\
&=\begin{pmatrix}
2C_{l-1,-1,l-2}^{l,2,l-1} \cdot
\Re\overline{u_{*,-1}^{(2,B)}}u_{*,l-2}^{(l-1,B)} &
2C_{l-1,-1,l-2}^{l,2,l-1} \cdot
\Im\overline{u_{*,-1}^{(2,B)}}u_{*,l-2}^{(l-1,B)} \\
2C_{l-1,-1,l-2}^{l,1,l-1} \cdot
\Re\overline{u_{*,-1}^{(1,A)}}u_{*,l-2}^{(l-1,B)} &
2C_{l-1,-1,l-2}^{l,1,l-1} \cdot
\Im\overline{u_{*,-1}^{(1,A)}}u_{*,l-2}^{(l-1,B)}
\end{pmatrix}\\
&\partial_{v_{l-2}^{(ls)},w_{l-2}^{(ls)}}
\Big(B_{(l,s),(2,B),(l-1,A)},B_{(l,s),(1,A),(l-1,A)}\Big)(\eta_*)\\
&=\begin{pmatrix}
2C_{l-2,1,l-1}^{l,2,l-1} \cdot
\Re\overline{u_{*,1}^{(2,B)}}u_{*,l-1}^{(l-1,A)} &
2C_{l-2,1,l-1}^{l,2,l-1} \cdot
\Im\overline{u_{*,1}^{(2,B)}}u_{*,l-1}^{(l-1,A)} \\
2C_{l-2,1,l-1}^{l,1,l-1} \cdot
\Re\overline{u_{*,1}^{(1,A)}}u_{*,l-1}^{(l-1,A)} &
2C_{l-2,1,l-1}^{l,1,l-1} \cdot
\Im\overline{u_{*,1}^{(1,A)}}u_{*,l-1}^{(l-1,A)}
\end{pmatrix}
\end{align*}
The Clebsch-Gordan coefficients here are non-zero by Lemma
\ref{lemma:CGnonvanishing} (for the term $C_{l-1,-1,l-2}^{l,2,l-1}$ of the second
matrix, this uses the condition $2 \neq l-1$ when $l \geq 4$). Then the
determinants of all three matrices are non-zero polynomials of six distinct
coordinates of $\eta_*$, except in the case of the first matrix for $l=4$. In
this case, $u^{(3,B)}_2$ and $u^{(l-1,B)}_{l-2}$ coincide, and the determinant
may be checked to be a non-zero polynomial of the four distinct coordinates
$v_{*,2}^{(3,B)},w_{*,2}^{(3,B)},v_{*,2}^{(2,A)},w_{*,2}^{(2,A)}$.

Combining these cases shows that $\der_{\eta^{(ls)}} B^{(ls)}(\eta_*)$ has full
column rank $2l+1$ as desired.\\

{\bf Part (a), $l=3$:}
We again fix $(A,B)=(1,2)$, and specialize to a point $\eta_*$ such that
\[\Type(3,A),\Type(2,A),\Type(1,A)=0, \qquad \Type(2,B),\Type(1,B)=1.\]
We then pick 7 rows of $\der_{\eta^{(3s)}} B^{(3s)}(\eta_*)$, indicated by the
left column of the below table. Applying (\ref{eq:cryoEMmcondition}),
the derivatives in $(v_m^{(3s)},w_m^{(3s)})$ are non-zero for only the values of
$m$ in the right column.

\pagebreak
\begin{table}[h]
\caption{ }
\begin{tabular}{c|c}
$(l',s')$ and $(l'',s'')$ & Values of $m$ \\
\hline
$(1,B)$ and $(2,B)$ & 1 \\
$(3,A)$ and $(2,A)$ & 1 \\
$(1,B)$ and $(2,A)$ & 2 \\
$(3,A)$ and $(1,A)$ & 2 \\
$(1,A)$ and $(2,A)$ & 1 and 3 \\
$(2,B)$ and $(2,A)$ & 1 and 3 \\
$(1,A)$ and $(2,B)$ & 2 and 0
\end{tabular}
\end{table}

Ordering the columns by
$v_1^{(3s)},w_1^{(3s)},v_2^{(3s)},w_2^{(3s)},v_3^{(3s)},w_3^{(3s)},w_0^{(3s)}$,
this $7 \times 7$ submatrix has a block lower-triangular structure
in the decomposition $7=2+2+2+1$. If $s \neq A$, then applying
(\ref{eq:dervmlsI}--\ref{eq:dereta0lsI}) and
(\ref{eq:dervmlsII}--\ref{eq:derwmlsII}), its diagonal blocks
are given explicitly by
\begin{align*}
\partial_{v_1^{(3s)},w_1^{(3s)}}
\Big(B_{(3,s),(1,B),(2,B)},B_{(3,s),(3,A),(2,A)}\Big)(\eta_*)
&=\begin{psmallmatrix}
2C_{1,0,1}^{3,1,2} \cdot \Re \overline{u_{*,0}^{(1B)}} u_{*,1}^{(2B)} &
2C_{1,0,1}^{3,1,2} \cdot \Im \overline{u_{*,0}^{(1B)}} u_{*,1}^{(2B)} \\
2C_{1,-3,-2}^{3,3,2} \cdot \Re \overline{u_{*,-3}^{(3A)}} u_{*,-2}^{(2A)} &
2C_{1,-3,-2}^{3,3,2} \cdot \Im \overline{u_{*,-3}^{(3A)}} u_{*,-2}^{(2A)}
\end{psmallmatrix}\\
\partial_{v_2^{(3s)},w_2^{(3s)}}
\Big(B_{(3,s),(1,B),(2,A)},B_{(3,s),(3,A),(1,A)}\Big)(\eta_*)
&=\begin{psmallmatrix}
2C_{2,0,2}^{3,1,2} \cdot \Re \overline{u_{*,0}^{(1B)}} u_{*,2}^{(2A)} &
2C_{2,0,2}^{3,1,2} \cdot \Im \overline{u_{*,0}^{(1B)}} u_{*,2}^{(2A)} \\
2C_{2,-3,-1}^{3,3,1} \cdot \Re \overline{u_{*,-3}^{(3A)}} u_{*,-1}^{(1A)} &
2C_{2,-3,-1}^{3,3,1} \cdot \Im \overline{u_{*,-3}^{(3A)}} u_{*,-1}^{(1A)}
\end{psmallmatrix}\\
\partial_{v_3^{(3s)},w_3^{(3s)}}
\Big(B_{(3,s),(1,A),(2,A)},B_{(3,s),(2,B),(2,A)}\Big)(\eta_*)
&=\begin{psmallmatrix}
2C_{3,-1,2}^{3,1,2} \cdot \Re \overline{u_{*,-1}^{(1A)}} u_{*,2}^{(2A)} &
2C_{3,-1,2}^{3,1,2} \cdot \Im \overline{u_{*,-1}^{(1A)}} u_{*,2}^{(2A)} \\
2C_{3,-1,2}^{3,2,2} \cdot \Re \overline{u_{*,-1}^{(2B)}} u_{*,2}^{(2A)} &
2C_{3,-1,2}^{3,2,2} \cdot \Im \overline{u_{*,-1}^{(2B)}} u_{*,2}^{(2A)}
\end{psmallmatrix}\\
\partial_{w_0^{(3s)}} B_{(3,s),(1,A),(2,B)}(\eta_*)
&=2C_{0,1,1}^{3,1,2} \cdot \Im \overline{u_{*,1}^{1A}} u_{*,1}^{2B}
\end{align*}
Here $u_{*,0}^{(1B)}=\i w_{*,0}^{(1B)}$ depends on only one rather than two
non-zero coordinate of $\eta_*$; nonetheless, one may still check that
the determinants of the above three matrices are generically non-zero.
If $s=A$, then the second rows of the first two matrices above have coefficients
4 instead of 2, from applying (\ref{eq:dervmlsII}--\ref{eq:derwmlsII}) in place 
of (\ref{eq:dervmlsI}--\ref{eq:derwmlsI}), but this does not affect their
ranks. Thus these blocks are generically non-singular,
so $\der_{\eta^{(3s)}} B^{(3s)}(\eta_*)$ has full column rank 7.\\

{\bf Part (a), $l=2$:} We specialize to a point $\eta_*$ such that
\[\Type(2,A),\Type(1,A),\Type(0,A)=0, \qquad \Type(2,B),\Type(1,B)=1.\]
(Here $\Type(0,A)=0$ means simply that $v_{*,0}^{(0A)}$ is non-zero.)
We pick the following 5 rows of $\der_{\eta^{(2s)}} B^{(2s)}(\eta_*)$, for
which the derivatives in $(v_m^{(2s)},w_m^{(2s)})$ are non-zero for only the
following corresponding values of $m$.

\begin{table}[h]
\caption{ }
\begin{tabular}{c|c}
$(l',s')$ and $(l'',s'')$ & Values of $m$ \\
\hline
$(1,B)$ and $(1,B)$ & 0 \\
$(1,B)$ and $(1,A)$ & 1 \\
$(2,B)$ and $(0,A)$ & 1 \\
$(2,A)$ and $(0,A)$ & 2 \\
$(1,A)$ and $(1,A)$ & 0 and 2
\end{tabular}
\end{table}

Ordering the columns by
$v_0^{(2s)},v_1^{(2s)},w_1^{(2s)},v_2^{(2s)},w_2^{(2s)}$,
this $5 \times 5$ submatrix has a block lower-triangular structure
in the decomposition $5=1+2+2$. If $s \notin \{A,B\}$, these blocks are
\begin{align*}
\partial_{v_0^{(2s)}} B_{(2,s),(1,B),(1,B)}(\eta_*)&=
C_{0,0,0}^{2,1,1} \cdot \left|u_{*,0}^{(1B)}\right|^2\\
\partial_{v_1^{(2s)},w_1^{(2s)}}
\Big(B_{(2,s),(1,B),(1,A)},B_{(2,s),(2,B),(0,A)}\Big)(\eta_*)&=
\begin{psmallmatrix}
2C_{1,0,1}^{2,1,1} \cdot \Re \overline{u_{*,0}^{(1B)}} u_{*,1}^{(1A)} &
2C_{1,0,1}^{2,1,1} \cdot \Im \overline{u_{*,0}^{(1B)}} u_{*,1}^{(1A)} \\
2C_{1,-1,0}^{2,2,0} \cdot \Re \overline{u_{*,-1}^{(2B)}} u_{*,0}^{(0A)} &
2C_{1,-1,0}^{2,2,0} \cdot \Im \overline{u_{*,-1}^{(2B)}} u_{*,0}^{(0A)}
\end{psmallmatrix}\\
\partial_{v_2^{(2s)},w_2^{(2s)}}
\Big(B_{(2,s),(2,A),(0,A)},B_{(2,s),(1,A),(1,A)}\Big)(\eta_*)&=
\begin{psmallmatrix}
2C_{2,-2,0}^{2,2,0} \cdot \Re \overline{u_{*,-2}^{(2A)}} u_{*,0}^{(0A)} &
2C_{2,-2,0}^{2,2,0} \cdot \Im \overline{u_{*,-2}^{(2A)}} u_{*,0}^{(0A)}\\
2C_{2,-1,1}^{2,1,1} \cdot \Re \overline{u_{*,-1}^{(1A)}} u_{*,1}^{(1A)} &
2C_{2,-1,1}^{2,1,1} \cdot \Im \overline{u_{*,-1}^{(1A)}} u_{*,1}^{(1A)}
\end{psmallmatrix}
\end{align*}
If $s=A$ or $s=B$, then the first row of the third matrix or second row of the
second matrix should have coefficient 4 in place of 2, but this does not
affect their ranks. These blocks are generically non-singular, so
$\der_{\eta^{(2s)}} B^{(2s)}(\eta_*)$ has full column rank 5.\\

{\bf Part (b), $l=1$, $s \geq 3$:}
Note that $(A,B,s)=(1,2,s)$ are distinct indices because
$s \geq 3$. We specialize to a point $\eta_*$ such that 
\[\Type(1,s),\Type(1,A),\Type(0,A),\Type(0,B)=0, \qquad \Type(1,B)=1.\]
We pick the following 3 rows of $\der_{\eta^{(1s)}} B^{(1s)}(\eta_*)$, for
which the derivatives in $(v_m^{(1s)},w_m^{(1s)})$ are non-zero for only the
following corresponding values of $m$.

\begin{table}[h]
\caption{ }
\begin{tabular}{c|c}
$(l',s')$ and $(l'',s'')$ & Values of $m$ \\
\hline
$(1,s)$ and $(0,A)$ & 1 \\
$(1,A)$ and $(0,B)$ & 1 \\
$(1,B)$ and $(0,B)$ & 0
\end{tabular}
\end{table}

Ordering the columns by $v_1^{(1s)},w_1^{(1s)},w_0^{(1s)}$, this $3 \times 3$
submatrix has a block lower-triangular structure
in the decomposition $3=2+1$, with diagonal blocks
\begin{align*}
\partial_{v_1^{(1s)},w_1^{(1s)}}
\Big(B_{(1,s),(1,s),(0,A)},B_{(1,s),(1,A),(0,B)}\Big)(\eta_*)&=
\begin{psmallmatrix}
4C_{1,-1,0}^{1,1,0} \cdot \Re \overline{u_{*,-1}^{(1s)}} u_{*,0}^{(0A)} &
4C_{1,-1,0}^{1,1,0} \cdot \Im \overline{u_{*,-1}^{(1s)}} u_{*,0}^{(0A)}\\
2C_{1,-1,0}^{1,1,0} \cdot \Re \overline{u_{*,-1}^{(1A)}} u_{*,0}^{(0B)} &
2C_{1,-1,0}^{1,1,0} \cdot \Im \overline{u_{*,-1}^{(1A)}} u_{*,0}^{(0B)}
\end{psmallmatrix}\\
\partial_{w_0^{(1s)}} B_{(1,s),(1,B),(0,B)}(\eta_*)&=
-\i C_{0,0,0}^{1,1,0} \cdot \overline{u_{*,0}^{(1B)}} u_{*,0}^{(0A)}
\end{align*}
These blocks are generically non-singular (where we use
that $s$ and $A$ are distinct for the first block),
so $\der_{\eta^{(1s)}} B^{(1s)}(\eta_*)$ has full column rank 3.\\

{\bf Part (b), $l=1$, $s=2$:}
Consider the two columns of $\der_{\eta^{(12)}}
B^{(12)}$ corresponding to $\partial_{v_1^{(12)},w_1^{(12)}}$ and the two
rows corresponding to $B_{(1,2),(1,1),(0,1)},B_{(1,2),(1,2),(0,1)}$.
This $2 \times 2$ submatrix is (for any $\eta_*$)
\begin{align*}
&\partial_{v_1^{(12)},w_1^{(12)}}
\Big(B_{(1,2),(1,1),(0,1)},B_{(1,2),(1,2),(0,1)}\Big)(\eta_*)
=\begin{psmallmatrix} 2C_{1,-1,0}^{1,1,0} \cdot
\Re \overline{u_{*,-1}^{(11)}}u_{*,0}^{(01)} &
2C_{1,-1,0}^{1,1,0} \cdot \Im \overline{u_{*,-1}^{(11)}}u_{*,0}^{(01)} \\
4C_{1,-1,0}^{1,1,0} \cdot \Re \overline{u_{*,-1}^{(12)}}u_{*,0}^{(01)} &
4C_{1,-1,0}^{1,1,0} \cdot \Im \overline{u_{*,-1}^{(12)}}u_{*,0}^{(01)} \end{psmallmatrix}
\end{align*}
This is generically non-singular, so 
$\der_{v_1^{(12)},w_1^{(12)}} B^{(12)}(\eta_*)$ has full column rank 2.\\

{\bf Part (b), $l=1$, $s=1$:}
Consider the column and row of $\der_{\eta^{(11)}} B^{(11)}$ corresponding to
\[\partial_{v_1^{(11)}} B_{(1,1),(1,1),(0,1)}(\eta_*)
=4C_{1,-1,0}^{1,1,0} \Re \overline{u_{*,-1}^{(11)}}u_{*,0}^{(01)}.\]
This is generically non-zero, so $\der_{v_1^{(11)}} B^{(11)}(\eta_*)$ has
full column rank 1.\\

{\bf Part (c), $l=0$:} For any
$s \in \{1,\ldots,S_0\}$, $\eta^{(0s)}=v_0^{(0s)}$
is a single real variable. Applying $\langle 0,0;0,0|0,0 \rangle=1$,
we have from (\ref{eq:Bls}) that
\[\partial_{\eta^{(0s)}} B_{(0,s),(0,s),(0,s)}(\eta_*)
=\partial_{v_0^{(0s)}} (v_{*,0}^{(0s)})^3=3(v_{*,0}^{(0s)})^2,\]
which is generically non-zero. Thus
$\der_{\eta^{(0s)}} B^{(0s)}(\eta_*)$ has full column rank 1.
\end{proof}

\subsection{Projected cryo-EM} \label{sec:projCEM}

\subsubsection{Function bases}
We describe in further detail the function bases and the forms of the
projection map and rotational action on the basis coefficients in
the projected cryo-EM model of
Section \ref{subsec:projectedcryoEM}, following a setup of the model that is
similar to that of \cite[Section 5.5 and Appendix A.4]{bandeira2017estimation}.

Parametrizing $\R^3$ by spherical coordinates
$(\rho,\phi_1,\phi_2)$, define a product basis
\begin{equation}\label{eq:projectedbasis3D}
\hat{j}_{lsm}(\rho,\phi_1,\phi_2)=\tilde{z}_s(\rho)y_{lm}(\phi_1,\phi_2)
\quad \text{ for } s \geq 1, \quad l \geq 0, \quad m \in \{-l,\ldots,l\}
\end{equation}
where $\{y_{lm}\}$ are the complex spherical harmonics
(\ref{eq:complexharmonics}), and $\{\tilde{z}_s:s \geq 1\}$
is a system of radial basis functions $\tilde{z}_s:[0,\infty) \to \R$.
Similarly, parametrizing $\R^2$ by polar coordinates $(\rho,\phi_2)$
where $\rho \geq 0$ is the radius and $\phi_2 \in [0,2\pi)$ is the angle, let
\[b_m(\phi_2)=(2\pi)^{-1/2}e^{\i m\phi_2}\]
and define a corresponding product basis
\begin{equation}\label{eq:projectedbasis}
\hat{j}_{sm}(\rho,\phi_2)=\tilde{z}_s(\rho)b_m(\phi_2)
\quad \text{ for } s \geq 1, \quad m \in \mathbb{Z}.
\end{equation}
We choose the radial basis functions $\{\tz_s:s \geq 1\}$ to satisfy a modified
orthogonality relation
\begin{equation}\label{eq:2Dradial}
\int_0^\infty \rho\,\tz_s(\rho)\tz_{s'}(\rho)\der\rho=\1\{s=s'\}
\end{equation}
with weight $\rho$ instead of $\rho^2$, which ensures that $\{\hat{j}_{sm}\}$
are orthonormal in $L_2(\R^2,\C)$.
We write $\{j_{lsm}\}$ for the 3-D inverse Fourier transform of
$\{\hat{j}_{lsm}\}$, and $\{j_{sm}\}$ for the 2-D inverse Fourier transform of
$\{\hat{j}_{sm}\}$.

For any function $f \in L_2(\R^3)$, let $\hat{f}$ be its Fourier transform
as defined in (\ref{eq:fouriertransform}), and let
\[\widehat{\Pi \cdot f}(k_1,k_2)=
\int_{\R^2} e^{-2\pi \i(k_1x_1+k_2x_2)}(\Pi \cdot f)(x_1,x_2)\der x_1 \der x_2\]
be the 2-D Fourier transform of its tomographic projection.
By the Fourier-slice relation,
\begin{equation}\label{eq:fourierslice}
\widehat{\Pi \cdot f}(k_1,k_2)=\hat{f}(k_1,k_2,0).
\end{equation}
In spherical coordinates $(\rho,\phi_1,\phi_2)$ for $\R^3$
and polar coordinates $(\rho,\phi_2)$ for $\R^2$, this
corresponds to
\[\widehat{\Pi \cdot f}(\rho,\phi_2)=\hat{f}(\rho,\pi/2,\phi_2)\]
where we restrict $\phi_1=\pi/2$. This restriction of
each complex spherical harmonic $y_{lm}$
in (\ref{eq:complexharmonics}) is given by
\[y_{lm}(\pi/2,\phi_2)=p_{lm} \cdot b_m(\phi_2)\]
where $b_m$ is the function defined above, and
\begin{align}
p_{lm}&=(-1)^m \sqrt{\frac{(2l+1)}{2}\frac{(l-m)!}{(l+m)!}} \cdot
P_{lm}(0)\notag\\
&=\1\{l+m \text{ is even}\} \times \frac{(-1)^{(l+m)/2}\sqrt{(2l+1)/2}}{2^l l!}
\binom{l}{(l+m)/2}\sqrt{(l-m)!(l+m)!}\label{eq:plm}
\end{align}
the second equality applying Lemma \ref{lem:legendre-value}. 
Note that these coefficients satisfy a sign symmetry
\begin{equation}\label{eq:psymmetry}
p_{lm}=(-1)^m p_{l,-m}.
\end{equation}
Specializing (\ref{eq:fourierslice}) to $f=j_{lsm}$, we then have
\[\widehat{\Pi \cdot j_{lsm}}(\rho,\phi_2)
=p_{lm} \cdot \tilde{z}_s(\rho) b_m(\phi_2)=p_{lm} \cdot
\hat{j}_{sm}(\rho,\phi_2).\]
Then taking inverse Fourier transforms,
\begin{equation}\label{eq:Pijaction}
\Pi \cdot j_{lsm}=p_{lm} \cdot j_{sm}.
\end{equation}

We consider the space of functions $f:\R^3 \to \R$ that are
$(L,S_0,\ldots,S_L)$-bandlimited in the basis $\{j_{lsm}\}$,
admitting the first representation in (\ref{eq:projcryoEMbandlimited}),
\[f=\sum_{(l,s,m) \in \cI} u_m^{(ls)} \cdot j_{lsm}\]
where the index set $\cI$ is defined in (\ref{eq:indices}).
Then (\ref{eq:Pijaction}) shows that $\Pi \cdot f$ is also
bandlimited in the basis $\{j_{sm}\}$, admitting the first representation in
(\ref{eq:projbandlimited}),
\[\Pi \cdot f=\sum_{(s,m) \in \tcI} \tu_m^{(s)} \cdot j_{sm}\]
where the index set $\tcI$ is defined in (\ref{eq:indicesproj}).
Define a real basis $\{h_{lsm}\}$ from $\{j_{lsm}\}$ by (\ref{eq:hlsm}),
so that the coefficients $u \in \C^d$ for the former and $\theta \in \R^d$ for 
the latter are related by the unitary transform
$u=\hat{V}^*\theta$ in (\ref{eq:hatv-trans}). 
Similarly, define a real basis $\{h_{sm}\}$ from $\{j_{sm}\}$ by
\[h_{sm}=\begin{cases}
\frac{1}{\sqrt{2}}\big(j_{s,-m}+(-1)^m j_{sm}\big) & \text{ if } m>0 \\
j_{s0} & \text{ if } m=0 \\
\frac{\i}{\sqrt{2}}\big(j_{sm}-(-1)^m j_{s,-m}\big) & \text{ if } m<0,
\end{cases}\]
where the coefficients $u \in \C^{\td}$ for the former and $\theta \in \R^{\td}$
for the latter are also related by a unitary transform $\tu=\tV^*\ttheta$,
defined as
\begin{equation}\label{eq:tu-ttheta}
\tu_m^{(s)}=\begin{cases}
\frac{(-1)^m}{\sqrt{2}}(\ttheta_{|m|}^{(s)}-\i\ttheta_{-|m|}^{(s)}) &
\text{ if } m>0\\
\ttheta_0^{(s)} & \text{ if } m=0\\
\frac{1}{\sqrt{2}}(\ttheta_{|m|}^{(s)}+\i\ttheta_{|m|}^{(s)}) & \text{ if } m<0.
\end{cases}
\end{equation}
It may be checked from the orthonormality of $\{\hat{j}_{sm}\}$ and $\{j_{sm}\}$ in $L_2(\R^2,\C)$ that these functions $\{h_{sm}\}$ are also
orthonormal in $L_2(\R^2,\R)$.
This yields the latter two real representations in
(\ref{eq:projcryoEMbandlimited}) and (\ref{eq:projbandlimited}).

Recall that since $f$ is real-valued, the coefficients $u_m^{(ls)}$
satisfy the sign symmetry (\ref{eq:usymmetry}). Similarly,
since $\Pi \cdot f$ is real-valued, its Fourier transform satisfies
\[\widehat{\Pi \cdot f}(\rho,\phi_2)=\overline{\widehat{\Pi \cdot f}
(\rho,\pi+\phi_2)}\]
where $(\rho,\pi+\phi_2)$ is the reflection of $(\rho,\phi_2)$ about the origin.
Then the coefficients $\tu_m^{(s)}$ of $\widehat{\Pi \cdot f}$ in the basis
$\{\hat{j}_{sm}\}$ must satisfy the analogous sign symmetry
\begin{equation}\label{eq:tusymmetry}
\tu_m^{(s)}=(-1)^m\overline{\tu_{-m}^{(s)}}.
\end{equation}

The identity (\ref{eq:Pijaction}) and the relations $u=\hat{V}^*\theta$ and
$\tu=\tV^*\ttheta$ show that the tomographic projection $\Pi$ is a linear map
from $\theta \in \R^d$ to $\ttheta \in \R^{\td}$, defined as
\begin{equation}\label{eq:cryoEMPi}
\Pi=\tV \cdot \Pi^\C \cdot \hV^* \in \R^{\td \times d},
\qquad \Pi_{(s',m'),(l,s,m)}^\C=\1\{s=s'\} \cdot \1\{m=m'\} \cdot p_{lm}.
\end{equation}
Here, $\Pi^\C$ is the corresponding linear map from $u \in \C^d$ to $\tu \in
\C^{\td}$, and $p_{lm}$ are the values defined in (\ref{eq:plm}).
The action of the rotation $f \mapsto f_\frakg$ on $\theta \in \R^d$ is as
previously
described in Lemma \ref{lem:so3-cryo-unproj-act}, and this expresses the model
in the general form of (\ref{eq:projectedmodel}) for projected orbit recovery.

\begin{remark}
Note that if $\{\tz_s:s=1,\ldots,S\}$ has the same
linear span as $\{z_s:s=1,\ldots,S\}$ used in the unprojected cryo-EM model of Section \ref{subsec:cryoEM},
then the two spaces of bandlimited functions (\ref{eq:projcryoEMbandlimited})
and (\ref{eq:cryoEMbandlimited}) coincide. However, we caution that
here under the orthogonality relation (\ref{eq:2Dradial}),
the \emph{unprojected} basis $\{h_{lsm}\}$ is not
orthonormal for this function space, so our parametrization $f \mapsto \theta$
here is not an isometric parametrization of $f \in L_2(\R^3)$.
\end{remark}

\subsubsection{Terms of the high noise series expansion}

We describe the explicit forms of $\ts_1(\theta)$,
$\ts_2(\theta)$, and $\ts_3(\theta)$. Recalling the entries
$p_{lm}$ in (\ref{eq:plm}), define
\begin{align}
Q_{kl}&=\frac{(-1)^{k+l}}{(2k+1)(2l+1)}
\sum_{q=-(k \wedge l)}^{k \wedge l} p_{kq}^2p_{lq}^2,\label{eq:Qkl}\\
M_{k,k',k'',l,l',l''}
&=\frac{(-1)^{k''+l''}}{(2k'' + 1)(2l'' + 1)}
\mathop{\sum_{q=-(k \wedge l)}^{k \wedge l}
\sum_{q'=-(k' \wedge l')}^{k' \wedge l'}
\sum_{q''=-(k'' \wedge l'')}^{k'' \wedge l''}}_{q''=q+q'}\notag\\
&\hspace{1in}\langle k,q;k',q'|k'',q'' \rangle \langle l,q;l',q'|l'',q''
\rangle p_{kq}p_{k'q'}p_{k''q''}p_{lq}p_{l'q'}p_{l''q''}.\label{eq:Mkl}
\end{align}
Recall also $u^{(ls)}(\theta)$ and
$B_{(l,s),(l',s'),(l'',s'')}(\theta)$ from (\ref{eq:uls}) and (\ref{eq:Bls}).

\begin{theorem}\label{thm:projectedcryoEM-mom}
For any $L \geq 1$ and $S_0,\ldots,S_L \geq 1$,
\begin{align*}
\ts_1(\theta)&=\frac{p_{00}^2}{2}\sum_{s=1}^{S_0} \Big(u^{(0s)}(\theta)
-u^{(0s)}(\theta_*)\Big)^2\\
\ts_2(\theta)&=\frac{1}{4}\sum_{k,l=0}^L Q_{kl}
\sum_{s,s'=1}^{S_k \wedge S_l}
\Big(\langle u^{(ks)}(\theta),u^{(ks')}(\theta) \rangle
-\langle u^{(ks)}(\theta_*),u^{(ks')}(\theta_*) \rangle\Big)\\
&\hspace{2in}\times\Big(\langle u^{(ls)}(\theta),u^{(ls')}(\theta) \rangle
-\langle u^{(ls)}(\theta_*),u^{(ls')}(\theta_*) \rangle\Big)\\
\ts_3(\theta)&=\frac{1}{12}
\mathop{\sum_{k,k',k'',l,l',l''=0}^L}_{|k-k'| \leq k'' \leq k+k' \text{ and }
|l-l'| \leq l'' \leq l+l'} M_{k,k',k'',l,l',l''}
\sum_{s=1}^{S_k \wedge S_l} \sum_{s'=1}^{S_{k'} \wedge S_{l'}}
\sum_{s''=1}^{S_{k''} \wedge S_{l''}}\\
&\phantom{===} \Big(B_{(k,s),(k',s'),(k'',s'')}(\theta) - B_{(k, s), (k', s'), (k'', s'')}(\theta_*)\Big)
\Big(B_{(l,s),(l',s'),(l'',s'')}(\theta) - B_{(l,s),(l',s'),(l'',s'')}(\theta_*)\Big).
\end{align*}
\end{theorem}
\begin{proof}
Recall from Lemma \ref{lem:skform} that
\begin{equation} \label{eq:tsk-exp}
\ts_k(\theta) = \frac{1}{2(k!)} \E_{g, h}\left[ \langle \Pi  g  \theta, \Pi  h  \theta \rangle^k
- 2 \langle \Pi  g  \theta, \Pi  h  \theta_*\rangle^k
+ \langle \Pi  g  \theta_*, \Pi  h  \theta_*\rangle^k\right].
\end{equation}
Consider two different real coefficient vectors $\theta,\vartheta \in \R^d$,
with corresponding complex coefficients $u=\hat{V}^*\theta$ and
$v=\hat{V}^*\vartheta$.\\

\noindent \textbf{Case $k = 1$:} Notice that
\begin{align*}
\langle \Pi g\theta,\Pi h\vartheta \rangle
=\Big\langle (\tV^*\Pi \hV)(\hV^*g \hV)u,(\tV^*\Pi \hV)(\hV^*h \hV)v
\Big\rangle
=\langle D(\frakg)u,(\Pi^\C)^*\Pi^\C D(\frakh)v\rangle,
\end{align*}
where $D(\frakg),D(\frakh)$ are the block-diagonal matrices in (\ref{eq:DcryoEM}). The form of $\Pi^\C$ from (\ref{eq:cryoEMPi}) yields
\[
({\Pi^\C}^*\Pi^\C)_{lsm,l's'm'}=\1\{s=s'\} \cdot \1\{m=m'\}p_{lm}p_{l'm'}
\]
so that 
\begin{align}
\langle \Pi g\theta,\Pi h\vartheta \rangle
&=\sum_{k,l=0}^L \sum_{s=1}^{S_k \wedge S_l}
\sum_{m,q=-k}^k \sum_{n,r=-l}^l
\overline{D^{(k)}_{qm}(\frakg)u_m^{(ks)}} \cdot
\1\{q=r\}p_{kq}p_{lr} \cdot D^{(l)}_{rn}(\frakh)v_n^{(ls)}.
\label{eq:PiguPihv}
\end{align}
Applying (\ref{eq:ED}) to take the expectation, we preserve only the terms for
$k=l=m=q=n=r=0$, yielding
\[
\E_{g,h}[\langle \Pi g\theta,\Pi h\vartheta \rangle]
=\sum_{s=1}^{S_0} p_{00}^2 \overline{u_0^{(0s)}}v_0^{(0s)}.
\]
Recalling that $u^{(0s)}=u_0^{(0s)}$ and $v^{(0s)}=v_0^{(0s)}$
are real-valued by (\ref{eq:hatv-trans}), and
substituting into (\ref{eq:tsk-exp}), we obtain
\[
\ts_1(\theta)
= \frac{p_{00}^2}{2} \sum_{s = 1}^{S_0} \Big(u^{(0s)}(\theta) - u^{(0s)}(\theta_*)\Big)^2.
\]

\noindent \textbf{Case $k = 2$:} 
We square both sides of (\ref{eq:PiguPihv}) and apply the relations
(\ref{eq:WignerDorthog}), (\ref{eq:usymmetry}), and (\ref{eq:psymmetry})
to get
\begin{align*}
&\E_{g,h}[\langle \Pi g\theta,\Pi h\vartheta \rangle^2]\\
&=\sum_{k,l=0}^L \sum_{s,s'=1}^{S_k \wedge S_l}
\sum_{m,q=-k}^k \sum_{n,r=-l}^l
\frac{(-1)^{m+q+n+r}}{(2k+1)(2l+1)}
\overline{u_m^{(ks)}u_{-m}^{(ks')}}
\cdot \1\{q=r\}p_{k,q}p_{k,-q}p_{l,r}p_{l,-r}
\cdot v_n^{(ls)}v_{-n}^{(ls')}\\
&=\sum_{k,l=0}^L \sum_{s,s'=1}^{S_k \wedge S_l}
\sum_{m,q=-k}^k \sum_{n,r=-l}^l \frac{(-1)^{k+l}}{(2k+1)(2l+1)}
\overline{u_m^{(ks)}}u_m^{(ks')} \cdot \1\{q=r\}p_{kq}^2p_{lr}^2
\cdot v_n^{(ls)}\overline{v_n^{(ls')}}\\
&=\sum_{k,l=0}^L \frac{(-1)^{k+l}}{(2k+1)(2l+1)}
\sum_{s,s'=1}^{S_k \wedge S_l} \langle u^{(ks)},u^{(ks')} \rangle \cdot
\overline{\langle v^{(ls)},v^{(ls')} \rangle}
\sum_{q=-(k \wedge l)}^{k \wedge l} p_{kq}^2p_{lq}^2\\
&= \sum_{k,l=0}^L Q_{kl} \sum_{s,s'=1}^{S_k \wedge S_l}
\langle u^{(ks)},u^{(ks')} \rangle \cdot \overline{\langle
v^{(ls)},v^{(ls')}\rangle}
\end{align*}
where we have substituted $Q_{kl}$ from (\ref{eq:Qkl}) in the last equality.
By the isometry $\langle u^{(ks)},u^{(ks')} \rangle
=\langle \theta^{(ks)},\theta^{(ks')} \rangle$, both inner products on the last
line are real. Then applying this to (\ref{eq:tsk-exp}),
\begin{align*}
\ts_2(\theta) 
&=\frac{1}{4}\sum_{k,l=0}^L Q_{kl}
\sum_{s,s'=1}^{S_k \wedge S_l}
\left(\langle u^{(ks)}(\theta),u^{(ks')}(\theta) \rangle 
- \langle u^{(ks)}(\theta_*),u^{(ks')}(\theta_*) \rangle \right)\\
&\hspace{2in}\times \left(\langle u^{(ls)}(\theta),u^{(ls')}(\theta) \rangle
- \langle u^{(ls)}(\theta_*),u^{(ls')}(\theta_*) \rangle \right).
\end{align*}

\noindent \textbf{Case $k = 3$:} 
Let us introduce the abbreviations
\[\sum_{k,k',k'',l,l',l''}=\mathop{\sum_{k,k',k'',l,l',l''=0}^L}_{|k-k'| \leq k'' \leq k+k',\;|l-l'| \leq
l'' \leq l+l'} \qquad \text{ and } \qquad
\sum_{s,s',s''}=\sum_{s=1}^{S_k \wedge S_l}
\sum_{s'=1}^{S_{k'} \wedge S_{l'}}\sum_{s''=1}^{S_{k''} \wedge S_{l''}}.\]
For given indices $m,m',q,q',n,n',r,r'$, let us write as shorthand
$m''=m+m'$, $n''=n+n'$, $q''=q+q'$, and $r''=r+r'$.
We cube both sides of (\ref{eq:PiguPihv}) and apply the relation
(\ref{eq:WignerDtriple}) to obtain
\begin{align*}
\E_{g,h}[\langle \Pi g\theta,\Pi h\vartheta\rangle^3]
&=\sum_{k,k',k'',l,l',l''} \sum_{s,s',s''}
\sum_{m,q=-k}^k \sum_{m',q'=-k'}^{k'}
\sum_{n,r=-l}^l \sum_{n',r'=-l'}^{l'}
\frac{(-1)^{m''+q''+n''+r''}}{(2k''+1)(2l''+1)}\\
&\hspace{0.5in}
\times C^{k,k',k''}_{q,q',q''}C^{k,k',k''}_{m,m',m''}
C^{l,l',l''}_{r,r',r''}C^{l,l',l''}_{n,n',n''}
\1\{q=r\}\1\{q'=r'\}\\
&\hspace{0.5in} \times
\overline{u_m^{(ks)}u_{m'}^{(k's')}u_{-m''}^{(k''s'')}}
v_n^{(ls)}v_{n'}^{(l's')}v_{-n''}^{(l''s'')}
p_{k,q}p_{k',q'}p_{k'',-q''}p_{l,r}p_{l',r'}p_{l'',-r''}.
\end{align*}
Let us apply, by (\ref{eq:usymmetry}) and (\ref{eq:psymmetry}),
\[\overline{u^{(k''s'')}_{-m''}}=(-1)^{m''+k''}u_{m''}^{(k''s'')},
\qquad \overline{v^{(l''s'')}_{-n''}}=(-1)^{n''+l''}v_{n''}^{(l''s'')},\]
\[p_{k'',-q''}=(-1)^{q''}p_{k'',q''},
\qquad p_{l'',-r''}=(-1)^{r''}p_{l'',r''}.\]
Recalling $B_{(l,s),(l',s'),(l'',s'')}(\theta)$ from (\ref{eq:Bls}), which is
real-valued, and substituting the form of $M_{k,k',k'',l,l',l''}$ in
(\ref{eq:Mkl}), the above may be written succinctly as
\[\E_{g,h}[\langle \Pi g\theta,\Pi h\vartheta\rangle^3]
=\sum_{k,k',k'',l,l',l''} M_{k,k',k'',l,l',l''}\\
\sum_{s,s',s''}
B_{(k,s),(k',s'),(k'',s'')}(\theta)B_{(l,s),(l',s'),(l'',s'')}(\vartheta).\]
Then by (\ref{eq:tsk-exp}), we find
\begin{align*}
\ts_3(\theta) &=\frac{1}{12}
\sum_{k,k',k'',l,l',l''} M_{k,k',k'',l,l',l''}
\sum_{s,s',s''}\\
&\phantom{==}
\Big(B_{(k,s),(k',s'),(k'',s'')}(\theta) - B_{(k, s), (k', s'), (k'', s'')}(\theta_*)\Big)
\Big(B_{(l,s),(l',s'),(l'',s'')}(\theta) -
B_{(l,s),(l',s'),(l'',s'')}(\theta_*)\Big).
\end{align*}
\end{proof}

\subsubsection{Transcendence degrees}

We now prove Theorem \ref{thm:projectedcryoEM} on the sequences of transcendence
degrees.

\begin{proof}[Proof of Theorem \ref{thm:projectedcryoEM}]
Theorem \ref{thm:cryoEM} shows $\trdeg(\cR^\G)=d-3$.
We compute $\trdeg(\tcR^\G_{\leq m})$ for $m = 1, 2, 3$ using
Lemma \ref{lem:trdeg}. For $m=1$,
\[
\nabla^2 \ts_1(\theta_*) = p_{00}^2 \sum_{s = 1}^{S_0} \nabla u^{(0s)}(\theta_*)
\nabla u^{(0s)}(\theta_*)^\top.
\]
Since $p_{00} \neq 0$, this shows
$\trdeg(\tcR^\G_{\leq 1}) = \rank(\nabla^2 \ts_1(\theta_*)) = S_0$ as in
Theorem \ref{thm:cryoEM}.

For $m = 2$,
\begin{multline*}
\nabla^2 \ts_1(\theta_*) + \nabla^2 \ts_2(\theta_*) = p_{00}^2 \sum_{s =
1}^{S_0} \nabla u^{(0s)}(\theta_*) \nabla u^{(0s)}(\theta_*)^\top\\
+ \frac{1}{2} \sum_{k, l = 0}^L Q_{kl} \sum_{s, s' = 1}^{S_k \wedge S_l}
\nabla[\langle u^{(ks)}(\theta), u^{(ks')}(\theta)\rangle] \nabla[\langle
u^{(ls)}(\theta), u^{(ls')}(\theta)\rangle]^\top\Big|_{\theta = \theta_*}.
\end{multline*}
Recall the form of $Q_{kl}$ from (\ref{eq:Qkl}).
Define the index sets
\[\cJ=\{(k,r,r'):0 \leq k \leq L,\,1 \leq r,r' \leq S_k\}, \qquad
\cL=\{(q,t,t'):-L \leq q \leq L,\,1 \leq t,t' \leq S\}\]
where $S=\max_{l=0}^L S_l$, and define a matrix
$D \in \R^{|\cJ| \times |\cL|}$ with the entries
\begin{align*}
D_{krr',qtt'} &:= \frac{(-1)^k}{2k + 1} \cdot \frac{p^2_{kq}}{\sqrt{2}} \cdot 
\1_{t = r} \1_{t' = r'} \1_{-k \leq q \leq k}.
\end{align*}
This definition satisfies, for any $(k,r,r'),(l,s,s') \in \cJ$,
\[\frac{1}{2}Q_{kl}\1_{r=s}\1_{r'=s'}
=\sum_{(q,t,t') \in \cL}
D_{krr',qtt'}D_{lss',qtt'}=(DD^\top)_{krr',lss'}.\]
Then, defining the matrices $G^0$ and $G$ with columns
\begin{align*}
G^0_s &:= p_{00} \nabla u^{(0s)}(\theta_*) \qquad \text{ for $1 \leq s \leq S_0$}\\
G_{krr'} &:= \nabla[\langle u^{(k r)}(\theta),
u^{(kr')}{(\theta)}\rangle] \Big|_{\theta = \theta_*} \qquad \text{ for }
(k,r,r') \in \cJ,
\end{align*}
we have
\begin{align*}
\nabla^2 \ts_1(\theta_*) + \nabla^2 \ts_2(\theta_*)
&=G^0(G^0)^\top+\sum_{(k,r,r') \in \cJ} \sum_{(l,s,s') \in \cJ}
\frac{1}{2}Q_{kl} \1_{r=s}\1_{r'=s'}G_{krr'}G_{lss'}^\top\\
&=G^0(G^0)^\top+GDD^\top G^\top=[GD \mid G^0][GD \mid G^0]^\top
\end{align*}
Here, the square submatrix of $D$ consisting of the columns
$(q,t,t') \in \cJ \subset \cL$ is lower triangular with non-zero diagonal,
because $p_{kk} \neq 0$ for any $k=0,\ldots,L$. Then the column span of
$GD$ coincides with that of $G$. Since the column span of
$G^0$ is also contained in that of $G$ for generic $\theta_*$, we conclude that
\[\trdeg(\tcR^\G_{\leq 2}) = \rank(\nabla^2 \ts_1(\theta_*) + \nabla^2
\ts_2(\theta_*)) = \rank\Big([GD \mid G^0]\Big)=\rank(G).\]
Then $\trdeg(\tcR^\G_{\leq 2})=\trdeg(\cR^\G_{\leq 2})$ as in Theorem
\ref{thm:cryoEM}.

For $m = 3$, we have $\trdeg(\cR^\G_{\leq 3}) \leq \trdeg(\cR^\G) = d - 3$, so
it suffices to show
$\rank(\nabla^2 \ts_3(\theta_*)) \geq d - 3$.  For this, we first write a
more convenient form for $\ts_3(\theta)$ and its Hessian at
$\theta = \theta_*$.  Recall $S = \max_{l=0}^L S_l$ and define the index sets
\begin{align*}
\cQ &= \Big\{((q,r),(q',r'),(q'',r'')):\;-L \leq q,q',q'' \leq L,\;\; q+q'+q''=0,\;\;1 \leq r,r',r'' \leq S\Big\}\\
\cH &= \Big\{((l,s),(l',s'),(l'',s'')):\; 0\leq l,l',l'' \leq L,\;\;|l - l'|
\leq l'' \leq l + l',\;\;1 \leq s \leq S_l,\;1 \leq s' \leq S_{l'},\;
1 \leq s'' \leq S_{l''}\Big\}.
\end{align*}
Define a matrix $N \in \R^{|\cQ| \times |\cH|}$ entrywise by
\begin{align}
N_{(q,r),(q',r'),(q'',r'')}^{(l,s),(l',s'),(l'',s'')}
&=\1\{r=s,r'=s',r''=s''\} \cdot \1\{|q| \leq l,|q'| \leq l',|q''| \leq
l''\}\nonumber\\
&\hspace{1in} \cdot \frac{(-1)^{l''+q''}}{2l''+1}
\cdot \langle l,q;l',q'|l'',-q'' \rangle
p_{lq}p_{l'q'}p_{l''q''},\label{eq:Ndef}
\end{align}
where the subscript is the row index in $\cQ$ and the superscript is the
column index in $\cH$. In the expression (\ref{eq:Mkl})
for $M_{k,k',k'',l,l',l''}$, let us flip the sign of $q''$ and apply
(\ref{eq:psymmetry}) to write this as
\begin{align*}
M_{k,k',k'',l,l',l''}&=\frac{(-1)^{k''+q''}}{2k''+1}\frac{(-1)^{l''+q''}}{2l''+1}
\mathop{\sum_{|q| \leq k \wedge l} \sum_{|q'| \leq k' \wedge l'}
\sum_{|q''| \leq k'' \wedge l''}}_{q+q'+q''=0}\\
&\hspace{1in}\langle k,q;k',q'|k'',-q'' \rangle
\langle l,q;l',q'|l'',-q'' \rangle p_{kq}p_{k'q'}p_{k''q''}
p_{lq}p_{l'q'}p_{l''q''}.
\end{align*}
Then
\begin{align*}
\ts_3(\theta)&=\frac{1}{12}\big(B(\theta)-B(\theta_*)\big)^\top N^\top N
\big(B(\theta)-B(\theta_*)\big)
\end{align*}
Applying the chain rule to differentiate this twice at $\theta=\theta_*$, we
obtain
\[\nabla^2 \ts_3(\theta_*)
=\frac{1}{6}\der B(\theta_*)^\top N^\top N\der B(\theta_*),\]
so $\rank\big(\nabla^2 \ts_3(\theta_*)\big) =\rank\big(N \cdot \der
B(\theta_*)\big)$. We remark that, here, we cannot 
reduce this directly to $\rank(\der B(\theta_*))$ using the the above argument
that showed $\rank(G DD^\top G^\top)=\rank(G)$ for $m=2$.
This is because for large $L$, the matrix $N$ may have $O(L^3)$ columns but only
$O(L^2)$ rows. Then $N^\top N$ is of low-rank, in contrast to $DD^\top$ above
which is nonsingular.

To analyze $\rank(N \cdot \der B(\theta_*))$,
recall the linear reparametrization by the coordinates $\eta(\theta)$ in
(\ref{eq:etals}) and (\ref{eq:eta}).
Then equivalently
\[\rank\big(\nabla^2 \ts_3(\theta_*)\big)
=\rank\big(N \cdot \der_\eta B(\eta_*)\big)\]
The proof of Theorem \ref{thm:cryoEM} verified that
$\rank(\der_\eta B(\eta_*))=d-3$ for generic $\eta_* \in \R^d$. In fact,
let
\[D(\eta_*)=\text{submatrix of } \der_\eta B(\eta_*)
\text{ with columns }
\partial_{w_0^{(12)}},\partial_{w_1^{(11)}},\partial_{w_0^{(11)}} \text{
removed.}\]
Then Lemma \ref{lemma:S3cryoEMrankincrease} shows that
$D(\eta_*)$ has full column rank $d-3$ for generic $\eta_* \in \R^d$. Applying
\[\rank(N \cdot \der_\eta B(\eta_*)) \geq \rank(N \cdot D(\eta_*)),\]
it then suffices to show that $N \cdot D(\eta_*)$ also has full column rank
$d-3$ for generic $\eta_* \in \R^d$.

For this, we define the following submatrices of $N$ and $D(\eta_*)$.
For each $k \in \{0,1,\ldots,L\}$, define the index sets
\begin{align*}
\cH^{(k)}&=\Big\{((l,s),(l',s'),(l'',s'')) \in \cH:\;\max(l,l',l'')=k \Big\},\\
\cQ^{(k)}&=\Big\{((q,r),(q',r'),(q'',r'')) \in \cQ:\;\max(|q|,|q'|,|q''|)=k
\Big\},\\
\cV^{(k)}&=\Big\{\text{coordinates } v^{(ls)}_m,w^{(ls)}_m \text{ of } \eta:
\;l=k\Big\} 
\qquad \text{ if } k \neq 1,\\
\cV^{(1)}&=\Big\{\text{coordinates } v^{(1s)}_m,w^{(1s)}_m \text{ of } \eta
\Big\} \Big\backslash \Big\{w_0^{(12)},w_1^{(11)},w_0^{(11)}\Big\}.
\end{align*}
Let $N_k \in \R^{|\cQ^{(k)}| \times |\cH^{(k)}|}$
be the submatrix of $N$ containing the rows in $\cQ^{(k)}$ and columns in
$\cH^{(k)}$, and let $D_k(\eta_*) \in \R^{|\cH^{(k)}| \times |\cV^{(k)}|}$
be the submatrix of $D(\eta_*)$ containing the rows in
$\cH^{(k)}$ and columns in $\cV^{(k)}$. Similarly, define $N_{\leq k}$ and
$D_{\leq k}(\eta_*)$ to contain rows and columns of
$\cQ^{(l)},\cH^{(l)},\cV^{(l)}$ for $l \leq k$.
Note that $D_k(\eta_*)$ and $D_{\leq k}(\eta_*)$
depend only on the coordinates of $v^{(ls)}_m$ and $w^{(ls)}_m$ where $l \leq k$, by
the definition of $\cH^{(k)}$ and the form of each function
$B_{(l,s),(l',s'),(l'',s'')}$.

We prove by induction on $L$ the claim that $N \cdot D(\eta_*)$ has full
column rank for generic $\eta_* \in \R^d$.
Lemma \ref{lemma:SO3projinduction}(a) below shows that for $L=1$, there exists
some $\eta_*$ where $N \cdot D(\eta_*)$ has full column rank. Then
$N \cdot D(\eta_*)$ has full column rank also for generic $\eta_*$
by Fact \ref{fact:fullrank}, establishing the base case $L=1$.

For the inductive step, we establish a block structure on
$N$ and $D(\eta_*)$. Block the rows and columns of $N$ by
$(\cQ \setminus \cQ^{(L)},\cQ^{(L)})$ and $(\cH \setminus \cH^{(L)},\cH^{(L)})$,
and those of $D(\eta_*)$ by
$(\cH \setminus \cH^{(L)},\cH^{(L)})$ and $(\eta \setminus
\cV^{(L)},\cV^{(L)})$. Note that
$N^{(l,s),(l',s'),(l'',s'')}_{(q,r),(q',r'),(q'',r'')}=0$ unless
$\max(|q|,|q'|,|q''|) \leq \max(l,l',l'')$, and also
$B_{(l,s),(l',s'),(l'',s'')}$ does not depend on any variable $v_m^{(ks)}$ or
$w_m^{(ks)}$ where $k>\max(l,l',l'')$. Thus $N$ and $D(\eta_*)$
have the block structures
\[N=\begin{pmatrix} A & B \\
0 & N_L \end{pmatrix}, \qquad D(\eta_*)=\begin{pmatrix} X(\eta_*) & 0 \\
Y(\eta_*) & D_L(\eta_*) \end{pmatrix}\]
for some matrices $A,B,X(\eta_*),Y(\eta_*)$.

Let us now specialize to $\eta_* \in \R^d$ where
\begin{equation}\label{eq:projspecialeta}
v_{*,m}^{(Ls)}=w_{*,m}^{(Ls)}=0 \text{ for all } s=1,\ldots,S_L
\text{ and } m=-L,\ldots,L.
\end{equation}
The above matrix $Y(\eta_*)$ contains the
derivatives in variables $\{v_m^{(ks)},w_m^{(ks)}:k<L\}$ of the
functions $B_{(l,s),(l',s'),(l'',s'')}$ where $\max(l,l',l'')=L$. By the form of
$B_{(l,s),(l',s'),(l'',s'')}$, any such
derivative vanishes for $\eta_*$ satisfying (\ref{eq:projspecialeta}),
so $Y(\eta_*)=0$ and
\[
N \cdot D(\eta_*)=\begin{pmatrix} A \cdot X(\eta_*) & B \cdot
D_L(\eta_*) \\ 0 & N_L \cdot D_L(\eta_*) \end{pmatrix}.
\]
The induction hypothesis for $L-1$ is exactly the statement that
the upper-left block $A \cdot X(\eta_*)$ has full column rank for all
generic values of the coordinates $\{v_{*,m}^{(ks)},w_{*,m}^{(ks)}:k \leq L-1\}$.
Applying Fact \ref{fact:fullrank} and Lemma
\ref{lemma:SO3projinduction}(b) below with $k=L$, restricting to $\eta_*$
satisfying (\ref{eq:projspecialeta}) and for generic values of the remaining
coordinates $\{v_{*,m}^{(ks)},w_{*,m}^{(ks)}:k \leq L-1\}$, the lower-right block
$N_L \cdot D_L(\eta_*)$ also has full column rank.
Then there exists a point $\eta_* \in \R^d$ satisfying (\ref{eq:projspecialeta})
where
$N \cdot D(\eta_*)$ has full column rank.
Then $N \cdot D(\eta_*)$ has full column rank also
for generic $\eta_* \in \R^d$, completing the induction and the proof.
\end{proof}

\begin{lemma}\label{lemma:SO3projinduction}
If $S_l \geq 4$ for $0 \leq l \leq L$, then we have the following.
\begin{enumerate}[(a)]
\item There exists a point $\eta_* \in \R^d$ such that
$N_{\leq 1} \cdot D_{\leq 1}(\eta_*)$ has full column rank.
\item For each $k \geq 2$, there exists a point $\eta_* \in \R^d$ such that
$v_{*,m}^{(ks)}=w_{*,m}^{(ks)}=0$ for all $s \in \{1,\ldots,S_k\}$ and $m \in
\{-k,\ldots,k\}$, and $N_k \cdot D_k(\eta_*)$ has full column rank.
\end{enumerate}
\end{lemma}
\begin{proof}[Proof of Lemma \ref{lemma:SO3projinduction}]
{\bf Part (a):} Recall from the form of $p_{lm}$ in (\ref{eq:plm})
that $p_{lm}=0$ if $l+m$ is odd and $p_{lm} \neq 0$ if $l+m$ is even.
Then, for $\max\{l, l', l''\} \leq 1$, 
the non-vanishing of Clebsch-Gordon coefficients in
Lemma \ref{lemma:CGnonvanishing} and the definition of $N$ in (\ref{eq:Ndef})
imply that
\begin{equation}\label{eq:Nstructurebasecase}
N^{(l, s), (l', s'), (l'', s'')}_{(q, r), (q', r'), (q'', r'')} \neq 0
\text{ if and only if }l = |q|, l' = |q'|, l'' = |q''|, r = s, r' = s', r'' =
s''.
\end{equation}
For each $((l,s),(l',s'),(l'',s'')) \in \cH^{(0)} \cup \cH^{(1)}$ where
$l=l'+l''$, take the row
$((-l,s),(l',s'),(l'',s'')) \in \cQ^{(0)} \cup \cQ^{(1)}$ of $N_{\leq 1}$.
It suffices to exhibit $\eta_*$ such that the submatrix of corresponding rows of
$N_{\leq 1} \cdot D_{\leq 1}(\eta_*)$ has full column rank. 
Observation (\ref{eq:Nstructurebasecase}) implies that each such row of $N_{\leq 1}$
has exactly one non-zero entry, which is given by
$N^{(l, s), (l', s'), (l'', s'')}_{(-l, s), (l', s'), (l'', s'')}$.
Then it suffices to show that the submatrix of
$D_{\leq 1}(\eta_*)$ consisting of the rows
$((l,s),(l',s'),(l'',s'')) \in \cH^{(0)} \cup \cH^{(1)}$ where
$l=l'+l''$ has full column rank. But this has been exhibited already in
Lemma \ref{lemma:S3cryoEMrankincrease}, because the proof of
Lemma \ref{lemma:S3cryoEMrankincrease}(b--c) in fact only used rows of
$\der_{\eta^{(0s)}} B^{(0s)}$ and $\der_{\eta^{(1s)}} B^{(1s)}$
for which $(l,l',l'')=(0,0,0)$ or $(1,1,0)$, both
satisfying $l=l'+l''$. This completes the proof of (a).\\

{\bf Part (b), $k=2$ and $k=3$:} The argument is similar to part (a). Observe
first that when $\eta_*$ satisfies $v_{*,m}^{(ks)}=w_{*,m}^{(ks)}=0$ for all
$s$ and $m$, the rows of $D_k(\eta_*)$ indexed by
$((l,s),(l',s'),(l'',s'')) \in \cH^{(k)}$ having more than one index
$l,l',l''$ equal to $k$ are identically 0. Let $D_k(\eta_*)'$ be the submatrix
of $D_k(\eta_*)$ with
these rows removed, and let $N_k'$ be the submatrix of $N_k$ with the
corresponding columns removed. Then $N_k \cdot D_k(\eta_*)=N_k' \cdot
D_k(\eta_*)'$.

For each remaining tuple $((l,s),(l',s'),(l'',s'')) \in \cH^{(k)}$ where
$l=l'+l''$, consider
the row
\[((-l,s),(l',s'),(l'',s'')) \in \cQ^{(k)}\] of $N_k'$.
Note that we must have $(l,l',l'')=(2,1,1)$ if $k=2$, and
$(l,l',l'')=(3,1,2)$ or $(3,2,1)$ if $k=3$. Each such row has the non-zero entry
$N_{(-l,s),(l',s'),(l'',s'')}^{(l,s),(l',s'),(l'',s'')}$ as above, and this
is the only non-zero entry in the row: Indeed, if
$((j,r),(j',r'),(j'',r''))$ is a column of $N_k'$ where
$N_{(-l,s),(l',s'),(l'',s'')}^{(j,r),(j',r'),(j'',r'')} \neq 0$, then by
definition of $N$ in (\ref{eq:Ndef})
we must have $(s,s',s'')=(r,r',r'')$, $j \geq l$,
$j' \geq l'$, $j'' \geq l''$, and each of $j-l$, $j'-l'$, $j''-l''$ is even.
Columns of $N_k'$ must satisfy $(j,j',j'') \in \{(2,1,1),(3,1,2),(3,2,1)\}$,
and this forces $(j,j',j'')=(l,l',l'')$. So the non-zero entry in this row of
$N_k'$ is unique, as claimed.

Then it suffices to check that the
submatrix of rows of $D_k(\eta_*)$ indexed by $((l,s),(l',s'),(l'',s'')) \in
\cH^{(k)}$ where $l',l''<k$ and $k=l=l'+l''$ has full column rank.
This was not exhibited in the proof of Lemma \ref{lemma:S3cryoEMrankincrease}
(which used rows where $l'+l''>l$ strictly)
but we may show this here by a similar argument, assuming now the 
availability of 4 different spherical frequencies: Fix spherical
frequencies $(A,B,C,D)=(1,2,3,4)$, and consider $\eta_*$ satisfying
\[\Type(l',A),\Type(l',B)=0, \quad \Type(l',C),\Type(l',D)=1 \quad
\text{ for all } \quad l' \in \{1,\ldots,k-1\}.\]
Recall that this means
$v_{*,m}^{(l'A)},w_{*,m}^{(l'A)},v_{*,m}^{(l'B)},w_{*,m}^{(l'B)}=0$ unless
$m=l'$, and $v_{*,m}^{(l'C)},w_{*,m}^{(l'C)},v_{*,m}^{(l'D)},w_{*,m}^{(l'D)}=0$
unless $m=l'-1$. Then, for $\partial_{v_m^{(ls)}} B_{(l,s),(l',s'),(l'',s'')},
\partial_{w_m^{(ls)}} B_{(l,s),(l',s'),(l'',s'')}$ to be non-zero, this
requires as in (\ref{eq:cryoEMmcondition})
\begin{equation}\label{eq:projcryoEMmcondition2}
m \in \Big\{\big|(l'-\Type(l',s'))-(l''-\Type(l'',s''))\big|,
(l'-\Type(l',s'))+(l''-\Type(l'',s''))\Big\}.
\end{equation}

For $k=2$, we choose the following 5 rows of $D_k(\eta_*)$, with the following
corresponding values of $m$ satisfying (\ref{eq:projcryoEMmcondition2}):
\pagebreak

\begin{table}[h]
\caption{ }
\begin{tabular}{l|l}
$(l', s')$ and $(l'', s'')$ & Values of $m$ \\ \hline
$(1, C)$ and $(1, C)$ & $0$ \\ 
$(1, A)$ and $(1, A)$ & $0, 2$ \\ 
$(1, B)$ and $(1, B)$ & $0, 2$ \\ 
$(1, A)$ and $(1, C)$ & $1$ \\ 
$(1, B)$ and $(1, D)$ & $1$
\end{tabular}
\end{table}
Ordering the columns by
$v_0^{(2s)},v_2^{(2s)},w_2^{(2s)},v_1^{(2s)},w_1^{(2s)}$, the resulting
$5 \times 5$ submatrix is block lower-triangular with diagonal blocks
\begin{align*}
\partial_{v_0^{(2s)}} B_{(2,s),(1,C),(1,C)}(\eta_*)&=
C_{0,0,0}^{2,1,1} \cdot \left|u_{*,0}^{(1C)}\right|^2\\
\partial_{v_2^{(2s)},w_2^{(2s)}}
\Big(B_{(2,s),(1,A),(1,A)},B_{(2,s),(1,B),(1,B)}\Big)(\eta_*)&=
\begin{psmallmatrix}
2C_{2,-1,1}^{2,1,1} \cdot \Re \overline{u_{*,-1}^{(1A)}} u_{*,1}^{(1A)} &
2C_{2,-1,1}^{2,1,1} \cdot \Im \overline{u_{*,-1}^{(1A)}} u_{*,1}^{(1A)}\\
2C_{2,-1,1}^{2,1,1} \cdot \Re \overline{u_{*,-1}^{(1B)}} u_{*,1}^{(1B)} &
2C_{2,-1,1}^{2,1,1} \cdot \Im \overline{u_{*,-1}^{(1B)}} u_{*,1}^{(1B)}
\end{psmallmatrix}\\
\partial_{v_1^{(2s)},w_1^{(2s)}}
\Big(B_{(2,s),(1,A),(1,C)},B_{(2,s),(1,B),(1,D)}\Big)(\eta_*)&=
\begin{psmallmatrix}
2C_{1,-1,0}^{2,1,1} \cdot \Re \overline{u_{*,-1}^{(1A)}} u_{*,0}^{(1C)} &
2C_{1,-1,0}^{2,1,1} \cdot \Im \overline{u_{*,-1}^{(1A)}} u_{*,0}^{(1C)}\\
2C_{1,-1,0}^{2,1,1} \cdot \Re \overline{u_{*,-1}^{(1B)}} u_{*,0}^{(1D)} &
2C_{1,-1,0}^{2,1,1} \cdot \Im \overline{u_{*,-1}^{(1B)}} u_{*,0}^{(1D)}
\end{psmallmatrix}
\end{align*}
These blocks are generically non-singular, so this submatrix of $D_2(\eta_*)$
has full column rank.

For $k=3$, we choose the following 7 rows of $D_k(\eta_*)$, with the following
corresponding values of $m$ satisfying (\ref{eq:projcryoEMmcondition2}):

\begin{table}[h]
\caption{ }
\begin{tabular}{l|l}
$(l', s')$ and $(l'', s'')$ & values of $m$ \\ \hline
$(1, D)$ and $(2, B)$ & $2$ \\ 
$(1, C)$ and $(2, A)$ & $2$ \\ 
$(1, C)$ and $(2, C)$ & $1$ \\ 
$(1, D)$ and $(2, D)$ & $1$ \\
$(1, A)$ and $(2, A)$ & $1, 3$ \\ 
$(1, B)$ and $(2, B)$ & $1, 3$ \\
$(1, A)$ and $(2, C)$ & $2, 0$ \\ 
\end{tabular}
\end{table}

Ordering the columns by
$v_2^{(3s)},w_2^{(3s)},v_1^{(3s)},w_1^{(3s)},v_3^{(3s)},w_3^{(3s)},w_0^{(3s)}$,
the resulting $7 \times 7$ submatrix is block lower-triangular
with diagonal blocks
\begin{align*}
\partial_{v_2^{(3s)},w_2^{(3s)}}
\Big(B_{(3,s),(1,D),(2,B)},B_{(3,s),(1,C),(2,A)}\Big)(\eta_*)&=
\begin{psmallmatrix}
2C_{2,0,2}^{3,1,2} \cdot \Re \overline{u_{*,0}^{(1D)}} u_{*,2}^{(2B)} &
2C_{2,0,2}^{3,1,2} \cdot \Im \overline{u_{*,0}^{(1D)}} u_{*,2}^{(2B)}\\
2C_{2,0,2}^{3,1,2} \cdot \Re \overline{u_{*,0}^{(1C)}} u_{*,2}^{(2A)} &
2C_{2,0,2}^{3,1,2} \cdot \Im \overline{u_{*,0}^{(1C)}} u_{*,2}^{(2A)}
\end{psmallmatrix}\\
\partial_{v_1^{(3s)},w_1^{(3s)}}
\Big(B_{(3,s),(1,C),(2,C)},B_{(3,s),(1,D),(2,D)}\Big)(\eta_*)&=
\begin{psmallmatrix}
2C_{1,0,1}^{3,1,2} \cdot \Re \overline{u_{*,0}^{(1C)}} u_{*,1}^{(2C)} &
2C_{1,0,1}^{3,1,2} \cdot \Im \overline{u_{*,0}^{(1C)}} u_{*,1}^{(2C)}\\
2C_{1,0,1}^{3,1,2} \cdot \Re \overline{u_{*,0}^{(1D)}} u_{*,1}^{(2D)} &
2C_{1,0,1}^{3,1,2} \cdot \Im \overline{u_{*,0}^{(1D)}} u_{*,1}^{(2D)}
\end{psmallmatrix}\\
\partial_{v_3^{(3s)},w_3^{(3s)}}
\Big(B_{(3,s),(1,A),(2,A)},B_{(3,s),(1,B),(2,B)}\Big)(\eta_*)&=
\begin{psmallmatrix}
2C_{3,-1,2}^{3,1,2} \cdot \Re \overline{u_{*,-1}^{(1A)}} u_{*,2}^{(2A)} &
2C_{3,-1,2}^{3,1,2} \cdot \Im \overline{u_{*,-1}^{(1A)}} u_{*,2}^{(2A)}\\
2C_{3,-1,2}^{3,1,2} \cdot \Re \overline{u_{*,-1}^{(1B)}} u_{*,2}^{(2B)} &
2C_{3,-1,2}^{3,1,2} \cdot \Im \overline{u_{*,-1}^{(1B)}} u_{*,2}^{(2B)}
\end{psmallmatrix}\\
\partial_{w_0^{(3s)}} B_{(3,s),(1,A),(2,C)}(\eta_*)&=
2C_{0,1,1}^{3,1,2} \cdot \Im \overline{u_{*,1}^{(1A)}}u_{*,1}^{2C}
\end{align*}
These blocks are again generically non-singular, so this submatrix of
$D_3(\eta_*)$ has full column rank.

This verifies that $N_k \cdot D_k(\eta_*)$ has full column rank for $k=2,3$.\\

{\bf Part (b), $k \geq 4$:}
As above, we fix $(A,B,C,D)=(1,2,3,4)$ and consider $\eta_*$ satisfying both
$v_{*,m}^{(ks)},w_{*,m}^{(ks)}=0$ for all $m,s$ and
\begin{equation}\label{eq:specialetaproj}
\Type(l',A),\Type(l',B)=0, \quad \Type(l',C),\Type(l',D)=1 \quad
\text{ for all } \quad l' \in \{1,\ldots,k-1\}.
\end{equation}
Columns of $D_k(\eta_*)$
correspond to derivatives in the coordinates $\cV^{(k)}$. We partition
these coordinates into blocks $s=1,\ldots,S_k$ and write
\[D_k(\eta_*)=[D_k^{(s)}(\eta_*):s=1,\ldots,S_k],\]
where columns of each $D_k^{(s)}(\eta_*)$ are indexed by
$\eta_0^{(ks)},v_1^{(ks)},w_1^{(ks)},\ldots,v_k^{(ks)},w_k^{(ks)}$.
It suffices to show that $N_k \cdot D_k^{(s)}(\eta_*)$ has full column rank
$2k+1$ for generic $\eta_*$ satisfying (\ref{eq:specialetaproj}), for
each fixed $s$. We do this by choosing $2k+1$ rows of $N_k$---call this
submatrix $N_k'$---and verifying that the corresponding $(2k+1) \times (2k+1)$
submatrix $N_k' \cdot D_k^{(s)}(\eta_*)$ is non-singular.

The argument for verifying non-singularity is different
from our preceding approaches in Lemmas \ref{lem:S2rankincrease} and
\ref{lemma:S3cryoEMrankincrease}. Let us first explain the high-level idea:
Rather than exhibiting a sparse structure for $N_k' \cdot D_k^{(s)}(\eta_*)$
where the rank may be explicitly checked, we study the determinant
\begin{equation}\label{eq:Pdet}
P(\eta_*)=\det [N_k' \cdot D_k^{(s)}(\eta_*)]
\end{equation}
and show that this is not identically 0 as a polynomial of the non-zero
coordinates of $\eta_*$. We introduce a special degree-$(2k+1)$ monomial
\begin{equation}\label{eq:Mpoly}
M = \Big(w^{(1C)}_0 w^{(1D)}_0 v^{(1A)}_1 v^{(1B)}_1\Big)^2\left(
\prod_{j = 2}^{\lfloor k/2 \rfloor - 1} v^{(jA)}_j v^{(jB)}_j v^{(jC)}_{j - 1}
v^{(jD)}_{j - 1}\right)
v^{(\lfloor k/2 \rfloor C)}_{\lfloor k/2 \rfloor - 1} \Big(v^{(\lfloor k/2
\rfloor A)}_{\lfloor k/2 \rfloor} v^{(\lfloor k/2 \rfloor D)}_{\lfloor k/2 \rfloor - 1}\Big)^{\mathbf{1}\{\text{$k$ odd}\}},
\end{equation}
where all variables appearing in $M$ are coordinates of $\eta_*$ which
are not fixed to be zero. We then write
\begin{equation}\label{eq:PoverM}
P=(P/M) \cdot M+Q
\end{equation}
where $Q$ are the terms of $P$ not divisible by $M$, $(P/M) \cdot
M$ are the terms which \emph{are} divisible by $M$, and $P/M$ denotes
their quotient by $M$. It suffices to show that $P/M$ is not identically 0.

We now describe the choice of $2k+1$ rows of $N_k'$
that allows us to verify this claim $P/M \neq 0$. We restrict to rows
$((-k,s),(q',s'),(q'',s'')) \in \cQ^{(k)}$ of $N_k$ where the first pair is
fixed to be $(-k,s)$, and where $q' \in \{1,\ldots,\lfloor k/2 \rfloor\}$.
This requires $-k+q'+q''=0$, so $q''=k-q' \in \{k-1,\ldots,\lceil k/2 \rceil\}$.
We index such rows by $(q',s'),(q'',s'')$. For any such row $(q',s'),(q'',s'')$,
we apply the following two observations:
\begin{itemize}
\item By definition of $N$ in (\ref{eq:Ndef}),
each non-zero entry in this row of $N_k$ belongs to a column
$((k,s),(l',s'),(l'',s'')) \in \cH^{(k)}$
where $l',l'' \in \{1,\ldots,k-1\}$ and
\begin{equation}\label{eq:lparity}
l' \geq q', \qquad l'' \geq q'', \qquad l'-q',l''-q'' \text{ are even}.
\end{equation}
\item As in (\ref{eq:cryoEMmcondition}), for this row
$((k,s),(l',s'),(l'',s'')) \in \cH^{(k)}$ of $D_k^{(s)}(\eta_*)$,
the entries in the
columns $\partial_{v_m^{(ks)}},\partial_{w_m^{(ks)}}$ can be non-zero only when
\begin{equation}\label{eq:projcryoEMmcondition}
m \in \Big\{\big|(l'-\Type(l',s'))-(l''-\Type(l'',s''))\big|,
(l'-\Type(l',s'))+(l''-\Type(l'',s''))\Big\}.
\end{equation}
\end{itemize}
Combined, these yield the important observation that,
fixing a row $(q',s'),(q'',s'')$ of $N_k \cdot D_k^{(s)}(\eta_*)$ and a pair of
columns $\partial_{v_m^{(ks)}},\partial_{w_m^{(ks)}}$ (or a single column in the
case $m=0$) for a specific index $m \in
\{0,\ldots,k\}$, these two entries (or one entry) of $N_k \cdot D_k^{(s)}(\eta_*)$ are
homogenous degree-2 polynomials, whose degree-2 monomials are each a product of
some variable $v_{\cdot}^{l'\cdot},w_{\cdot}^{l'\cdot}$ and some variable
$v_\cdot^{l''\cdot},w_\cdot^{l''\cdot}$ where
$l',l''$ satisfy both conditions (\ref{eq:lparity}) and
(\ref{eq:projcryoEMmcondition}).

Table \ref{tab:rowchoice} now specifies an explicit choice of
$2k+1$ rows $(q',s'),(q'',s'')$ of
$N_k$ to form $N_k'$, and indicates which columns
$\partial_{v_m^{(ks)}},\partial_{w_m^{(ks)}}$ of each corresponding
row of $N_k \cdot D_k^{(s)}(\eta_*)$
can depend on some variable $v_\cdot^{q'\cdot},w_\cdot^{q'\cdot}$ with the same
spherical frequency as the first row index $q'$.
For example: If $(s',s'',q',q'')=(A,A,1,k-1)$, then (\ref{eq:lparity}) forces
$l''=k-1$. In order for a term of this row to depend on
$v_{\cdot}^{1\cdot},w_{\cdot}^{1\cdot}$, we must then have $l'=1$.
Then $\Type(l',s')=1$ and $\Type(l'',s'')=k-1$, so the condition
(\ref{eq:projcryoEMmcondition}) implies that only columns corresponding to 
$m \in \{k,k-2\}$ can depend on such variables
$v_{\cdot}^{1\cdot},w_{\cdot}^{1\cdot}$. This yields the first row of the
table. If $(s',s'',q',q'')=(C,C,3,k-3)$, then (\ref{eq:lparity}) forces $l'' \in
\{k-3,k-1\}$. For this row to depend on $v_{\cdot}^{3\cdot},w_{\cdot}^{3\cdot}$,
we must have $l'=3$. Then $\Type(l',s')=2$ and $\Type(l'',s'') \in \{k-4,k-2\}$,
so $m \in \{k-6,k-4,k-2,k\}$, and this yields the $13^\text{th}$ row of the
table. The remaining rows are deduced by the same type of reasoning. (The
sequences $(2,4,\ldots)$, $(3,5,\ldots)$, etc.\ in Table \ref{tab:rowchoice}
denote some sequences of
consecutive even and odd integers, whose exact last elements will not be
important for our later arguments.)

\begin{table}
\caption{ }\label{tab:rowchoice}
\begin{tabular}{l|llll|l}
&$s'$ & $s''$ & $q'$ & $q''$ & $m$ s.t.\ 
$\partial_{v_m^{(ks)}},\partial_{w_m^{(ks)}}$ can depend on
$v_\cdot^{q'\cdot},w_\cdot^{q'\cdot}$ \\ \hline
&$A$ & $A$ & $1$ & $k - 1$ & $k, k - 2$ \\
&$B$ & $B$ & $1$ & $k - 1$ & $k, k - 2$ \\
&$C$ & $C$ & $1$ & $k - 1$ & $k - 2$ \\
&$D$ & $D$ & $1$ & $k - 1$ & $k - 2$ \\
&$A$ & $C$ & $1$ & $k - 1$ & $k - 1, k - 3$ \\
&$B$ & $D$ & $1$ & $k - 1$ & $k - 1, k - 3$ \\ 
&$C$ & $A$ & $1$ & $k - 1$ & $k - 1$ \\
&$D$ & $B$ & $1$ & $k - 1$ & $k - 1$ \\ \hline
&$C$ & $C$ & $2$ & $k - 2$ & $k - 4$ and $k-2$\\
&$D$ & $D$ & $2$ & $k - 2$ & $k - 4$ and $k-2$\\
&$A$ & $C$ & $2$ & $k - 2$ & $k - 5$ and $k-1$ \\
&$B$ & $D$ & $2$ & $k - 2$ & $k - 5$ and $k-1$ \\ \hline
&$C$ & $C$ & $3$ & $k - 3$ & $k - 6, k - 4, k - 2$ and $k$ \\
&$D$ & $D$ & $3$ & $k - 3$ & $k - 6, k - 4, k - 2$ and $k$ \\
&$A$ & $C$ & $3$ & $k - 3$ & $k - 7, k - 5$ and $k-1$ \\
&$B$ & $D$ & $3$ & $k - 3$ & $k - 7, k - 5$ and $k-1$ \\ \hline
& & & \vdots & \vdots & \vdots \\ \hline
& $C$ & $C$ & $\lfloor k/2 \rfloor - 1$ & $\lceil k/2 \rceil + 1$ &
$2,4,\ldots$ and $k-2,k$ if $k$ even; $3,5,\ldots$ and $k-2,k$ if $k$ odd \\
& $D$ & $D$ & $\lfloor k/2 \rfloor - 1$ & $\lceil k/2 \rceil + 1$ &
$2,4,\ldots$ and $k-2,k$ if $k$ even; $3,5,\ldots$ and $k-2,k$ if $k$ odd \\
& $A$ & $C$ & $\lfloor k/2 \rfloor - 1$ & $\lceil k/2 \rceil + 1$ &
$1,3,\ldots$ and $k-1$ if $k$ even; $2,4,\ldots$ and $k-1$ if $k$ odd \\
& $B$ & $D$ & $\lfloor k/2 \rfloor - 1$ & $\lceil k/2 \rceil + 1$ &
$1,3,\ldots$ and $k-1$ if $k$ even; $2,4,\ldots$ and $k-1$ if $k$ odd \\ \hline
& $C$ & $C$ & $\lfloor k/2 \rfloor$ & $\lceil k/2 \rceil$ &
$0,2,\ldots$ and $k-2,k$ if $k$ even; $1,3,\ldots$ and $k-2,k$ if $k$ odd \\
(if $k$ odd) &$D$ & $D$ & $\lfloor k/2 \rfloor$ & $\lceil k/2 \rceil$ &
$1,3,\ldots$ and $k-2,k$ \\
(if $k$ odd) &$A$ & $C$ & $\lfloor k/2 \rfloor$ & $\lceil k/2 \rceil$ &
$0,2,\ldots$ and $k-1$
\end{tabular}
\end{table}

Order the rows of $N_k' \cdot D_k^{(s)}(\eta_*)$ in the order listed in Table
\ref{tab:rowchoice}, and the columns in
the ordering of \emph{decreasing $m$}:
\[v_k^{(ks)},w_k^{(ks)},v_{k-1}^{(ks)},w_{k-1}^{(ks)},v_{k-2}^{(ks)},w_{k-2}^{(ks)}\ldots,
v_1^{(ks)},w_1^{(ks)},\eta_0^{(ks)}.\]
Consider the block decomposition for both rows and columns of $N_k' \cdot D_k^{(s)}(\eta_*)$ with respect to
\[2k+1=8+4+4+\ldots+4+\begin{cases} 1 & \text{ if } k \text{ even } \\
3 & \text{ if } k \text{ odd. } \end{cases}\]
Let $P=\det(N_k' \cdot D_k^{(s)}(\eta_*))$ as defined in (\ref{eq:Pdet}),
and let $P_1,P_2,\ldots,P_{\lfloor k/2\rfloor}$ be the determinants
of the diagonal blocks in this decomposition.
Recall that we wish to show $P/M \neq 0$. Let us factor $M$
in (\ref{eq:Mpoly}) correspondingly as
$M=M_1M_2\ldots M_{\lfloor k/2 \rfloor}$ where
\begin{align*}
M_1&=(w^{(1C)}_0 w^{(1D)}_0 v^{(1A)}_1 v^{(1B)}_1)^2\\
M_j&=v_j^{(jA)}v_j^{(jB)}v_{j-1}^{(jC)}v_{j-1}^{(jD)} \quad \text{ for } j=2,\ldots,
\lfloor k/2 \rfloor-1\\
M_{\lfloor k/2 \rfloor}&=v^{(\lfloor k/2 \rfloor C)}_{\lfloor k/2 \rfloor - 1}
\Big(v^{(\lfloor k/2 \rfloor A)}_{\lfloor k/2 \rfloor} v^{(\lfloor k/2 \rfloor D)}_{\lfloor k/2 \rfloor
- 1}\Big)^{\mathbf{1}\{\text{$k$ odd}\}}.
\end{align*}
The degrees of $M_1,M_2,M_3,\ldots$ coincide with the above block sizes
$8,4,4,\ldots$. Furthermore, each $M_j$ depends on only variables
$v_\cdot^{j\cdot},w_\cdot^{j\cdot}$ having spherical frequency $j$.

Observe now that:
\begin{itemize}
\item Only the $8 \times 8$ upper-left diagonal block of $N_k' \cdot
D_k^{(s)}(\eta_*)$ has entries depending on variables
$v_\cdot^{1\cdot},w_\cdot^{1\cdot}$. This is
because all monomials in rows 9 onwards are a product of some
$v_\cdot^{l'\cdot},w_\cdot^{l'\cdot}$ with some
$v_\cdot^{l''\cdot},w_\cdot^{l''\cdot}$ where
$l' \geq q' \geq 2$ and $l'' \geq q'' \geq 2$. In the first 8 rows,
the table indicates that only the first 8 columns (corresponding to
$m=k,k-1,k-2,k-3$) can depend on $v_\cdot^{1\cdot},w_\cdot^{1\cdot}$.
\item Furthermore, any degree-2 monomial in this 
$8 \times 8$ block that depends on
$v_\cdot^{1\cdot},w_\cdot^{1\cdot}$ must
have as its second variable $v_\cdot^{l''\cdot},w_\cdot^{l''\cdot}$ for some
$l'' \geq k-1>\lfloor k/2 \rfloor$ strictly.
\item Removing this first row block and column block of size 8, only the
$4 \times 4$ upper-left diagonal block of the remaining matrix has entries
depending on $v_\cdot^{2\cdot},w_\cdot^{2\cdot}$, by the same reasoning.
Furthermore, any degree-2 monomial in this $4 \times 4$ block that depends on
$v_\cdot^{2\cdot},w_\cdot^{2\cdot}$ must have as its second variable
$v_\cdot^{l''\cdot},w_\cdot^{l''\cdot}$ for some $l'' \geq k-2>
\lfloor k/2 \rfloor$ strictly.
\item Removing also this second row block and column block of size 4, 
only the $4 \times 4$ upper-left remaining diagonal block has entries
depending on $v_\cdot^{3\cdot},w_\cdot^{3\cdot}$, etc. This argument can be
continued inductively until the last block.
\end{itemize}
These observations imply that the terms of $P$ divisible by $M$ must have
the factorization
\begin{equation}\label{eq:PoverMfactorization}
P/M=(P_1/M_1)(P_2/M_2)\ldots (P_{\lfloor k/2 \rfloor}/M_{\lfloor k/2 \rfloor}),
\end{equation}
where, analogously to (\ref{eq:PoverM}), each factor $P_j/M_j$ is the
polynomial that is the quotient by $M_j$ of those terms of
$P_j$ which are exactly divisible by $M_j$.

To complete the proof, we check by direct computation that each polynomial
$P_j/M_j$ on the right side of (\ref{eq:PoverMfactorization}) is non-zero.

{\bf Verification that $P_1/M_1 \neq 0$:}
Consider, as an example, the entry in the first row $(A,A,1,k-1)$ and
first column $\partial_{v_k^{(ks)}}$ of $N_k' \cdot D_k^{(s)}(\eta_*)$.
This entry is the inner product
\[\big(N_{(-k,s),(1,A),(k-1,A)}\big)^\top \big(\partial_{v_k^{(ks)}}
B(\eta_*)\big)
=\!\!\!\!\!\!\!\!\sum_{((k,s),(l',s'),(l'',s'')) \in \cH^{(k)}}\!\!\!
N_{(-k,s),(1,A),(k-1,A)}^{(k,s),(l',s'),(l'',s'')} \cdot
\partial_{v_k^{(ks)}} B_{(k,s),(l',s'),(l'',s'')}(\eta_*).\]
Importantly, only the single term indexed by $((k,s),(l',s'),(l'',s''))
=((k,s),(1,A),(k-1,A))$ of this sum depends on any variable appearing in the
monomial $M_1$. This is because, from the definition of $N$ in (\ref{eq:Ndef}),
all other non-zero entries
$N_{(-k,s),(1,A),(k-1,A)}^{(k,s),(l',s'),(l'',s'')}$ of this row of $N$
have $l' \geq 3$ and $l''=k-1$. We introduce the shorthand
\[Y_q=N_{(-k,s),(q,s'),(k-q,s'')}^{(k,s),(q,s'),(k-q,s'')},\]
where this notation fixes $k$ and uses that this value does not actually
depend on $(s,s',s'')$. Note that by (\ref{eq:Ndef}) and the non-vanishing of
Clebsch-Gordan coefficients in Lemma \ref{lemma:CGnonvanishing}, $Y_q \neq
0$ for every $q=0,1,\ldots,k$. Then, applying (\ref{eq:dervmlsI}),
(\ref{eq:specialetaproj}), and the sign symmetry (\ref{eq:usymmetry}),
the above single term is
\begin{align*}
N_{(-k,s),(1,A),(k-1,A)}^{(k,s),(1,A),(k-1,A)}
\cdot \partial_{v_k^{(ks)}} B_{(k,s),(1,A),(k-1,A)}(\eta_*)
&=Y_1 \cdot 2C_{k,-1,k-1}^{k,1,k-1} \Re \overline{u_{-1}^{(1A)}}
u_{k-1}^{(k-1,A)}\\
&=2Y_1C_{k,-1,k-1}^{k,1,k-1}\Big(v_1^{(1A)}v_{k-1}^{(k-1,A)}
-w_1^{(1A)}w_{k-1}^{(k-1,A)}\Big).
\end{align*}
Only the first of these two summands depends on a variable in
$M_1$, namely $v_1^{(1A)}$. We write its quotient by this variable
$v_1^{(1A)}$ in the upper-left entry of the first table below.

As a second example, consider the entry in row $(C,C,1,k-1)$ and column
$\partial_{v_{k-2}^{(ks)}}$ of $N_k' \cdot D_k^{(s)}(\eta_*)$. By the same
reasoning as above, the only term of
this entry which depends on a variable in $M_1$ is
\begin{align*}
N_{(-k,s),(1,C),(k-1,C)}^{(k,s),(1,C),(k-1,C)}
\cdot \partial_{v_{k-2}}^{(ks)} B_{(k,s),(1,C),(k-1,C)}
&=Y_1 \cdot 2C^{k,1,k-1}_{k-2,0,k-2} \cdot \Re \overline{u^{(1C)}_0}
u^{(k-1,C)}_{k-2}\\
&=2Y_1C^{k,1,k-1}_{k-2,0,k-2} \cdot w_0^{(1C)}w_{k-2}^{(k-1,C)}.
\end{align*}
Its quotient by the variable $w_0^{(1C)}$ appearing in $M_1$ is the $(3,3)$
entry of the first table below.

The entries of the $8 \times 8$ block for $P_1$ which
depend on some variable in $M_1$ are contained within two $4 \times 4$
submatrices, corresponding to the below two tables. Similar to the above
computations, each entry of each submatrix has at most 1 term depending on some
variable in $M_1$. We indicate the quotient of this term by the
corresponding variable of $M_1$ in the two tables below. For entries that
have no dependence on variables of $M_1$, we write this quotient as 0.

{\small\setlength\tabcolsep{1pt}
\begin{center}
\begin{tabular}{l|cccccc}
& $v^{ks}_k$ & $w^{ks}_k$ & $v^{k, s}_{k-2}$ & $w^{k, s}_{k-2}$ \\ \hline
$A,A,1,k-1$ & $2Y_1 C^{k, 1, k - 1}_{k, -1, k - 1} v^{k-1, A}_{k-1}$ & $2Y_1 C^{k, 1, k - 1}_{k, -1, k - 1} w^{k-1, A}_{k-1}$&
$2Y_1 C^{k, 1, k - 1}_{k - 2, 1, k - 1} v^{k-1, A}_{k-1}$ & $2Y_1 C^{k, 1, k -
1}_{k - 2, 1, k - 1} w^{k-1, A}_{k-1}$ \\
$B,B,1,k-1$ & $2Y_1 C^{k, 1, k - 1}_{k, -1, k - 1} v^{k-1, B}_{k-1}$ & $2Y_1 C^{k, 1, k - 1}_{k, -1, k - 1} w^{k-1, B}_{k-1}$&
$2Y_1 C^{k, 1, k - 1}_{k - 2, 1, k - 1} v^{k-1, B}_{k-1}$ & $2Y_1 C^{k, 1, k -
1}_{k - 2, 1, k - 1} w^{k-1, B}_{k-1}$ \\
$C,C,1,k-1$ & 0 & 0 & $2Y_1 C^{k, 1, k - 1}_{k-2, 0, k - 2} w^{k-1,
C}_{k-2}$ & $-2Y_1 C^{k, 1, k - 1}_{k-2, 0, k - 2} v^{k-1, C}_{k-2}$ \\
$D,D,1,k-1$ & 0 & 0 & $2Y_1 C^{k, 1, k - 1}_{k-2, 0, k - 2} w^{k-1,
D}_{k-2}$ & $-2Y_1
C^{k, 1, k - 1}_{k-2, 0, k - 2} v^{k-1, D}_{k-2}$
\end{tabular}
\begin{tabular}{l|cccccc}
& $v^{k,s}_{k-1}$ & $w^{k, s}_{k-1}$ &  $v^{k, s}_{k-3}$ & $w^{k, s}_{k-3}$ & \\ \hline
$A,C,1,k-1$ & $2Y_1 C^{k, 1, k - 1}_{k - 1, -1, k - 2} v^{k-1, C}_{k-2}$ & $2Y_1 C^{k, 1, k - 1}_{k - 1, -1, k - 2} w^{k-1, C}_{k-2}$
& $2Y_1 C^{k, 1, k - 1}_{k - 3, 1, k - 2} v^{k-1, C}_{k-2}$ & $2Y_1 C^{k, 1, k -
1}_{k - 3, 1, k - 2} w^{k-1, C}_{k-2}$ \\
$B,D,1,k-1$ & $2Y_1 C^{k, 1, k - 1}_{k - 1, -1, k - 2} v^{k-1, D}_{k-2}$ & $2Y_1 C^{k, 1, k - 1}_{k - 1, -1, k - 2} w^{k-1, D}_{k-2}$
& $2Y_1 C^{k, 1, k - 1}_{k - 3, 1, k - 2} v^{k-1, D}_{k-2}$ & $2Y_1 C^{k, 1, k -
1}_{k - 3, 1, k - 2} w^{k-1, D}_{k-2}$ \\
$C,A,1,k-1$ & $2Y_1C^{k, 1, k - 1}_{k - 1, 0, k - 1} w^{k-1,
A}_{k-1}$ & $-2Y_1C^{k, 1, k - 1}_{k - 1, 0, k - 1} v^{k-1, A}_{k-1}$ & 0 & 0 \\
$D,B,1,k-1$ & $2Y_1C^{k, 1, k - 1}_{k - 1, 0, k - 1} w^{k-1,
B}_{k-1}$ & $-2Y_1C^{k, 1, k - 1}_{k - 1, 0, k - 1} v^{k-1, B}_{k-1}$ & 0 & 0
\end{tabular}
\end{center}}

Now $P_1/M_1$ is the product of determinants of the above
two $4 \times 4$ matrices. By Lemma \ref{lemma:CGnonvanishing}, each
Clebsch-Gordan coefficient here is non-zero. Then the determinant of each
$4 \times 4$ matrix is the product of two
$2 \times 2$ determinants, each of which is a non-vanishing quadratic, so
$P_1/M_1 \neq 0$.

{\bf Verification that $P_j/M_j \neq 0$ for $j=2,\ldots,\lfloor k/2 \rfloor-1$:}
Consider, as an example, the entry of row $(C,C,3,k-3)$ and column
$\partial_{w_{k-6}}^{(ks)}$. This entry is the inner product
\[\big(N_{(-k,s),(3,C),(k-3,C)}\big)^\top\big(\partial_{w_{k-6}^{(ks)}}
B(\eta_*)\big).\]
The non-zero entries $N_{(-k,s),(3,C),(k-3,C)}^{(k,s),(l',C),(l'',C)}$ of this
row of $N$ must have $l' \geq 3$ odd and $l'' \in \{k-1,k-3\}$. We must then
have $l'=3$ in order for the corresponding entry $\partial_{w_{k-6}^{(ks)}}
B_{(k,s),(l',C),(l'',C)}(\eta_*)$ to depend on some variable of $M_3$, which
would be $v_2^{(3C)}$. Recalling the form of this
derivative in (\ref{eq:derwmlsI}), this forces $|m'|=2$ for the summation index
of (\ref{eq:derwmlsI}) corresponding to any term depends
on $v_2^{(3C)}$. Then $m+m' \in \{k-8,k-4\}$, since $m=k-6$.
Applying our specialization
(\ref{eq:projcryoEMmcondition}) to the derivative (\ref{eq:derwmlsI}),
this requires $l'' \in \{k-7,k-3\}$ in order for this term
to be non-zero. Combining with the above condition $l'' \in \{k-1,k-3\}$,
we must have $l''=k-3$. Thus, to
summarize, again only a single term of the sum constituting the above
inner-product depends on the variable $v_2^{(3C)}$
of $M_3$. This term is
\begin{align*}
N_{(-k,s),(3,C),(k-3,C)}^{(k,s),(3,C),(k-3,C)}
\cdot \partial_{w_{k-6}^{(ks)}} B_{(k,s),(3,C),(k-3,C)}(\eta_*)
&=Y_3 \cdot 2C_{k-6,2,k-4}^{k,3,k-3} \Im\overline{u_2^{(3C)}}u_{k-4}^{(k-3,C)}\\
&=2Y_3 C_{k-6,2,k-4}^{k,3,k-3}\Big(v_2^{(3C)}w_{k-4}^{(k-3,C)}
-w_2^{(3C)}v_{k-4}^{(k-3,C)}\Big).
\end{align*}
Only the first summand depends on $v_2^{(3C)}$, and its quotient by $v_2^{(3C)}$
is recorded in row 1 and column 2 of the table below, corresponding to $j=3$.

By the same reasoning, a similar simplification occurs for every
$j=2,\ldots,\lfloor k/2 \rfloor-1$ and every entry
of $P_j/M_j$. For general $j$, we may compute each entry of this $4 \times 4$ block that
depends on a variable of $M_j$, and the table records the quotient of this
entry by the corresponding variable of $M_j$.
\begin{center}
{\scriptsize\setlength\tabcolsep{1pt}
\begin{tabular}{l|cccc}
& $v^{k, s}_{k - 2j}$ & $w^{k, s}_{k - 2j}$ & $v^{k, s}_{k - 2j - 1}$ & $w^{k, s}_{k - 2j - 1}$ \\ \hline
$C, C, j, k-j$ & $2Y_j C^{k, j, k - j}_{k - 2j, j-1, k - j -1} 
v^{k - j, C}_{k - j-1}$ & $2Y_j C^{k, j, k - j}_{k - 2j, j-1, k - j -1}
w^{k - j, C}_{k - j-1}$ & 0 & 0 \\
$D, D, j, k-j$ & $2Y_j C^{k, j, k - j}_{k - 2j, j-1, k - j -1} 
v^{k - j, D}_{k - j-1}$ & $2Y_j C^{k, j, k - j}_{k - 2j, j-1, k - j -1}
w^{k - j, D}_{k - j-1}$ & 0 & 0 \\
$A, C, j, k-j$ & 0 & 0 & $2Y_j C^{k, j, k - j}_{k - 2j - 1, j, k - j - 1} 
v^{k - j, C}_{k - j - 1}$ & $2Y_j C^{k, j, k - j}_{k - 2j - 1, j, k - j
- 1}  w^{k - j, C}_{k - j - 1}$ \\
$B, D, j, k-j$ & 0 & 0 & $2Y_j C^{k, j, k - j}_{k - 2j - 1, j, k - j - 1} 
v^{k - j, D}_{k - j - 1}$ & $2Y_j C^{k, j, k - j}_{k - 2j - 1, j, k - j
- 1}  w^{k - j, D}_{k - j - 1}$
\end{tabular}}
\end{center}
Then $P_j/M_j$ is the determinant of this $4 \times 4$ matrix, which is a
product of two $2 \times 2$ determinants.
By Lemma \ref{lemma:CGnonvanishing}, these Clebsch-Gordan coefficients are
non-zero, so each $2 \times 2$ determinant is a non-vanishing quadratic, and
$P_j/M_j \neq 0$.

{\bf Verification that $P_{\lfloor k/2 \rfloor}/M_{\lfloor k/2 \rfloor} \neq 0$:}
If $k$ is even, we have $\eta_0^{(ks)}=v_0^{(ks)}$, and
\[P_{k/2}(\eta_*)=\big(N_{(-k,s),(k/2,C),(k/2,C)}\big)^\top
\big(\partial_{v_0^{(ks)}} B(\eta_*)\big).\]
The non-zero
elements $N_{(-k,s),(k/2,C),(k/2,C)}^{(k,s),(l',C),(l'',C)}$ have $l' \geq k/2$ 
and $l'' \geq k/2$ with $l'-k/2$ and $l'-k/2$ both even. From
(\ref{eq:dereta0lsI}) and the specialization (\ref{eq:projcryoEMmcondition}),
the only term of this inner-product depending on 
$M_{k/2}=v_{(k/2)-1}^{(k/2)C}$ arises from $l'=l''=k/2$, and this term is
\begin{align*}
&N_{(-k,s),(k/2,C),(k/2,C)}^{(k,s),(k/2,C),(k/2,C)}
\cdot \partial_{v_0^{(ks)}} B_{(k,s),(k/2,C),(k/2,C)}(\eta_*)\\
&=Y_{k/2}C_{0,k/2-1,k/2-1}^{k,k/2,k/2}\left|u_{k/2-1}^{(k/2)C}\right|^2
=Y_{k/2}C_{0,k/2-1,k/2-1}^{k,k/2,k/2}\left((v_{k/2-1}^{(k/2)C})^2
+(w_{k/2-1}^{(k/2)C})^2\right).
\end{align*}
So $P_{k/2}/M_{k/2}=Y_{k/2}C_{0,k/2-1,k/2-1}^{k,k/2,k/2} \cdot
v_{k/2-1}^{(k/2)C}$, which is non-zero.

If $k$ is odd, then $\eta_0^{(ks)}=w_0^{(ks)}$. In the $3 \times 3$ submatrix
corresponding to $P_{\lfloor k/2 \rfloor}$, again each entry has at most 1 term
depending on some variable of $M_{\lfloor k/2 \rfloor}$. The below table
records the quotient of this term by the corresponding variable.
\begin{center}
{\scriptsize\setlength\tabcolsep{1pt}
\begin{tabular}{l|cccc}
& $v^{k, s}_1$ & $w^{k, s}_1$ & $w^{k, s}_0$ \\ \hline
$C, C, \lfloor k/2 \rfloor,\lceil k/2 \rceil$ & $2Y_{\lfloor k/2 \rfloor} C^{k, \lfloor k/2 \rfloor, \lceil k/2 \rceil}_{1, \lfloor k/2
\rfloor-1, \lceil k/2 \rceil-1} v^{\lceil k/2 \rceil, C}_{\lceil k/2 \rceil -
1}$ &
$2Y_{\lfloor k/2 \rfloor} C^{k, \lfloor k/2 \rfloor, \lceil k/2
\rceil}_{1, \lfloor k/2 \rfloor-1, \lceil k/2 \rceil-1} w^{\lceil k/2 \rceil,
C}_{\lceil k/2 \rceil - 1}$ & 0 \\
$D, D, \lfloor k/2 \rfloor, \lceil k/2 \rceil$ & $2Y_{\lfloor k/2
\rfloor} C^{k, \lfloor k/2 \rfloor, \lceil k/2 \rceil}_{1, \lfloor k/2
\rfloor-1, \lceil k/2 \rceil-1} v^{\lceil k/2 \rceil, D}_{\lceil k/2 \rceil -
1}$ &
$2Y_{\lfloor k/2 \rfloor} C^{k, \lfloor k/2 \rfloor, \lceil k/2
\rceil}_{1, \lfloor k/2 \rfloor-1, \lceil k/2 \rceil-1} w^{\lceil k/2 \rceil,
D}_{\lceil k/2 \rceil - 1}$ & 0 \\
$A, C, \lfloor k/2 \rfloor, \lceil k/2 \rceil$ & 0 & 0 & $2Y_{\lfloor
k/2 \rfloor}C^{k, \lfloor k/2 \rfloor, \lceil k/2 \rceil}_{0, \lfloor k/2
\rfloor, \lfloor k/2 \rfloor} w^{\lceil k/2 \rceil, C}_{\lceil k/2 \rceil-1}$
\end{tabular}}
\end{center}
Then $P_{\lfloor k/2 \rfloor}/M_{\lfloor k/2 \rfloor}$ is the determinant of
this $3 \times 3$ matrix, which is non-zero.

Combining the above, we have shown $P/M \neq 0$ as desired. This
completes the proof of part (b) also for $k \geq 4$.
\end{proof}

\section{Details of the numerical simulations}\label{appendix:simulations}

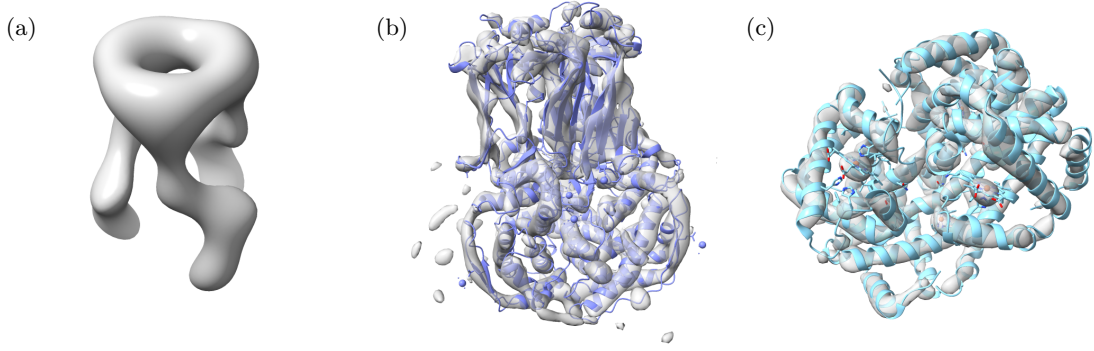
\begin{figure}
\begin{picture}(420,130)
\put(0,0){\includegraphics[width=0.3\textwidth]{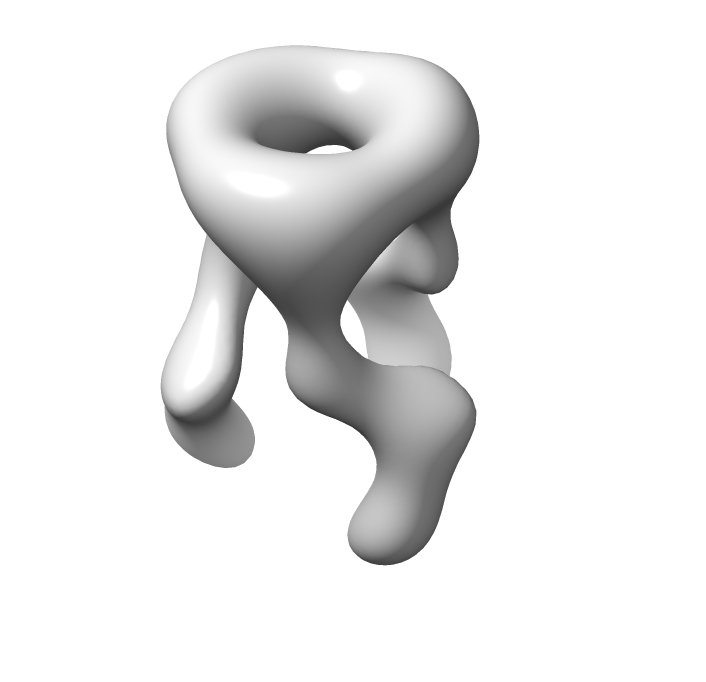}}
\put(140,0){\includegraphics[width=0.3\textwidth]{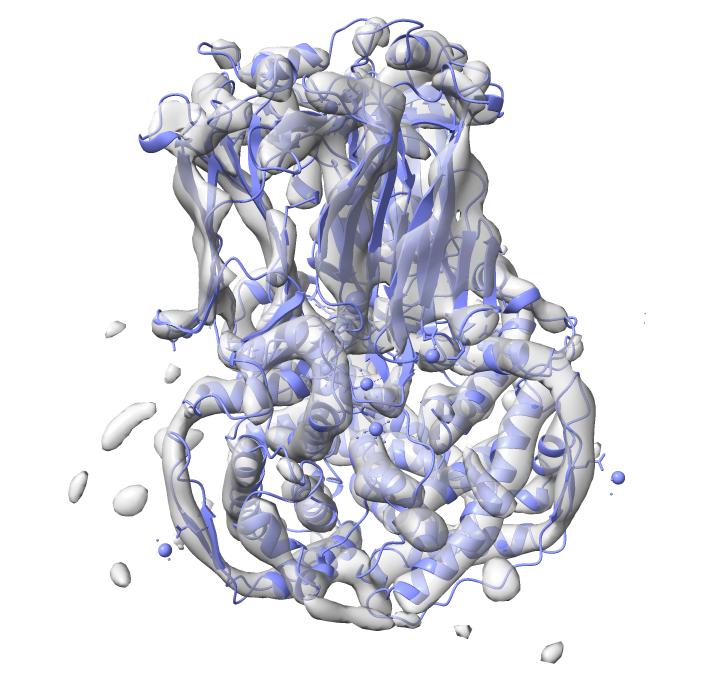}}
\put(280,0){\includegraphics[width=0.3\textwidth]{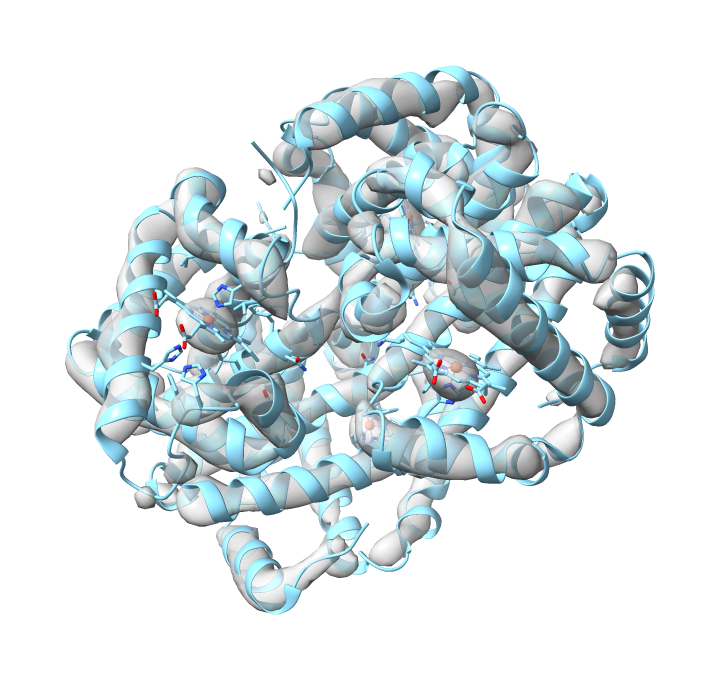}}
\put(0,120){{\small(a)}}
\put(140,120){{\small(b)}}
\put(280,120){{\small(c)}}
\end{picture}
\caption{(a) 24.6\AA-resolution and (b) 8.2\AA-resolution low-pass filtered maps
for the rotavirus VP6 trimer, prior to performing basis approximation as
depicted in Figure \ref{fig:rotavirus}. (c) 7.0\AA-resolution
low-pass filtered map for hemoglobin, prior to performing basis approximation as
depicted in Figure \ref{fig:hemoglobin}.}\label{fig:lowpass}
\end{figure}

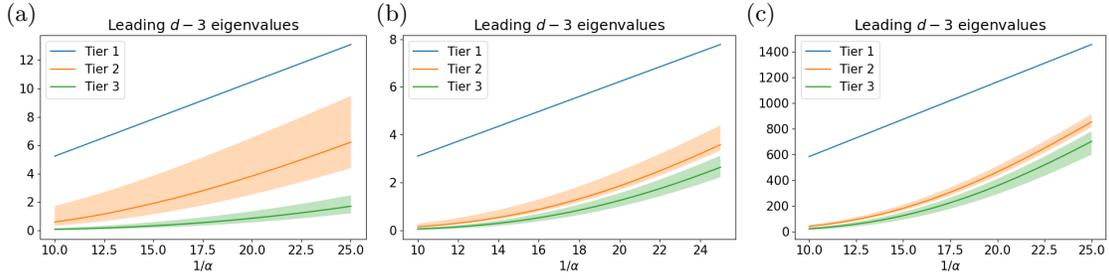
\begin{figure}
\begin{picture}(420,100)
\put(0,0){\includegraphics[width=0.3\textwidth]{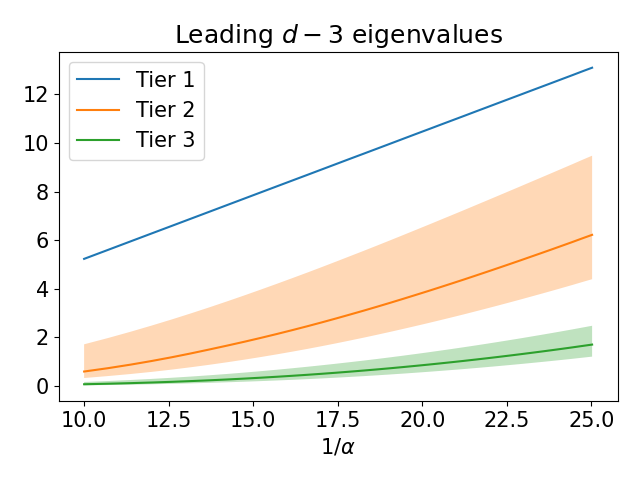}}
\put(140,0){\includegraphics[width=0.3\textwidth]{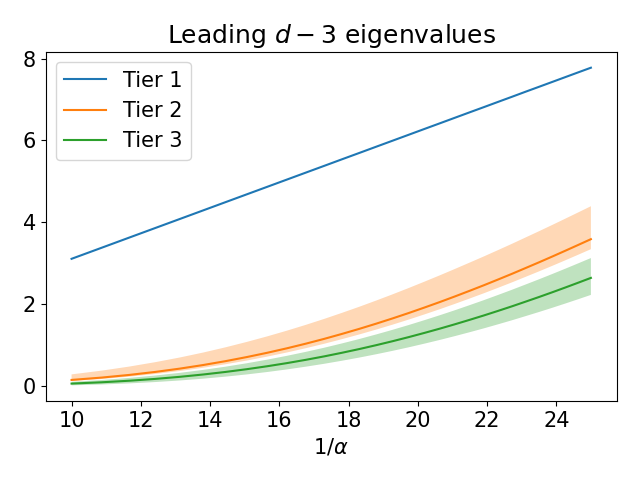}}
\put(280,0){\includegraphics[width=0.3\textwidth]{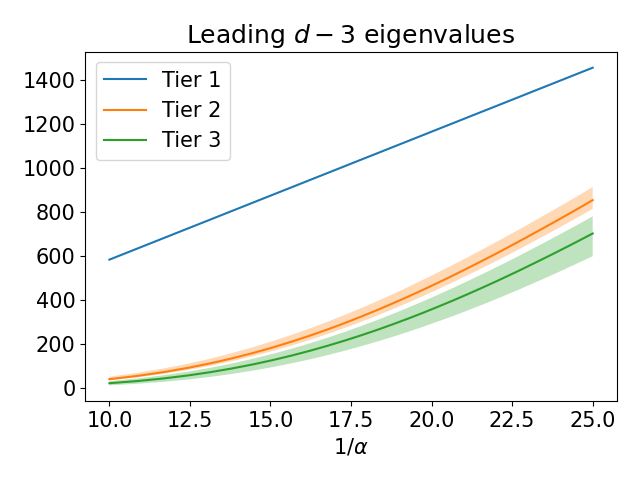}}
\put(0,100){{\small(a)}}
\put(140,100){{\small(b)}}
\put(280,100){{\small(c)}}
\end{picture}
\caption{Leading $d-3$ eigenvalues of the observed Fisher information matrices
depicted in Figures \ref{fig:rotavirus} and \ref{fig:hemoglobin}, plotted
against a common scaling $1/\alpha \propto \sigma^{-2}$, for (a) the
405-dimensional approximation for rotavirus VP6, (b) the 4410-dimensional
approximation for rotavirus VP6, and (c) the 3528-dimensional approximation for
hemoglobin. Lines depict the median within each of the three tiers, and bands
depict the 10th to 90th percentiles.}\label{fig:alltiers}
\end{figure}

We used rotavirus VP6 and hemoglobin maps
publicly available on EMDB (EMDB-1461 and EMDB-3650). We recentered the
rotavirus map EMDB-1461 to have center-of-mass at the origin, and zero-padded it
into a cubical volume of dimensions $141 \times 141 \times 141$. The hemoglobin
volume EMDB-3650 is already cubical. We centered the values of both maps to
have mean value 0.\\

\noindent {\bf Fourier quadrature.} 
We computed the Fourier transform $\hat{f}$ of both maps over a spherical grid
in the Fourier domain, using the FINUFFT library developed
in \cite{barnett2019parallel}. Parametrizing Fourier space by spherical
coordinates $(\rho,\phi_1,\phi_2)$, we computed $\hat{f}$ on a linearly spaced
grid of $150 \times 250 \times 250$ points $(\rho,\phi_1,\phi_2) \in
[0,1/R] \times [0,\pi] \times [0,2\pi)$,
where $R$ is the spatial-domain distance from the origin to the boundary of the 
cubical volume. All integrals in the Fourier domain were computed using the
weighted quadrature defined by this discrete grid with weight proportional to
$\rho^2 \sin \phi_1$.\\

\noindent {\bf Low-pass filter and basis approximation.}
For each frequency threshold $v$, we performed low-pass filtering by simple
truncation of the Fourier transform to radii $\rho \in [0,v]$. We then
iteratively defined radial functions $z_1,\ldots,z_S:[0,v] \to \R$
satisfying the orthogonality (\ref{eq:zorthogonality}), so that for each
$s=1,\ldots,S$, the partial basis $\{z_1,\ldots,z_s\}$ maximizes the
total power of the projection of the Fourier transform $\hat{f}$ onto the
function space
\[\Big\{z_1(\rho)h_1(\phi_1,\phi_2)+\ldots+z_s(\rho)h_s(\phi_1,\phi_2):
h_1,\ldots,h_s \in L_2(\sS^2,\C)\Big\}.\]

In detail, let us denote $u=(\phi_1,\phi_2)$, and $\der u=\sin \phi_1\,\der
\phi_1\,\der \phi_2$ as the surface area measure on $\sS^2$.
Then the projection of $\hat{f}$ onto the above space is defined
explicitly by
\[h_i(u)=\int_0^v \hat{f}(\rho,u)\overline{z_i(\rho)}
\cdot \rho^2 \der \rho \quad \text{ for each } i=1,\ldots,s.\]
The projected power is then
\begin{align*}
\text{Power}&=\int_{\sS^2}\int_0^v \left|\sum_{i=1}^s z_i(\rho)h_i(u)\right|^2
\cdot \rho^2\,\der \rho\,\der u
=\int_{\sS^2} \sum_{i=1}^s |h_i(u)|^2\,\der u\\
&=\int_0^v\int_0^v \sum_{i=1}^s \overline{z_i(\rho)} C(\rho,\rho')z_i(\rho')
\cdot \rho^2\,\der \rho\cdot {\rho'}^2\,\der \rho'
\end{align*}
where we have defined the cross-covariance of $\hat{f}(\rho,\cdot)$ and
$\hat{f}(\rho',\cdot)$ as
\[C(\rho,\rho')=\int_{\sS^2} \hat{f}(\rho,u)\overline{\hat{f}(\rho',u)}\der u.\]
From the orthogonality (\ref{eq:zorthogonality}), the maximizing functions
$z_1,\ldots,z_S$ are such that $\{\rho\,z_1(\rho),\ldots,\rho\,z_S(\rho)\}$
are the $S$ leading eigenfunctions (orthogonal with respect to the standard
unweighted $L_2$-inner-product on $[0,v]$) of the weighted cross-covariance kernel
\[K(g,h)=\int_0^v \int_0^v \overline{g(\rho)}C(\rho,\rho')h(\rho)
\cdot \rho\,\der \rho \cdot \rho'\der \rho'.\]
We approximated this kernel $K$ by its $M \times M$ matrix discretization
$K_{\text{mat}}=(C(\rho,\rho')\rho\,\rho')_{\rho,\rho'}$ where $M$ is the
number of radial quadrature points $\rho \in [0,v]$. We approximated its
eigenfunctions by the eigenvectors of $K_{\text{mat}}$. As the eigenfunctions
of $K$ correspond to $\rho\,z_i(\rho)$, we divided the
eigenvectors of $K_{\text{mat}}$
by $\rho$ to obtain the values of the radial basis functions
$z_1,\ldots,z_S$ along the above radial quadrature.

The final function basis over $\R^3$ was obtained as a product of 
$\{z_s:s=1,\ldots,S\}$ with the spherical harmonics as described
in Section \ref{subsec:cryoEM}. We computed the spherical harmonics on the
above $250 \times 250$ quadrature points $(\phi_1,\phi_2)$ using
the implementation of the spherical
harmonics provided by the \texttt{sph\_harm} function in scipy. Finally, basis
coefficients $\theta_*$ were computed by integration in the Fourier domain,
as approximated by the above quadrature.\\

\noindent {\bf $\SO(3)$ quadrature.}
We computed the empirical Hessian $\nabla^2 R_n(\theta_*)$ by approximating the
integral over $\SO(3)$ in the definition of the log-likelihood using a weighted
discrete quadrature on $\SO(3)$. Parametrizing $\SO(3)$ by the Euler angles
$(\alpha,\beta,\gamma)$, we used a discrete grid of $40 \times 40 \times 40$
values $(\alpha,\beta,\gamma) \in [0,2\pi) \times [0,\pi] \times [0,2\pi)$,
with linearly-spaced grid points and equal weights for $(\alpha,\gamma)$.
For $\beta$ we also used linearly-spaced points
$\beta_1,\ldots,\beta_{40}$, with weights $w_1,\ldots,w_{40}$ computed by numerically solving the equations
\begin{equation}
	\sum_{i=1}^{40} w_i D^{(l)}_{0,0}\left( (0,\beta_i,0) \right)  = \begin{cases}
	1, & \text{if } l=0 \\
	0, & \text{otherwise}
	\end{cases}
\end{equation}
for $0 \leq l < 40$. Here $D^{(l)}_{0,0}\left( (0,\beta_i,0) \right)$ is the
middle element of the Wigner-D matrix $D^{(l)}(\frakg)$, evaluated at $\frakg
\in \SO(3)$ having Euler angles $(0,\beta_i,0)$.
With proper normalization, this yields an approximate quadrature for
band-limited functions over $\SO(3)$. 

We note that 
computing $\nabla^2 R_n(\theta_*)$ for many samples is quite computationally
intensive, scaling as $O(nd^2 \cdot |\text{quad}|)$ where $|\text{quad}|$ is the
number of $\SO(3)$ quadrature points, and we were consequently
limited in the size of this quadrature. Since this quadrature may provide an
imperfect approximation to integration over true Haar measure on $\SO(3)$,
to mitigate some of the discretization effects, we generated samples
$y_1,\ldots,y_n$ using random rotations also drawn from the weighted discrete
distribution over $\SO(3)$ defined by this quadrature, rather than from the Haar
measure. This does not fully address the numerical inaccuracy, but at
least ensures that the true rotation for each sample $y_i$ belongs
to our discretization of $\SO(3)$. This becomes numerically important for
smaller values of the noise variance $\sigma^2$, when the posterior
distribution of the rotation $\frakg \in \SO(3)$ for each sample becomes more
localized over $\SO(3)$.

We translated each quadrature point of $\SO(3)$ to a rotational element in
$\O(d)$ via the representation described in Lemma
\ref{lem:so3-sphere-act}. We used the implementation of Wigner-D matrices
provided by the third-party \texttt{quaternion} and
\texttt{spherical\_functions} Python libraries. These elements of $\O(d)$ were
applied to $\theta_*$ (with additive Gaussian noise) to generate the samples
$y_1,\ldots,y_n$, and the corresponding quadrature on $\O(d)$ was
used also to compute the integrals over $\SO(3)$ in the analytic expression
for $\nabla^2 R_n(\theta_*)$.\\

\noindent {\bf Visualizations.}
The molecular graphics in Figures \ref{fig:rotavirus} and \ref{fig:hemoglobin} 
were rendered using the UCSF ChimeraX software \cite{pettersen2021ucsf}. We used
atomic structures publicly available on PDB (PDB:1QHD and PDB:5NI1). We aligned
the rotavirus VP6 structure PDB:1QHD to the 8.2\AA-resolution map
depicted in Figure \ref{fig:lowpass}(a) using the ChimeraX ``Fit in Map'' tool.
The hemoglobin structure PDB:5NI1 is already pre-aligned to the analyzed map.

\section{Cryogenic Electron Microscopy (cryo-EM)}\label{sec:cryoem}

Cryo-EM is a technology for  determining the spatial structure of macromolecules. 
In recent years, cryo-EM has become increasingly popular in structural biology.
Thanks to technological advancement in hardware and algorithms in the last
decade, cryo-EM now allows scientists to routinely recover structures at near-atomic resolutions. 
Unlike popular X-ray crystallography techniques for structure determination,
cryo-EM does not require the samples to be crystallized. This gives cryo-EM an
advantage, in particular, for molecules that are difficult to crystallize, and in heterogeneous samples. 

In a typical cryo-EM study, a solution with the molecule of interest is
flash-frozen in a thin layer of ice. The particles are sufficiently sparse and
the layer sufficiently thin so that when viewed from above, the molecules rarely overlap. Each particle is trapped in the ice at an unknown random orientation. The sample is then inserted into the microscope; an electron beam is transmitted through the sample and then recorded by a camera. A detailed description of the procedure can be found, inter alia, in \cite{BookGlaeser}.

The procedure produces a 2D tomographic projection of each particle. For
simplicity, in this work we have omitted some effects such as the filters
applied in the process (i.e.\ the contrast transfer function) and the problem of
centering the particles, which is less crucial at lower resolutions.  
In addition, we assume the experimental distribution of viewing directions is
uniform.
A more comprehensive description can be found in  \cite{BookGlaeser},  and a more detailed mathematical description can be found in \cite{bendory2020single}.
The simplified imaging model is summarized by the equation
	    \begin{equation} \label{eq:cryoemmodel:space}
I(r_1, r_2) = \int_{\mathbb{R}} f(\frakg^{-1} {\bm {r} } ) \der r_3 =
\int_{\mathbb{R}} \left(\frakg \cdot f\right)\left( (r_1,r_2,r_3) \right) \der r_3,
	    \end{equation}
where ${\bm r}=(r_1,r_2,r_3)$, the function $f:\R^3 \to \R$ is the electric
potential, and $\frakg \in \SO(3)$ is a rotation. We have expressed this
equivalently in the main text as (\ref{eq:cryoEMmodel}).
In other words, the image is obtained by integrating the $z$ axis of the volume rotated to viewing direction $\frakg$ (which is not known to us).
The interactions of the electrons with the sample lead to rapid deterioration in
the quality of the sample, and very few electrons can be used to record the
images before the sample becomes unusable. Therefore, the measurements are
characterized by low signal-to-noise ratios.

Taking the Fourier transform of Eq.\ (\ref{eq:cryoemmodel:space}) with appropriate normalization yields
	    \begin{equation} \label{eq:cryoemmodel:fourier}
		    \hat{I}(\omega_1, \omega_2) =  \left(\frakg \cdot  \hat{f}\right)\left( (\omega_1,\omega_2,0) \right) ,
	    \end{equation}
where $\hat{I}$ is the Fourier transform of the image and $\hat{f}$ is the
Fourier transform of the density map. In other words, in the Fourier domain, the
tomographic projection can be conveniently described as taking a slice of the
Fourier transform of the volume, at the plane that passes through the origin and
is perpendicular to the viewing direction. We have expressed this
relation in our analysis as (\ref{eq:fourierslice}), which is characterized in
our specific function bases by the projection operator $\Pi^\C$ in (\ref{eq:cryoEMPi}).

\subsection{Cryo-Electron Tomography (Cryo-ET, ``Unprojected Cryo-EM'')}\label{sec:cryoET}

Cryo-ET is based on the same technology as cryo-EM. However, in cryo-ET several images are taken of each particle, with the sample tilted in a different direction for every image. 
As in the classic cryo-EM problem above, the relative rotation angles of the
different particles are unknown. However, the relative tilt angles of images of
{\em the same particle} are known.
By the Fourier-slice relation (\ref{eq:cryoemmodel:fourier}), each image
is a slice of the Fourier domain, and thus a dense set of slices from different
viewing directions of {\em the same particle} can be used to reconstruct an
entire 3D volume.
Unfortunately, due to physical limitations, the tilt angles cannot cover all
viewing directions,
and a series of cryo-ET tilt images is typically noisier than a cryo-EM image. 
Thus cryo-ET provides a method of obtaining noisy 3D maps of individual
particles, whose relative rotations across different particles
are unknown as in the cryo-EM problem.
Cryo-ET is more commonly used to study larger samples (e.g.\ entire cells), but
is also used in the study of smaller particles. For additional information, see
\cite{eisenstein2019improved,turk2020promise}. In the main text,
we have referred to this problem of reconstructing a map from \emph{unprojected}
and rotated 3D volumes also as the ``unprojected cryo-EM'' model.

A simplified model of cryo-ET, after the tilt series has been reassembled to a 3D function, has the form:
	    \begin{equation} \label{eq:cryoETmodel:spatial}
		    f_\frakg (\omega) = \left( \frakg \cdot  {f} \right)(\omega) = f(  \frakg^{-1} \cdot \omega),
	    \end{equation}
	    or, in the Fourier domain:
	    \begin{equation} \label{eq:cryoETmodel:fourier}
		     \hat{f}_\frakg (\omega) = \left( \frakg \cdot  \hat{f} \right)(\omega) = \hat{f}(  \frakg^{-1} \cdot \omega).
	    \end{equation}
	    The (projected) cryo-EM model, is related to the cryo-ET model (``unprojected cryo-EM'') through the tomographic projection operator $\Pi$; 
	    in the Fourier domain, $\Pi$ has the form
	    \begin{equation} \label{eq:cryoEMmodel:Pi}
		     \left(\Pi \hat{f}\right) (\omega_1, \omega_2) = \hat{f} (\omega_1, \omega_2, 0).
	    \end{equation}
	    
\bibliographystyle{alpha}
\bibliography{continuous_orbit}

\end{document}